\documentclass[a4paper, 10pt]{article}
\usepackage{amssymb,amsmath,amsthm,dsfont}
\usepackage{graphicx,subfigure}
\usepackage{multirow, color}
\usepackage[left = 2.5cm, right = 2.5cm, top = 2.5cm, bottom = 2.5cm]{geometry}

\newtheorem{corollary}{Corollary}
\newtheorem{prop}{Proposition}
\newtheorem{thm}{Theorem}
\newtheorem{lemma}{Lemma}
\theoremstyle{definition}
\newtheorem{defn}{Definition}
\newtheorem{remark}{Remark}
\newtheorem{ex}{Example}

\newcommand{\black}{\color{black}}
\newcommand{\blue}{\color{blue}}

\begin{document}

  \title{ Theoretical Results on  Fractionally Integrated Exponential
    Generalized Autoregressive Conditional Heteroskedastic  Processes}

\author{S\'ilvia R.C. Lopes \ and Taiane   S. Prass\footnote{Corresponding author. E-mail:
   taianeprass@gmail.com}\vspace{0.2cm}\\
  Mathematics Institute - UFRGS\vspace{0.1cm}\\
  Porto Alegre - RS - Brazil}

\maketitle

  \begin{abstract}
        Here we present a theoretical study on the main properties of
        Fractionally Integrated Exponential Generalized Autoregressive Conditional
        Heteroskedastic (FIEGARCH)  processes.  We analyze the conditions for
        the existence, the invertibility, the stationarity and the ergodicity of
        these processes.  We prove that, if $\{X_t\}_{t \in \mathds{Z}}$ is a
        FIEGARCH$(p,d,q)$ process then,   under mild conditions,
        $\{\ln(X_t^2)\}_{t\in\mathds{Z}}$ is an ARFIMA$(q,d,0)$,  that is, an
        autoregressive fractionally integrated moving average process.   The
        convergence order for the polynomial coefficients that describes the
        volatility is presented and results related to the spectral representation
        and to the covariance structure of both processes
        $\{\ln(X_t^2)\}_{t\in\mathds{Z}}$ and $\{\ln(\sigma_t^2)\}_{t\in\mathds{Z}}$
        are also discussed.  Expressions for the kurtosis and the asymmetry measures
        for any stationary FIEGARCH$(p,d,q)$ process are also derived.  The $h$-step
        ahead forecast for the processes $\{X_t\}_{t \in \mathds{Z}}$,
        $\{\ln(\sigma_t^2)\}_{t\in\mathds{Z}}$ and $\{\ln(X_t^2)\}_{t\in\mathds{Z}}$
        are given with their respective mean square error forecast.  The work also
        presents a Monte Carlo simulation study showing how to generate, estimate
        and forecast based on six different FIEGARCH models. The forecasting
        performance of  six models belonging to the class of autoregressive
        conditional heteroskedastic models (namely, ARCH-type models)  and radial basis models
        is compared through an empirical application to Brazilian  stock market exchange index.    \\



  \end{abstract}

\section{Introduction}

Financial time series present an important characteristic known as volatility
which can be defined/mea\-sured in different ways but it is not directly
observable.  A common approach, but not unique, is to define the volatility as
the conditional standard deviation (or the conditional variance) of the process
and use heteroskedastic models to describe it.

ARCH-type models, proposed by \cite{EN82}, constitute one of the main classes of
econometric models used for representing the dynamic evolution of volatilities.
Another popular one is the class of Stochastic Volatility (SV) models (see,
\cite{BREA98} and references therein).   In both cases, ARCH-type and SV
models, the stochastic process $\{X_t\}_{t\in\mathds{Z}}$ can be written as
\[
X_t = \sigma_tZ_t, \quad \mbox{for all } t \in \mathds{Z},
\]
where $\{Z_t\}_{t\in\mathds{Z}}$ is a sequence of independent identically
distributed (i.i.d.) random variables, with zero mean and variance equal to one,
and $\sigma_t := \mbox{Var}(X_t|\mathcal{F}_{t-1})$, where $\mathcal{F}_{t-1}$
denotes the sigma field generated by the past informations until time $t-1$.  An
important difference between these two classes is that, for ARCH-type models,
$\mathcal{F}_{t} := \sigma(\{X_s\}_{s\leq t})$ or $\mathcal{F}_{t} :=
\sigma(\{Z_s\}_{s\leq t})$, while for SV models $\mathcal{F}_{t} :=
\sigma(\{Z_s, \eta_s\}_{s\leq t})$, where $\{\eta_t\}_{t\in\mathds{Z}}$ is a
sequence of latent variables, independent of $\{Z_t\}_{t\in\mathds{Z}}$.
Therefore,  the volatility of a SV process is specified as a latent
variable which is not directly observable and this can make the estimation
challenging, which is a known drawback of this class of models.

By ARCH-type models we mean not only the ARCH$(p)$ model,  proposed by \cite{EN82},  where
\[
\sigma_t^2 = \alpha_0 + \sum_{i = 1}^p\alpha_i X_{t-k}^2, \quad \mbox{for all }
t\in \mathds{Z},
\]
 (which characterizes the volatility as a function of powers of past
observed values, consequently, the volatility can be observed one-step ahead),
but also the several generalizations that were lately proposed to properly model
the dynamics of the volatility.  Among the generalizations of the ARCH model are
the  Generalized ARCH (GARCH) processes, proposed by \cite{ENBO86}, and the
Exponential GARCH (EGARCH) processes, proposed by \cite{NE91}.  These models are given,
respectively, by \eqref{figarch} and \eqref{fiegarch} below by setting $d = 0$.
The usual definition of $\sigma_t^2$ for a GARCH($p^*,q$) model, namely,
\[
\sigma_t^2 = \omega + \sum_{i = 1}^{p^*} \alpha_i X_{t-i}^2 + \sum_{j =
  1}^q\beta_{j}\sigma_{t-j}^2, \quad \mbox{for all} \quad t \in\mathds{Z},
\]
is obtained from \eqref{figarch} by letting $p^* := \max\{p,q\}$ and  $\alpha(z)
= \sum_{i = 1}^{p^*}\alpha_iz^i := \beta(z) - \phi(z)$,  where $\beta(z) :=  1 -
\sum_{j=1}^q \beta_jz^j$ and $\phi(z) :=  1 -\sum_{k=1}^p\phi_kz^k$.

ARCH, GARCH and EGARCH are all short memory models.  Among the generalizations
that capture the effects of long-memory characteristic in the conditional
variance are the Fractionally Integrated GARCH (FIGARCH), proposed by
\cite{BAEA96}, and the Fractionally Integrated EGARCH (FIEGARCH), introduced by
\cite{BOMI96}.   For a FIGARCH$(p,d,q)$, $\sigma_t^2$ is given by
\begin{equation}\label{figarch}
 \bigg[1 -  \sum_{j=1}^q \beta_j\mathcal{B}^k\bigg]\sigma_{t}^2 =  \omega + \bigg(1 -  \sum_{j=1}^q \beta_j\mathcal{B}^k - \bigg[1 -  \sum_{k=1}^p\phi_k\mathcal{B}^k\bigg](1-\mathcal{B})^{d}\bigg)X_t^2,  \quad \mbox{for all } t\in \mathds{Z},
\end{equation}
while for a FIEGARCH$(p,d,q)$, $\sigma_t^2$ is defined through the relation,
\begin{align}
  \ln(\sigma_t^2)  &= \omega+\frac{1 - \sum_{i=1}^p\alpha_i\mathcal{B}^i}{1 - \sum_{j=1}^q\beta_j\mathcal{B}^j}(1-\mathcal{B})^{-d}\big(\theta Z_{t-1} + \gamma [|Z_{t-1}| - \mathds{E}(|Z_{t-1}|)]\big)\nonumber\\
  &:=
  \omega+\frac{\alpha(\mathcal{B})}{\beta(\mathcal{B})}(1-\mathcal{B})^{-d}g(Z_{t-1}),
  \quad \mbox{for all } t\in \mathds{Z},\label{fiegarch}
\end{align}
where $\mathcal{B}$ is the backward shift
operator defined by $\mathcal{B}^k (X_t) = X_{t-k}$, for all $k\in\mathds{N}$,
and  $(1-\mathcal{B})^d$ is the operator defined by its Maclaurin series
expansion as,
\begin{equation*}
  (1-\mathcal{B})^d = \sum_{k=0}^{\infty}\frac{\Gamma(k-d)}{\Gamma(k+1)\Gamma(-d)} :=
  \sum_{k=0}^{\infty}\delta_{d,k}\,\mathcal{B}^k := \delta_d(\mathcal{B}),
\end{equation*}
with $\Gamma(\cdot)$  the gamma function.

FIEGARCH models have not only the capability of modeling clusters of volatility
(as in the ARCH and GARCH models) and capturing its asymmetry\footnote{By
  asymmetry we mean that the volatility reacts in an asymmetrical form to the
  returns, that is, volatility tends to rise in response to ``bad'' news and to
  fall in response to ``good'' news.}  (as in the EGARCH models) but they also
take into account the characteristic of long memory in the volatility (as in the
FIGARCH models, with the advantage of been weakly stationary if $d < 0.5$).
Besides non-stationarity (in the weak sense), another drawback of the
FIGARCH$(p,d,q)$ models is that we must have $d \geq 0$ and the polynomial
coefficients in its definition must satisfy some restrictions so the conditional
variance will be positive.  FIEGARCH$(p,d,q)$ models do not have this problem
since the variance is defined in terms of the logarithm function.

Some authors argue that the long memory behavior observed in the sample
autocorrelation and periodogram functions of financial time series could
actually be caused by the non-stationarity property.  According to
\cite{MIST99}, long range behavior could be just an artifact due to structural
changes. On the other hand, \cite{MIST99} also argue that, when modeling return
series with large sample size, considering a single GARCH model is unfeasible
and that the best alternative would be to update the parameter values along the
time.  As an alternative to the traditional heteroskedastic models, \cite{MI00}
presents a regime switching model that, combined with heavy tailed
distributions, presents the long memory characteristic.

It is our belief that FIEGARCH models are a competitive alternative for modeling
large sample sized data, especially because they avoid parameter updating. Also,
as we prove in this work, FIEGARCH processes are weakly stationary if and only
if $d<0.5$ and hence, non-stationarity can be easily identified. Moreover,
\cite{SAEA06} analyze the daily returns of the Tunisian stock market and rule
out the random walk hypothesis.  According to the authors, the rejection of this
hypothesis seems to be due to substantial non-linear dependence and not to
non-stationarity in the return series and, after comparing several ARCH-type
models they concluded that a stationary FIEGARCH model provides the best fit for
the data.  Furthermore, \cite{JA09} presents a sub period investigation of long
memory and structural changes in volatility. The authors consider FIEGARCH
models to examine the long run persistence of stock return volatility for 23
developing markets for the period of January 2000 to October 2007.  No clear
evidence that long memory characteristic could be attributed to structural
changes in volatility was found.

 Although, in practice, often a simple FIEGARCH$(p,d,q)$ model with $p, q
\in \{0,1\}$ suffices to fully describe financial time series (for instance,
\cite{JA09} and \cite{RUVE08}, consider FIEGARCH$(0,d,1)$ models and
\cite{SAEA06} considers FIEGARCH$(1,d,1)$ models), there are evidences that for
some financial time series higher values of $p$ and $q$ are in fact necessary
(\cite{PRLO1},\cite{PRLO2},\cite{ PR08}).  In this work we present a theoretical
study on the main properties of FIEGARCH$(p,d,q)$ processes, for any $p,q \geq
0$.

One of the contributions of the paper is to extend, for any $p$ and $q$, the
results already known in the literature for $p, q \in\{0,1\}$ or $d=0$. In
particular,   we provide the expressions for the asymmetry and
kurtosis measures of FIEGARCH$(p,d,q)$ process, for all $p,q \geq 0$. These
results extends the one in \cite{RUVE08} where only the case $p=0$ and $q=1$
was considered and only the kurtosis measure was derived.

Another contribution of this work is the ARFIMA representation of
$\{\ln(X_t^2)\}_{t\in\mathds{Z}}$, when $\{X_t\}_{t\in\mathds{Z}}$ is a FIEGARCH
process, which is derived in the paper.  This results is
very useful in model identification and parameter estimation since the
literature of ARFIMA models is well developed (see \cite{LO08} and references
therein) and, to the best of our knowledge, this result is absent in the
literature.

To derive the properties of $\{\ln(X_t^2)\}_{t\in\mathds{Z}}$,  we first
investigate the conditions for the existence of power series representation for
$\lambda(z) = \alpha(z)[\beta(z)]^{-1}(1-z)^{-d}$ and the behavior of the
coefficients in this representation.  This study is fundamental not only for
simulation purposes but also to draw conclusions on the autocorrelation and
spectral density functions decay of the non-observable process
$\{\ln(\sigma_t^2)\}_{t\in\mathds{Z}}$ and the observable one
$\{\ln(X_t^2)\}_{t\in\mathds{Z}}$.  We also provide a recurrence formula to
calculate the coefficients of the series expansion of $\lambda(\cdot)$,  for any $p,q \geq 0$. This
recurrence formula allows to easily simulate FIEGARCH processes.

The fact that  $\{\ln(\sigma_t^2)\}_{t\in\mathds{Z}}$ is
an ARFIMA$(q,d,p)$ process and the result that any FIEGARCH process is a
martingale difference with respect to the natural filtration
$\{\mathcal{F}_{t}\}_{t\in\mathds{Z}}$, where $\mathcal{F}_{t} := \sigma(\{Z_s\}_{s\leq t})$,
are applied to obtain the $h$-step ahead forecast
for the processes $\{X_t\}_{t \in \mathds{Z}}$ and $\{X^2_t\}_{t \in
  \mathds{Z}}$.  We also present the $h$-step ahead forecast for both
$\{\ln(\sigma_t^2)\}_{t\in\mathds{Z}}$ and $\{\ln(X_t^2)\}_{t\in\mathds{Z}}$
processes, with their respective mean square error forecast.  To the best of our
knowledge, formal proofs for these expressions are not given in the literature
of FIEGARCH$(p,d,q)$ processes.

We also present a simulation study including generation, estimation and
forecasting features of FIEGARCH models.  Despite the fact that the
quasi-likelihood is one of the most applied methods in non-linear process
estimation, asymptotic results for FIEGARCH processes are still an open question
(see \cite{PA07})\footnote{The asymptotic properties for the quasi-likelihood
  method are well established for ARCH/GARCH models (see, for instance,
  \cite{LEHA94}, \cite{LU96}, \cite{BEEA03}, \cite{BEHO03} and \cite{HAYA03})
  and also for EGARCH models (see, for instance, \cite{STMI06}).}. Therefore, we
consider here a simulation study to investigate the finite sample
performance of the estimator. Since it is expected that, the better the fit, the better the
forecasting, we also investigate the fitted models' forecasting performance.

The paper is organized as follows: Section \ref{fiesection} presents the formal
definition of FIEGARCH process and its theoretical properties.  We give a
recurrence formula to obtain the coefficients in the power series expansion of
the polynomial that describes the volatility and we show their asymptotic
properties.  The autocovariance and spectral density functions of the processes
$\{\ln(\sigma_t^2)\}_{t\in\mathds{Z}}$ and $\{\ln(X_t^2)\}_{t\in\mathds{Z}}$ are
also presented and analyzed. The asymmetry and kurtosis measures of any
stationary FIEGARCH process are also presented. Section \ref{forecastingsection}
presents the theoretical results regarding the forecasting.  Section
\ref{simulationsection} presents a Monte Carlo simulation study including the
generation of FIEGARCH time series, estimation of the model parameters and the
forecasting based on the fitted model.  Section \ref{analysisObserved} presents
the analysis of an observed time series and the comparison of the forecasting
performance for different ARCH-type and radial basis models.  Section \ref{conclusions} concludes
the paper.

\section{FIEGARCH Process}\label{fiesection}

In this section we present the \emph{Fractionally Integrated Exponential
  Generalized Autoregressive Conditional Heteroskedastic} process
(FIEGARCH). This class of processes, introduced by \cite{BOMI96}, describes not
only the volatility varying on time and the volatility clusters (known as
ARCH/GARCH effects) but also the volatility long-range dependence and its
asymmetry.

Here, we present some results related to the existence, stationarity and
ergodicity for these processes. We analyze the autocorrelation and the spectral
density functions decay for both $\{\ln(\sigma_t^2)\}_{t\in\mathds{Z}}$ and
$\{\ln(X_t^2)\}_{t\in\mathds{Z}}$ processes.  Conditions for the existence of a
series expansion for the polynomial that describes the volatility are given and
a recurrence formula to calculate the coefficients of this expansion is
presented. We also discuss the coefficients asymptotic behavior. We observe
that  if $\{X_t\}_{t \in \mathds{Z}}$ is a FIEGARCH$(p,d,q)$ process then
$\{\ln(\sigma_t^2)\}_{t\in\mathds{Z}}$ is an ARFIMA$(q,d,p)$ process and we
prove that, under mild conditions, $\{\ln(X_t^2)\}_{t\in\mathds{Z}}$ is an
ARFIMA$(q,d,0)$ process with correlated innovations. We present the expression
for the kurtosis and the asymmetry measures for any stationary FIEGARCH$(p,d,q)$
process.

Throughout the paper, given $a\in \mathds{R} \cup \{-\infty,+\infty\}$, $f(x) =
O(g(x))$ means that $|f(x)|\le c|g(x)|$, for some $c>0$, as $x \rightarrow a$;
$f(x) = o(g(x))$ means that $f(x)/g(x) \rightarrow 0$, as $x \rightarrow a$;
$f(x) \sim g(x)$ means that $ f(x)/g(x) \rightarrow 1$, as $x \rightarrow a$.
We also say that $f(x) \approx g(x)$, as $x \rightarrow \infty$, if for any
$\varepsilon > 0$, there exists $x_0 \in \mathds{R}$ such that $|f(x) - g(x)| <
\varepsilon$, for all $x \geq x_0$. Also, given any set $T$, $T^*$ corresponds
to the set $T\backslash\{0\}$ and $\mathbb{I}_{A}(\cdot)$ is the indicator
function defined as $\mathbb{I}_{A}(z) = 1$, if $z\in A$, and 0, otherwise.

From now on, let $(1-\mathcal{B})^d$ be the operator defined by its Maclaurin
series expansion as,
\begin{equation}\label{binomialexp}
  (1-\mathcal{B})^d = \sum_{k=0}^{\infty}\frac{\Gamma(k-d)}{\Gamma(k+1)\Gamma(-d)} :=
  \sum_{k=0}^{\infty}\delta_{d,k}\,\mathcal{B}^k := \delta_d(\mathcal{B}),
\end{equation}
where $\Gamma(\cdot)$ is the gamma function, $\mathcal{B}$ is the backward shift
operator defined by $\mathcal{B}^k (X_t) = X_{t-k}$, for all $k\in\mathds{N}$,
and the coefficients $\delta_{d,k}$ are such that $\delta_{d,0} = 1$ and
$\delta_{d,k-1} = \delta_{d,k-1}\big( \frac{k-1-d}{k}\big)$, for all $k\geq
1$.

\begin{remark}
  Note that expression \eqref{binomialexp} is valid only for non-integer values
  of $d$. When $d\in\mathds{N}$, $(1-\mathcal{B})^d$ is merely the difference
  operator $1-\mathcal{B}$ iterated $d$ times.  Also, one observe that, upon
  replacing $d$ by $-d$, the operator $(1-\mathcal{B})^{-d}$ has the same
  binomial expansion as the polynomial given in \eqref{binomialexp}, that is

  \begin{equation}\label{pik}
    (1-\mathcal{B})^{-d} = \sum_{j=0}^{\infty} \delta_{-d,j}\mathcal{B}^{j} := \sum_{k=0}^{\infty}
    \pi_{d,k}\mathcal{B}^{k},
  \end{equation}
  where $\pi_{d,j} = \delta_{-d,j}$, for all $j\in\mathds{N}$. Moreover,
  $\pi_{d,k} \sim \frac{1}{\Gamma(d)\, k^{1-d}}$, as $k\rightarrow \infty$ (see
  \cite{PR08}).  Therefore, $\pi_{d,k} = O(k^{d-1})$, as $k$ goes to infinity.
\end{remark}

Suppose that $\{Z_t\}_{t \in \mathds{Z}}$ is a sequence of independent and
identically distributed (i.i.d.) random variables, with zero mean and variance
equal to one.  Let $\alpha(\cdot)$ and $\beta(\cdot)$ be the polynomials of
order $p$ and $q$ defined, respectively, by

\begin{equation}
  \alpha(z) = \sum_{i=0}^{p}(-\alpha_i)z^i = 1-\sum_{i=1}^{p}\alpha_iz^i \quad \mbox{ and}
  \quad \beta(z) = \sum_{j=0}^{q}(-\beta_j)z^j = 1-\sum_{j=1}^{q}\beta_jz^j,\label{alphabeta}
\end{equation}
with $\alpha_0 = \beta_0 = -1$.  We assume that $\beta(z)\neq 0$, if $|z|\leq
1$, and that $\alpha(\cdot)$ and $\beta(\cdot)$ have no common roots. These
conditions assure that the operator
$\frac{\alpha(\mathcal{B})}{\beta(\mathcal{B})}$ is well defined.

\begin{defn}\label{definitionfie}
  Let $\{X_t\}_{t \in \mathds{Z}}$ be the stochastic process defined as
  \begin{align} X_t& \, = \, \sigma_tZ_t,\label{generalproc}\\
    \ln(\sigma_t^2)&\, = \,
    \omega+\frac{\alpha(\mathcal{B})}{\beta(\mathcal{B})}(1-\mathcal{B})^{-d}g(Z_{t-1}),
    \quad \mbox{ for all } t\in\mathds{Z},\label{fieprocess}
  \end{align}
  where $\omega \in \mathds{R}$ and $g(\cdot)$ is defined by
  \begin{equation}\label{functiong}
    g(Z_{t})=\theta Z_t +
    \gamma\left[|Z_t|-\mathds{E}(|Z_t|)\right],\quad \mbox{ for all }
    t\in\mathds{Z}, \ \mbox{ with } \theta, \gamma \in \mathds{R}.
  \end{equation}
  Then  $\{X_t\}_{t \in \mathds{Z}}$ is a \emph{Fractionally Integrated} EGARCH
  \emph{process}, denoted by FIEGARCH$(p,d,q)$.
\end{defn}

\begin{figure}[!ht]
    \centering
    \mbox{
      \subfigure[ $\{x_t\}_{t =1}^{2000}$]
      {\includegraphics[width = 0.32\textwidth]{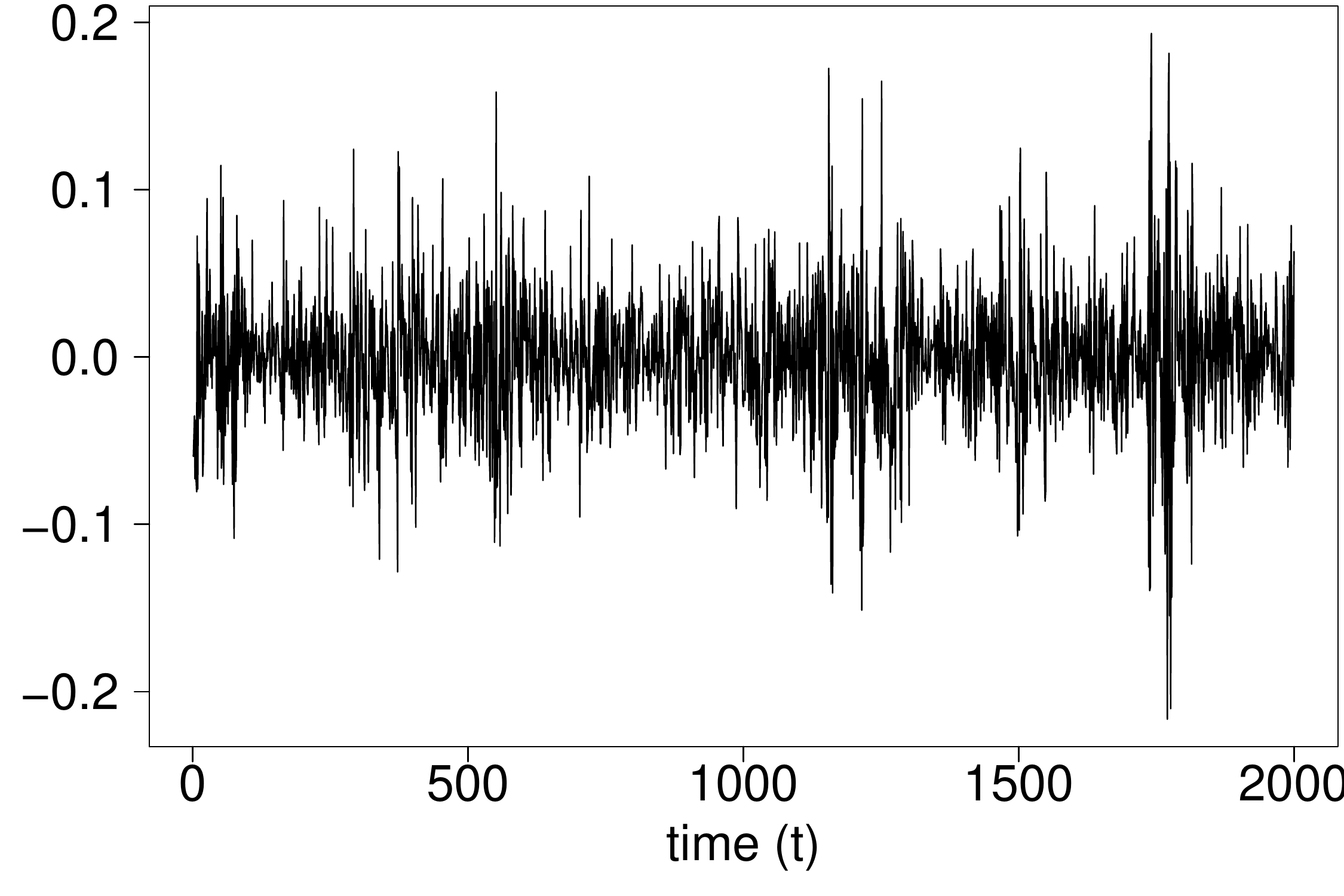}}
      \subfigure[$\{\sigma_t^2\}_{t=1}^{2000}$]
      {\includegraphics[width = 0.32\textwidth]{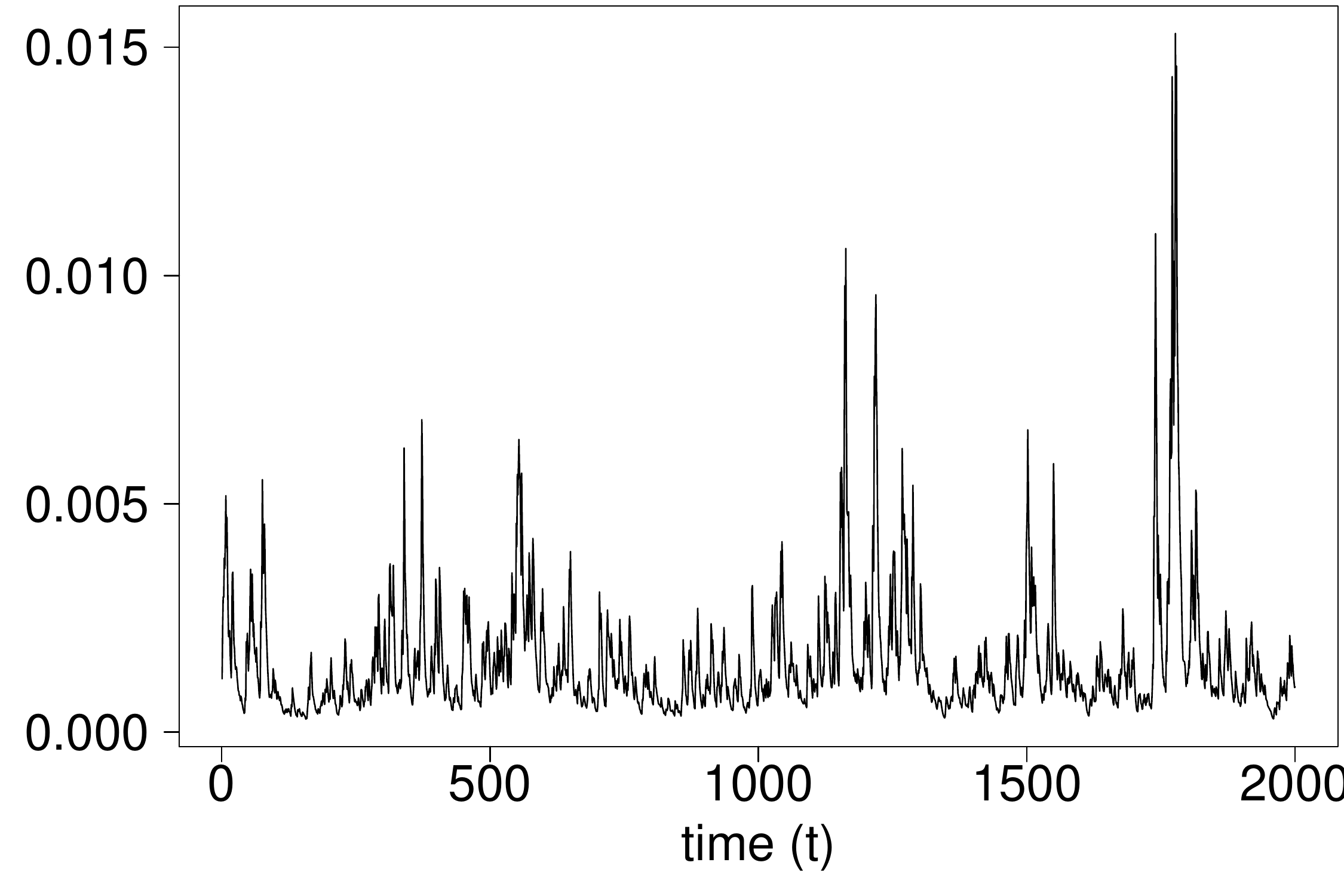}}
      \subfigure[$\{\ln(\sigma_t^2)\}_{t=1}^{2000}$]
      {\includegraphics[width = 0.32\textwidth]{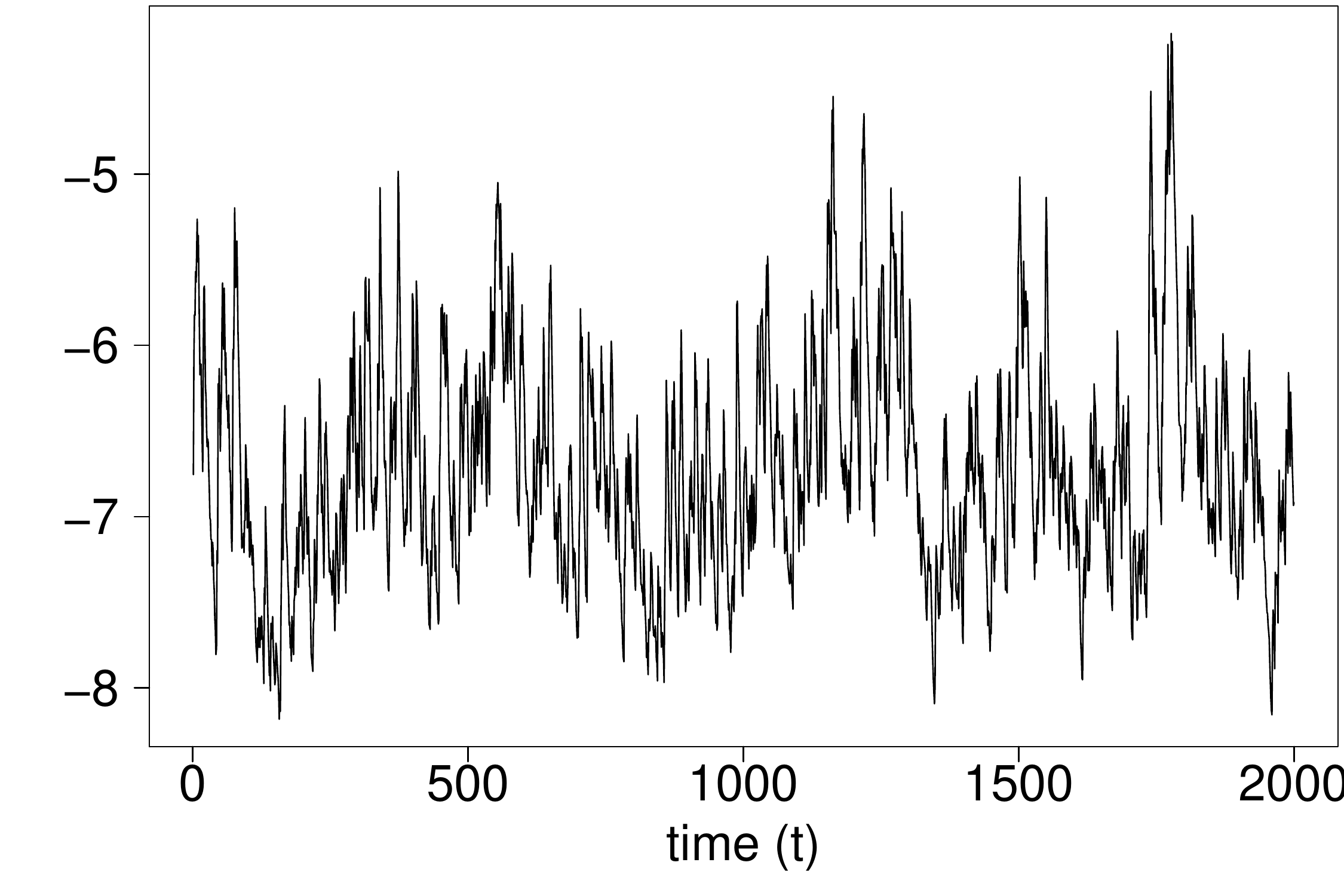}} }

\mbox{
      \subfigure[ $\{x_t\}_{t
        =1}^{2000}$]{\includegraphics[width = 0.32\textwidth]{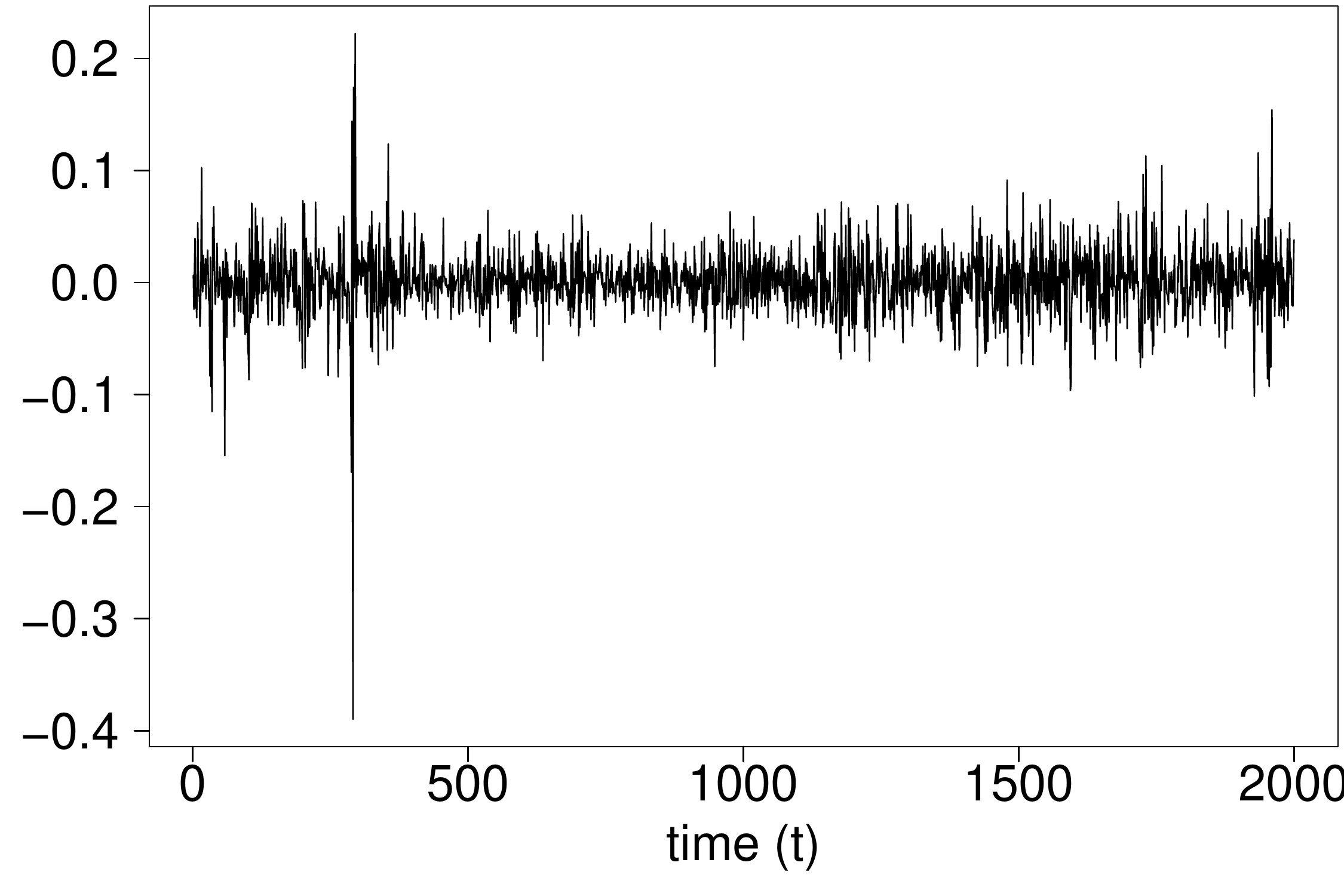}}
      \subfigure[$\{\sigma_t^2\}_{t=1}^{2000}$
      ]{\includegraphics[width = 0.32\textwidth]{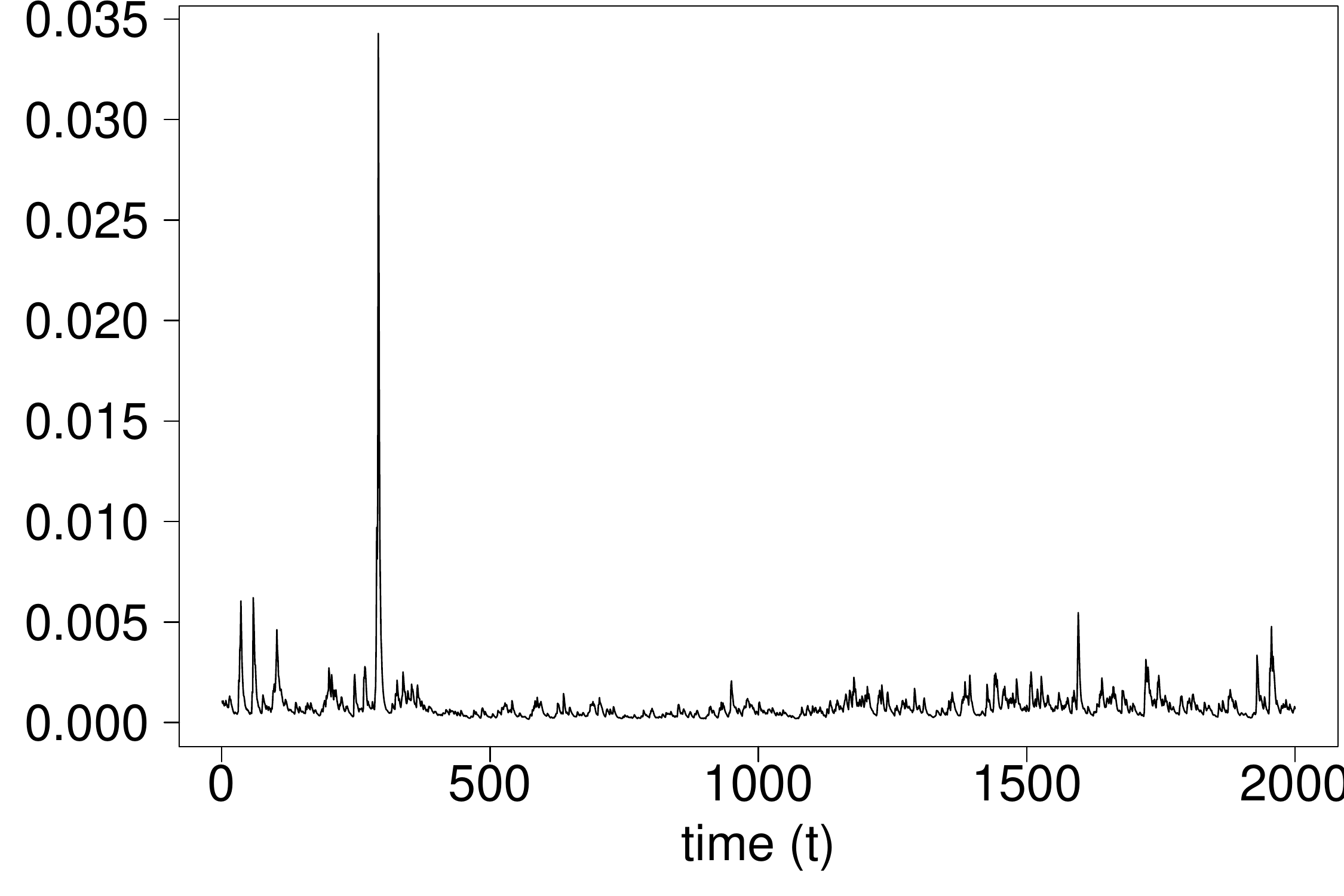}}
      \subfigure[$\{\ln(\sigma_t^2)\}_{t=1}^{2000}$]
      {\includegraphics[width =  0.32\textwidth]{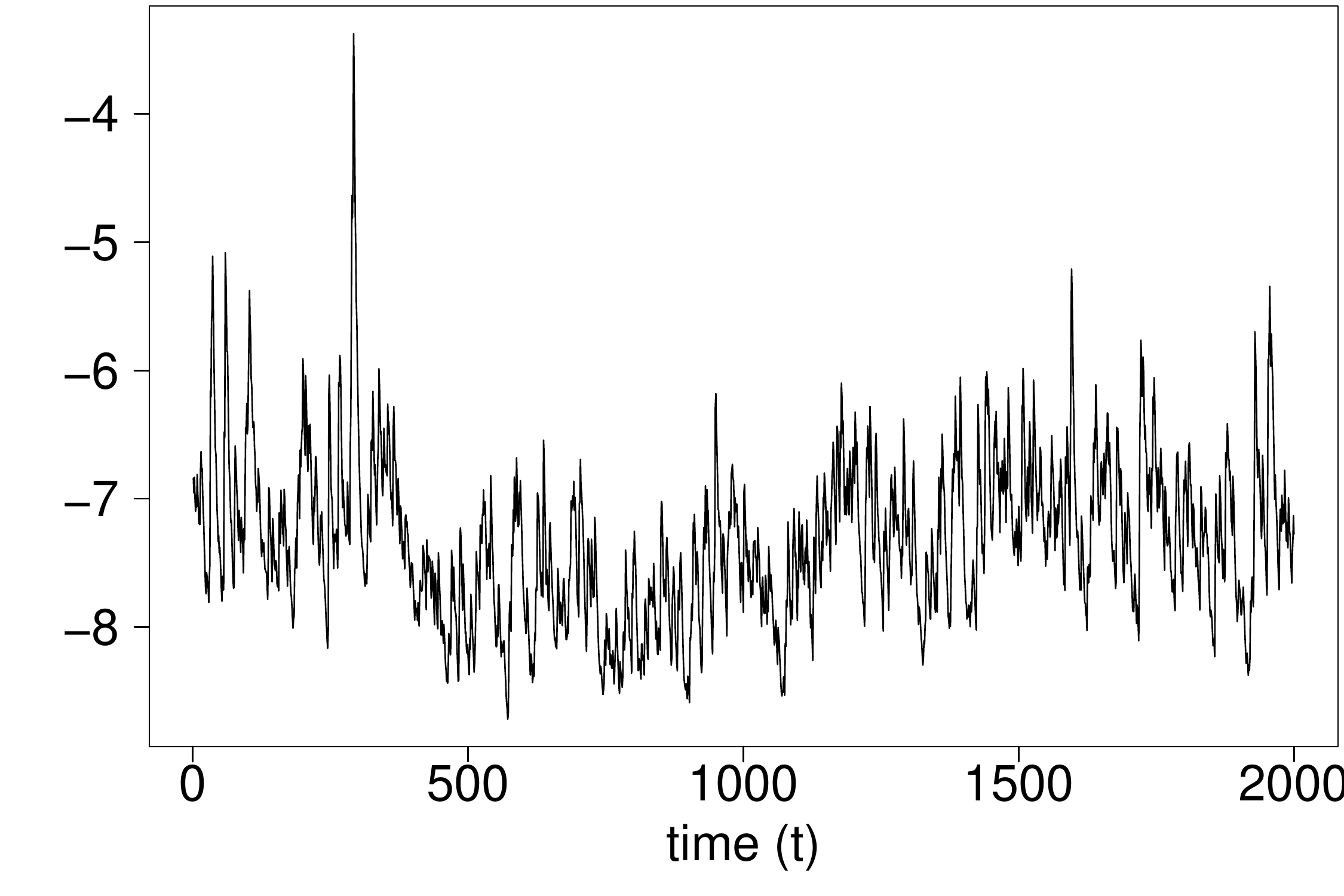}} }

    \caption{ Samples from  FIEGARCH$(0,d,1)$ processes,  with $n = 2,000$
      observations, considering $Z_0 \sim \mathcal{N}(0,1)$ (first row)
      and  $Z_0 \sim \mbox{GED}(1.5)$ (second row).    (a) and (d) show the
      time series $\{x_t\}_{t=1}^{n}$;  (b) and (e) show the conditional variance of
      $\{x_t\}_{t=1}^{n}$;   (c) and (f)  show the logarithm of the conditional variance.}\label{fiegsim}
  \end{figure}

  \begin{ex}
     Figure \ref{fiegsim} presents samples from  FIEGARCH$(0,d,1)$ processes,
    with $n = 2,000$ observations, considering  two  different underlying
    distributions.    To obtain these samples we consider Definition
    \ref{definitionfie} and two different distributions for $Z_0$.    For
this simulation we set   $d = 0.3578$, $\theta = -0.1661$, $\gamma = 0.2792$,
$\omega = -7.2247$ and   $\beta_1 = 0.6860$.  These are the parameter values of
the FIEGARCH model   fitted to the Bovespa index log-returns in Section
\ref{analysisObserved}.   Figures \ref{fiegsim} (a) - (c) consider $Z_0\sim
    \mathcal{N}(0,1)$ and show, respectively,  the  time series
    $\{x_t\}_{t=1}^{n}$,   the conditional variance
    $\{\sigma_t^2\}_{t=1}^n$ and  the logarithm of
    the conditional variance $\{\ln(\sigma_t^2)\}_{t=1}^n$.
Figures \ref{fiegsim} (d) - (e) show the same time series as in Figures
\ref{fiegsim} (a) - (c) when the distribution for  $Z_0$ is the Generalized
Error Distribution (GED), with tail-thickness parameter $\nu = 1.5$.
 \end{ex}

\begin{remark}
  Note that, in Definition \ref{definitionfie}, no conditions on the parameter
  $d$ are imposed.  Necessary and sufficient conditions on the parameter $d$, to
  guarantee the existence of the stochastic process
  $\{\ln(\sigma_t^2)\}_{t\in\mathds{Z}}$, satisfying \eqref{fieprocess}, are
  discussed in the sequel. Also notice that when $d=0$, we obtain the well known
  EGARCH process.
\end{remark}

 For practical purpose, it is important to observe that slightly different
definitions of FIEGARCH processes are found in the literature. Usually it is
easy to show that, under certain conditions, the different definitions are
equivalent to \eqref{fiegarch}.  For instance, \cite{ZIWA05} defines the
conditional variance of a FIEGARCH process through the equation
\[
\bigg[1-\sum_{j=1}^q \beta_j \mathcal{B}^j\bigg](1-\mathcal{B})^d\ln(\sigma_t^2) =
a + \sum_{i=0}^{p}[\psi_i|Z_{t-1-i}|+\gamma_i Z_{t-1-i}].
\]
This is the definition considered, for instance, in the software S-Plus (see
\cite{ZIWA05}) and it is equivalent to \eqref{fiegarch} whenever $d>0$, $a =
-\gamma\mathds{E}(|Z_{t}|)[1-\sum_{i=1}^p \alpha_i]$, $\psi_0 =
\gamma$, $\gamma_0 =\theta$, $\psi_i = -\gamma\alpha_{i}$ and $\gamma_i =
-\theta\alpha_{i}$ for all $1\leq i \leq p$.  This equivalence is mentioned in
\cite{PA07} and a detailed proof is provided in \cite{PR08}.  In \cite{RUVE08}
only the case $p=0$ and $q = 1$ is considered and
$\{\ln(\sigma_t^2)\}_{t\in\mathds{Z}}$ is defined as
\begin{equation}\label{ruiz}
  (1-\phi\mathcal{B})(1-\mathcal{B})^d\ln(\sigma_t^2)  =
  \omega^* +  \alpha\big[|Z_{t-1}| - \sqrt{2/\pi}\big] + \gamma^* Z_{t-1}, \quad \mbox{for all } t\in\mathds{Z},
\end{equation}
where $\{Z_t\}_{t\in\mathds{Z}}$ is a Gaussian white noise process with variance
equal to one.  This is
the definition considered, for instance, in the G@RCH package version 4.0 of
\cite{LP05}.  Notice that, by setting $\phi = \beta_1$, $\alpha = \gamma$ and
$\gamma^* = \theta$, \eqref{ruiz} is equivalent to \eqref{fieprocess} if and
only if the equality $\omega^* = (1-\beta)(1-\mathcal{B})^d\omega$ holds.

\begin{remark}
    We observe that  the theory presented here can  be easily adapted to
  a more general framework than \eqref{fieprocess} (which  uses the same notation as in
  \cite{BOMI96}) by considering
  \begin{equation}
    \ln(\sigma_t^2) = \omega_t + \sum_{k=0}^{\infty}\lambda_{k}g(Z_{t-1-k}):= \omega_t + \lambda(\mathcal{B})g(Z_{t-1}), \quad \mbox{ for all $t\in\mathds{Z}$},\label{fiegarchNelson}
  \end{equation}
  where  $\{\omega_t\}_{t\in \mathds{Z}}$ and  $\{\lambda_{k}\}_{k\in
    \mathds{N}}$ are real, nonstochastic, scalar sequences for which the process
  $\{\ln(\sigma_t^2)\}_{t\in\mathds{Z}}$ is well defined, $\{Z_t\}_{t\in
    \mathds{Z}}$ is white noise process with variance not necessarily equal to
  one and $g(\cdot)$ is any measurable function.   In particular, Theorems
  \ref{N1} and \ref{N2} below, which are stated and proved in \cite{NE91},
  assume that $\ln(\sigma_t^2)$ is given by \eqref{fiegarchNelson} (the notation
  was adapted to reflect the one used in this work), with
  $\{Z_t\}_{t\in\mathds{Z}}$ and $g(\cdot)$ as in Definition
  \ref{definitionfie}.  Although \eqref{fiegarchNelson} is more general than
  \eqref{fieprocess}, in practice the applicability of the model is somewhat
  limited given that the parameter estimation is far more complicated when
  compared to the model \eqref{fieprocess}.
\end{remark}

Notice that the function $g(\cdot)$ can be rewritten as
\begin{equation}
  g(Z_{t})=\left\{\begin{array}{cc}
      (\theta + \gamma)Z_t - \gamma\mathds{E}(|Z_t|), & Z_t\geq 0; \vspace{.2cm}\\
      (\theta - \gamma)Z_t - \gamma\mathds{E}(|Z_t|) ,& Z_t<0. \\
    \end{array} \right.\nonumber
\end{equation}
This expression clearly shows the asymmetry in response to positive and negative
returns.  Also, it is easy to see that $g(\cdot)$ is non-linear if $\theta \neq
0$ and the asymmetry is due to the values of $\theta \pm \gamma$.  While the
parameter $\theta$, also known in the literature as \emph{leverage parameter},
shows the return's sign effect, the parameter $\gamma$ denotes the return's
magnitude effect.  Therefore, the model is able to capture the fact that a
negative return usually results in higher volatility than a positive one.
Proposition \ref{Prop1} below  presents  the properties of the stochastic process
$\{g(Z_t)\}_{t\in\mathds{Z}}$.  Although the proof  is
straightforward,  these properties are extremely important to prove the
results stated in the sequel.

\begin{prop}\label{Prop1}
  Let $\{Z_t\}_{t \in \mathds{Z}}$ be a sequence of i.i.d.  random variables,
  with $\mathds{E}(|Z_0|) < \infty$.  Let $\{g(Z_t)\}_{t\in\mathds{Z}}$ be defined by
  \eqref{functiong} and assume that $\theta$ and $\gamma$ are not both equal to
  zero.  Then $\{g(Z_t)\}_{t\in\mathds{Z}}$ is a strictly stationary and ergodic
  process.  If $\mathds{E}(Z_0^2) <\infty$, then $\{g(Z_t)\}_{t\in\mathds{Z}}$
  is also weakly stationary with zero mean (therefore a white noise process) and
  variance $\sigma^2_g$ given by
  \begin{equation}\label{eq:sigmag}
    \sigma^2_g = \theta^2 + \gamma^2
    -[\gamma\mathds{E}(|Z_0|)]^2 + 2\, \theta \, \gamma \, \mathds{E}(Z_0|Z_0|).
  \end{equation}
\end{prop}
\begin{proof}
   See \cite{PR08}.
\end{proof}

 Theorem \ref{N1} below provides a criterion for stationarity and ergodicity of
EGARCH (FIEGARCH) processes.  As pointed out by \cite{NE91}, the stationarity
and ergodicity criterion in Theorem \ref{N1} is exactly the same as for a
general linear process with finite variance innovations.  Obviously, different
definitions of $\lambda(\cdot)$ in \eqref{fiegarchNelson} will lead to different
conditions for the criterion in Theorem \ref{N1} to hold.  In  \cite{NE91} it is
stated  that, in many applications, an ARMA process provides a parsimonious
parametrization for $\{\ln(\sigma_t^2)\}_{t\in \mathds{Z}}$.  In this case,
$\lambda(\cdot)$ is defined as $\lambda(z) = \alpha(z)[\beta(z)]^{-1}$, $|z|
\leq 1$, where $\alpha(\cdot)$ and $\beta(\cdot)$ are the polynomials given in
\eqref{alphabeta}, leading to an EGARCH$(p,q)$ process. For this model, the
criterion in Theorem \ref{N1} holds whenever the roots of $\beta(\cdot)$ are
outside the closed disk $\{z : |z| \leq 1\}$. We shall later discuss  the
condition for the criterion in Theorem \ref{N1} to hold when
$\{\ln(\sigma_t^2)\}_{t\in \mathds{Z}}$ is defined by \eqref{fieprocess},
leading to a FIEGARCH$(p,d,q)$ process.

\begin{thm}\label{N1}
  Define $\{\sigma_t^2\}_{t\in\mathds{Z}}$, $\{X_t\}_{t\in\mathds{Z}}$ and
  $\{Z_t\}_{t\in\mathds{Z}}$ by
  \begin{align}
    X_t &= \sigma_tZ_t; \qquad   Z_t  \sim i.i.d., \quad  \mathds{E}(Z_t) = 0 \quad
    \mbox{and} \quad  \mathrm{Var}(Z_t) = 1;\label{Neeq1}\\
    \ln(\sigma_t^2) &= \omega_t + \sum_{k=1}^{\infty}\lambda_{k}g(Z_{t-k}),
    \quad \lambda_1 = 1; \qquad g(Z_t) = \theta Z_t + \gamma\left[|Z_t| -
    \mathds{E}(|Z_t|)\right];\label{Neeq2}
  \end{align}
  where $\{\omega_t\}_{t\in \mathds{Z}}$ and $\{\lambda_{k}\}_{k\in
    \mathds{N}^*}$ are real, nonstochastic, scalar sequences, and assume that
  $\theta$ and $\gamma$ do not both equal zero. Then
  $\{e^{-\omega_t}\sigma_t^2\}_{t\in\mathds{Z}}$,
  $\{e^{-\omega_t/2}X_t\}_{t\in\mathds{Z}}$ and $\{\ln(\sigma_t^2) -
  \omega_t\}_{t\in\mathds{Z}}$ are strictly stationary and ergodic and
  $\{\ln(\sigma_t^2) - \omega_t\}_{t\in\mathds{Z}}$ is covariance stationary if
  and only if $\sum_{k=1}^\infty \lambda_k^2 < \infty$.  If $\sum_{k=1}^\infty
  \lambda_k^2 = \infty$, then $|\ln(\sigma_t^2) - \omega_t| = \infty$ almost
  surely. If $\sum_{k=1}^\infty \lambda_k^2 < \infty$, then for $k>0$,
  $\mathrm{Cov}\big(Z_{t-k}, \ln(\sigma_t^2)\big) = \lambda_k[\theta + \gamma
  \mathds{E}(Z_t|Z_t|)]$, and $\mathrm{Cov}\big(\ln(\sigma_t^2),
  \ln(\sigma_{t-k}^2)\big) = \mathrm{Var}\big(g(Z_t)\big)\sum_{j =
    1}^\infty\lambda_j\lambda_{j+k}$.
\end{thm}
\begin{proof}
   See  theorem 2.1 in \cite{NE91}.
\end{proof}

 Theorem \ref{N2} shows the existence of the $r$th moment  for the
random variables $X_t$ and $\sigma_t^2$, defined by \eqref{Neeq1}-\eqref{Neeq2},
when $\sum_{k=1}^\infty \lambda_k ^2 < \infty$ and the distribution of $Z_0$ is
the Generalized Error Distribution (GED).

\begin{thm}\label{N2}
  Define $\{\sigma_t^2, X_t\}_{t\in\mathds{Z}}$ by \eqref{Neeq1}-\eqref{Neeq2},
  and assume that $\theta$ and $\gamma$ do not both equal zero.  Let
  $\{Z_t\}_{t\in\mathds{Z}}$ be i.i.d. GED with mean zero, variance one, and
  tail-thickness parameter $\nu > 1$, and let $\sum_{k=1}^\infty \lambda_k ^2 <
  \infty$. Then $\{e^{-\omega_t}\sigma_t^2\}_{t\in\mathds{Z}}$ and
  $\{e^{-\omega_t/2}X_t\}_{t\in\mathds{Z}}$ possess finite, time-invariant
  moments of arbitrary order. Further, if $0 <r < \infty$, conditioning
  information at time 0 drops out of the forecast $r$th moments of
  $e^{-\omega_t}\sigma_t^2$ and $e^{-\omega_t/2}X_t$, as $t\to \infty$:
  \begin{align*}
    \underset{t\to\infty}{\mathrm{plim}} &\big[ \mathds{E}(e^{-r\omega_t}\sigma_t^{2r}\big|Z_0,Z_{-1},Z_{-2}, \cdots) - \mathds{E}(e^{-r\omega_t}\sigma_t^{2r})\big] = 0 \quad \mbox{and}\\
    \underset{t\to\infty}{\mathrm{plim}} &\big[ \mathds{E}(e^{-r\omega_t/2}|X_t|^{r}\big|Z_0,Z_{-1},Z_{-2}, \cdots) - \mathds{E}(e^{-r\omega_t/2}|X_t|^{r})\big] = 0,\\
  \end{align*}
  where $\mathrm{plim}$ denotes the limit in probability.
\end{thm}
\begin{proof}
   See  theorem 2.2  in \cite{NE91}.
\end{proof}

From now on, let $\lambda(\cdot)$ be the polynomial defined by
\begin{equation}\label{lambdapoly}
  \lambda(z)=\frac{\alpha(z)}{\beta(z)}(1-z)^{-d}
  :=\sum_{k=0}^{\infty}\lambda_{d,k}z^k, \quad \mbox{ for all } |z| < 1,
\end{equation}
where $\alpha(\cdot)$ and $\beta(\cdot)$ are defined in \eqref{alphabeta}. Since
it is assumed that $\beta(\cdot)$ has no roots in the closed disk $\{z:|z|\leq
1\}$, and also $\alpha(\cdot)$ and $\beta(\cdot)$ have no common roots, the
function $\lambda(z)$ is analytic in the open disc $\{z:|z|<1\}$ ( if $d \leq
0$, in the closed disk $\{z:|z|\leq 1\}$).  Therefore, it has a unique power
series representation and  \eqref{fieprocess} can be rewritten,
equivalently, as
\begin{equation}\label{bol}
  \ln(\sigma_t^2) = \omega +\sum_{k=0}^\infty \lambda_{d,k}g(Z_{t-1-k}) = \omega + \lambda(\mathcal{B})g(Z_{t-1}) , \quad \mbox{for all } \, t\in\mathds{Z}.
\end{equation}
Notice that, with this definition we obtain a particular case of parametrization
\eqref{fiegarchNelson}.

Theorem \ref{convorder} below gives the convergence order of the coefficients
$\lambda_{d,k}$, as $k$ goes to infinity.  This  theorem is important for
two reasons.  First, it provides an approximation for $\lambda_{d,k}$, as $k\to
\infty$,  and this  result plays an important role when choosing the
truncation point in the series representation for simulation purposes.
Second, and most important, the asymptotic representation provided in this
theorem plays the key role to establish the necessary condition for square summability
of $\{\lambda_{d,k}\}_{k\in\mathds{N}}$.  More specifically,  from Theorem
\ref{convorder} one concludes that $\{\lambda_{d,k}\}_{k\in\mathds{N}} \in
\ell^2$ if and only if $d < 0.5$ and $\{\lambda_{d,k}\}_{k\in\mathds{N}} \in
\ell^1$ whenever $d < 0$.

\begin{thm}\label{convorder}
  Let $\lambda(\cdot)$ be the polynomial defined by \eqref{lambdapoly}. Then,
  for all $k \in \mathds{N}$, the coefficients $\lambda_{d,k}$ satisfy \small
  \begin{equation}\label{conv}
    \lambda_{d,k} \sim  \frac{1}{\Gamma(d)\,k^{1-d}}\frac{\alpha(1)}{\beta(1)},
    \quad \mbox{ as  $k\rightarrow \infty$}.
  \end{equation}
  Consequently, $\lambda_{d,k} = O(k^{d-1})$, as $k$ goes to infinity.
\end{thm}

\begin{proof}
  Denote $\beta(z)^{-1}$ by $f(z)$. Since $\beta(\cdot)$ has no roots in the
  closed disk $\{z:|z|\leq1\}$, one has
  \begin{equation}\label{betapoly}
    \beta(z)^{-1}:= f(z)=\sum_{k=0}^{\infty}f_kz^k,\quad \mbox{ where } f_k =
    \dfrac{f^{(k)}(0)}{k!}, \, \mbox{ for all } k \in \mathds{N}.
  \end{equation}
  From expressions \eqref{pik}, \eqref{lambdapoly} and \eqref{betapoly} it
  follows that
  \begin{equation}\label{new}
    \lambda(z) =\sum_{k=0}^{\infty}\bigg[\sum_{i=0}^{\min\{p,k\}}\!\!\!(-\alpha_{i})
    \sum_{j=0}^{k-i}\pi_{d,k-i-j}f_j\bigg]z^k.
  \end{equation}
  From \eqref{new}, one has
  \[
  \lambda_{d,k} =
  \sum_{i=0}^{\min\{p,k\}}\!\!\!(-\alpha_{i})\sum_{j=0}^{k-i}\pi_{d,k-i-j}f_j,
  \quad \mbox{for all } k\in \mathds{N}.
  \]
  In particular, $\lambda_{d,k}=\displaystyle\sum_{i=0}^{p}(-\alpha_{i})
  \sum_{j=0}^{k-i}\pi_{d,j}f_{k-i-j}, \ \mbox{ for all }k>p$.

  Moreover, since $f_k\rightarrow 0$, as $k \rightarrow \infty$, it follows that
  for all $\varepsilon>0$, there exists $k_0>0$, such that, for a given $m>0$
  and for all $k>k_0$, $|\pi_{d,j}f_{k-i-j}|<\frac{\varepsilon}{m},$ for all
  $0\leq j\leq m$ and $0\leq i\leq p$.  Hence, for $k$ sufficiently large,
  \[
  \lambda_{d,k} \sim\sum_{i=0}^{p}
  (-\alpha_{i})\!\!\sum_{j=m+1}^{k-i}\pi_{d,j}f_{k-i-j}.
  \]
  Notice that, since $\pi_{d,k} \sim \frac{1}{\Gamma(d)\, k^{1-d}}$, as $k
  \rightarrow \infty$, one can choose $m_0$ such that $\pi_{d,k}\sim\pi_{d,k-i}
  \sim \pi_{d,j}$, for all $m_0< m+1\leq j \leq k-i$ and $0\leq i \leq
  p$. Consequently,
  \[
  \lambda_{d,k} \sim
  \pi_{d,k}\sum_{i=0}^{p}(-\alpha_{i})\!\!\!\!\!\!\sum_{j=0}^{k-i-(m+1)}\!\!\!\!\!\!f_{j}
  \sim\pi_{d,k}\sum_{i=0}^{p}(-\alpha_{i})\sum_{j=0}^{\infty}f_{j}.
  \]
  However, $\sum_{j=0}^{\infty}f_{j} = f(1) = \dfrac{1}{\beta(1)}$ and
  $\pi_{d,k} \sim \frac{1}{\Gamma(d)\, k^{1-d}}$, as $k\to \infty$. So, we have
  \begin{equation*}
    \lambda_{d,k} \sim \pi_{d,k}\,\frac{\alpha(1)}{\beta(1)} \sim \frac{1}{\Gamma(d)\,
      k^{1-d}}\frac{\alpha(1)}{\beta(1)}.
  \end{equation*}
  It follows that $\lambda_{d,k}\to 0$ and $\lambda_{d,k}k^{1-d}\to
  \frac{1}{\Gamma(d)\, }\frac{\alpha(1)}{\beta(1)}$, as $k \to \infty$. Hence,
  $\lambda_{d,k} = O(k^{d-1})$, as $k \to\infty$, which concludes the proof.
\end{proof}

Proposition \ref{coefficientsfiegarch} presents a recurrence formula for
calculating the coefficients $\lambda_{d,k}$, for all $k\in\mathds{N}$. This
formula is used to generate the FIEGARCH time series in the simulation study
presented in Section \ref{simulationsection}.

\begin{prop}\label{coefficientsfiegarch}
  Let $\lambda(\cdot)$ be the polynomial defined by \eqref{lambdapoly}. The
  coefficients $\lambda_{d,k}$, for all $k \in \mathds{N}$, are given by
  \begin{equation}\label{coefflambda}
    \lambda_{d,0} = 1 \quad \mbox{ and} \quad \lambda_{d,k}= -\alpha_k^* +
    \sum_{i=0}^{k-1}\lambda_i \sum_{j=0}^{k-i} \beta_j^*\delta_{d,k-i-j},\ \mbox{ for all }\ k\geq
    1,
  \end{equation}
  where the coefficients $\delta_{d,k}$, for all $k\in\mathds{N}$, are given in
  \eqref{binomialexp} and
  \begin{equation}\label{def} \alpha_m^* := \left\{
      \begin{array}{ccc} \alpha_m, & \mbox{ if} &0
        \leq m \leq p;\vspace{.3cm}\\
        0, & \mbox{ if} & m > p
      \end{array} \right. \quad \mbox{ and}\quad \beta_m^* := \left\{
      \begin{array}{ccc} \beta_m, & \mbox{ if} & 0\leq
        m \leq q;\vspace{.3cm}\\
        0, & \mbox{ if} & m > q.
      \end{array} \right.
  \end{equation}
\end{prop}

\begin{proof}
  Let $\lambda(\cdot)$ be defined by \eqref{lambdapoly}. Consequently,
  \begin{equation}\label{equality}
    \alpha(z)=\beta(z)(1-z)^d\sum_{k=0}^{\infty}\lambda_{d,k}z^k.
  \end{equation}
  By defining $\beta_k^*$ as in expression \eqref{def}, for all
  $k\in\mathds{N}$, and upon considering expression \eqref{binomialexp},
  observing that $\delta_{d,0} = -1 = \beta_0$, the right hand side of
  expression \eqref{equality} can be rewritten as
  \begin{align}\label{formula2}
    \beta(z)(1-z)^d\sum_{k=0}^{\infty}\lambda_{d,k}z^k
    &=\bigg[\sum_{k=0}^{\infty}\bigg(\sum_{j=0}^{k}-\beta_{j}^*\delta_{d,k-j}\bigg)z^k\bigg]\sum_{k=0}^{\infty}\lambda_{d,k}z^k
    =
    \sum_{k=0}^{\infty}\bigg[\sum_{i=0}^{k}\lambda_{d,i}\bigg(-\sum_{j=0}^{k-i}\beta_{j}^*\delta_{d,k-i-j}\bigg)\bigg]z^k\nonumber\\
    &=\sum_{k=0}^{\infty}\bigg[\lambda_{d,k}-\sum_{i=0}^{k-1}\lambda_{d,i}\sum_{j=0}^{k-i}\beta_{j}^*\delta_{d,k-i-j}\bigg]z^k.
  \end{align}

  Now, by setting $\alpha_k^*$ as in expression \eqref{def}, for all
  $k\in\mathds{N}$, from expression \eqref{formula2} one concludes that the
  equality \eqref{equality} holds if and only if,
  \[
  -\alpha_0^* = \lambda_{d,0} \quad \mbox{ and} \quad -\alpha_k^* =
  \lambda_{d,k}-\sum_{i=0}^{k-1}\lambda_{d,i}\sum_{j=0}^{k-i}
  \delta_{d,k-i-j}\beta_{j}^*, \quad \mbox{ for all } k\geq 1.
  \]
  Therefore, expression \eqref{coefflambda} holds.  It is easy to see that by
  replacing the coefficients $\lambda_{d,k}$, given by \eqref{coefflambda}, in
  the expression \eqref{formula2}, for all $k\in \mathds{N}$, we get
  $\sum_{k=0}^{\infty}(-\alpha_k^*)z^k=\alpha(z)$, which completes the proof.
\end{proof}

 The applicability of Theorem \ref{N1} to long memory models was briefly
mentioned (without going into details) in \cite{NE91}.  Corollary
\ref{strictstatfiegarch} below is a direct application of Theorem
\ref{convorder} and provides a simple condition for the criterion in Theorem
\ref{N1} to hold when $\{\ln(\sigma_t^2)\}_{t\in \mathds{Z}}$ is defined by
\eqref{fieprocess}, which leads to a long memory model whenever $d>0$.

\begin{corollary}\label{strictstatfiegarch}
   Let $\{X_t\}_{t \in \mathds{Z}}$ be a \emph{FIEGARCH}$(p,d,q)$ process,
  given in Definition \ref{definitionfie}.  If $d<0.5$,
  $\{\ln(\sigma_t^2)\}_{t\in\mathds{Z}}$ is stationary (weakly and strictly),
  ergodic and the random variable $\ln(\sigma_t^2)$ is almost surely finite, for
  all $t \in \mathds{Z}$.  Moreover, $\{X_t\}_{t \in \mathds{Z}}$ and
  $\{\sigma_t^2\}_{t\in\mathds{Z}}$ are strictly stationary and ergodic
  processes.
\end{corollary}

\begin{proof}
  Let $\{\lambda_{d,k}\}_{k\in\mathds{N}}$ be defined by \eqref{lambdapoly} and
  rewrite \eqref{fieprocess} as \eqref{bol}.  Observe that, by Theorem
  \ref{convorder}, the condition $d<0.5$ implies that $\sum_{k=0}^\infty
  \lambda_{d,k}^2 < \infty$.  Therefore, the results follow from Theorem
  \ref{N1} by taking $\omega_t := \omega$, for all $t\in\mathds{Z}$, and
  $\lambda_k := \lambda_{d,k-1}$, for all $k \geq 1$.
\end{proof}

 The square summability of $\{\lambda_{d,k}\}_{k\in\mathds{N}}$ implies that
the process $\{\ln(\sigma_t^2)\}_{t\in\mathds{Z}}$ is stationary (weakly and
strictly), ergodic and the random variable $\ln(\sigma_t^2)$ is almost surely
finite, for all $t \in \mathds{Z}$ (see Theorem \ref{N1}).  Now,  since
$\{g(Z_t)\}_{t\in\mathds{Z}}$ is a white noise (Proposition \ref{Prop1}), it
follows immediately that $\{\ln(\sigma_t^2)\}_{t\in\mathds{Z}}$ is an
ARFIMA$(q,d,p)$ process (for details on ARFIMA processes see, for instance,
\cite{BRDA91}, \cite{LO08}).  This result is very useful, not only for
forecasting purposes (see Section \ref{forecastingsection}) but also, to
conclude the following properties

\begin{itemize}
\item[{\bf P1:}] if $d <0.5$, the autocorrelation function of the process
  $\{\ln(\sigma_t^2)\}_{t\in\mathds{Z}}$ is such that
  \[
  \rho_{\ln(\sigma^2_t)}(h)\sim ch^{2d-1}, \quad \mbox{as } \ h \to \infty,
  \]
  where $c\neq 0$, and the spectral density function of the process
  $\{\ln(\sigma_t^2)\}_{t\in\mathds{Z}}$ is such that
  \[
  f_{\ln(\sigma^2_t)}(\lambda)=\frac{\sigma_g^2}{2\pi}
  \frac{|\alpha(e^{-i\lambda})|^2}{|\beta(e^{-i\lambda})|^2}|1-e^{-i\lambda}|^{-2d
  }\sim \frac{\sigma_g^2}{2\pi}\left[
    \frac{\alpha(1)}{\beta(1)}\right]^2\!\lambda^{-2d}, \quad \mbox{as} \
  \lambda \to 0,
  \]
  where $\sigma_g^2 = \mathrm{Var}\big(g(Z_t)\big)$ is given in \eqref{eq:sigmag};

\item[{\bf P2:}] if $d \in (-1, 0.5)$ and $\alpha(z)\neq 0$, for $|z|\leq 1$,
  the process $\{\ln(\sigma_t^2)\}_{t\in\mathds{Z}}$ is invertible,  that
  is,
  \[
  \lim_{m\rightarrow \infty} \mathds{E}\bigg(\bigg|\sum_{k=0}^m
  \tilde\lambda_{d,k}\big[\ln(\sigma_{t-k}^2) - \omega -
  g(Z_{t-1})\big]\bigg|^{r}\bigg) = 0, \quad \mbox{for all} \quad 0 <r \leq 2,
  \]
  where $\displaystyle\sum_{k=0}^{\infty}\tilde\lambda_{d,k}z^k = \tilde \lambda(z) :=
  \lambda^{-1}(z) = \dfrac{\beta(z)}{\alpha(z)}(1-z)^d, \quad |z|<1.$
\end{itemize}

\begin{remark}
   The proof of {\bf P1} can be found in \cite{BRDA91}, theorem 13.2.2.
  Regarding {\bf P2},  in the literature one usually find that an
  ARFIMA$(p,d,q)$ process is invertible for $|d| < 0.5$ (see, for instance,
  \cite{BRDA91}, theorem 13.2.2).  However, \cite{BL85} already proved that this
  range can be extended to $d\in(-1, 0.5)$, for an ARFIMA$(0,d,0)$ and, more
  recently, \cite{BOPA07} show that this result actually holds for any
  ARFIMA$(p,d,q)$.
\end{remark}

Corollary \ref{strictstatfiegarch} shows that $\{X_t\}_{t\in\mathds{Z}}$ and
$\{\sigma^2_t\}_{t\in\mathds{Z}}$ are strictly stationary and ergodic processes.
However, as mentioned in \cite{NE91}, this does not imply weakly stationarity
when the random variable $Z_t$, for $t \in \mathds{Z}$, is such that either its
mean or its variance is not finite.  Theorem \ref{N2} considers the GED function
and proves the existence of the moment of order $r>0$, for the random variables
$X_t$ and $\sigma_t^2$, for all $t \in \mathds{Z}$, when the process
$\{\ln(\sigma_t^2)\}_{t\in\mathds{Z}}$ is defined in terms of a square summable
sequence of coefficients.   Corollary \ref{estacionar} below applies the
result of Theorem \ref{convorder} to state a simple condition so that Theorem
\ref{N2} holds for FIEGARCH$(p,d,q)$ processes.

\begin{corollary}\label{estacionar}
  Let $\{X_t\}_{t \in \mathds{Z}}$ be a \emph{FIEGARCH}$(p,d,q)$ process, given
  in Definition \ref{definitionfie}.  Assume that $\theta$ and $\gamma$ are not
  both equal to zero and that $\{Z_t\}_{t \in \mathds{Z}}$ is a sequence of
  i.i.d. \emph{GED}$(v)$ random variables, with $v>1$, zero mean and variance
  equal to one.  If $d<0.5$, then $\mathds{E}(X_t^r)<\infty$ and
  $\mathds{E}(\sigma_t^{2r})<\infty$, for all $t\in \mathds{Z}$ and $r>0$.
\end{corollary}

\begin{proof}
  Let $\{\lambda_{d,k}\}_{k\in\mathds{N}}$ be defined by \eqref{lambdapoly} and
  rewrite \eqref{fieprocess} as \eqref{bol}.  Define $\omega_t := \omega$, for
  all $t\in\mathds{Z}$, and $\lambda_k := \lambda_{d,k-1}$, for all $k\geq1$.
  Observe that, from Theorem \ref{convorder}, $d<0.5$ implies
  $\sum_{k=1}^{\infty}\lambda_{k}^2<\infty$.  Therefore,  the assumptions of
  Theorem \ref{N2} hold and the results follow.
\end{proof}

 As a consequence of Corollary \ref{estacionar}, if $d<0.5$ and $\{Z_t\}_{t
  \in \mathds{Z}}$ is a sequence of i.i.d. GED$(v)$ random variables, with
$v>1$, zero mean and variance equal to one, then $\mathds{E}(X_t^4) < \infty$
(consequently,  $\mathds{E}(X_t^3) < \infty$) and both, the asymmetry ($A_X$) and
kurtosis ($K_X$) measures of $\{X_t\}_{t\in \mathds{Z}}$ exist.  Now, since $
\mathds{E}(X_t^r) = \mathds{E}(\sigma_t^r)\mathds{E}(Z_t^r)$, for all $r>0$ (it
follows from the independence of $\sigma_t$ and $Z_t$), $\mathds{E}(X_t) = 0$
and $\mathds{E}(Z_t^2) = 1$, the measures $A_X$ and $K_X$ can be rewritten as
\begin{equation}\label{ex32}
  A_X :=
  \frac{\mathds{E}(X_t^3)}{\big[\mathds{E}(X_t^2)\big]^{3/2}}=\frac{\mathds{E}(\sigma_t^3)\mathds{E}(Z_t^3)}{\big[\mathds{E}(\sigma_t^2)\big]^{3/2}} \quad \mbox{and} \quad K_X :=
  \frac{\mathds{E}(X_t^4)}{\big[\mathds{E}(X_t^2)\big]^2}=\frac{\mathds{E}(\sigma_t^4)\mathds{E}(Z_t^4)}{\big[\mathds{E}(\sigma_t^2)\big]^2},
  \quad \mbox{for all } t\in\mathds{Z}.
\end{equation}

 An expression for $K_X$ (as a function of the FIEGARCH model parameters)
was already given in \cite{RUVE08} by assuming that $\{Z_t\}_{t\in\mathds{Z}}$
is a Gaussian white noise with variance equal to one, $d>0$, $p=0$, $q = 1$ and by
defining $\{\ln(\sigma_t^2)\}_{t\in\mathds{Z}}$ through expression \eqref{ruiz}.
According, with that definition, it can be shown that $K_X$ can be written as
\[
K_X=3\,
\frac{\prod_{j=1}^{\infty}\mathds{E}\big(\exp\{2\lambda_{j}g(Z_{t-j})\}\big)}{\Big[\prod_{j=1}^{\infty}\mathds{E}\big(\exp\{
  \lambda_{j}g(Z_{t-j})\}\big)\Big]^2}, \quad \mbox{with }
\left\{ \begin{array}{l}
    g(Z_{t}) = \theta Z_{t} +  \gamma[|Z_{t}| - \sqrt{2/\pi}], \quad t\in\mathds{Z}\\
    \lambda_j = \displaystyle\sum_{i=0}^{j-1}
    \frac{\Gamma(i+d)}{\Gamma(i+1)\Gamma(d)}\beta^{j-i-1}, \quad j \in
    \mathds{N}^* \mbox{ and } d > 0.
  \end{array}
\right.
\]
In Proposition \ref{kurtosisAsymmetry} bellow we consider stationary
FIEGARCH$(p,d,q)$ processes (therefore,  $d < 0.5$) with
$\{\ln(\sigma_t^2)\}_{t\in\mathds{Z}}$ defined by \eqref{fieprocess} and show
that a similar expression holds for any $p,q \geq 0$.  We do not impose that
$\{Z_t\}_{t\in\mathds{Z}}$ is a Gaussian white noise since Corollary
\ref{strictstatfiegarch} shows that the asymmetry and kurtosis measures exist
for a larger class of FIEGARCH models.

\begin{prop}\label{kurtosisAsymmetry}
  Let $\{X_t\}_{t \in \mathds{Z}}$ be a stationary \emph{FIEGARCH}$(p,d,q)$
  process, given in Definition \ref{definitionfie}.  If
  $\mathds{E}(X_0^3)<\infty$, the asymmetry measure of $\{X_t\}_{t \in
    \mathds{Z}}$ is given by
  \[
  A_X=\mathds{E}(Z_0^3)\, \frac{\prod_{k=0}^{\infty}\mathds{E}\Big(\!\exp\Big{\{}
    \frac{3}{2}\lambda_{d,k}g(Z_{0})\Big{\}}\Big)}{\big[\prod_{k=0}^{\infty}\mathds{E}\big(\!\exp\{
    \lambda_{d,k}g(Z_{0})\}\big)\big]^{3/2}}
  \]
  and, if $\mathds{E}(X_0^4)<\infty$, the kurtosis measure of $\{X_t\}_{t \in
    \mathds{Z}}$ is given by
  \[
  K_X=\mathds{E}(Z_0^4)\,
  \frac{\prod_{k=0}^{\infty}\mathds{E}(\exp\{
    2\lambda_{d,k}g(Z_{0})\})}{\big[\prod_{k=0}^{\infty}\mathds{E}\big(\!\exp\{
    \lambda_{d,k}g(Z_{0})\}\big)\big]^2},
  \]
  where $\lambda_{d,k}$ are given in \eqref{lambdapoly} and $g(\cdot)$ is
  defined by \eqref{functiong}.
\end{prop}

\begin{proof}
   Let $\{X_t\}_{t \in \mathds{Z}}$ be any stationary FIEGARCH$(p,d,q)$
  process and $\lambda(\cdot)$ be the polynomial defined by \eqref{lambdapoly}.
  Notice that, since $\{g(Z_t)\}_{t\in\mathds{Z}}$ is a sequence of
  i.i.d. random variables, from \eqref{fieprocess} it follows that
  \begin{equation}\label{ex3}
    \mathds{E}(\sigma_0^r) =
    e^{\frac{r\omega}{2}}\prod_{k=0}^{\infty}\mathds{E}\Big(\!\exp\Big\{\frac{r}{2}
    \lambda_{d,k}g(Z_{0})\Big\}\Big),  \quad \mbox{for all } r > 0.
  \end{equation}
  From the fact that $\sigma_t$ and $Z_t$ are independent random variables one has
  \begin{equation*}
    \mathds{E}(|X_t|^r) = \mathds{E}(|X_0|^r) =
    \mathds{E}(|Z_0|^r)  \mathds{E}(|\sigma_0|^r), \quad \mbox{for all } t\in\mathds{Z}
    \mbox{ and } r>0.
  \end{equation*}
  Thus, given $r >0$, $\mathds{E}(X_0^r)<\infty$ if and only if
  $\mathds{E}(\sigma_0^r)$ and $\mathds{E}(Z_0^r)$ are both finite. Therefore,
  if $\mathds{E}(X_0^3)<\infty$ (analogously, $\mathds{E}(X_0^4)<\infty$), the
  asymmetry (analogously, the kurtosis) measure exists, and expression
  \eqref{ex3} converges, for any $r\leq 3$ (analogously, $r\leq 4$).  Upon
  replacing \eqref{ex3} in \eqref{ex32} we conclude the proof.
\end{proof}

\begin{ex}
  Figure \ref{kurtosis} shows the theoretical value of the kurtosis measure, as
  a function of the parameter $d$, for any FIEGARCH$(0,d,1)$ process, with
  Gaussian noise and parameters $\theta = -0.1661$, $\gamma = 0.2792$, $\omega =
  -7.2247$ and (a) $\beta_1 = 0.6860$ (b) $\beta_1 = - 0.6860$.  The parameter
  values considered in Figure \ref{kurtosis} (a) are the same ones (except for
  $d$) considered in Figure \ref{fiegsim} (a).  For the specific model
  considered in Figure \ref{fiegsim}, $d = 0.3578$ and the theoretical value of
  the kurtosis measure is 5.6733. The sample kurtosis value of the simulated
  time series presented in Figure \ref{fiegsim} (a) is 5.3197, which is very
  close to the theoretical one.  It is easy to see that, while in Figure
  \ref{kurtosis} (a) the kurtosis values increase exponentially as $d$
  increases, in Figure \ref{kurtosis} (b) the kurtosis values decrease for $-0.5
  \leq d < 0.3$ and increase for $0.3 \leq d \leq 0.5$.

  \begin{figure}[!ht]
    \centering \mbox{
      \subfigure[$\beta_1 = 0.6860$]{
        \includegraphics[ width = 0.32\textwidth]{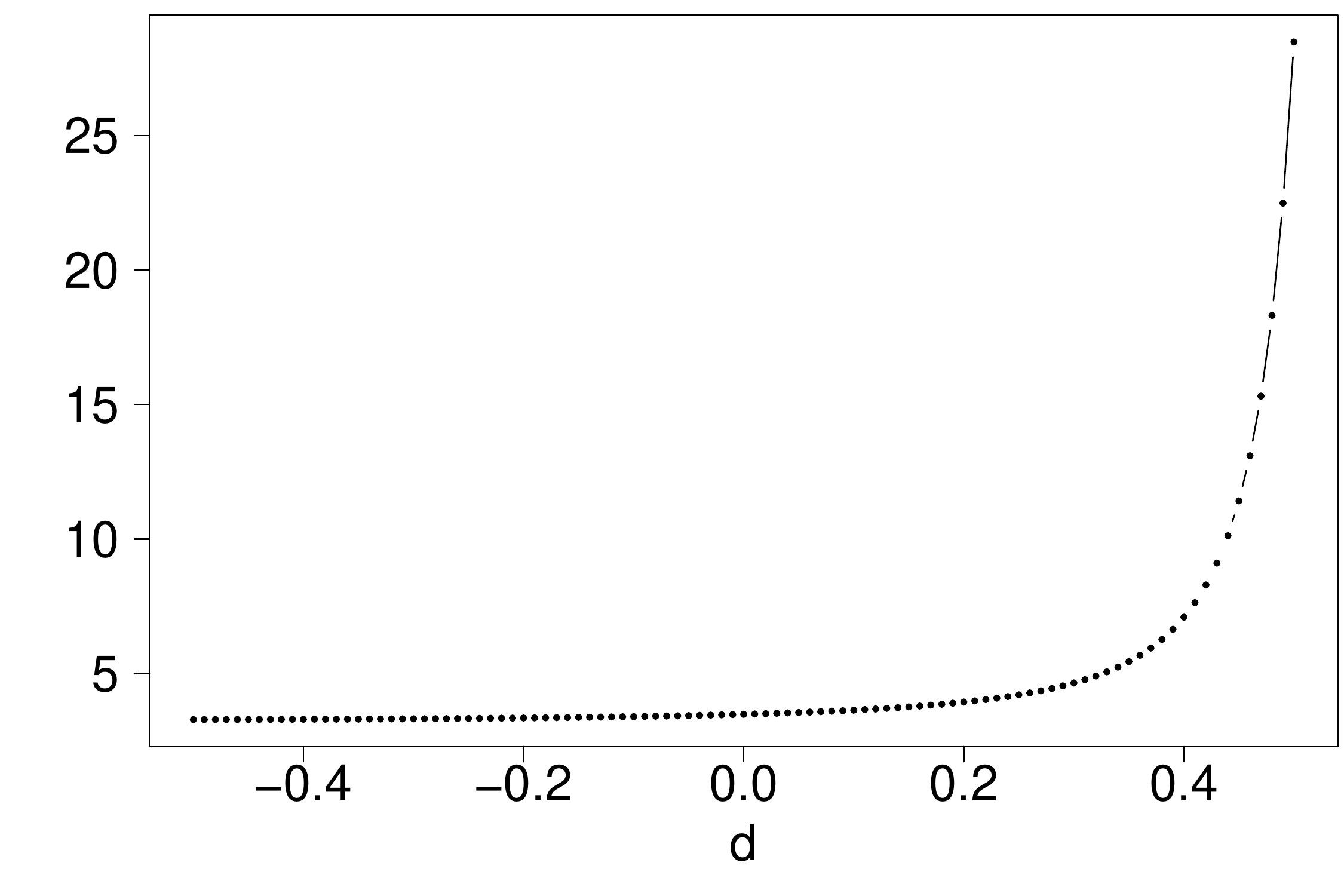}} \hspace{1cm}
      \subfigure[$\beta_1 = -0.6860$]
      {\includegraphics[width = 0.32\textwidth]{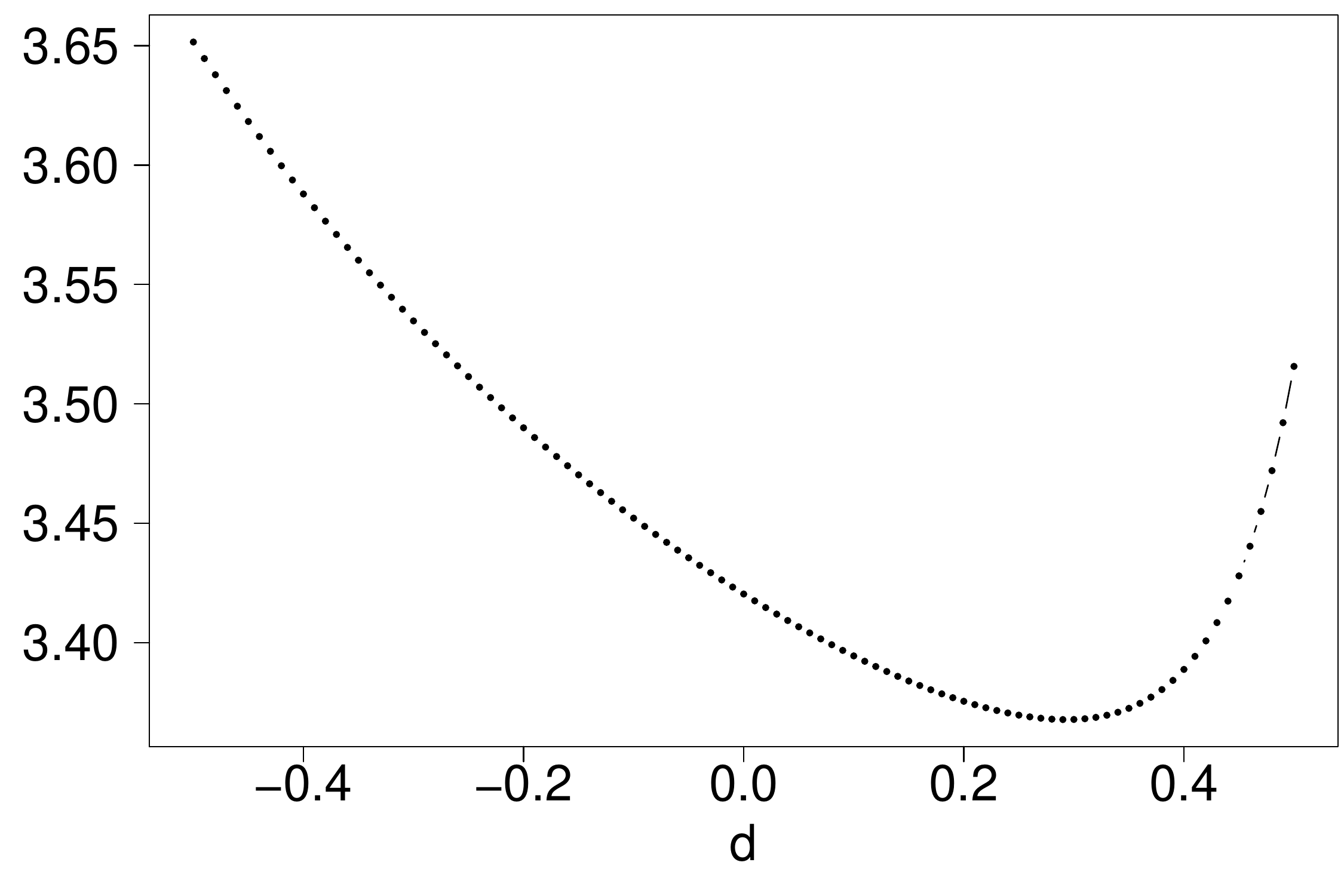}}
    }
    \caption{(a) Kurtosis measure of a FIEGARCH$(0,d,1)$ process as a function
      of the parameter $d$ with $\theta = -0.1661$, $\gamma = 0.2792$, $\omega =
      -7.2247$ and $\beta_1 = 0.6860$; (b) Kurtosis measure of a
      FIEGARCH$(0,d,1)$ process with the same parameters as in (a) but with
      $\beta_1$ replaced by $\beta_1 = - 0.6860$. }\label{kurtosis}
  \end{figure}
\end{ex}

 Although $\{\ln(\sigma_t^2)\}_{t\in\mathds{Z}}$ is an ARFIMA process, in
practice  it cannot be directly observed and frequently, knowing its
characteristics may not be sufficient for model identification and estimation
purposes.  On the other hand, $\{X_t\}_{t \in \mathds{Z}}$ is an observable
process and so is $\{\ln(X_t^2)\}_{t\in\mathds{Z}}$.  By noticing that, from
\eqref{generalproc}, one can rewrite
\[
\ln(X_t^2) = \ln(\sigma_t^2) + \ln(Z_t^2), \quad \mbox{for all} \quad
t\in\mathds{Z},
\]
and now it is clear that the properties of $\{\ln(\sigma_t^2)\}_{t\in\mathds{Z}}$ are
useful to characterize the process $\{\ln(X_t^2)\}_{t\in\mathds{Z}}$.  This
approach was already considered in the literature for parameter estimation
purposes.  For instance,  \cite{PEZA08} and \cite{HUEA05} consider models
such that $X_t$ can be written as in \eqref{generalproc}, but $\sigma_t$ can
have a more general definition than \eqref{fieprocess}.  While \cite{PEZA08}
consider maximum likelihood and Whittle's method of estimation in the class of
exponential volatility models, especially the EGARCH ones, \cite{HUEA05}
consider different semiparametric estimators of the memory parameter in general
signal plus noise models.  In both cases, to obtain an estimator by Whittle's
method, the authors consider the spectral density function of
$\{\ln(X_t^2)\}_{t\in\mathds{Z}}$.

In what follows, we focus our attention in the case where $X_t$ can be written
as in \eqref{generalproc}, and $\sigma_t$ is defined through the expression
\eqref{fieprocess} and we present some properties of the process
$\{\ln(X_t^2)\}_{t\in\mathds{Z}}$. In particular, we show that, under mild
conditions, this process also has an ARFIMA representation.  To the best of our
knowledge no formal proofs of these results are given in the literature of
FIEGARCH$(p,d,q)$ processes, especially the ARFIMA$(q,d,0)$ representation of
$\{\ln(X_t^2)\}_{t\in\mathds{Z}}$.

\begin{thm}\label{arfimafieg}
  Let $\{X_t\}_{t \in \mathds{Z}}$ be a \emph{FIEGARCH}$(p,d,q)$ process, given
  in Definition \ref{definitionfie}.  If $\mathds{E}([\ln(Z_0^2)]^2)<\infty$ and
  $d<0.5$, then the process $\{\ln(X_t^2)\}_{t\in\mathds{Z}}$ is well defined
  and it is stationary (weakly and strictly) and ergodic. Moreover, the
  autocovariance function of $\{\ln(X_t^2)\}_{t\in\mathds{Z}}$ is given by
  \begin{equation}\label{covlnx}
    \mbox{\large$\gamma$}_{\ln(\mbox{\tiny$X$}^2)}(h) =
    \sigma^2_g\sum_{k=0}^\infty\lambda_{d,k}\lambda_{d, k+|h|} + \mbox{\rm Var}\big(\!\ln(Z_{t}^2)\big)\mathbb{I}_{\{0\}}(h) +
    \lambda_{d,|h|-1}\mathcal{K}{\hspace{1pt}}\mathbb{I}_{\mathds{Z}^*}(h), \quad \mbox{for all } h\in\mathds{Z}.
  \end{equation}
  where $\sigma_g^2$ is given in \eqref{eq:sigmag} and
  $\mathcal{K}{\hspace{1pt}} = \mbox{\rm Cov}\big( g(Z_{0}),\ln(Z_{0}^2)\big)$.
\end{thm}

\begin{proof}
  Assume that $\mathds{E}([\ln(Z_0^2)]^2) < \infty$ and $d<0.5$.  Let
  $\{\lambda_{d,k}\}_{k\in\mathds{Z}}$ be given by \eqref{lambdapoly} and
  rewrite \eqref{fieprocess} as \eqref{bol}.

  Observe that $\mathds{E}([\ln(Z_0^2)]^2) < \infty$ implies
  $\mathds{E}(|\ln(Z_0^2)|) < \infty$ and thus $|\ln(Z_t^2)|$ is finite with
  probability one, for all $t\in\mathds{Z}$.  Since $d<0.5$, it follows that
  $\ln(\sigma_t^2)$ is finite with probability one, for all $t\in\mathds{Z}$
  (see Corollary \ref{strictstatfiegarch}).  Therefore, $\ln(X_t^2)$ is finite with
  probability one, for all $t\in\mathds{Z}$, and hence the stochastic process
  $\{\ln(X_t^2)\}_{t\in\mathds{Z}}$ is well defined.  The strict stationarity
  and ergodicity of $\{\ln(X_t^2)\}_{t\in\mathds{Z}}$ follow immediately from
  the measurability of $\ln(Z_t^2) + \ln(\sigma_t^2)$ and the i.i.d. property of
  $\{Z_t\}_{t \in \mathds{Z}}$ (see \cite{DU91}).  To prove that
  $\{\ln(X_t^2)\}_{t\in\mathds{Z}}$ is also weakly stationary notice that
  $\mathds{E}([\ln(Z_0^2)]^2) < \infty$ implies $\mbox{\rm Var}(\ln(Z_t^2)) <
  \infty$, $d< 0 .5$ implies $\mbox{\rm Var}(\ln(\sigma_t^2)) < \infty$
  (see Corollary \ref{strictstatfiegarch}) and the independence of $\{Z_t\}_{t \in
    \mathds{Z}}$ implies that $ \ln(Z_t^2)$ and $\ln(\sigma_t^2)$ are
  independent random variables.  Hence,
  \[
  \mbox{\rm Var}(\ln(X_t^2)) = \mbox{\rm Var}(\ln(\sigma_t^2)) + \mbox{\rm
    Var}(\ln(Z_t^2)) < \infty, \ \ \mbox{ for all } \ \ t\in \mathds{Z}.
  \]

   To complete the proof it remains to show that the autocovariance
  function $\mbox{\large$\gamma$}_{\ln(\mbox{\tiny$X$}^2)}(h)$, for all
  $h\in\mathds{Z}$, is given by expression \eqref{covlnx}.  From the definition
  of $\ln(X_t^2)$, it follows that
  \begin{align}
    \mbox{\rm Cov}\big(\!\ln(X_{t+h}^2),\ln(X_t^2)\big) = & \ \mbox{\rm
      Cov}\big(\!\ln(\sigma_{t+h}^2),\ln(\sigma_t^2)\big)
    + \mbox{\rm Cov}\big(\!\ln(Z_{t+h}^2),\ln(Z_t^2)\big)\nonumber\\
    &+\, \mbox{\rm Cov}\big(\!\ln(\sigma_{t+h}^2),\ln(Z_t^2)\big)+\mbox{\rm
      Cov}\big(\!\ln(Z_{t+h}^2),\ln(\sigma_t^2)\big). \label{covariance}
  \end{align}
   Theorem \ref{N1} shows that
  \[
  \mbox{\rm Cov}\big(\!\ln(\sigma_{t+h}^2),\ln(\sigma_t^2)\big) =
  \sigma^2_g\sum_{k=0}^\infty\lambda_{d,k}\lambda_{d, k+|h|}, \quad \mbox{ for
    all } h \in\mathds{Z}.
  \]
  From the independence of the random variables $\ln(Z_{t}^2)$, for all
  $t\in\mathds{Z}$, and from expression \eqref{bol}, we have
  \[
  \mbox{\rm Cov}\big(\!\ln(Z_{t+h}^2),\ln(Z_{t}^2)\big)\! =\left\{
    \begin{array}{r}
      0, \quad \quad \mbox{ if } h\neq0;\vspace{0.15cm}\\
      \hspace{-3pt} \mbox{\rm Var}\big(\ln(Z_{t}^2)\big), \mbox{ if } h=0\phantom{,}\\
    \end{array} \right.
  \mbox{ and \ }
  \mbox{\rm Cov}\hspace{-1pt}\big(\hspace{-1pt}\ln(\sigma_{t+h}^2),\ln(Z_t^2)\big)=\left\{\begin{array}{r}
      \hspace{-3pt}\lambda_{d,h-1}\mathcal{K}{\hspace{1pt}}, \mbox{ if } h>0;\vspace{0.15cm}\\
      0, \quad \quad \mbox{if } h \leq 0.
    \end{array}\right.
  \]
  where $\mathcal{K}{\hspace{1pt}} = \mbox{\rm Cov}\big(
  g(Z_{0}),\ln(Z_{0}^2)\big)$. Since $\mbox{\rm
    Cov}\big(\ln(Z_{t+h}^2),\ln(\sigma_t^2)\big) = \mbox{\rm
    Cov}\big(\ln(\sigma_{u-h}^2), \ln(Z_{u}^2)\big)$, with $u =t+h$, one
  concludes that
  \[
  \mbox{\rm Cov}\big(\!\ln(\sigma_{t+h}^2),\ln(Z_t^2)\big)+\mbox{\rm
    Cov}\big(\!\ln(Z_{t+h}^2),\ln(\sigma_t^2)\big) = \mbox{\rm
    Cov}\big(\ln(\sigma_{t+|h|}^2),\ln(Z_t^2)\big) =
  \lambda_{d,|h|-1}\mathcal{K}{\hspace{1pt}}\mathbb{I}_{\mathds{Z}^*}(h).
  \]
  By replacing these results on expression \eqref{covariance} we conclude that
  the autocovariance function of $\{\ln(X_t^2)\}_{t\in\mathds{Z}}$ is given by
  \eqref{covlnx}.
\end{proof}

  \begin{figure}[!ht]
    \centering
    \mbox{
      \subfigure[$\mbox{\large$\gamma$}_{\ln(\mbox{\tiny$X$}^2)}(\cdot)$]
      {\includegraphics[width = 0.32\textwidth]{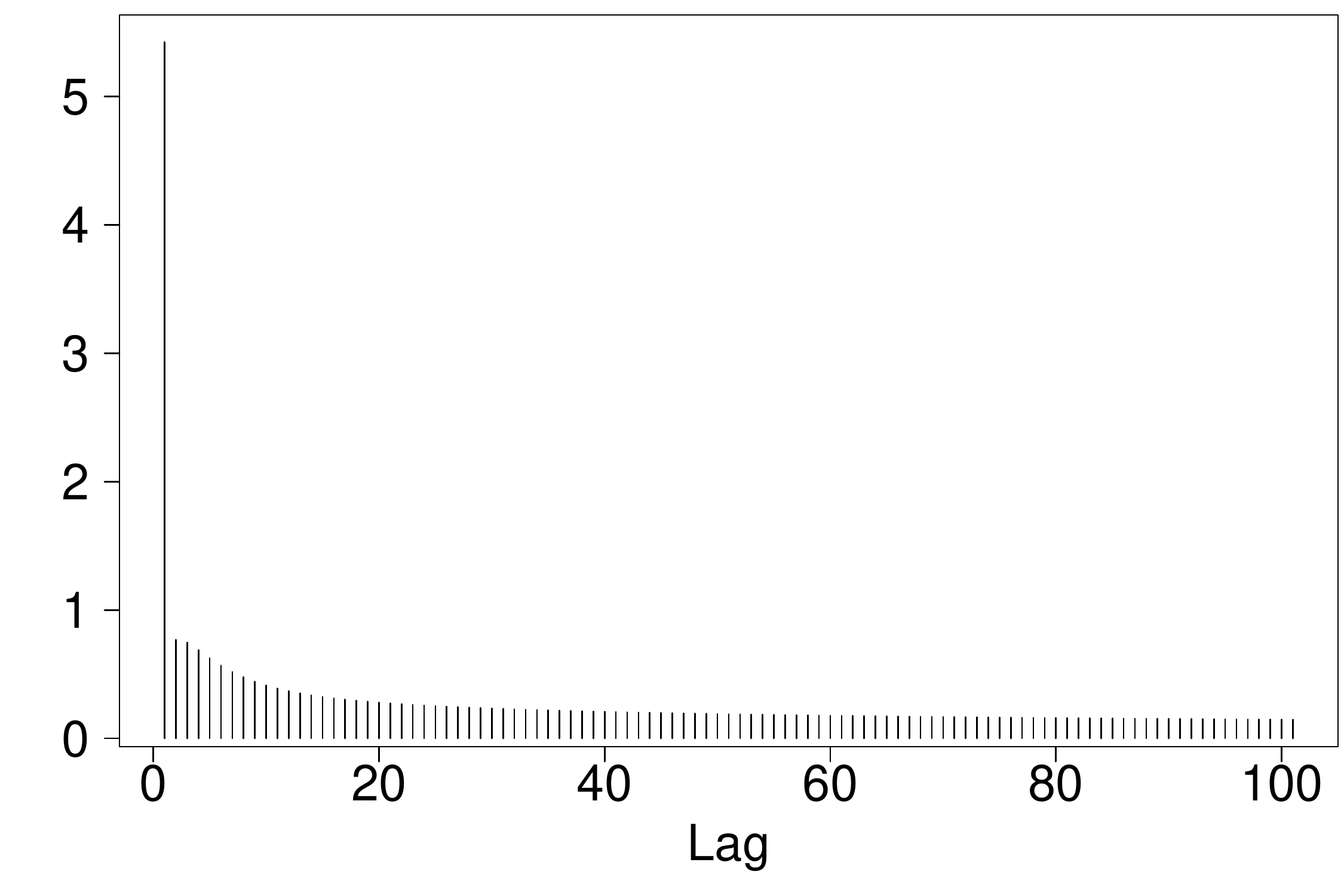}}
      \subfigure[$\mbox{\large$\hat{\gamma}$}_{\ln(\mbox{\tiny$X$}^2)}(\cdot)$]
      {\includegraphics[width = 0.32\textwidth]{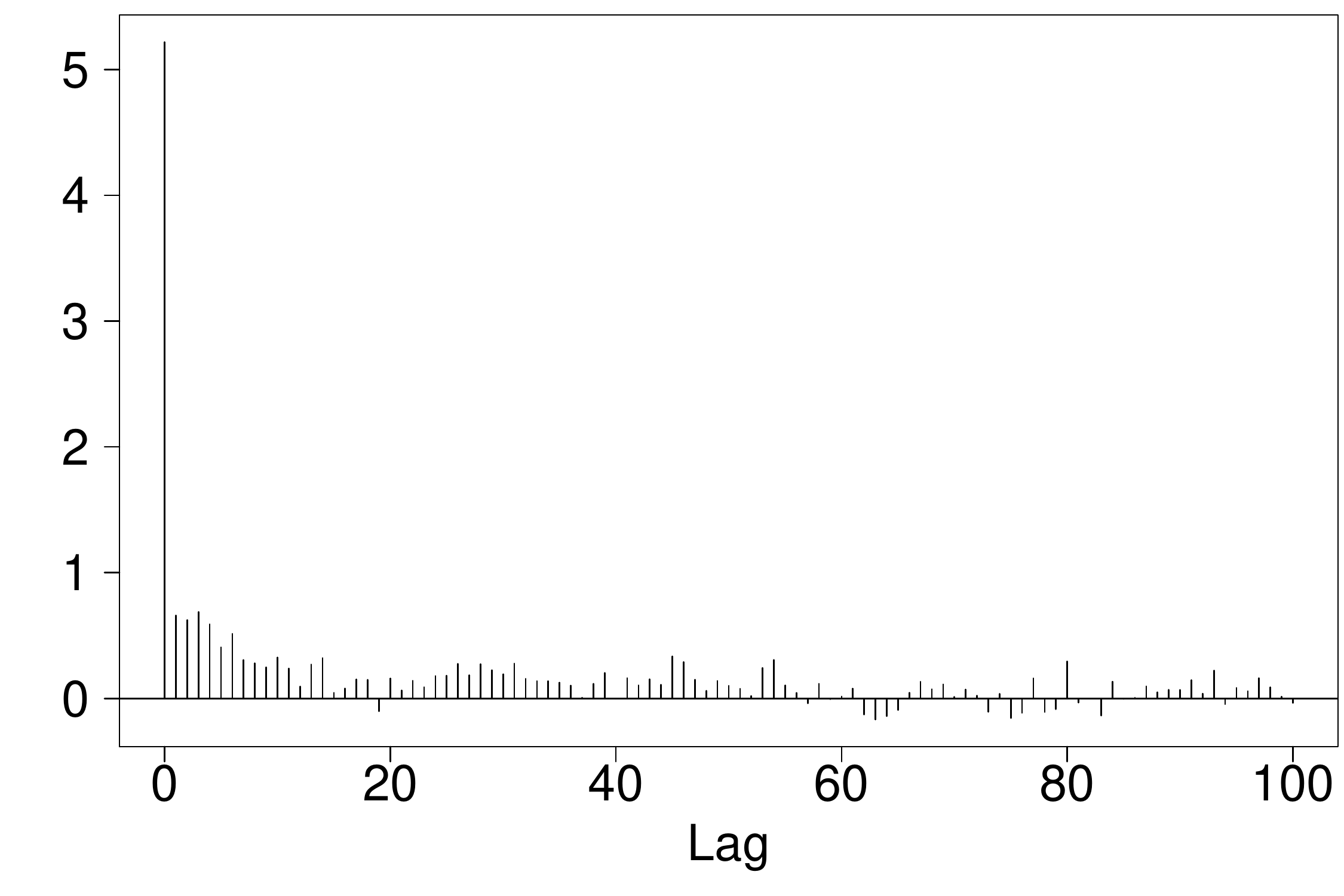}}
      \subfigure[$\mbox{\large$\hat{\gamma}$}_{\ln(r^2)}(\cdot)$]
      {\includegraphics[ width =  0.32\textwidth]{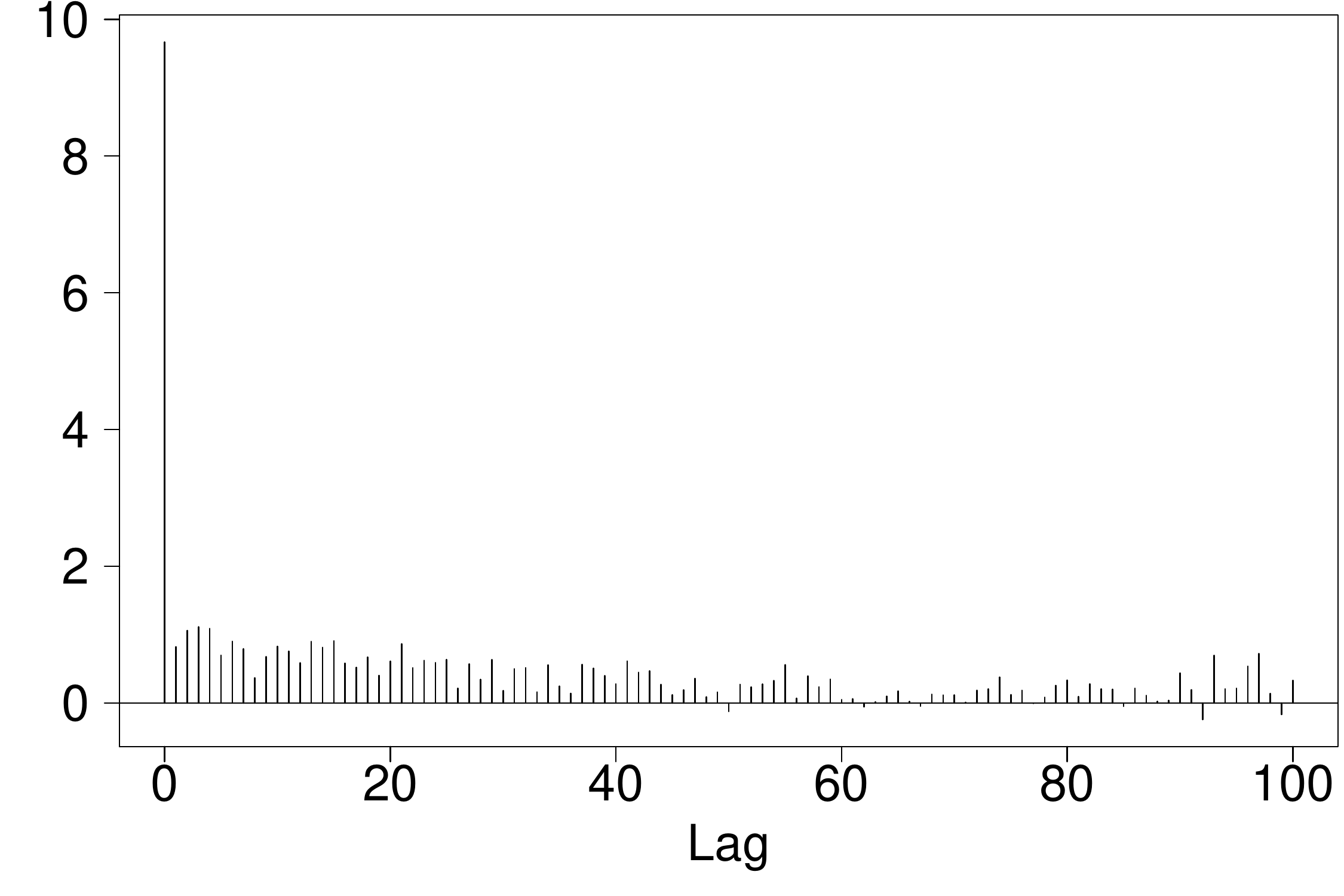}} }
    \caption{(a) Theoretical autocovariance function of the process
      $\{\ln(X_t^2)\}_{t\in\mathds{Z}}$, where $\{X_t\}_{t \in \mathds{Z}}$ is a
      FIEGARCH$(0,d,1)$ process and (b) sample autocovariance function of a time
      series $\{\ln(x_t^2)\}_{t=1}^{2000}$ derived from that FIEGARCH$(0,d,1)$
      process; (c) sample autocovariance function of the time series
      $\{\ln(r_t^2)\}_{t=1}^{1717}$, where $\{r_t\}_{t=1}^{1717}$ is the Bovespa
      index log-returns time series.} \label{covfie}
  \end{figure}

\begin{ex}
  Figure \ref{covfie} (a) presents the theoretical autocovariance function of
  the process $\{\ln(X_t^2)\}_{t\in\mathds{Z}}$, where $\{X_t\}_{t \in
    \mathds{Z}}$ is a FIEGARCH$(0,d,1)$ process, with the same parameter values
  considered in Figures \ref{fiegsim} and \ref{spectral}. Figure \ref{covfie}
  (b) shows the sample autocovariance function of the time series
  $\{\ln(x_t^2)\}_{t=1}^{2000}$, where $\{x_t\}_{t=1}^{2000}$ is the simulated
  time series presented in Figure \ref{fiegsim}. Figure \ref{covfie} (c)
  presents the sample autocovariance function of the time series
  $\{\ln(r_t^2)\}_{t=1}^n$, where $\{r_t\}_{t=1}^n$ is the Bovespa index
  log-returns time series (see Section \ref{analysisObserved}).  By comparing
  the three graphs in Figure \ref{covfie}, one concludes that all three
  functions present a similar behavior.  Since the sample autocovariance
  function is an estimator of the theoretical autocovariance function, it is
  expected that their graphics will have the same behavior. The similarity
  between the decay in the graphs in Figure \ref{covfie} (b) and (c) indicates
  that a FIEGARCH model seems appropriate for fitting the Bovespa index
  log-returns time series.

\end{ex}

\begin{ex}\label{example1}
 Theorem \ref{arfimafieg} provides the expression for the autocorrelation
function of $\{\ln(X_t^2)\}_{t\in\mathds{Z}}$.   The spectral density
function of the process $\{\ln(X_t^2)\}_{t\in\mathds{Z}}$ is given by (see
\cite{HUEA05})
\begin{align*}
  f_{\ln(\mbox{\tiny$X$}_t^2)}(\lambda) &= f_{\ln(\sigma_t^2)}(\lambda) + \frac{\mathcal{K}{\hspace{1pt}}}{\pi}\mathds{R}e\big(e^{-\mathfrak{i}
    \lambda}\Lambda(\lambda)\big) + f_{\ln(Z_t^2)}(\lambda) \\
    & = \frac{\sigma_g^2}{2\pi}
  \frac{|\alpha(e^{-i\lambda})|^2}{|\beta(e^{-i\lambda})|^2}|1-e^{-i\lambda}|^{-2d
  } + \frac{\mathcal{K}{\hspace{1pt}}}{\pi}\mathds{R}e\big(e^{-\mathfrak{i}
    \lambda}\Lambda(\lambda)\big)  + \frac{1}{2\pi}\mbox{Var}(\ln(Z_0^2)),   \quad \mbox{for all } \lambda \in
  [0,\pi],
\end{align*}
where  $\sigma_g^2 = \mathrm{Var}\big(g(Z_t)\big)$ is given in \eqref{eq:sigmag}, $\mathcal{K}{\hspace{1pt}} = \mbox{\rm Cov}\big(g(Z_0),\ln(Z_0^2)\big)$,
$\mathds{R}e(z)$ is the real part of $z$ and $\Lambda(z) :=
\lambda(e^{-\mathfrak{i} z})$.   As an illustration, Figure \ref{spectral} (a) shows the spectral density
  function of the process $\{\ln(X_t^2)\}_{t\in\mathds{Z}}$, where $\{X_t\}_{t
    \in \mathds{Z}}$ is any FIEGARCH$(0,d,1)$ process with $d = 0.3578$, $\theta
  = -0.1661$, $\gamma = 0.2792$, $\omega = -7.2247$ and $\beta_1 =
  0.6860$, assuming $Z_0 \sim
      \mathcal{N}(0,1)$ (dashed line) and $Z_0 \sim \mbox{GED}(1.5)$ (continuous
      line).   The corresponding values of  $\sigma_g^2$,
    $\mathcal{K}$ and $\mbox{Var}(\!\ln(Z_0^2))$,  used in the computation
    of $f_{\ln(\mbox{\tiny$X$}_t^2)}(\cdot)$,  are given in Table \ref{values}.
   Figure \ref{spectral} (b) shows the periodogram function of the time
  series $\{\ln(x_t^2)\}_{t=1}^{2000}$, where $\{x_t\}_{t=1}^{2000}$ is the
  simulated time series presented in Figure \ref{fiegsim} (a), with $Z_0 \sim \mathcal{N}(0,1)$. Figure
  \ref{spectral} (c) shows the periodogram function of the time series
  $\{\ln(r_t^2)\}_{t=1}^{1717}$, where $\{r_t\}_{t=1}^{1717}$ is the Bovespa
  index log-returns time series.

  \begin{table}[!ht]
    \centering
    \renewcommand{\arraystretch}{1.2}
  \caption{Theoretical values for the expectation and variance of functions of
    $Z_0$ and the corresponding values of   $\sigma_g^2$ e
    $\mathcal{K}$ considering the Gaussian and the Generalized Error distribution
    functions.  In both cases  $\theta = -0.1661$ and  $\gamma
    = 0.2792$.}\label{values} \vspace{0.2cm}
  {\footnotesize
\begin{tabular*}{1\textwidth}{@{\extracolsep{\fill}} ccccccccc}
\hline
Distribution && $\mathds{E}(|Z_0|)$  &
$\mathds{E}(|Z_0|\ln(Z_0^2))$ & $\mathds{E}(\ln(Z_0^2))$ &
$\mbox{Var}(\ln(Z_0^2))$ &&  $\sigma_g^2$ & $\mathcal{K}$\\
\cline{1-1} \cline{3-6} \cline{8-9}
$\mathcal{N}(0,1)$ && 0.7979 & 0.0925  &-1.2704 & 4.9348  && 0.0559
& 0.3088\\
GED$(1.5)$ &&  0.7674  & 0.0975& -1.4545 & 5.4469 && 0.0596 & 0.3389\\
\hline
\end{tabular*}}
\end{table}

    \begin{figure}[!htb]
    \centering
    \mbox{
      \subfigure[$f_{\ln(\mbox{\tiny$X$}_t^2)}(\cdot)$]
      {\includegraphics[ width =  0.32\textwidth]{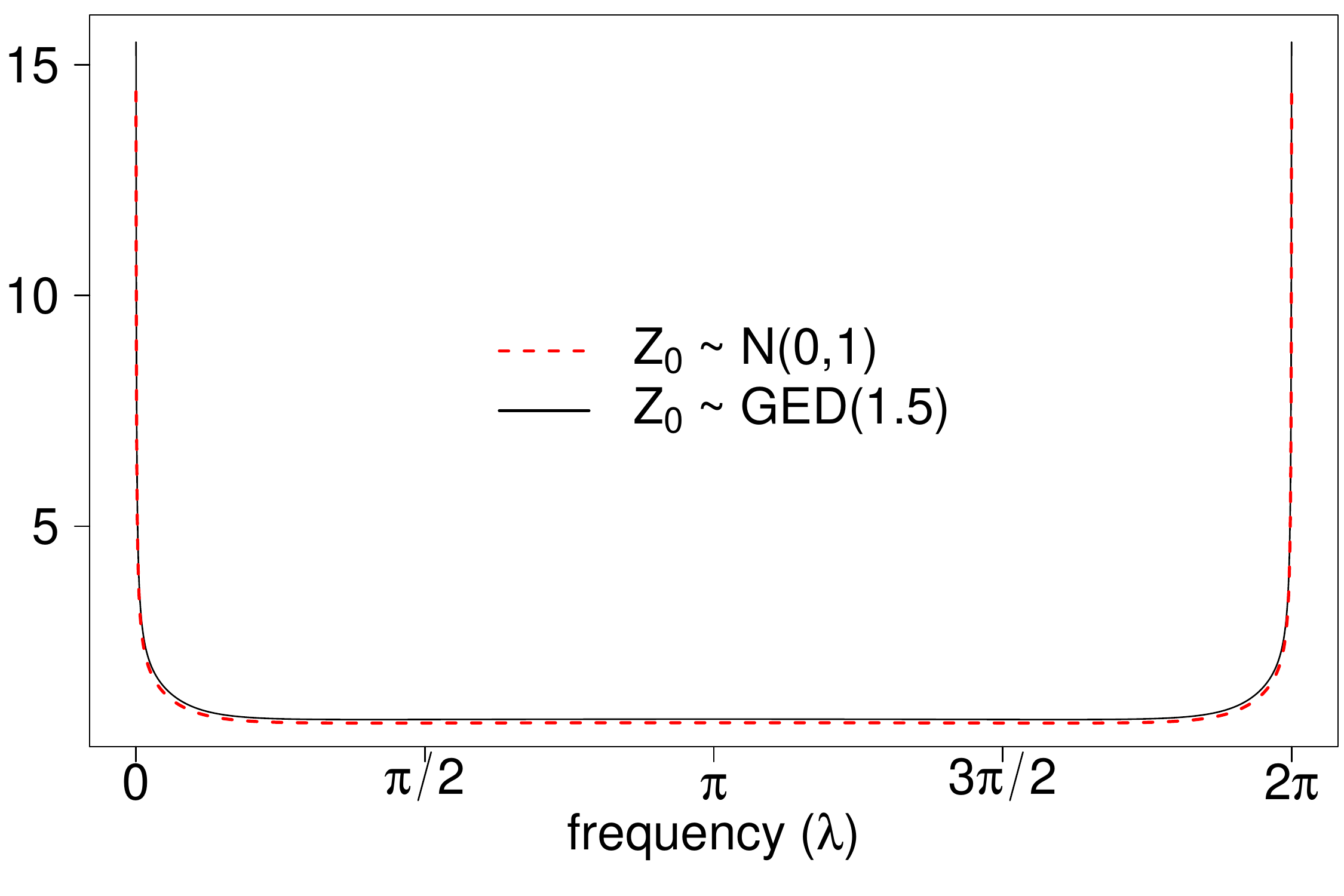}}
      \subfigure[$I_{\ln(\mbox{\tiny$X$}_t^2)}(\cdot)$]
      {\includegraphics[width = 0.32\textwidth]{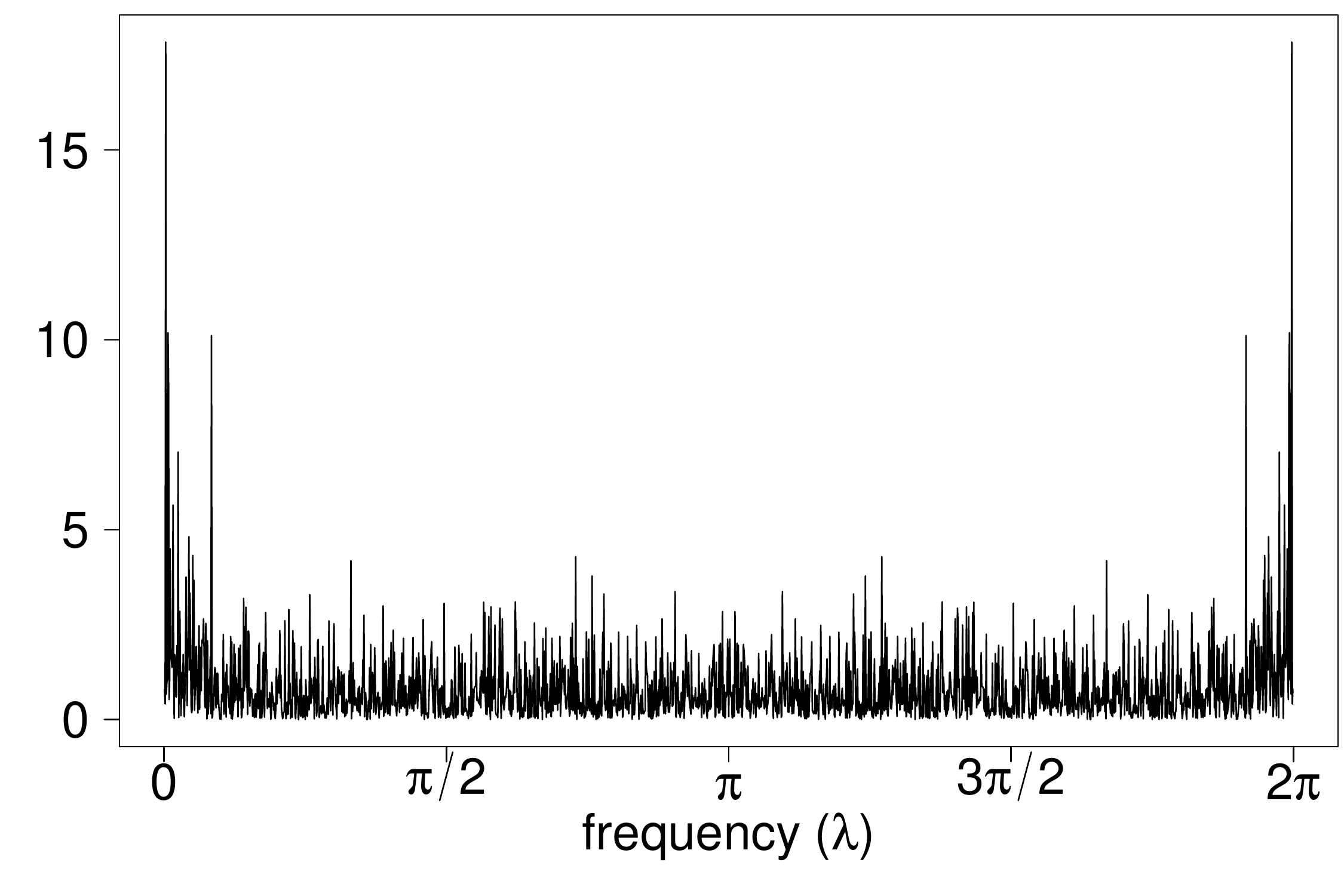}}
      \subfigure[$I_{\ln(r_t^2)}(\cdot)$]
      {\includegraphics[width = 0.32\textwidth]{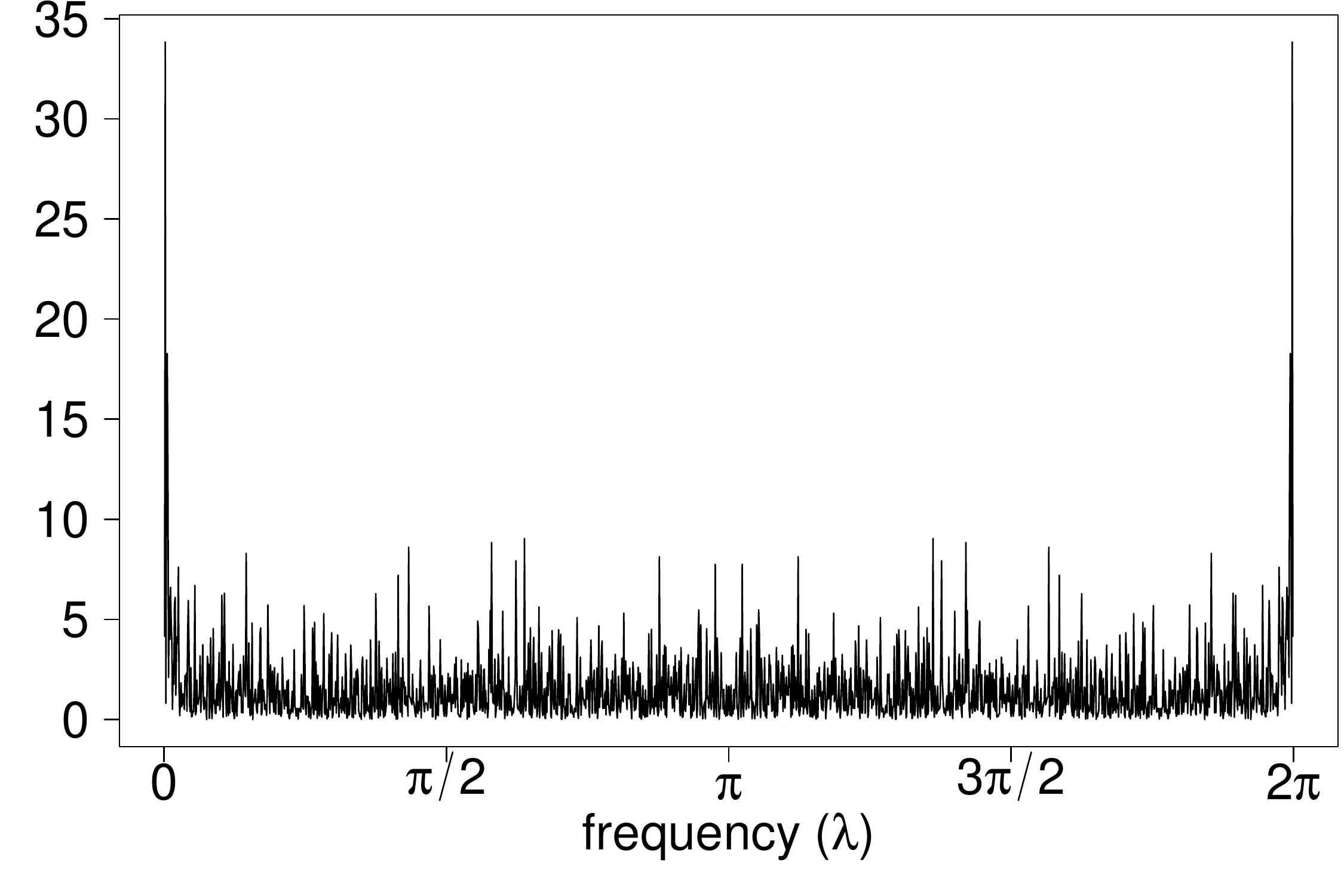}}
    }
    \caption{(a) Theoretical spectral density function of the process
      $\{\ln(X_t^2)\}_{t \in \mathds{Z}}$, where $\{X_t\}_{t \in \mathds{Z}}$ is
      a FIEGARCH$(0,d,1)$ process with $d = 0.3578$, $\theta = -0.1661$, $\gamma
      = 0.2792$, $\omega = -7.2247$ and  $\beta_1 = 0.6860$, assuming $Z_0 \sim
      \mathcal{N}(0,1)$ (dashed line) and $Z_0 \sim \mbox{GED}(1.5)$ (continuous
      line);  (b) periodogram
      function related to a time series $\{x_t\}_{t=1}^{2000}$ derived from this
      FIEGARCH$(0,d,1)$ process with $Z_0 \sim
      \mathcal{N}(0,1)$; (c) the periodogram function related to the
      time series $\{\ln(r_t^2)\}_{t=1}^{1717}$, where $\{r_t\}_{t=1}^{1717}$ is
      the Bovespa index log-returns time series.} \label{spectral}
  \end{figure}

 Figure \ref{spectral} (a) indicates that the probability
  distribution of $Z_0$ may not be evident from the periodogram function given
  the similarity between the graphs of $f_{\ln(\mbox{\tiny$X$}_t^2)}(\cdot)$.
  The   small difference on   the values of  $f_{\ln(\mbox{\tiny$X$}_t^2)}(\cdot)$
  for $Z_0 \sim \mathcal{N}(0,1)$ and  $Z_0 \sim \mbox{GED}(1.5)$ is explained
  by the fact that  the values of
$\sigma_g^2$, $\mathcal{K}$ and  $f_{\ln(Z_t^2)}(\lambda)$
are relatively close for $Z_0 \sim \mathcal{N}(0,1)$ and  $Z_0 \sim
\mbox{GED}(1.5)$ as shown in Table \ref{values}.  Moreover,
  one observes that the graphs in Figure \ref{spectral} (b) and (c) present
  similar behavior, indicating that a FIEGARCH model may be adequate to fit the
  data.      On the other hand,    Figure \ref{spectral} (a)  shows evidence that
  the underlying probability  distribution of  $\{r_t\}_{t=1}^{1717}$ may not be
  the same as   $\{x_t\}_{t=1}^{2,000}$.    In fact,  we  apply the two-sample
  Kolmogorov-Smirnov test to verify  the  hypothesis that    $I_{\ln(\mbox{\tiny$X$}_t^2)}(\cdot)$  and
  $I_{\ln(r_t^2)}(\cdot)$  have the same  probability distribution.  We also
  apply  the test to the standardized versions of   $I_{\ln(\mbox{\tiny$X$}_t^2)}(\cdot)$  and
  $I_{\ln(r_t^2)}(\cdot)$ (that is, we subtracted the sample mean and divided by
  the sample standard deviation).   In the first case the test rejects  the null
  hypothesis  ($\alpha = 0.05$,  test statistic  = 0.2208,
  $\mbox{p-value} < 2.2\times 10^{-16}$).  In the second case (standardized version) the
  test did not reject the null hypothesis ($\alpha = 0.05$,  test statistic  =
  0.0285, $\mbox{p-value}  =  0.8472$).

  To further  investigate whether the correct probability distribution of $Z_0$ can be
  identified through the periodogram function we consider the same time series
  $\{x_t\}_{t=1}^{2000}$ as in Figure \ref{spectral} (b) and perform the
  Kolmogorov-Smirnov hypothesis test as  described in  \cite{BRDA91}, pages 339
  - 342,  considering both cases  $Z_0 \sim   \mathcal{N}(0,1)$ and $Z_0 \sim
  \mbox{GED}(1.5)$.   Recall that
\begin{itemize}
\item   the null hypothesis of the test is that $\ln(X_t^2)$ has spectral
  density function   $f_{\ln(\mbox{\tiny$X$}_t^2)}(\cdot)$;

\item  the testing procedure consists on ploting the Kolmogorov-Smirnov boundaries
\[
y = \frac{x-1}{m-1} \pm k_\alpha(m-1)^{-1/2},  \quad  1 \leq x  \leq m,  \quad
k_\alpha = \left\{\begin{array}{lcl}
    1.36, & \mbox{if} & \alpha  = 0.05;\\
    1.63, & \mbox{if} & \alpha  = 0.01;
  \end{array}
  \right.
\]
and the function $C(x)$ defined as
  \[
  C(x) = \left\{
    \begin{array}{lcl}
      0, & \mbox{if}& x < 1;\\
      Y_i, & \mbox{if}&  i \leq x < i+1, \quad \mbox{for }  i\in\{1, \cdots, m\};\\
      1,  & \mbox{if}& x \geq m;\\
    \end{array}
    \right.
  \]
with $Y_0 := 0$,  $Y_m := 1$ and
\[
  Y_i := \Bigg[\sum_{k=1}^i\frac{I_{\ln(\mbox{\tiny$X$}_t^2)}(\omega_k)}{f_{\ln(\mbox{\tiny$X$}_t^2)}(\omega_k)}\Bigg]\Bigg[
  \sum_{k=1}^m\frac{I_{\ln(\mbox{\tiny$X$}_t^2)}(\omega_k)}{f_{\ln(\mbox{\tiny$X$}_t^2)}(\omega_k)}\Bigg]^{-1}\!\!\!\!\!\!\!,
\quad  \quad \mbox{with} \quad \omega_k = \frac{2k\pi}{n}, \,\,
k\in\{1,\cdots, m\},
\]
where $m$ is the integer part of  $(n-1)/2$ and   $n$ is the time series sample
 size;

\item the null hypothesis is rejected if $C(\cdot)$ exits the boundaries for
some $ 1 \leq x  \leq m$.
  \end{itemize}

    \begin{figure}[!htb]
    \centering
    \mbox{
      \subfigure[$C_{1}(x)$]
      {\includegraphics[ width =  0.3\textwidth]{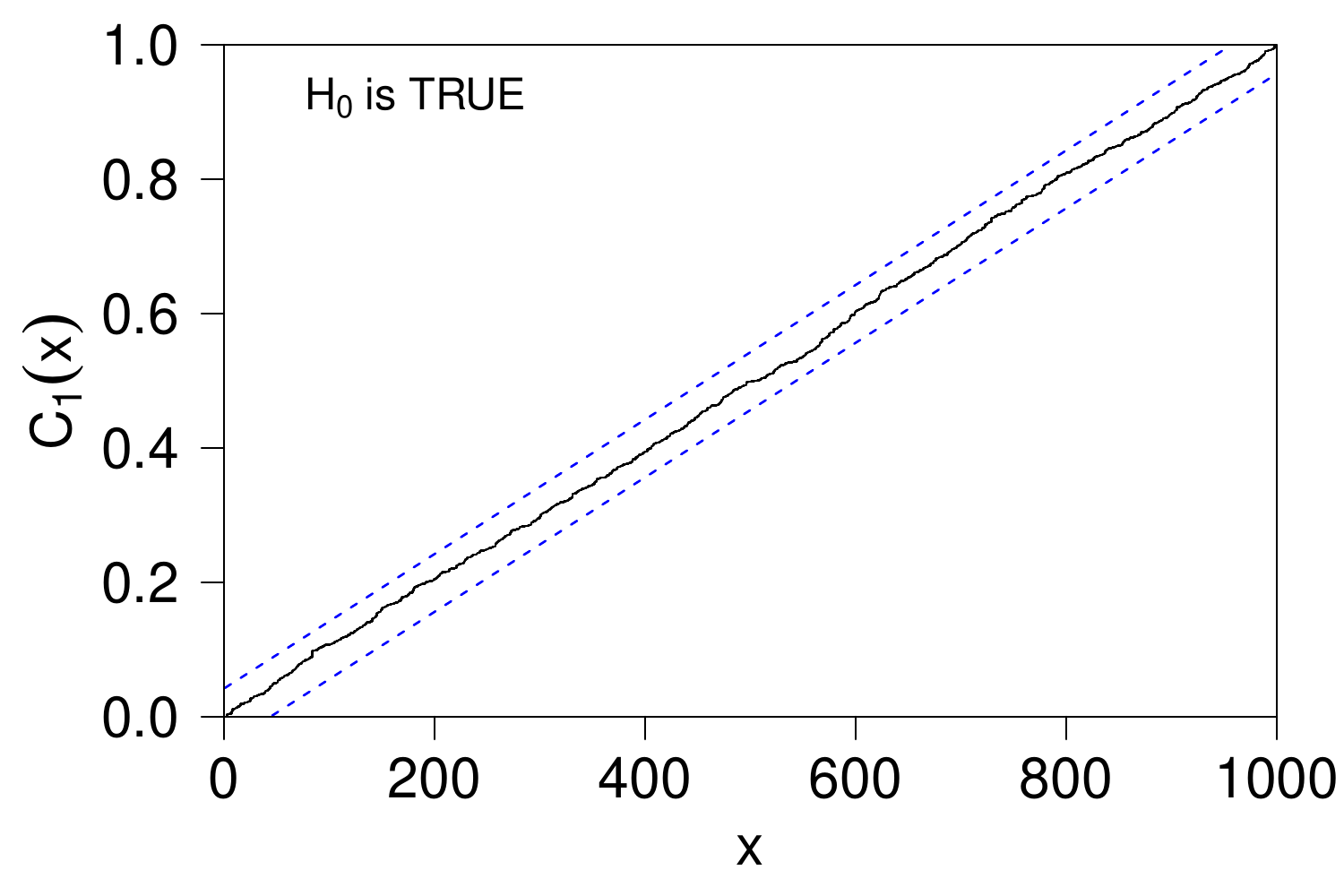}}
      \subfigure[$C_2(x)$]
      {\includegraphics[width = 0.3\textwidth]{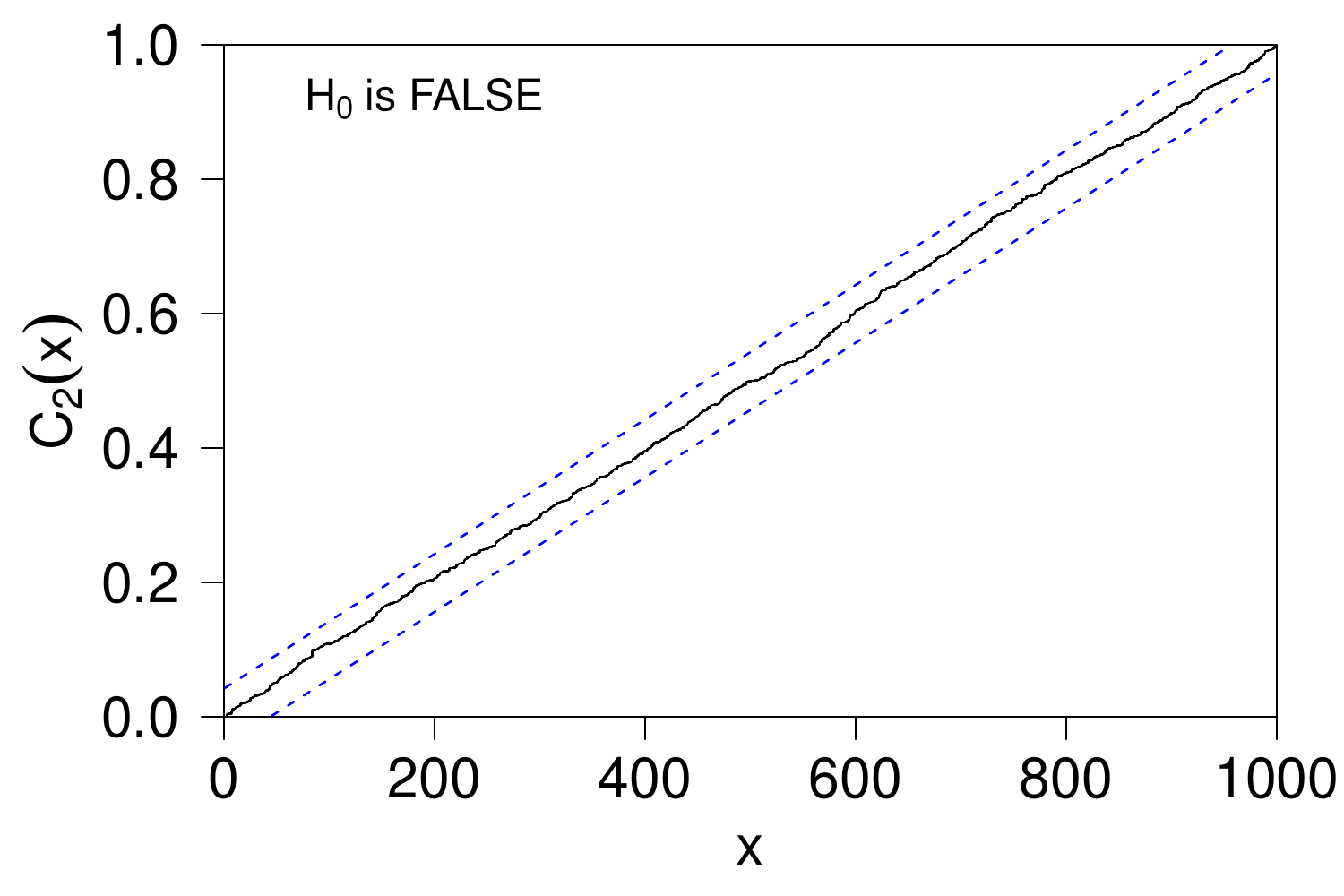}}
      \subfigure[$C_1(x) - C_2(x)$]
      {\includegraphics[width = 0.3\textwidth]{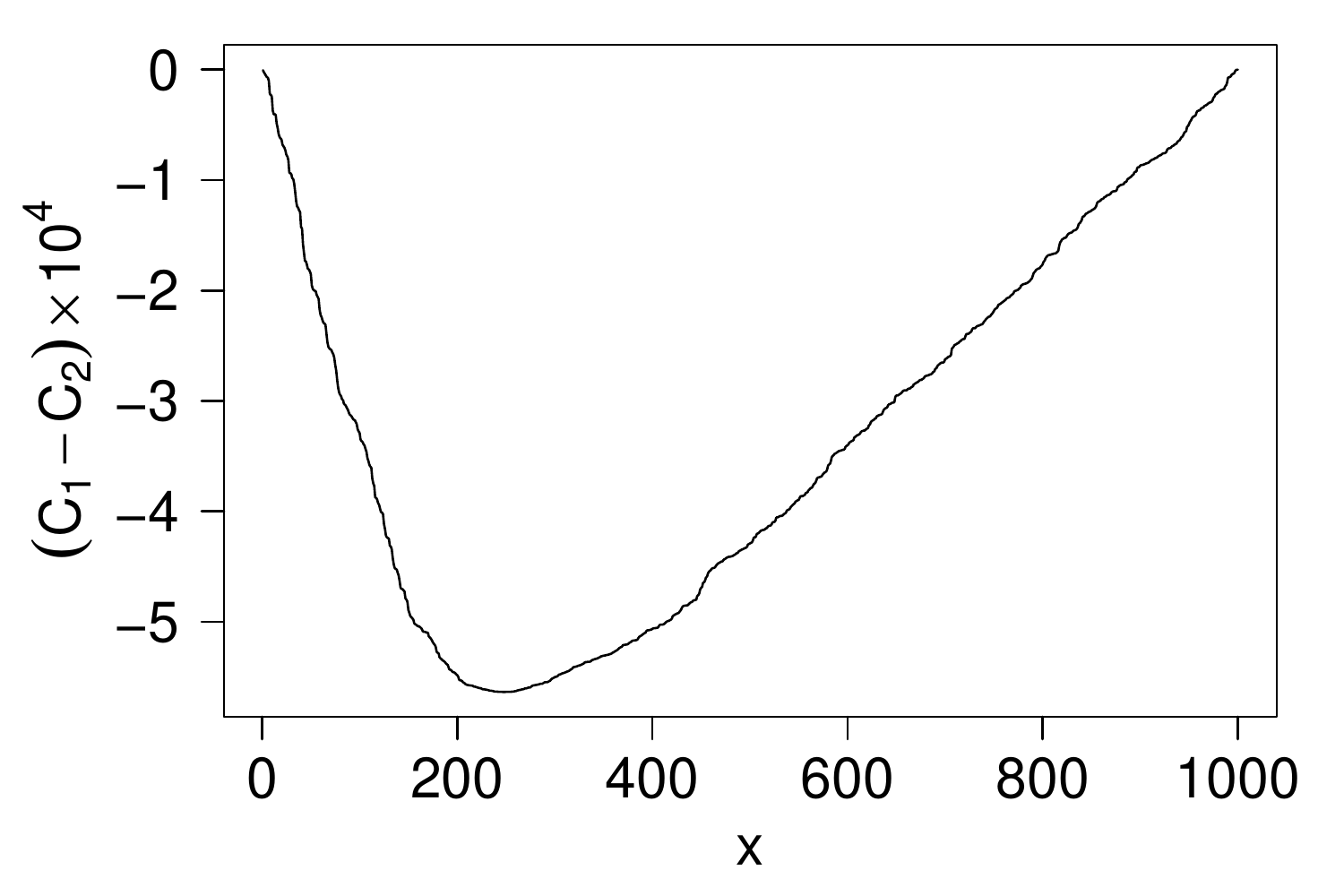}}
    }
    \caption{Function $C(x)$ and the Kolmogorov-Smirnov boundaries, with $\alpha =
      0.05$ (dashed
      line),   when  $\{x_t\}_{t =1}^{2,000}$ is a time series   derived from   a
      FIEGARCH$(0,d,1)$ process with  $Z_0 \sim
      \mathcal{N}(0,1)$   and  $f_{\ln(\mbox{\tiny$X$}_t^2)}(\cdot)$ is the
      theoretical spectral density function  of a FIEGARCH$(0,d,1)$ process with   (a)  $Z_0 \sim \mathcal{N}(0,1)$ (therefore,  the null hypothesis
      is true); (b) $Z_0 \sim  \mbox{GED}(1.5)$ (therefore,  the null hypothesis
      is false).  In all cases,    $d = 0.3578$, $\theta = -0.1661$, $\gamma
      = 0.2792$, $\omega = -7.2247$ and  $\beta_1 = 0.6860$.    On (c),  the
      difference between the values of $C(x)$ (multiplied by $10^4$)
      assuming, respectively,   $Z_0 \sim \mathcal{N}(0,1)$  and $Z_0 \sim \mbox{GED}(1.5)$.} \label{cpgram}
  \end{figure}

  The results of the tests are given in Figure \ref{cpgram}, where $C_1(\cdot)$
  and $C_2(\cdot)$ denote the values of $C(\cdot)$ obtained, respectively,
when     assuming $Z_0 \sim \mathcal{N}(0,1)$ and $Z_0 \sim \mbox{GED}(1.5)$.  From
  Figures \ref{cpgram} (a) and (b) one concludes that the Kolmogorov-Smirnov
  test does not reject the null hypothesis in both cases.  This result was
  expected given the small difference between the values of
  $f_{\ln(\mbox{\tiny$X$}_t^2)}(\cdot)$,  shown in Figure \ref{spectral} (a).  In
  fact,  by comparing Figures \ref{cpgram} (a) and (b), one observes no visible difference
  between those graphs.    Figure \ref{cpgram} (c) confirms that the difference
  is too small to be noticed since $|C_1(x) - C_2(x)| < 6\times
  10^{-4}$, for all $0 \leq x \leq 1000$.   This shows that, for  the   FIEGARCH
  process considered in Example \ref{example1},  the correct probability
  distribution of $Z_0$   cannot be   identified through the periodogram
  function, given that   the   Kolmogorov-Smirnov   hypothesis test failed to
  reject the null hypothesis when it was false.

\end{ex}

To conclude this section we present the following theorem  which shows that, under mild conditions,
$\{\ln(X_t^2)\}_{t\in\mathds{Z}}$ is an ARFIMA$(q,d,0)$ process with correlated
innovations.     This results is very useful in model identification and parameter estimation since the
literature of ARFIMA models is well developed (see \cite{LO08} and references
therein).

\begin{thm}
  Let $\{X_t\}_{t \in \mathds{Z}}$ be a \emph{FIEGARCH}$(p,d,q)$ process, given
  in Definition \ref{definitionfie}.  Suppose $|d|< 0.5$ and
  $\mathds{E}([\ln(Z_0^2)]^2)<\infty$. Then  $\{\ln(X_t^2)\}_{t\in\mathds{Z}}$
  is an \emph{ARFIMA}$(q,d,0)$ process given by
  \[
  \beta(\mathcal{B})(1-\mathcal{B})^{d}(\ln(X_t^2)-\omega)=\varepsilon_t,\quad
  \mbox{ for all } t\in\mathds{Z},
  \]
  where $\{\varepsilon_t\}_{t\in\mathds{Z}}$ is a stochastic process with zero
  mean and autocovariance function $\mbox{\large$\gamma$}_{\varepsilon}(\cdot)$
  given by
  \begin{equation}\label{varARFIMA}
    \mbox{\large$\gamma$}_{\varepsilon}(h) =  \left\{
      \begin{array}{ccc}
        \hspace{-3pt} \sigma_g^2\hspace{-2pt}\displaystyle\sum_{i=|h|}^{p}\hspace{-2pt}\alpha_i\alpha_{i-|h|} +
        \mathcal{K}{\hspace{1pt}}\hspace{-1pt}\displaystyle\sum_{i=0}^{p}\hspace{-1pt}\alpha_i\phi_{i+|h|+1} +
        \mathcal{K}{\hspace{1pt}}\hspace{-7pt}\displaystyle\sum_{i=|h|-1}^{p}\hspace{-6pt}\alpha_i\phi_{i-|h|+1} +
        \sigma_{\ell}^2\displaystyle\sum_{i=|h|}^{\infty}\phi_i\phi_{i-|h|} ,  & \mbox{if} & 0\leq |h|\leq p;\\
        \mathcal{K}{\hspace{1pt}}\alpha_p\phi_1 + \sigma^2_{\ell}\displaystyle\sum_{i=p+1}^{\infty}\phi_i\phi_{i-(p+1)} ,  & \mbox{if} & |h| =   p+1;\\
        \sigma^2_{\ell}\displaystyle\sum_{i=|h|}^{\infty}\phi_i\phi_{i-|h|}, & \mbox{if} & |h| > p+1,
      \end{array}
    \right.
  \end{equation}
  with $\{\phi_k\}_{k\in\mathds{N}}$ defined by
  \begin{equation}\label{phibeta}
    \phi(z) :=\beta(z)(1-z)^d = \sum_{k=0}^{\infty}\phi_kz^k, \quad \mbox{for } |z| < 1,
  \end{equation}
  $\sigma_g^2$ given in \eqref{eq:sigmag}, $\mathcal{K}{\hspace{1pt}} =
  \mbox{\rm Cov}(g(Z_0),\ln(Z_0^2))$, $\{\alpha_i\}_{i=0}^p$ given in
  \eqref{alphabeta} and $\sigma^2_{\ell} := \mbox{\rm Var}(\ln(Z_0^2))$.
\end{thm}

\begin{proof}
  Let $\{X_t\}_{t \in \mathds{Z}}$ be a FIEGARCH process.  From expressions
  \eqref{generalproc} and \eqref{fieprocess} we have
  \[
  \beta(\mathcal{B})(1-\mathcal{B})^d(\ln(X_t^2) -\omega) =\varepsilon_t, \quad
  \mbox{ for all } t\in\mathds{Z},
  \]
  where
  \begin{equation*}
    \varepsilon_t = \alpha(\mathcal{B})g(Z_{t-1})+\beta(\mathcal{B})(1-\mathcal{B})^d\ln(Z_t^2),\ \mbox{ for all }
    t\in\mathds{Z}.
  \end{equation*}
  In particular, if $d > 0$, we have
  $\beta(\mathcal{B})(1-\mathcal{B})^d\omega=0$ and
  $\beta(\mathcal{B})(1-\mathcal{B})^d\ln(X_t^2) =\varepsilon_t$, for all
  $t\in\mathds{Z}$.

  Now, suppose that $\mathds{E}([\ln(Z_0^2)]^2)<\infty$.  Since $\{Z_t\}_{t \in
    \mathds{Z}}$ is a sequence of i.i.d. random variables and $0 \leq
  |\mathds{E}(\ln(Z_0^2))| \leq \mathds{E}(|\ln(Z_0^2)|) \leq
  [\mathds{E}([\ln(Z_0^2)]^2)]^{1/2}$, one concludes that
  $\mathds{E}(\ln(Z_t^2)) = \mathds{E}(\ln(Z_0^2)) < \infty$, for all
  $t\in\mathds{Z}$.  Therefore,
  $\beta(\mathcal{B})(1-\mathcal{B})^d\mathds{E}(\ln(Z_t^2))=0$ and
  $\alpha(\mathcal{B})\mathds{E}(g(Z_{t-1}))=0$.  Consequently,
  $\mathds{E}(\varepsilon_t)=0$, for all $t\in\mathds{Z}$.

  Let $\phi(\cdot)$ be defined by expression \eqref{phibeta}.  Assume, for the
  moment, that $\mbox{\rm Var}(\varepsilon_t^2)<\infty$, for all
  $t\in\mathds{Z}$.  It follows that
  \begin{align}
    \mbox{\rm Cov}(\varepsilon_{t+h},\varepsilon_t) = & \, \mbox{\rm
      Cov}\big(\alpha(\mathcal{B})g(Z_{t+h-1}),\alpha(\mathcal{B})g(Z_{t-1})\big)+
    \mbox{\rm Cov}\big(\phi(\mathcal{B})\ln(Z_{t+h}^2),\phi(\mathcal{B})\ln(Z_{t}^2)\big)\nonumber\\
    &+\, \mbox{\rm
      Cov}\big(\alpha(\mathcal{B})g(Z_{t+h-1}),\phi(\mathcal{B})\ln(Z_{t}^2)\big)+
    \mbox{\rm
      Cov}\big(\phi(\mathcal{B})\ln(Z_{t+h}^2),\alpha(\mathcal{B})g(Z_{t-1})\big).\label{covep}
  \end{align}
  Since $\{g(Z_t)\}_{t\in\mathds{Z}}$ is a white noise process we have
  \[
  \mbox{\rm
    Cov}\big(\alpha(\mathcal{B})g(Z_{t+h-1}),\alpha(\mathcal{B})g(Z_{t-1})\big)
  =\left\{
    \begin{array}{ccc}
      \!\!\displaystyle\mbox{\rm Var}\big(g(Z_0)\big)\sum_{i=|h|}^{p}\alpha_i\alpha_{i-|h|},
      &\!\!\mbox{if} & |h| \leq p; \nonumber\\
      0, & \!\! \mbox{if} &\!\! |h| > p,
    \end{array} \right.
  \]
  which does not depend on $t\in\mathds{Z}$. From the independence of the random
  variables $Z_t$, for all $t \in \mathds{Z}$, one has
  \[
  \mbox{\rm
    Cov}\big(\alpha(\mathcal{B})g(Z_{t+h-1}),\phi(\mathcal{B})\ln(Z_{t}^2)\big)
  =\left\{\begin{array}{ccc}
      \!\!\mathcal{K}{\hspace{1pt}}\displaystyle\sum_{i=0}^{p}\ \alpha_i\phi_{i-h+1},&\mbox{if}& h<1;\vspace{.1cm}\nonumber\\
      \!\!\mathcal{K}{\hspace{1pt}}\!\!\displaystyle\sum_{i=h-1}^{p}\alpha_{i}\phi_{i-h+1},&
      \mbox{if} & 1\leq h\leq p+1;\vspace{.1cm}\nonumber \\ 0, & \mbox{if} & \
      h>p+1
    \end{array} \right.
  \]
  and
  \[
  \mbox{\rm
    Cov}\big(\phi(\mathcal{B})\ln(Z_{t+h}^2),\alpha(\mathcal{B})g(Z_{t-1})\big)
  =\left\{\begin{array}{ccc} 0, & \mbox{if} &
      h< -(p+1);\vspace{.1cm}\nonumber\\
      \mathcal{K}{\hspace{1pt}}\!\!\!\!\displaystyle\sum_{i=|h|-1}^{p}\!\!\!\!\alpha_i\phi_{i+h+1}, &\mbox{ if} &  -(p+1)\leq h\leq -1;\vspace{.1cm}\nonumber\\
      \mathcal{K}{\hspace{1pt}}\displaystyle\sum_{i=0}^{p}\alpha_{i}\phi_{i+h+1},
      &\mbox{ if} & h> -1,
    \end{array} \right.
  \]
  where $\mathcal{K}{\hspace{1pt}}:= \mbox{\rm Cov}\big(g(Z_0),\ln(Z_0^2)\big)$
  does not depend on $t\in\mathds{Z}$.  Also, from the independence of the
  random variables $\ln(Z_t^2)$, for all $t\in\mathds{Z}$, we have
  \[
  \mbox{\rm
    Cov}\big(\phi(\mathcal{B})\ln(Z_{t+h}^2),\phi(\mathcal{B})\ln(Z_{t}^2)\big)
  = \mbox{\rm Var}\big(\ln(Z_0^2)\big)\sum_{i=|h|}^{\infty}\phi_i\phi_{i-|h|},\quad
  \mbox{for all } h\in\mathds{Z}, \vspace{-0.4\baselineskip}
  \]
  which does not depend on $t\in\mathds{Z}$.

  Therefore, all four terms in expression \eqref{covep} do not depend on
  $t\in\mathds{Z}$ and expression \eqref{varARFIMA} holds.  Now, to validate
  expression \eqref{varARFIMA} we only need to show that $\mbox{\rm
    Var}(\varepsilon_t)<\infty$, for all $t\in\mathds{Z}$.  Notice that, since
  $\mathds{E}(\varepsilon_t) =0$, it follows that $\mathds{E}(\varepsilon_t^2) =
  \mbox{\rm Var}(\varepsilon_t) = \mbox{\large$\gamma$}_{\varepsilon}(0)$.  Upon
  replacing $h=0$ in \eqref{covep} one obtains
  \[
  \mbox{\large$\gamma$}_{\varepsilon}(0) = \mbox{\rm
    Var}\big(g(Z_0)\big)\displaystyle\sum_{i=0}^{p}\alpha_i^2 +
  2\mathcal{K}{\hspace{1pt}}\displaystyle\sum_{i=0}^{p}\alpha_i\phi_{i+1} +
  \mbox{\rm Var}\big(\ln(Z_0^2)\big)\displaystyle\sum_{i=0}^{\infty}\phi_i^2.
  \]
  By hypothesis, $\mathds{E}([\ln(Z_0^2)]^2)<\infty$ and $d\in (-0.5,0.5)$.  It
  follows that $\mbox{\rm Var}(\ln(Z_0^2))<\infty$ and
  $\sum_{i=0}^{\infty}\phi_i^2<\infty$.  We also know that $\mbox{\rm
    Var}(g(Z_0))<\infty$.  In order to show that $\mathcal{K}{\hspace{1pt}} <
  \infty$, notice that $\mathcal{K}{\hspace{1pt}}:=\mbox{\rm
    Cov}(g(Z_0),\ln(Z_0^2)) =\mathds{E}(g(Z_0)\ln(Z_0^2))$ and,
  since $\mathds{E}( Z_0^2)=1$ and $\mbox{\rm Var}(\ln(Z_0^2))<\infty$, from
  H\"older's inequality, we have $\mathds{E}(|Z_0|) < \infty$ and
  $\mathds{E}(\ln(Z_0^2)) < \infty$.  Then  from \eqref{functiong} it follows
  that
  \[
  \mathds{E}\big(g(Z_0)\ln(Z_0^2)\big) =\theta\, \mathds{E}\big(
  Z_0\ln(Z_0^2)\big) +\gamma \mathds{E}\big( |Z_0|\ln(Z_0^2)\big) - c,
  \]
  where $c :=\gamma\, \mathds{E}( |Z_0|)\mathds{E}(\ln(Z_0^2))<\infty$.  By
  using the fact that $2ab \leq a^2 + b^2$, for all $a,b\in\mathds{R}$, one
  concludes that
  \[
  \big|\mathds{E}(Z_t\ln(Z_t^2))\big| \leq
  \frac{1}{2}\left[\mathds{E}(Z_t^2) + \mathds{E}(\ln(Z_t^2))\right]<\infty
  \ \mbox{ and } \ \big|\mathds{E}(|Z_t|\ln(Z_t^2))\big| \leq
  \frac{1}{2}\left[\mathds{E}(Z_t^2) +
  \mathds{E}(\ln(Z_t^2))\right]<\infty.
  \]
  Hence $\mathds{E}(g(Z_0)\ln(Z_0^2)) <\infty$ and, consequently,
  $\mbox{\rm Cov}(g(Z_0),\ln(Z_0^2)) <\infty$ and
  $\mbox{\large$\gamma$}_{\varepsilon}(0)<\infty$.  Therefore, the result
  follows.
\end{proof}

\section{Forecasting}\label{forecastingsection}

Let $\{X_t\}_{t \in \mathds{Z}}$ be a FIEGARCH$(p,d,q)$ process, given in
Definition \ref{definitionfie}, and $\{x_t\}_{t=1}^{n}$ a time series obtained
from this process.  In this section, we prove that $\{X_t\}_{t \in \mathds{Z}}$
is a martingale difference with respect to the filtration
$\{\mathcal{F}_{t}\}_{t\in\mathds{Z}}$, where $\mathcal{F}_{t} :=
\sigma(\{Z_s\}_{s\leq t})$, and we provide the $h$-step ahead forecast for the
process $\{X_t\}_{t \in \mathds{Z}}$. Since the process
$\{\ln(\sigma_t^2)\}_{t\in\mathds{Z}}$, defined by \eqref{fieprocess}, has an
ARFIMA$(q,d,p)$ representation, the $h$-step ahead forecasting for this process
and its mean square error value can be easily obtained (for instance, see
\cite{LO08} and \cite{RELO99}).  This fact is used to provide an $h$-step ahead
forecast for $\{\ln(X_t^2)\}_{t\in\mathds{Z}}$ and the mean square error of
forecasting.  We also consider the fact that $\mathds{E}(X_t^2) =
\mathds{E}(\sigma_t^2)$, for all $t\in\mathds{Z}$, to provide an $h$-step ahead
forecast for both processes, $\{X_{t}^2\}_{t\in\mathds{Z}}$ and
$\{\sigma_{t}^2\}_{t\in\mathds{Z}}$, based on the predictions obtained from the
process $\{\ln(\sigma_t^2)\}_{t\in\mathds{Z}}$.  The notation used in this
section is introduced below.

\begin{remark}
  Let $Y_t$, for $t\in\mathds{Z}$, denote any random variable defined here. In
  the sequel we consider the following notation:\vspace{-0.2\baselineskip}
  \begin{itemize}
  \item we use the symbol ``\^{}'' to denote the $h$-ahead step forecast defined
    in terms of the conditional expectation, that is, $\hat Y_{t+h} =
    \mathds{E}(Y_{t+h}|\mathcal{F}_{t})$.  Notice that this is the best linear
    predictor in terms of mean square error value.  The symbols ``\~{}'' and
    ``\v{}'' are used to denote alternative estimators (e.g. $\tilde Y_{t+h}$
    and $\check Y_{t+h}$);

  \item for simplicity of notation, for the $h$-step ahead forecast of
    $\ln(Y_{t+h})$, we write $\hat{\ln}(Y_{t+h})$ instead of
    $\widehat{\ln(Y_{t+h})}$ (analogously for ``\~{}'' and ``\v{}'');

  \item we follow the approach usually considered in the literature and denote
    the $h$-ahead step forecast $Y_{t+h}^2$ as $\hat Y_{t+h}^2$ instead of
    $\widehat{Y_{t+h}^2}$.  If necessary, to avoid confusion, we will denote the
    square of $\hat Y_{t+h}$ as $(\hat Y_{t+h})^2$ (analogously for ``\~{}'' and
    ``\v{}'').
  \end{itemize}
\end{remark}

The following lemma shows that a FIEGARCH$(p,d,q)$ process is a martingale
difference with respect to $\{\mathcal{F}_{t}\}_{t\in\mathds{Z}}$. This result
is useful in the proof of Lemma \ref{predictors} that presents the $h$-step
ahead forecast of $X_{n+h}$, for a fixed value of $n\in\mathds{Z}$ and all
$h\geq 1$, and the $1$-step ahead forecast of $X_{n+1}^2$, given
$\mathcal{F}_{n}$.

\begin{lemma}\label{fiemd}
  Let $\{X_t\}_{t \in \mathds{Z}}$ be a \emph{FIEGARCH}$(p,d,q)$ process, given
  in Definition \ref{definitionfie} and $\mathcal{F}_{t} :=
  \sigma(\{Z_s\}_{s\leq t})$.  Then  the process $\{X_t\}_{t \in \mathds{Z}}$ is
  a martingale difference with respect to
  $\{\mathcal{F}_{t}\}_{t\in\mathds{Z}}$.
\end{lemma}

\begin{proof}
  From definition, $\sigma_t$ is a $\mathcal{F}_{t-1}$-measurable
  function. Moreover, for all $t\in\mathds{Z}$,
  $\mathds{E}(X_t)=\mathds{E}(\mathds{E}(X_t|\mathcal{F}_{t-1}))$ and
  $\mathds{E}(X_t|\mathcal{F}_{t-1})=\mathds{E}(\sigma_tZ_t|\mathcal{F}_{t-1})=\sigma_t\mathds{E}(Z_t|\mathcal{F}_{t-1})=0.$
  Therefore, the process $\{X_t\}_{t \in \mathds{Z}}$ is a martingale difference
  with respect to $\{\mathcal{F}_t\}_{t\in\mathds{Z}}$.
\end{proof}

\begin{lemma}\label{predictors}
  Let $\{X_t\}_{t \in \mathds{Z}}$ be a stationary \emph{FIEGARCH}$(p,d,q)$
  process, given by Definition \ref{definitionfie}. Then, for any fixed
  $n\in\mathds{Z}$, the $h$-step ahead forecast of $X_{n+h}$, for all $h>0$ and
  the $1$-step ahead forecast of $X_{n+1}^2$, given $\mathcal{F}_{n}$, are,
  respectively, $\hat{X}_{n+h} = 0$ and $\hat{X}_{n+1}^2 =\sigma_{n+1}^2$.
\end{lemma}

\begin{proof}
  From Lemma \ref{fiemd}, a FIEGARCH$(p,d,q)$ process is a martingale
  difference.  It follows that
  $\hat{X}_{n+h}=\mathds{E}(X_{n+h}|\mathcal{F}_{n}) =0$, for all $h>0$.  From
  definition, $\mathds{E}(X_{n+1}^2|\mathcal{F}_n) = \sigma_{n+1}^2$. Therefore,
  the $1$-step ahead forecast of $X_{n+1}^2$, given $\mathcal{F}_{n}$, is
  $\sigma_{n+1}^2$.  Moreover, if $\mathds{E}(X_{t}^4)<\infty$, for all
  $t\in\mathds{Z}$, then  this is the best forecast value in mean square error
  sense.
\end{proof}

To obtain the $h$-step ahead forecast for $X_{n+h}^2$, notice that $\sigma_t$
and $Z_t$ are independent and so are $\sigma_t^2$ and $Z_t^2$, for all
$t\in\mathds{Z}$.  Moreover, $\mathds{E}(Z_{n+h}^2|\mathcal{F}_n) =
\mathds{E}(Z_{n+h}^2) = 1$, for all $h>0$. It follows that
\[
\hat{X}_{n+h}^2 := \mathds{E}(X_{n+h}^2|\mathcal{F}_{n}) =
\mathds{E}(\sigma_{n+h}^2|\mathcal{F}_{n}) := \hat{\sigma}_{n+h}^2, \quad \mbox{
  for all} \quad h>0.
\]
While for ARCH/GARCH models, $\mathds{E}(\sigma_{n+h}^2|\mathcal{F}_{t})$ can be
easily calculated, for FIEGARCH processes, what is easy to derive is the
expression for the $h$-step ahead forecast for the process
$\{\ln(\sigma_t^2)\}_{t\in\mathds{Z}}$, for any $h>1$.  The expressions for
$\hat{\ln}(\sigma_{n+h}^2) := \mathds{E}(\ln(\sigma_{n+h}^2)|\mathcal{F}_{t})$
and for the mean square error of forecast are given in Proposition
\ref{hstepaheadX}.  We shall use this result to discuss the properties of the
predictor obtained by considering $\check \sigma_{n+h}^2 :=
\exp\{\hat{\ln}(\sigma_{n+h}^2)\}$, for all $h>0$.

\begin{prop}\label{hstepaheadX}
  Let $\{X_t\}_{t \in \mathds{Z}}$ be a \emph{FIEGARCH}$(p,d,q)$ process, given
  by Definition \ref{definitionfie}.  Then the $h$-step ahead forecast
  $\hat{\ln}(\sigma_{n+h}^2)$ of $\ln(\sigma_{n+h}^2)$, given $\mathcal{F}_{n} =
  \sigma(\{Z_t\}_{t\leq n})$, $n\in\mathds{N}$, is given by
  \begin{equation}
    \hat{\ln}(\sigma_{n+h}^2) =\omega + \sum_{k=0}^{\infty}\lambda_{d,k+h-1}\,g(Z_{n-k}),\
    \mbox{ for all } h>0.\label{exp1}
  \end{equation}
  Moreover, the mean square error forecast is equal to zero, if $h = 1$, and it
  is given by
  \begin{equation}\label{exp03}
    \mathds{E}\big(\big[\ln(\sigma_{n+h}^2)- \hat\ln(\sigma_{n+h}^2)\big]^2\big)
    =\sigma^2_g\sum_{k=0}^{h-2}\lambda_{d,k}^2, \quad \mbox{if } h \geq 2,
  \end{equation}
  where $\sigma^2_g := \mathds{E}([g(Z_0)]^2)$ is given in \eqref{eq:sigmag}.
\end{prop}

\begin{proof}
  Let
  $\hat\ln(\sigma_{n+h}^2):=\mathds{E}(\ln(\sigma_{n+h}^2)|\mathcal{F}_{n})$.
  Note that $\mathds{E}(g(Z_{t})|\mathcal{F}_{n})=\mathds{E}(g(Z_t))=0$, for all
  $t>n$, and $\mathds{E}(g(Z_{t})|\mathcal{F}_{n})=g(Z_t)$, for all $t\leq
  n$. By \eqref{bol}, one has
  \[
  \hat\ln(\sigma_{n+h}^2) = \omega
  +\sum_{k=0}^{\infty}\lambda_{d,k}\,\mathds{E}(g(Z_{n+h-1-k})|\mathcal{F}_{n})
  = \omega \ +\!  \sum_{k= h-1}^{\infty}\!\!\!\lambda_{d,k}\,g(Z_{n+h-1-k}),
  \]
  and expression \eqref{exp1} follows.

  Since $\ln(\sigma_{n+h}^2)$ is a function of $\{g(Z_s)\}_{s\leq{n+h-1}}$ and
  $\{g(Z_t)\}_{t\in\mathds{Z}}$ is a sequence of i.i.d. random variables with
  zero mean and variance $\sigma^2_g:= \mathds{E}([g(Z_0)]^2)$, we conclude that
  \[
  \mathds{E}\big(\big[\ln(\sigma_{n+h}^2)- \hat\ln(\sigma_{n+h}^2)\big]^2\big)=
  \mathds{E}\bigg(\bigg[\sum_{k=0}^{h-2}\lambda_{d,k}\,g(Z_{n+h-1-k})\bigg]^2\bigg)=\sigma^2_g\sum_{k=0}^{h-2}\lambda_{d,k}^2,
  \quad \mbox{if } h \geq 2,
  \]
  and zero if $h=1$.
\end{proof}

In practice, $\mathds{E}(\sigma_{n+h}^2|\mathcal{F}_{t})$ cannot be easily
calculated for FIEGARCH models and thus, a common approach is to predict
$\sigma_{n+h}^2$ through the relation $\check{\sigma}_{n+h}^2 :=
\exp\{\hat{\ln}(\sigma_{n+h}^2)\}$, with $ \hat{\ln}(\sigma_{n+h}^2)$ defined by
\eqref{exp1}, for all $h>0$.  As a consequence, a $h$-step ahead forecast for
$X_{n+h}^2$ is defined as $\check{X}_{n+h}^2 := \check{\sigma}_{n+h}^2$ and a
naive estimator for $\ln(X_{n+h}^2)$ is obtained by letting
\begin{equation}\label{forecastlx}
  \check{\ln}(X_{n+h}^2) := \ln(\check X_{n+h}^2) = \ln(\check \sigma_{n+h}^2) =
  \hat{\ln}(\sigma_{n+h}^2), \quad \mbox{ for all} \quad h>0.
\end{equation}
From expressions \eqref{generalproc} and \eqref{forecastlx}, it is obvious that
$\check \ln(X_{n+h}^2)$ is a biased estimator for $\ln(X_{n+h}^2)$, whenever
$\mathds{E}(\ln(Z_{n+h}^2)) \neq 0$. Proposition \ref{msfeXsq} gives the mean
square error forecast for the $h$-step ahead forecast of $\ln(X_{n+h}^2)$,
defined through expression \eqref{forecastlx}.

\begin{prop}\label{msfeXsq}
  Let $\check\ln(X_{n+h}^2)$, for all $h>0$, be the $h$-step ahead forecast of
  $\ln(X_{n+h}^2)$, given the filtration $\mathcal{F}_{n} =
  \sigma(\{Z_s\}_{s\leq n})$, defined by expression \eqref{forecastlx}. Then
  the mean square error forecast is given by
  \[
  \mathds{E}\big(\big[\ln(X_{n+h}^2)- \check \ln(X_{n+h}^2)\big]^2\big) =
  \sigma^2_g\sum_{k=0}^{h-2}\lambda_{d,k}^2 +
  \mathds{E}\big(\big[\ln(Z_{n+h}^2)\big]^2\big), \quad \mbox{where $\sigma^2_g
    := \mathds{E}\big([g(Z_0)]^2\big)$}.
  \]
\end{prop}

\begin{proof}
  By expression \eqref{forecastlx}, $\check{\ln}(X_{n+h}^2) :=
  \hat{\ln}(\sigma_{n+h}^2)$, for all $h>0$.  Thus, from expression
  \eqref{generalproc} and from Proposition \ref{hstepaheadX}, we have
  \begin{align}
    \mathds{E}\big(\big[\hspace{-1pt}\ln(X_{n+h}^2)\hspace{-1pt}-\check\ln(X_{n+h}^2)\big]^2\big)
    &=
    \mathds{E}\big(\big[\hspace{-1pt}\ln(X_{n+h}^2)-\ln(\hat{\sigma}_{n+h}^2)\big]^2\big)=
    \mathds{E}\big(\big[\hspace{-1pt}\ln(\sigma_{n+h}^2)+\ln(Z_{n+h}^2)-\ln(\hat{\sigma}_{n+h}^2)\big]^2\big)\nonumber\\
    &=\mathds{E}\bigg(\bigg[\sum_{k=0}^{h-2}\lambda_{d,k}\,g(Z_{n+h-1-k})+\ln(Z_{n+h}^2)\bigg]^2\bigg).\,\label{exp3}
  \end{align}
  By expanding the right hand side of expression \eqref{exp3} and using the fact
  that $\{Z_t\}_{t \in \mathds{Z}}$ is a sequence of i.i.d. random variables,
  the proposition follows immediately.
\end{proof}

\begin{remark}
  If the values of $X_t$ and $\sigma_t$ are known only for $t\in \{1,\cdots,n\}$
  then  the $h$-step ahead forecast $\hat\ln(\hat{\sigma}_{n+h}^2)$ of
  $\ln(\sigma_{n+h}^2)$, is approximated by
  \[
  \hat\ln(\sigma_{n+h}^2) \simeq \omega +
  \sum_{k=0}^{n-1}\lambda_{d,k+h-1}\,g(Z_{n-k}), \quad \mbox{for all } h>0,
  \]
  and, by definition, the same approximation follows for $\check
  \ln(X_{n+h}^2)$.  It is easy to see that, in this case, the mean square error
  of forecast values for the processes $\{\ln(\sigma_t^2)\}_{t\in\mathds{Z}}$
  and $\{\ln(X_t^2)\}_{t\in\mathds{Z}}$ are given, respectively, by
  \begin{align*}
    \mathds{E}\big(\big[\ln(\sigma_{n+h}^2)-\hat\ln(\sigma_{n+h}^2)\big]^2\big)
    &= \sigma^2_g\bigg(\, \sum_{k=0}^{h-2}\lambda_{d,k}^2\
    +\hspace{-5pt}\sum_{k=n+h-1}^{\infty}\hspace{-10pt}\lambda_{d,k}^2\bigg),
    \quad
    \mbox{ and}\\
    \quad \displaystyle
    \mathds{E}\big(\big[\ln(X_{n+h}^2)-\check\ln(X_{n+h}^2)\big]^2\big) & =
    \sigma^2_g\bigg(\, \sum_{k=0}^{h-2}\lambda_{d,k}^2\
    +\hspace{-5pt}\sum_{k=n+h-1}^{\infty}\hspace{-10pt}\lambda_{d,k}^2\bigg) +
    \mathds{E}(\left[\ln(Z_{n+h}^2)\right]^2,\quad \mbox{ for all } h>0.
  \end{align*}

\end{remark}

From Jensen's inequality, one concludes that
\[
\check{\sigma}_{n+h}^2 := \exp\{\hat\ln(\sigma_{n+h}^2)\} =
\exp\{\mathds{E}(\ln(\sigma_{n+h}^2)|\mathcal{F}_{t})\} \leq
\mathds{E}(\sigma_{n+h}^2|\mathcal{F}_{n}) :=\hat \sigma_{n+h}^2, \quad \mbox{
  for all } h>0,
\]
so that $\mathds{E}(\check{\sigma}_{n+h}^2 - \sigma_{n+h}^2) =
\mathds{E}(\mathds{E}( \check{\sigma}_{n+h} ^2 -
\sigma_{t+h}^2|\mathcal{F}_{n})) = \mathds{E}( \check{\sigma}_{n+h} ^2 -
\hat\sigma_{n+h}^2) \leq 0$, for all $h>0$.  In fact, from \eqref{conv} and
\eqref{exp1}, we have
\begin{equation}\label{sigmacheckconvergence}
  \check{\sigma}_{n+h}^2:=\exp\{\hat{\ln}(\sigma_{n+h}^2)\} = \exp\bigg\{\omega +
  \sum_{k=0}^{\infty}\lambda_{d,k+h-1}\,g(Z_{n-k})\bigg\}\overset{h\to
    \infty}{-\!\!\!\longrightarrow} e^{\omega} = \exp\{\mathds{E}(\ln(\sigma_{0}^2))\} .
\end{equation}

Another $h$-step ahead predictor for $\sigma_{n+h}^2$ can be defined as follows.
Consider an order 2 Taylor's expansion of the exponential function and write
\begin{align}
  \sigma_{n+h}^2 \, = & \,
  \exp\big{\{}\mathds{E}(\ln(\sigma_{n+h}^2)|\mathcal{F}_{n})\big{\}} +
  \left[\ln(\sigma_{n+h}^2) -   \mathds{E}(\ln(\sigma_{n+h}^2)|\mathcal{F}_{n})\right]\exp\big{\{}\mathds{E}(\ln(\sigma_{n+h}^2)|\mathcal{F}_{n})\big{\}} \nonumber\\
  & \, + \frac{1}{2}\left[\ln(\sigma_{n+h}^2) -
  \mathds{E}(\ln(\sigma_{n+h}^2)|\mathcal{F}_{n})\right]^2\exp\big{\{}\mathds{E}(\ln(\sigma_{n+h}^2)|\mathcal{F}_{n})\big{\}}
  + R_{n+h}, \quad \mbox{ for all } h>0, \label{sigmaTaylor}
\end{align}
From expression \eqref{sigmaTaylor}, a natural choice is to define a $h$-step
ahead predictor for $\sigma_{n+h}^2$ as
\begin{align}\label{sigmatilde}
  \tilde \sigma_{n+h}^2 :=
  \exp\big{\{}\mathds{E}(\ln(\sigma_{n+h}^2)|\mathcal{F}_{n})\big{\}} +
  \frac{1}{2}\mathds{E}(\left[\ln(\sigma_{n+h}^2) -
  \mathds{E}(\ln(\sigma_{n+h}^2)|\mathcal{F}_{n})\right]^2)\exp\big{\{}\mathds{E}(\ln(\sigma_{n+h}^2)|\mathcal{F}_{n})\big{\}},
\end{align}
for all $h>0$.

From expressions \eqref{exp1}, \eqref{exp03}  and \eqref{sigmatilde} one concludes that $\check
\sigma_{n+h}^2$ and $\tilde \sigma_{n+h}^2$ are related through the equation
\begin{equation}\label{relation}
  \tilde \sigma_{n+h}^2  = \left\{
    \begin{array}{ccc}
      \displaystyle  \exp\{\hat\ln(\sigma_{n+h}^2)\} = \check{\sigma}_{n+h}^2, & \mbox{if} & h =1;\vspace{0.2cm}\\
      \displaystyle \exp\{\hat\ln(\sigma_{n+h}^2)\}\bigg( 1
      +\frac{1}{2}\sigma^2_g\sum_{k=0}^{h-2}\lambda_{d,k}^2\bigg) =  \check{\sigma}_{n+h}^2\bigg( 1
      +\frac{1}{2}\sigma^2_g\sum_{k=0}^{h-2}\lambda_{d,k}^2\bigg), & \mbox{if } &  h > 1.
    \end{array}
  \right.
\end{equation}

Since $\sigma_{t+1}$ is a $\mathcal{F}_{t}$-measurable random variable, for all
$t\in\mathds{Z}$, we have $\mathds{E}(\tilde{\sigma}_{n+1} ^2 - \sigma_{n+1}^2)
= \mathds{E}(\check{\sigma}_{n+1} ^2 - \sigma_{n+1}^2) = 0$.  From equation
\eqref{sigmaTaylor}, we easily conclude that, for all $h> 1$,
\[
\mathds{E}(\tilde{\sigma}_{n+h} ^2 - \sigma_{n+h}^2) = -\mathds{E}(R_{n+h})\quad
\mbox{ and} \quad \mathds{E}(\check{\sigma}_{n+h} ^2 - \sigma_{n+h}^2) = - \bigg(
1+\frac{1}{2}\sigma^2_g\sum_{k=0}^{h-2}\lambda_{d,k}^2\bigg)\mathds{E}(\check{\sigma}_{n+h}^2)  -\mathds{E}(R_{n+h}).\\
\]
Therefore, the relation between the bias for the estimators $\check
\sigma_{n+h}^2$ and $\tilde \sigma_{n+h}^2$ is given by
\[
\mathds{E}(\tilde{\sigma}_{n+h} ^2 - \sigma_{n+h}^2) =
\mathds{E}(\check{\sigma}_{n+h} ^2 - \sigma_{n+h}^2) +
\mathds{E}(\check{\sigma}_{n+h}^2)\bigg( 1
+\frac{1}{2}\sigma^2_g\sum_{k=0}^{h-2}\lambda_{d,k}^2\bigg), \quad \mbox{for all
} h > 1.
\]

In Section \ref{simulationsection} we analyze the performance of $\tilde
\sigma_{n+h}^2$ through a Monte Carlo simulation study.

\section{Simulation Study}\label{simulationsection}

In this section we present a Monte Carlo simulation study to analyze the
performance of quasi-likelihood estimator and also the forecasting on
FIEGARCH$(p,d,q)$ processes.     Six different models are considered and,
from now on,  they  shall be  referred to as model M$i$, for $ i\in
\{1,\cdots,6\}$.  For all models we assume that the distribution of  $Z_0$ is
the Generalized Error Distribution (GED) with  tail-thickness parameter $\nu =
1.5$ (since $\nu < 2$ the tails are heavier than the Gaussian distribution).   The set of parameters considered in this study is
the same as in \cite{PRLO1} and \cite{PRLO2}\footnote{\cite{PRLO1} present a Monte Carlo simulation study on risk
measures estimation in time series derived from FIEGARCH process.  \cite{PRLO2}
analyze a portfolio composed by stocks from the Brazilian market Bovespa. The
authors consider the econometric approach to estimate the risk measure VaR and
use FIEGARCH models to obtain the conditional variance of the time series.}, except for models M5 and M6 (see Table
\ref{simpar}).  While model M5 considers $d = 0.49$,  which is close to the
non-stationary region ($d\geq 0.5$), model M6 considers $p=1$ and $q = 0$.   For
comparison, we shall consider for model M6 the same parameter values as
in model M3 (obviously,  with the necessary adjustments regarding $\alpha_1$
and $\beta_1$).
 We also present here the $h$ step-ahead forecast, for $h \in \{1,\cdots, 50\}$, for
the conditional variance of simulated FIEGARCH processes.

\subsection{Data Generating Process}

To generate samples from FIEGARCH$(p,d,q)$ processes we proceed as described
in steps {\bf DGP1} - {\bf DGP3}  below.  Notice that, while step 1 only needs to be repeated for
each model, steps 2 and 3 must be repeated for each model and each replication.
 The parameters value consider in this simulation study   are given in
Table \ref{simpar}.  For each model we consider $re = 1,000$ replications,   with sample size
   $N = 5,050$.

\begin{table}[!htb]
  \renewcommand{\arraystretch}{1.1}
  \centering
  \caption{ Parameters value for the models. By definition, M1:=
    FIEGARCH$(2,d,1)$; M2 := FIEGARCH$(0,d,4)$; M3 := FIEGARCH$(0,d,1)$; M4 :=
    FIEGARCH$(0,d,1)$,  M5 := FIEGARCH$(1,d,1)$ and  M6 :=
    FIEGARCH$(1,d,0)$.  }   \label{simpar}\vspace{0.2cm}
  {\footnotesize
    \begin{tabular*}{1\textwidth}{@{\extracolsep{\fill}} ccccccccccc}
      \hline
      \multirow{2}{*}{Model}& \multicolumn{10}{c}{\phantom{\Big{|}} Parameter}\\
      \cline{2-11}
      & $d$\phantom{\Big{|}} & $\theta$ & $\gamma$ & $\omega$ & $\alpha_1$& $\alpha_2$  & $ \beta_1$ & $ \beta_2$ & $ \beta_3$ & $\beta_4$
      \\
      \hline
      \vspace{-0.2cm}\\
      M1  & 0.4495 & -0.1245 & 0.3662 &  -6.5769 & -1.1190 &-0.7619& -0.6195&-&-&-\\
      M2& 0.2391 & -0.0456 & 0.3963 & -6.6278 & -& - & \ 0.2289 & 0.1941 & 0.4737 & -0.4441\\
      M3&  0.4312 &  -0.1095  & 0.3376 & -6.6829 &- & -&\ 0.5454&-&-&-\\
      M4 &  0.3578 & -0.1661 &0.2792 & -7.2247 & - & -& \
      0.6860&-&-&-\\
      M5 &  0.4900 & -0.0215 & 0.3700 & -5.8927 & 0.1409 & -& \
      -0.1611&-&-&-\\
       M6&  0.4312 &  -0.1095  & 0.3376 & -6.6829 &0.5454 & -& -&-&-&-\vspace{0.1cm}\\
      \hline
    \end{tabular*}
  }
\end{table}

  \vspace{1\baselineskip}
\noindent {\bf DGP1:}     Apply  the  recurrence formula given in Proposition
\ref{coefficientsfiegarch},  to  obtain the  coefficients  of the
polynomial $\lambda(z) = \sum_{k = 0}^\infty\lambda_{d,k}z^k$,  defined by \eqref{lambdapoly}.   For this simulation study the infinite sum \eqref{lambdapoly} is truncated at $m
= 50,000$.  To select the truncation point $m$ we consider Theorem
\ref{convorder} and the results presented in Table \ref{lambdacoeffsim}.

   From  Theorem \ref{convorder},  we have,
\[
\lambda_{d,k} \sim \frac{1}{\Gamma(d)k^{1-d}}\frac{\alpha(1)}{\beta(1)}, \quad
\mbox{as } k\rightarrow \infty,
\]
and we conclude that $\lambda_{d,k}=o(k^d)$ and $\lambda_{d,k}=O(k^{d-1})$, as
$k$ goes to infinity.   However, the speed of the convergence varies from model to
model, as we show in Table \ref{lambdacoeffsim}.  For simplicity,  in this
table, let
$Q_1(\cdot)$ and $Q_2(\cdot)$ be defined as
\[
Q_1(k) := \frac{\lambda_{d,k}}{k^d} \quad \mbox{ and } \quad Q_2(k) :=
\lambda_{d,k}\bigg(\frac{1}{\Gamma(d)k^{1-d}}\frac{\alpha(1)}{\beta(1)}\bigg)^{-1},
\quad \mbox{for all } k > 0.
\]

Table \ref{lambdacoeffsim} presents the values of the coefficients
$\lambda_{d,k}$, given in Proposition \ref{coefficientsfiegarch}, for $k \in
\{0; 10;$ 100; 1,000; 5,000; 10,000; 20,000; 50,000; 100,000$\}$, for each
simulated model M$i$, $i \in \{1,\cdots,6\}$. Note that, for $k\geq$ 5,000, the
coefficient values decrease slowly.  We also report in Table
\ref{lambdacoeffsim} $Q_1(k)$ and $Q_2(k)$ values for the correspondent
$\lambda_{d,k}$ value.  Note that, for $k \in \{$10,000; 50,000; 100,000$\}$,
the value $Q_1(k)$ is very close to zero, for all models. Also notice that,
while $Q_2(k)$ converges to 1 faster for model M1 than for the other models.

\begin{table}[!htb]
      \renewcommand{\arraystretch}{1.1}
  \centering
  \caption{ Coefficients $\lambda_{d,k}$ and the quotients $Q_1(k)$
    and $Q_2(k)$, for different values of $k$, for all models.} \label{lambdacoeffsim} \vspace{0.2cm}
  {\footnotesize
    \begin{tabular*}{1\textwidth}{@{\extracolsep{\fill}} lccccccccc}
      \hline
      \phantom{\Big{|}}$k$  & 0 &	10&	100&	1,000&	5,000&	10,000&	25,000&	 50,000&	100,000\\
      \hline
      \vspace{-0.2cm}\\
      \multicolumn{10}{l}{M1 :=   FIEGARCH$(2,d,1)$}\\
      $\lambda_{d,k}$ &1&	0.26537&	0.07167&	0.02015&	0.00830&	0.00567&	 0.00342&	0.00234&	0.00160\\
      $Q_1(k)$&- &	0.09426&	0.00904&	0.00090&	0.00018&	0.00009&	0.00004&	 0.00002&	0.00001\\
      $Q_2(k)$&-&	1.04410&	1.00173&	1.00017&	1.00003&	1.00002&
      1.00001&	1.00000&	1.00000\vspace{0.1cm}\\

      \multicolumn{10}{l}{M2 :=   FIEGARCH$(0,d,4)$}\\
      $\lambda_{d,k}$ &1&	-0.09039&	0.01450&	0.00251&	0.00074&	0.00043&	 0.00022&	0.00013&	0.00008\\
      $Q_1(k)$&- &	-0.05212&	0.00482&	0.00048&	0.00010&	0.00005&	0.00002&	 0.00001&	0.00000\\
      $Q_2(k)$&-&	-1.08434&	1.00292&	1.00027&	1.00005&	1.00003&
      1.00001&	1.00001&	1.00000\vspace{0.1cm}\\

      \multicolumn{10}{l}{M3 :=   FIEGARCH$(0,d,1)$}\\
      $\lambda_{d,k}$ &1&	0.31434&	0.07844&	0.02106&	0.00843&	0.00568&	 0.00337&	0.00227&	0.00153\\
      $Q_1(k)$&- &	0.11647&	0.01077&	0.00107&	0.00021&	0.00011&	0.00004&	 0.00002&	0.00001\\
      $Q_2(k)$&-&	1.08789&	1.00576&	1.00056&	1.00011&	1.00006&
      1.00002&	1.00001&	1.00001\vspace{0.1cm}\\

      \multicolumn{10}{l}{M4 :=   FIEGARCH$(0,d,1)$}\\
      $\lambda_{d,k}$ &1&	0.36874&	0.06738&	0.01517&	0.00539&	0.00345&	 0.00192&	0.00123&	0.00079\\
      $Q_1(k)$&- &	0.16178&	0.01297&	0.00128&	0.00026&	0.00013&	0.00005&	 0.00003&	0.00001\\
      $Q_2(k)$&-&	1.26414&	1.01350&	1.00129&	1.00026&	1.00013&
      1.00005&	1.00003&	1.00001\vspace{0.1cm}\\

      \multicolumn{10}{l}{M5 :=   FIEGARCH$(1,d,1)$}\\
      $\lambda_{d,k}$ &1&	0.12291&	0.03897&	0.01207&	0.00531&	0.00373&	 0.00234&	0.00164&	0.00115\\
      $Q_1(k)$&-&	0.03977&	0.00408&	0.00041&	0.00008&	0.00004&	0.00002&	 0.00001&	0.00000\\
      $Q_2(k)$&-&	 0.97189&	0.99720&	0.99972&	0.99994&	0.99997&	0.99999&	 0.99999&	1.00000\vspace{0.1cm}\\

      \multicolumn{10}{l}{M6 :=   FIEGARCH$(1,d,0)$}\\
  $\lambda_{d,k}$  &  1	&0.05472&	0.01599&	0.00435&	0.00174&	0.00117&	0.00070&	0.00047&	0.00032\\
 $Q_1(k)$&-&	0.02027&	0.00219&	0.00022&	0.00004&	0.00002&	0.00001&	0.00000&	0.00000\\
  $Q_2(k)$&-&	0.91632&	0.99192&	0.99919&	0.99984&	0.99992&
  0.99997&	0.99998 &	0.99999\vspace{0.1cm}\\
  \hline
    \end{tabular*}
  }
\end{table}

  \vspace{1\baselineskip}
\noindent {\bf DGP2:}  Set    $Z_{0}\sim \mathrm{GED}(\nu)$,   with $\nu = 1.5$,  and obtain an
     i.i.d. sample $\{z_{t}\}_{t = -m}^N$.

  \vspace{1\baselineskip}
\noindent {\bf DGP3:}  By considering   Definition \ref{definitionfie} and the
equality in \eqref{lambdapoly},  the sample $\{x_t\}_{t=1}^n$ is obtained
through the relation
\[
\ln(\sigma_t^2) = \sum_{k=0}^m\lambda_{d,k}g(z_{t-1-k})  \quad \mbox{and} \quad x_t = \sigma_tz_t,  \quad \mbox{for all } t=1,\cdots, N.
\]

\begin{remark}
 For parameter  estimation and forecasting   we shall consider sub-samples from
these time series,   with size $n\in \{2,000; 5,000\}$.  The sub-samples of
size $n = 2,000$ correspond to the last 2,000 values of the generated time
series (after removing the last 50 values which are used only  to compare the
out-of-sample forecasting performance of the models).   The value $n=2,000$ is the approximated size of the observed time series considered in
\cite{PRLO2}. The value $n = 5,000$ was chosen to analyze the estimators
asymptotic properties.
\end{remark}

\subsection{Estimation Procedure}
In this study we consider the quasi-likelihood method to estimate the parameters
of FIEGARCH models for the simulated time series.  Given any time series
$\{x_t\}_{t=1}^{n}$, this method assumes that $X_{t}|\mathcal{F}_{t-1}$, for all
$t\in\mathds{Z}$, is normally distributed. The vector of unknown parameters is
denoted by
\[
\boldsymbol{\eta}=(d;\omega;\theta;\lambda;\alpha_1,\cdots,\alpha_p;\beta_1,\cdots,\beta_q)'\
\in\mathds{R}^{p+q+4}
\]
and the estimator $\hat{\boldsymbol{\eta}}$ of $\boldsymbol{\eta}$ is the value
that maximizes
\begin{equation}
  \ln(\ell(\boldsymbol\eta; x_1, \cdots, x_n))=-\frac{n}{2}\ln(2\pi)-\frac{1}{2}
  \sum_{t=1}^{n}\left[\ln(\sigma_t^2)+\frac{x_t^2}{\sigma_t^2}\right], \!.\label{pseudolike}
\end{equation}

Since the processes $\{x_t\}_{t<1}$ and $\{z_t\}_{t<1}$ are unknown, we need to
consider a set $I_0$ of initial conditions in order to start the recursion and
to obtain the random variable $\ln(\sigma_t^2)$, for $t \in \{1, \cdots,
n\}$. Then we use these estimated values to solve \eqref{pseudolike}.  For this
simulation study we assume, as initial conditions, $ g(z_t) = 0$, $ \sigma_t^2 =
\hat{\sigma}_X^2$ and $x_t := \sigma_tz_t = 0$, whenever $ t<1$, where
$\hat{\sigma}_X^2$ is the sample variance of $\{x_t\}_{t=1}^{n}$.  This is the
initial set suggested by \cite{BOMI96}.  The random variables $\ln(\sigma_t^2)$,
for $t \in \{1, \cdots, n\}$, are then estimated upon considering the set $I_0$
of initial conditions and the known values $\{x_t\}_{t=1}^{n}$. The infinite sum
in the polynomial $\lambda(\cdot)$ is truncated at $m = n$, where $n$ is the
available sample size.

 \subsection{Performance Measures}

For any model, let  $\hat{\eta}_k$ denotes the estimate of  $\eta$  in the
$k$-th replication, where $k \in \{1,\cdots,re\}$, $re = 1,000$ and  $\eta$ is
any vector parameter given in Table \ref{simpar}.   To access the performance of
quasi-likelihood  procedure we  calculate the
mean $\bar{\eta}_ i$, the standard deviation ($sd$), the bias ($bias$), the mean
absolute error ($mae$) and the mean square error ($mse$) values,
 defined by
\[
\bar{\eta} := \frac{1}{re}\sum_{k=1}^{re} \hat{\eta}_k, \hspace{5pt} sd :=
\sqrt{\frac{1}{re}\sum_{k=1}^{re}(\hat{\eta}_k - \bar{\eta})^ 2}, \hspace{5pt}
bias := \frac{1}{re}\sum_{k=1}^{re}e_k, \hspace{5pt} mae :=
\frac{1}{re}\sum_{k=1}^{re}|e_k|, \hspace{5pt} \mbox{and} \hspace{5pt} mse
:=\frac{1}{re}\sum_{k=1}^{re}e^2_k,
\]
where $e_k:=\hat{\eta}_k -\eta$, for $k \in \{1,\cdots,re\}$.

\subsection{Estimation Results}

 Table \ref{resultssim} summarizes the results on the parameter estimation
 procedure.  Figures \ref{figm1} - \ref{figm6} present the kernel distribution
 of the  parameter estimators for each considered model when
 $n\in\{2,000;5,000\}$. These  graphs help to illustrate the results presented
 in Table \ref{resultssim}.

\begin{figure}[!htb]
   \vspace{-0.2cm} \centering
  \includegraphics[width = 0.2\textwidth]{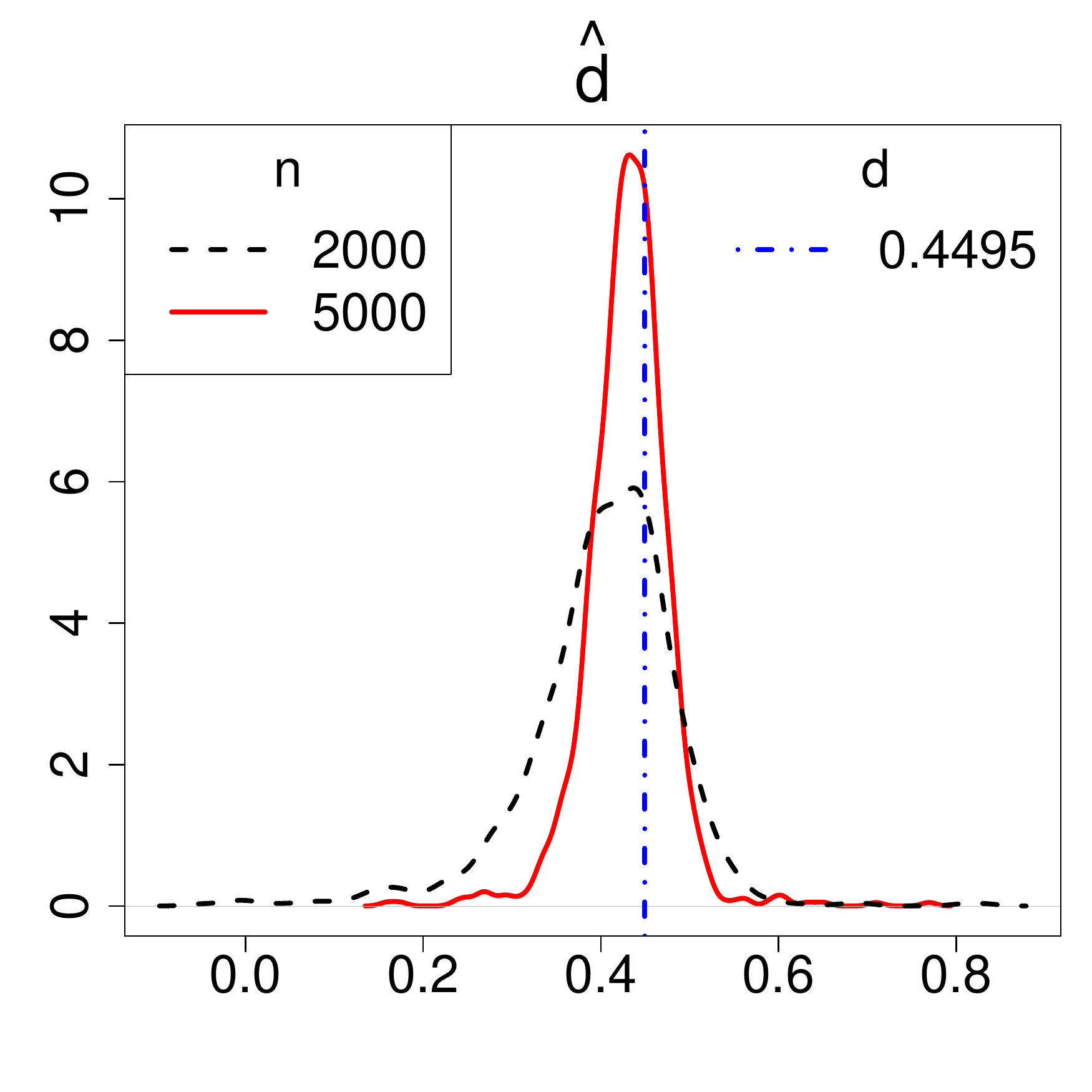}
  \includegraphics[width = 0.2\textwidth]{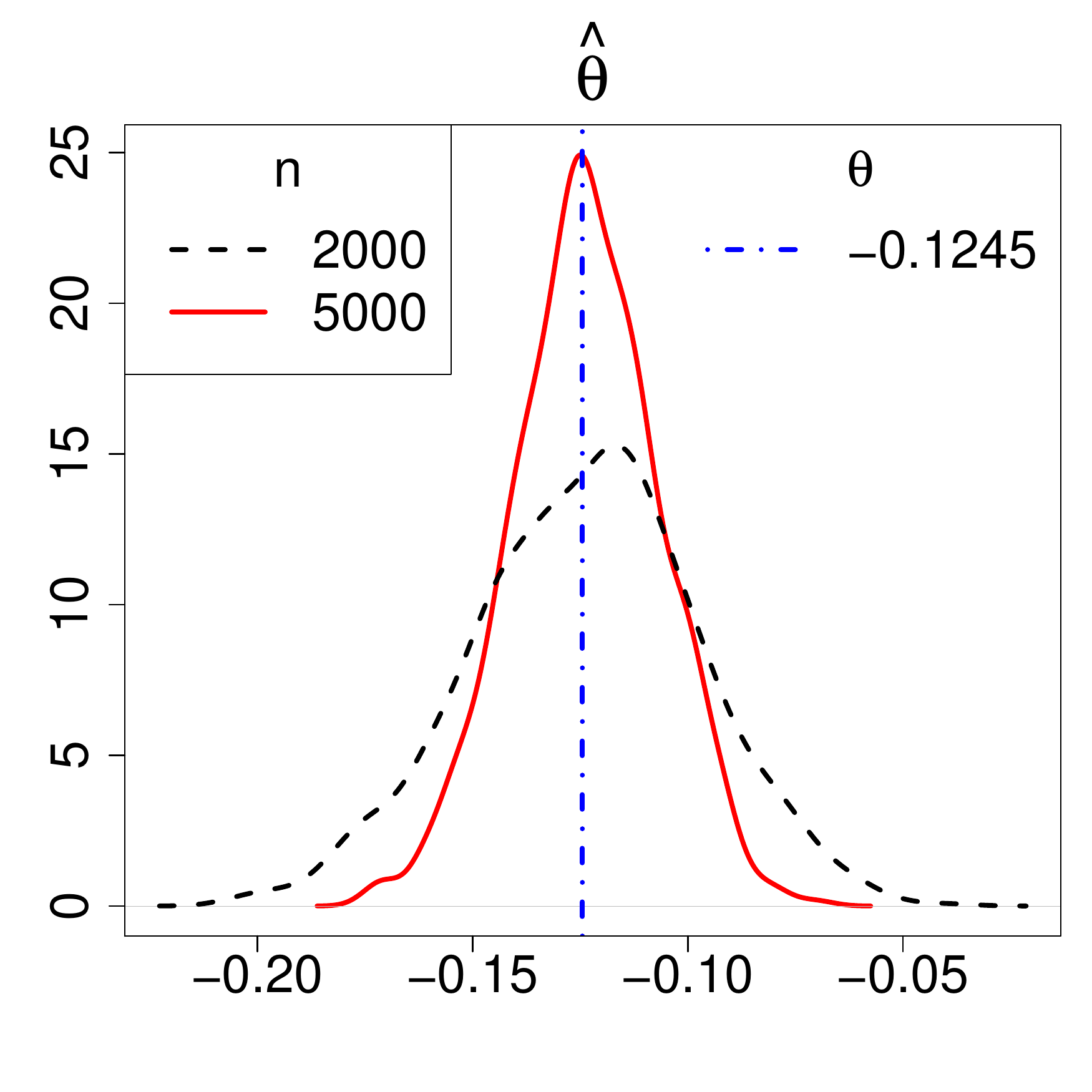}
  \includegraphics[width = 0.2\textwidth]{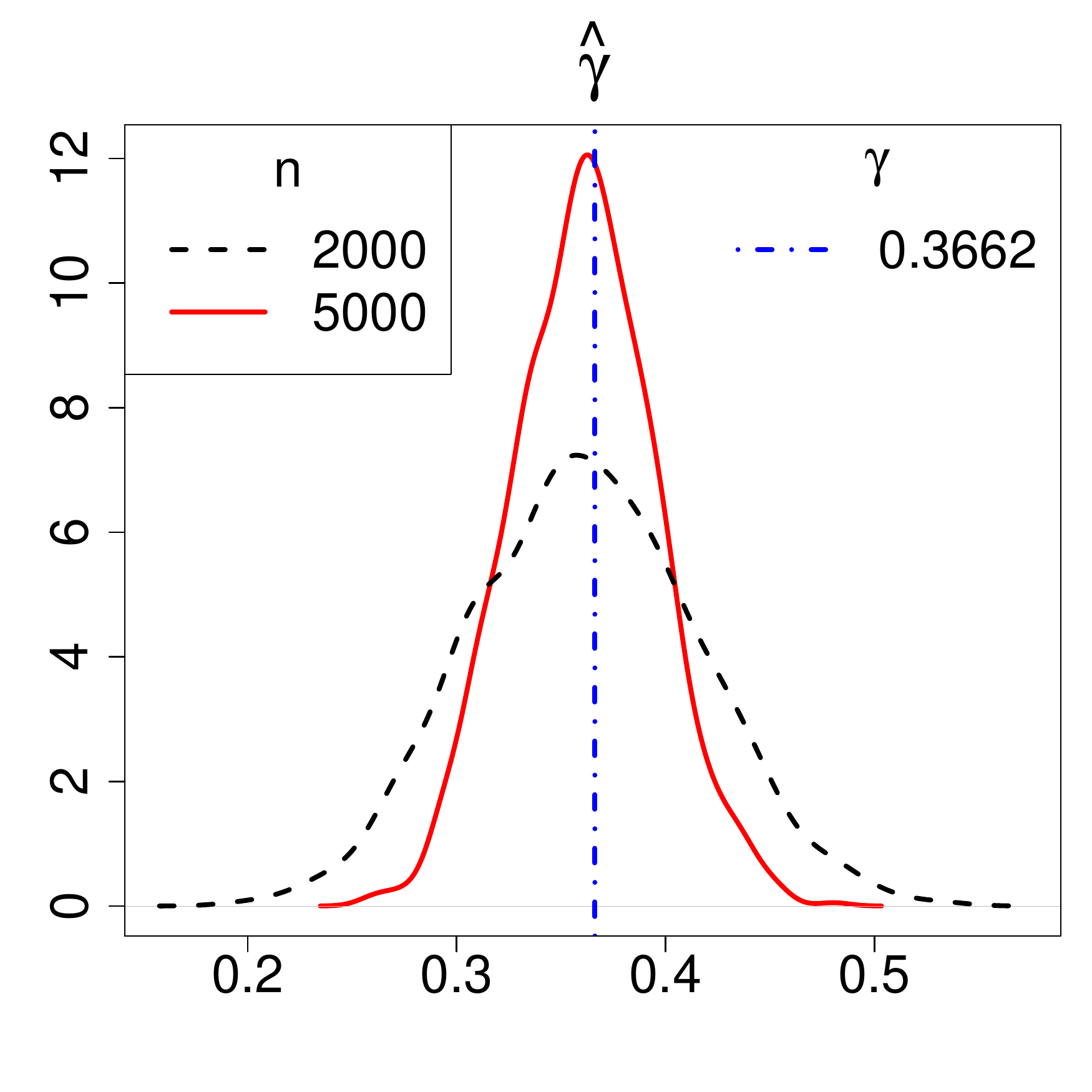}
  \includegraphics[width =   0.2\textwidth]{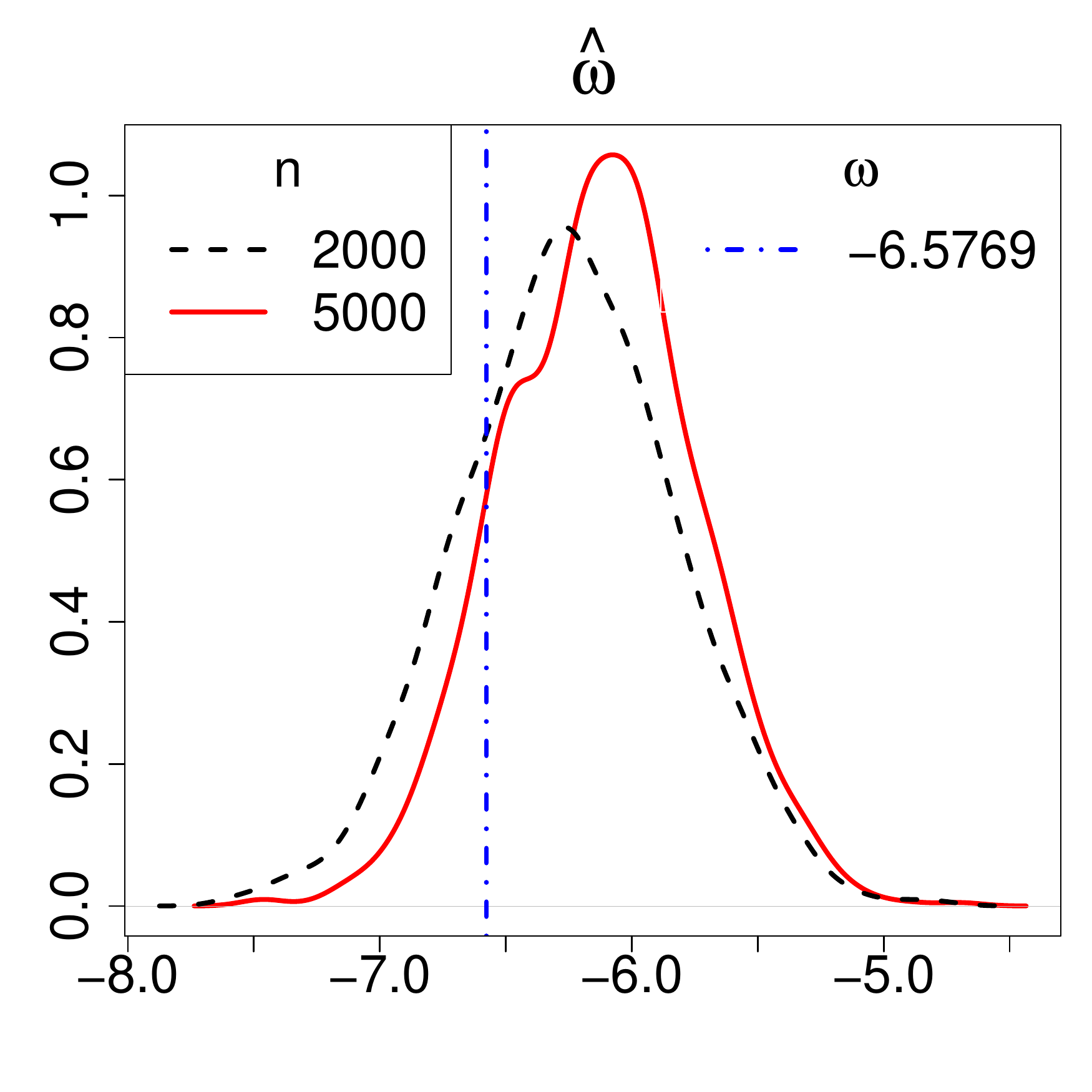}\\
  \includegraphics[width = 0.2\textwidth]{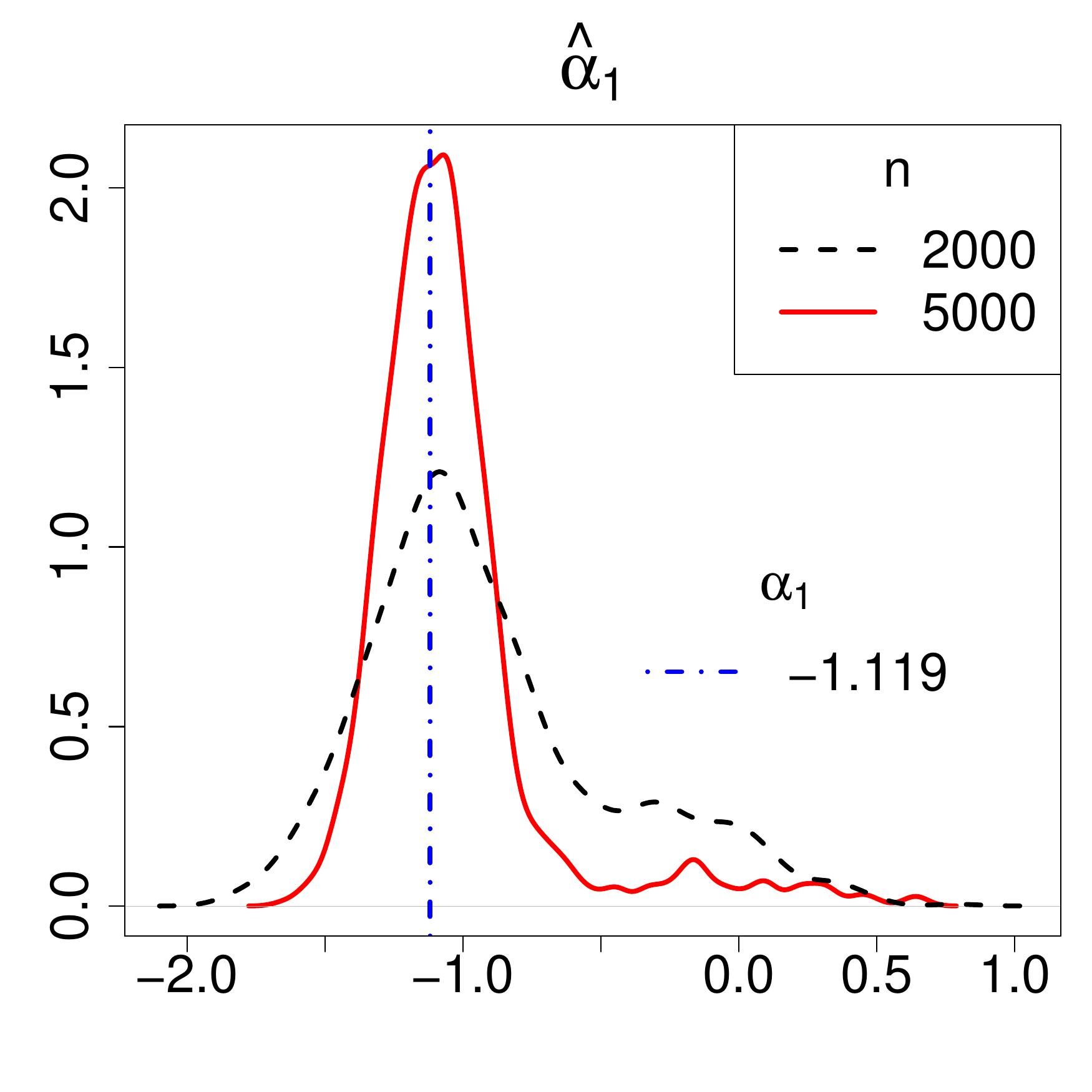}
  \includegraphics[width = 0.2\textwidth]{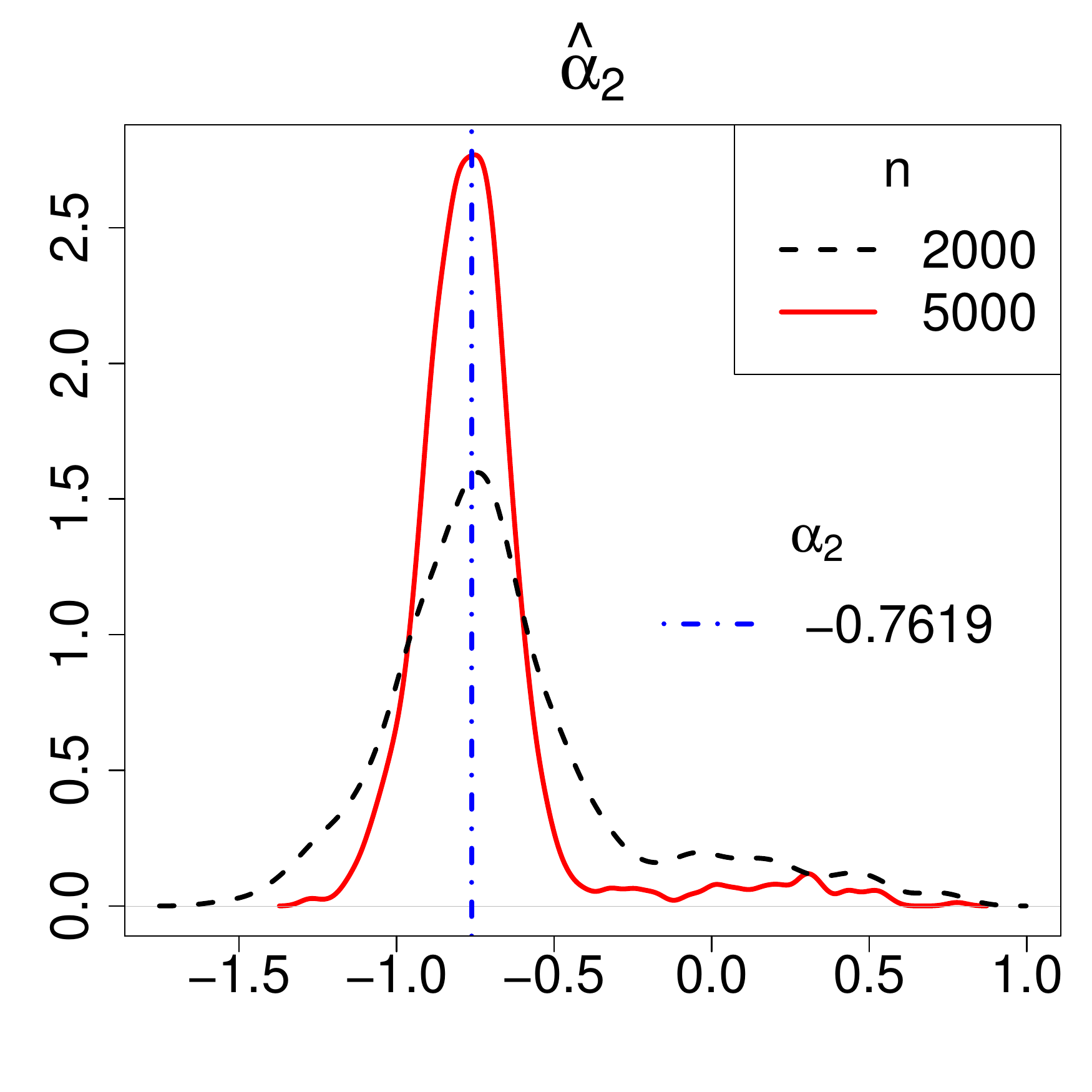}
  \includegraphics[width = 0.2\textwidth]{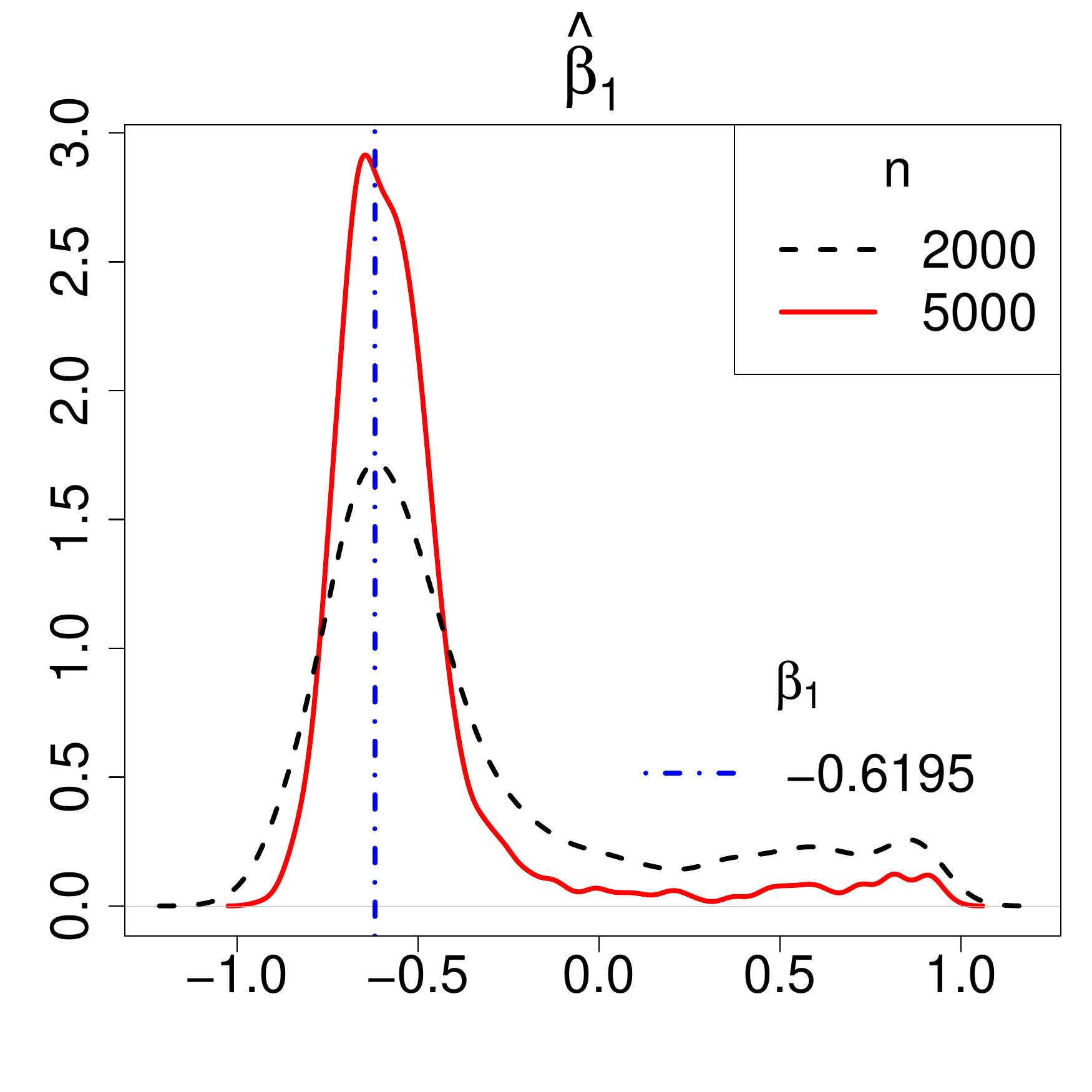}\vspace{-0.2cm}

   \caption{Kernel density function of the estimates for model M1, for $n \in
     \{2,000;\, 5,000\}$. }  \label{figm1}
 \end{figure}

 \begin{figure}[!htb]
  \vspace{-0.2cm} \centering
  \includegraphics[width = 0.2\textwidth]{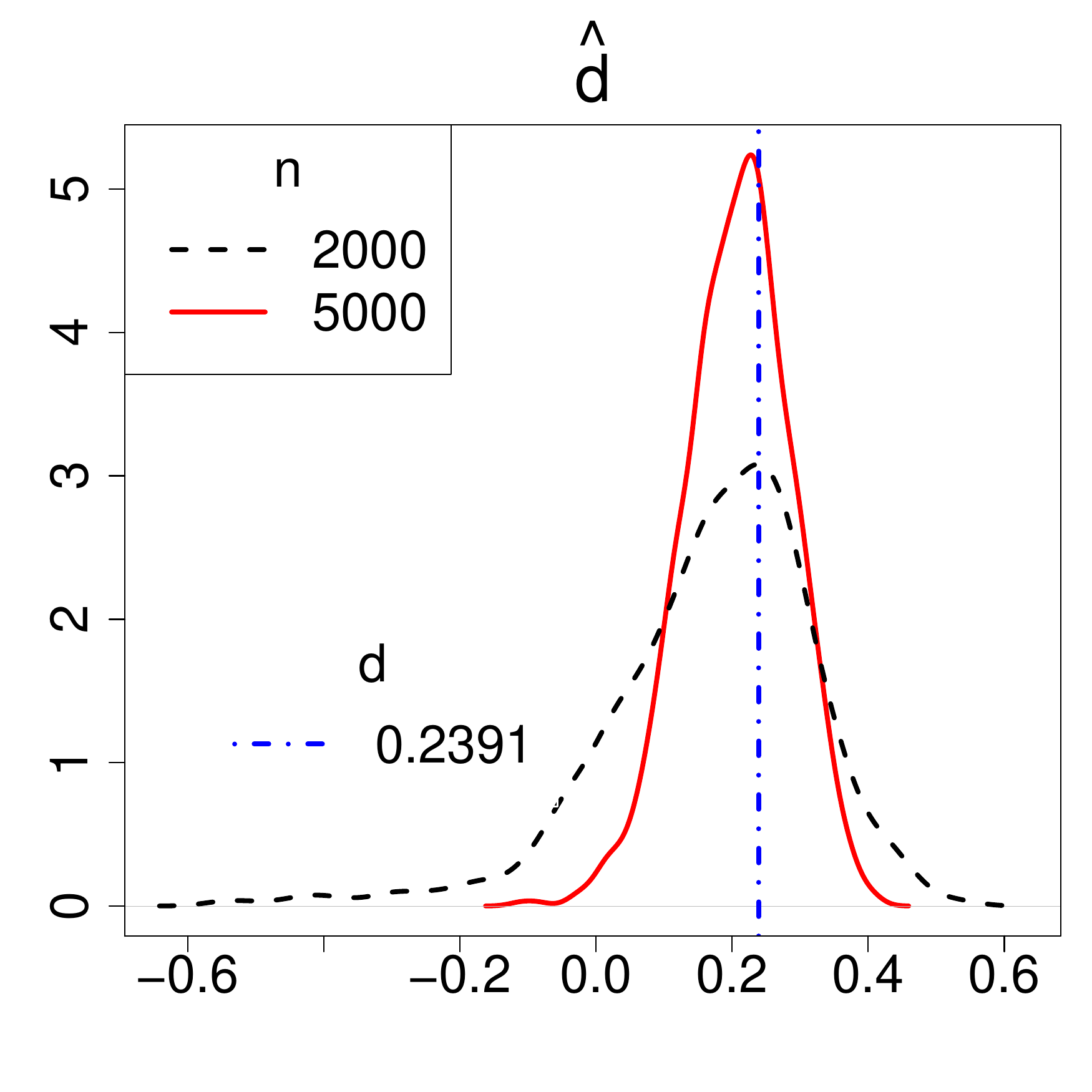}
  \includegraphics[width = 0.2\textwidth]{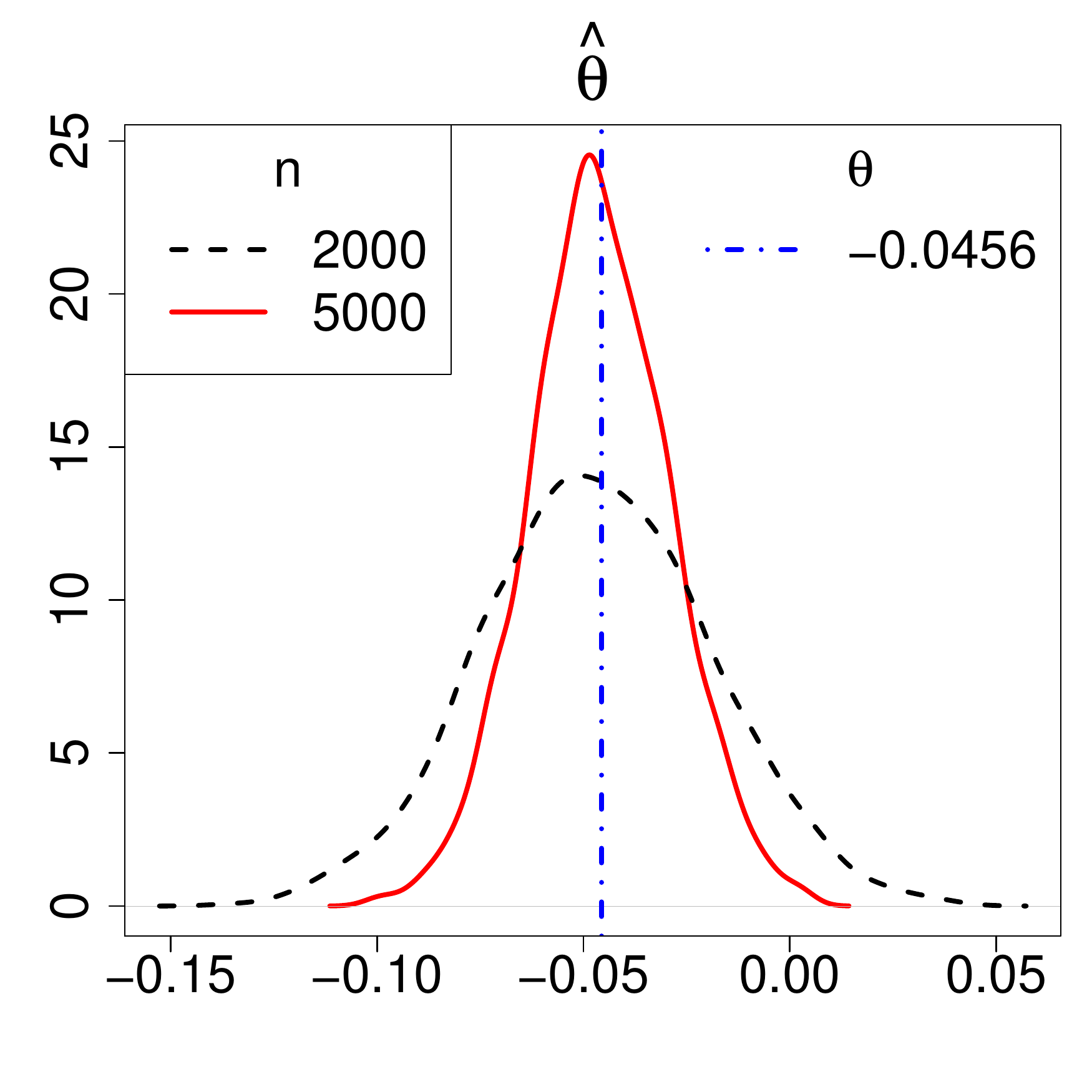}
  \includegraphics[width =0.2\textwidth]{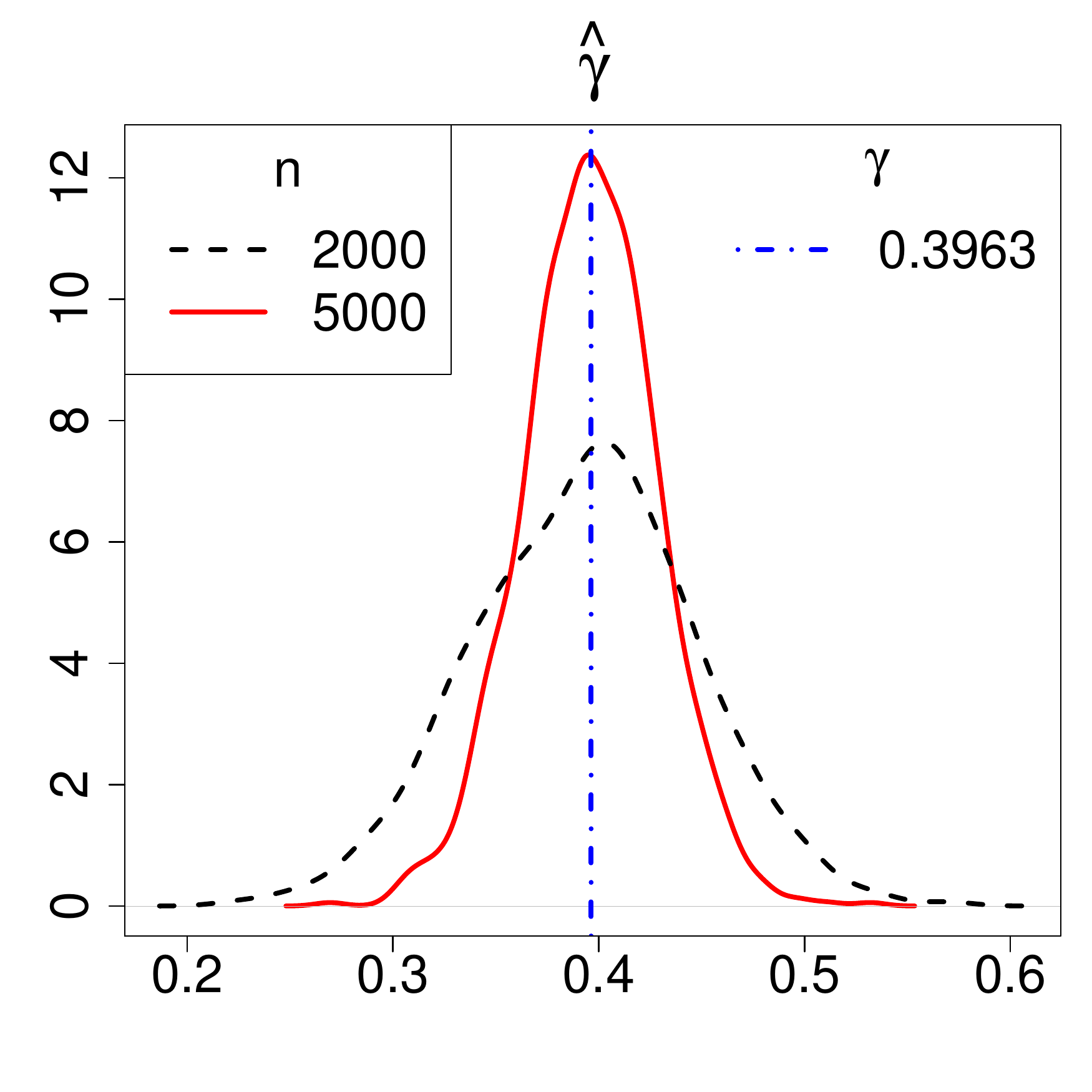}
  \includegraphics[width = 0.2\textwidth]{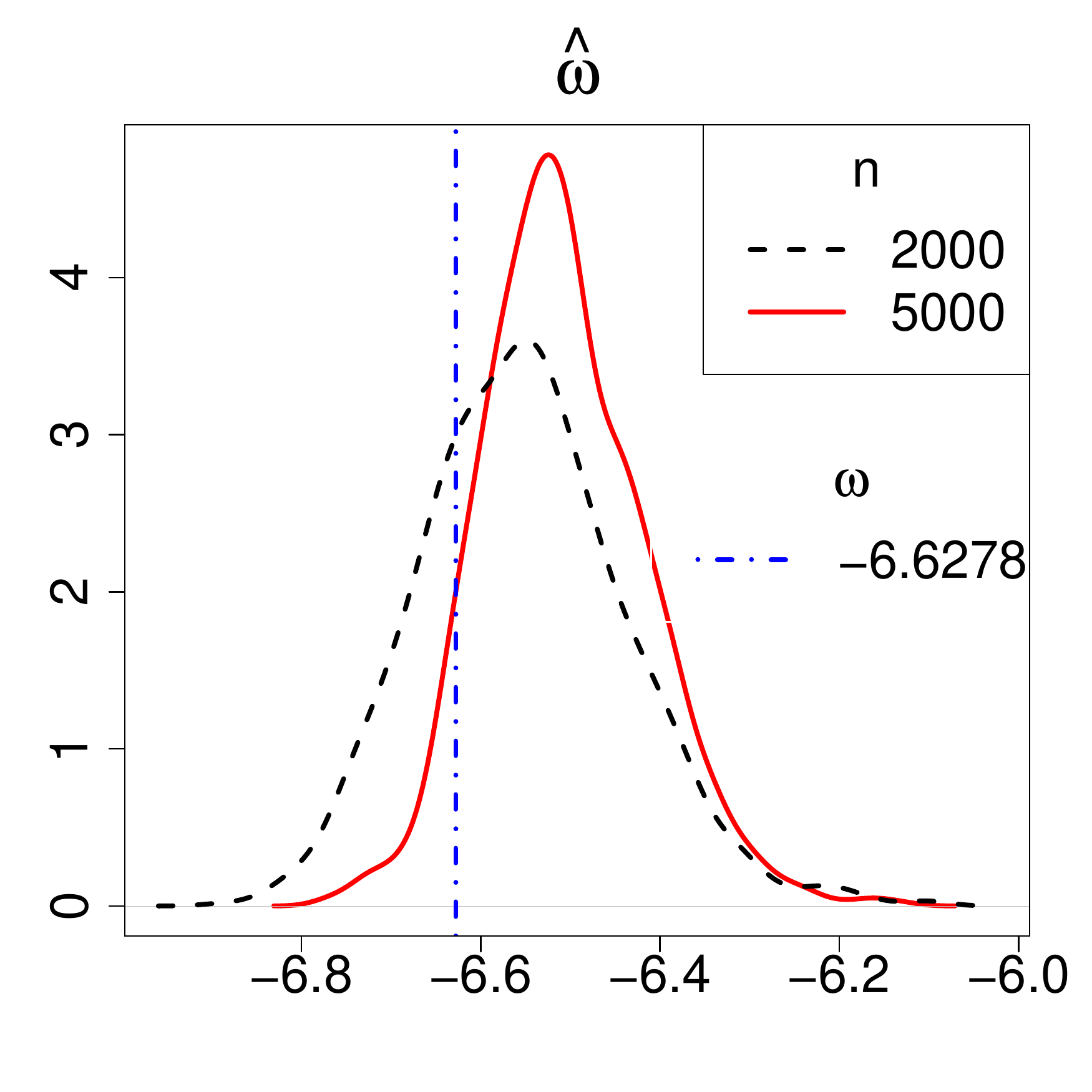} \\
  \includegraphics[width = 0.2\textwidth]{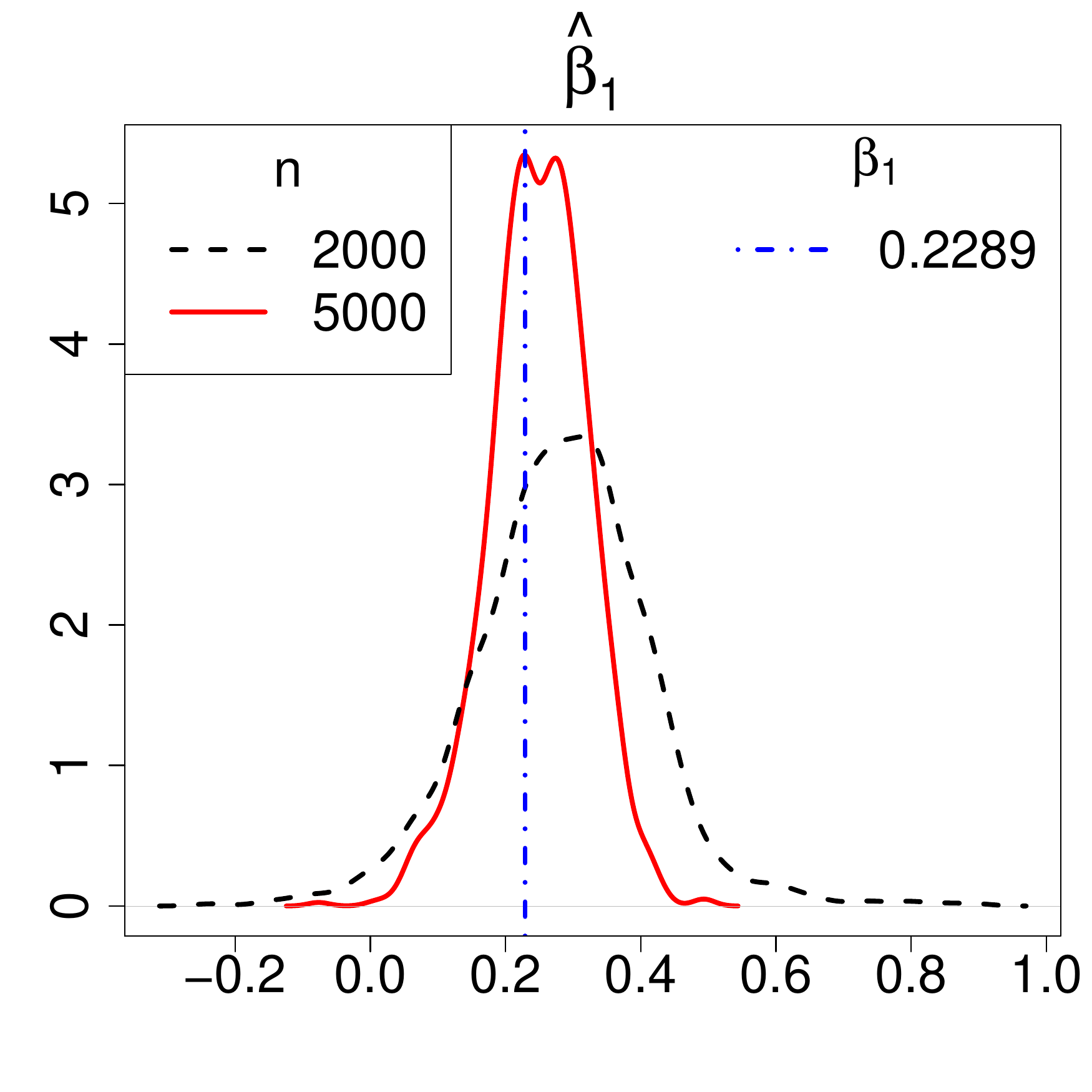}
  \includegraphics[width = 0.2\textwidth]{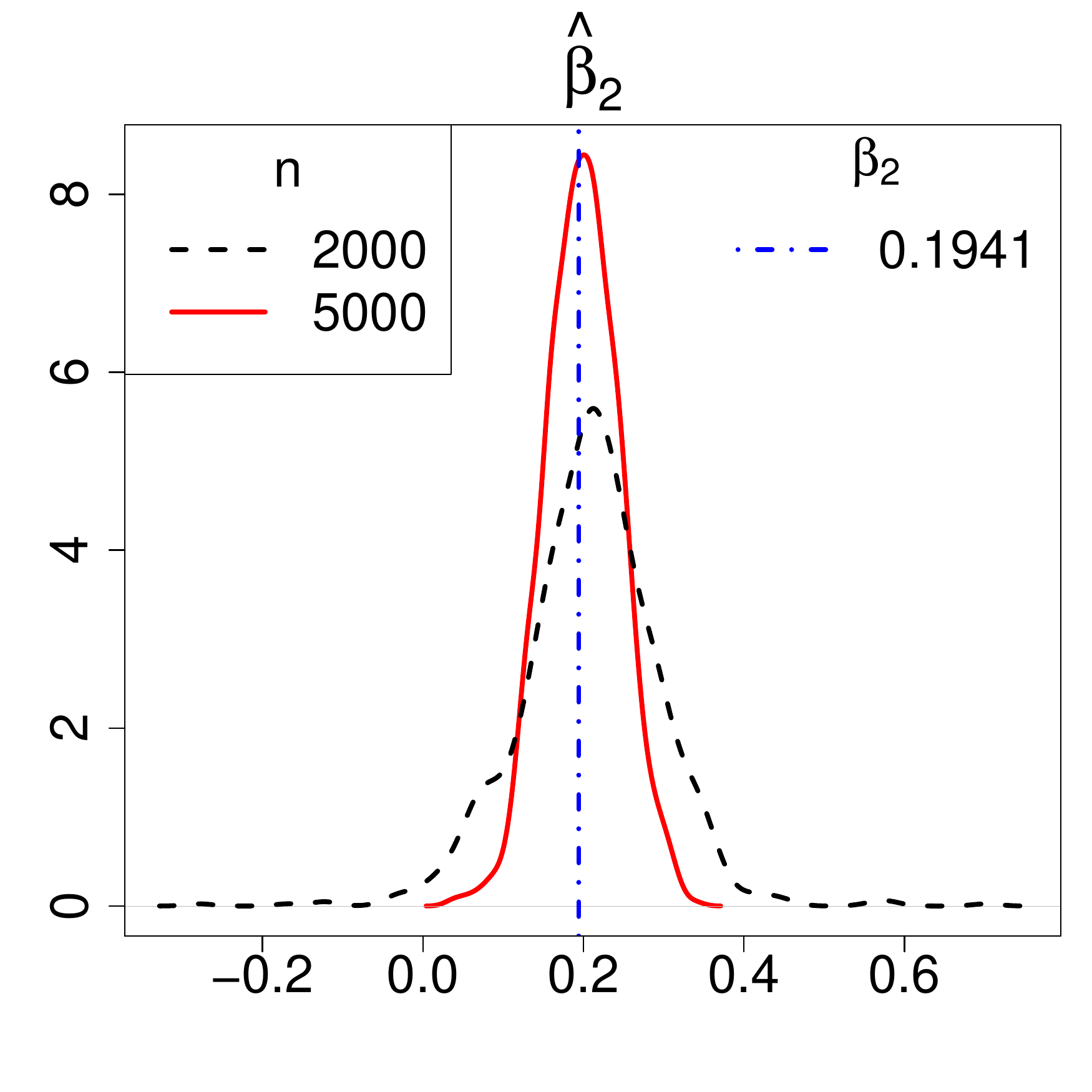}
  \includegraphics[width = 0.2\textwidth]{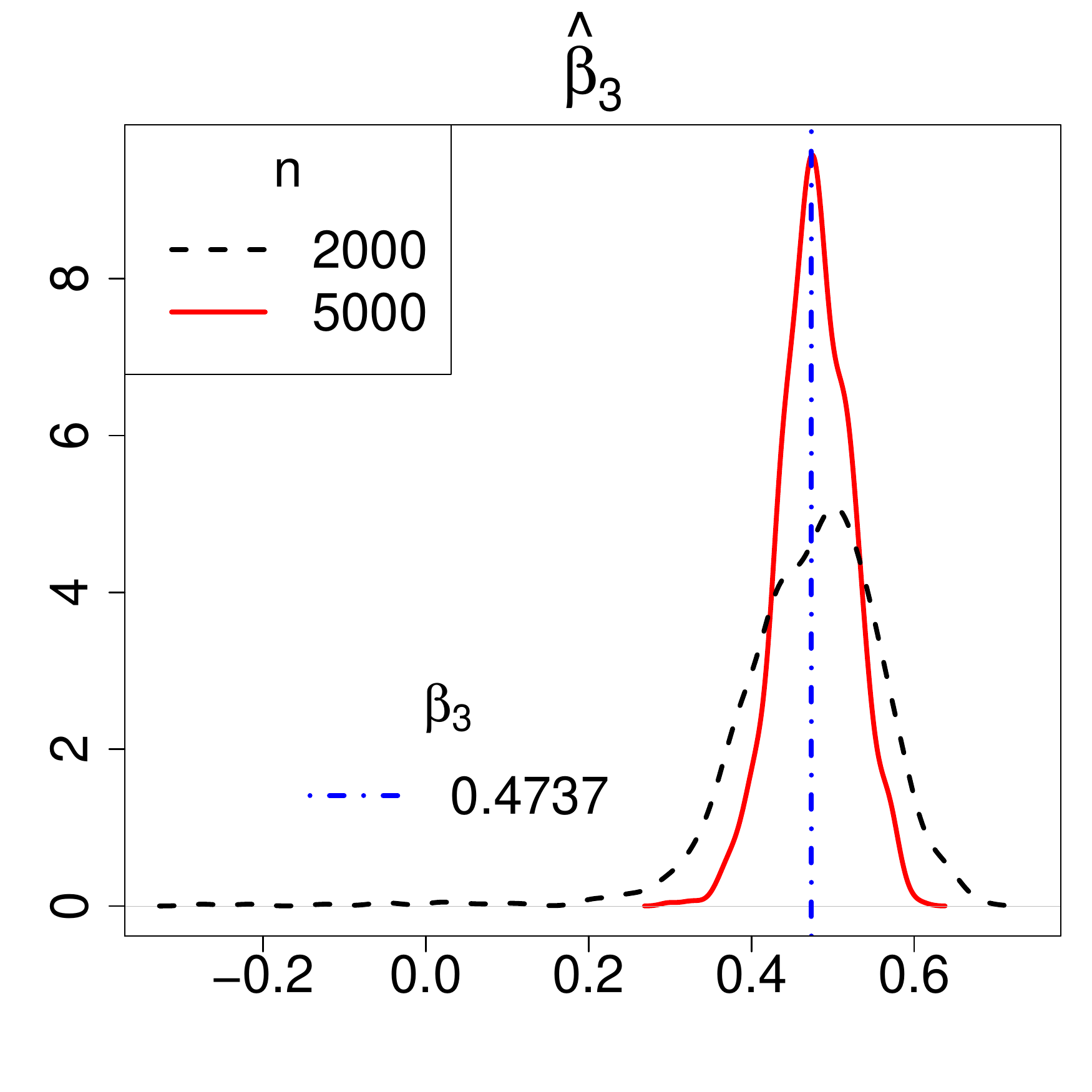}
  \includegraphics[width = 0.2\textwidth]{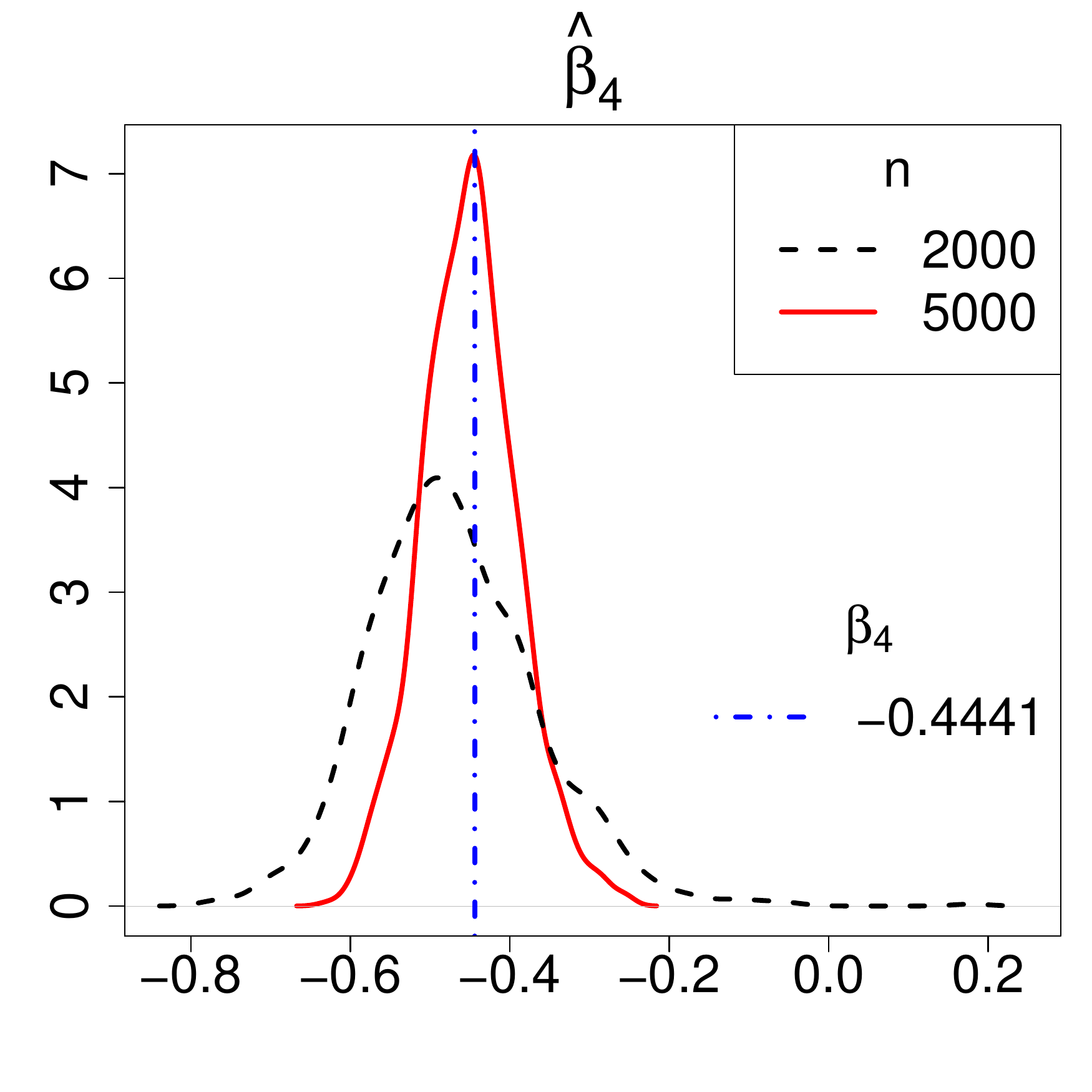}\vspace{-0.2cm}

   \caption{Kernel density function of the estimates for model M2, for $n \in
     \{2,000;\, 5,000\}$. }
 \end{figure}

 \begin{figure}[!htb]
  \vspace{-0.2cm} \centering
 \includegraphics[width = 0.2\textwidth]{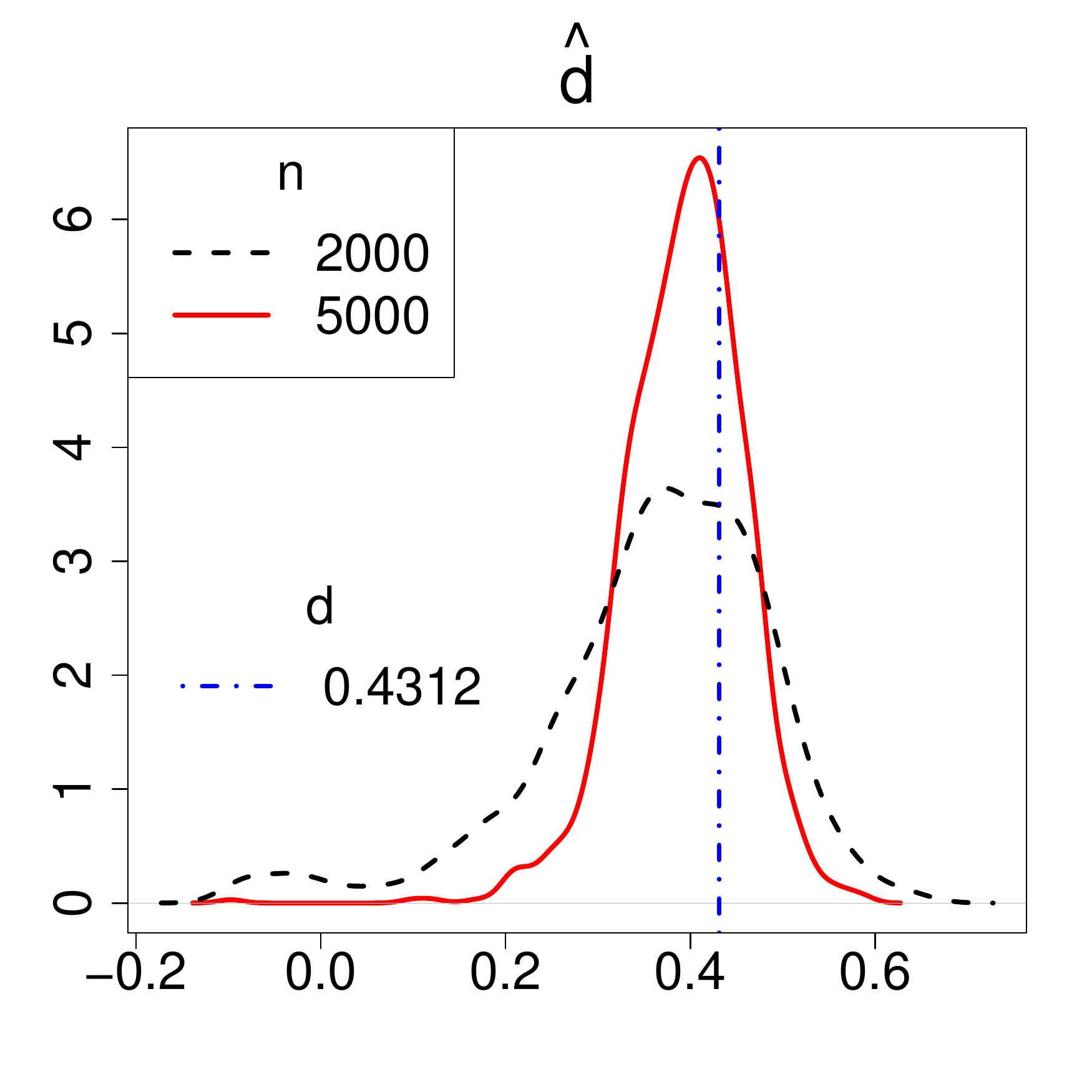}\!\!
  \includegraphics[width = 0.2\textwidth]{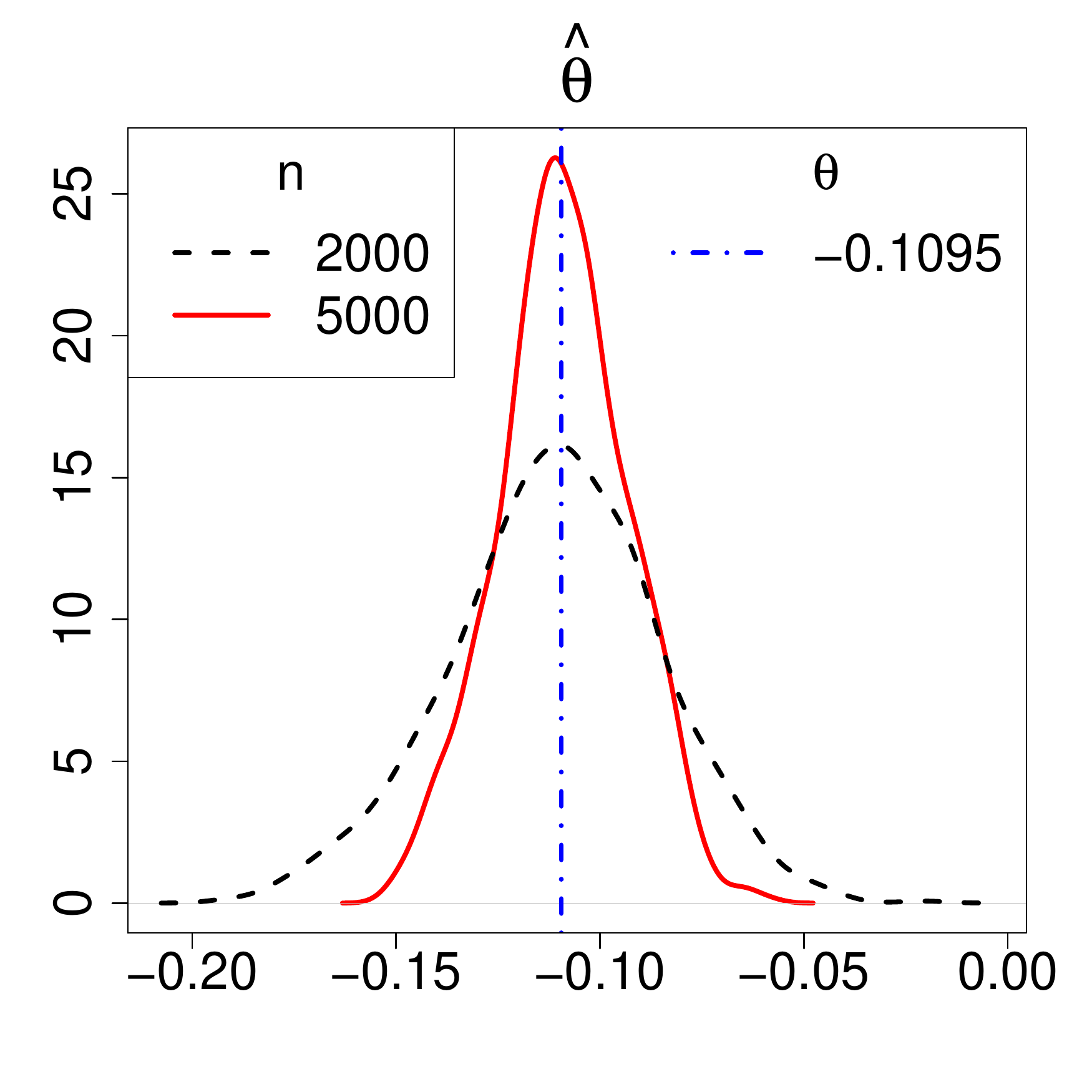}\!\!
  \includegraphics[width = 0.2\textwidth]{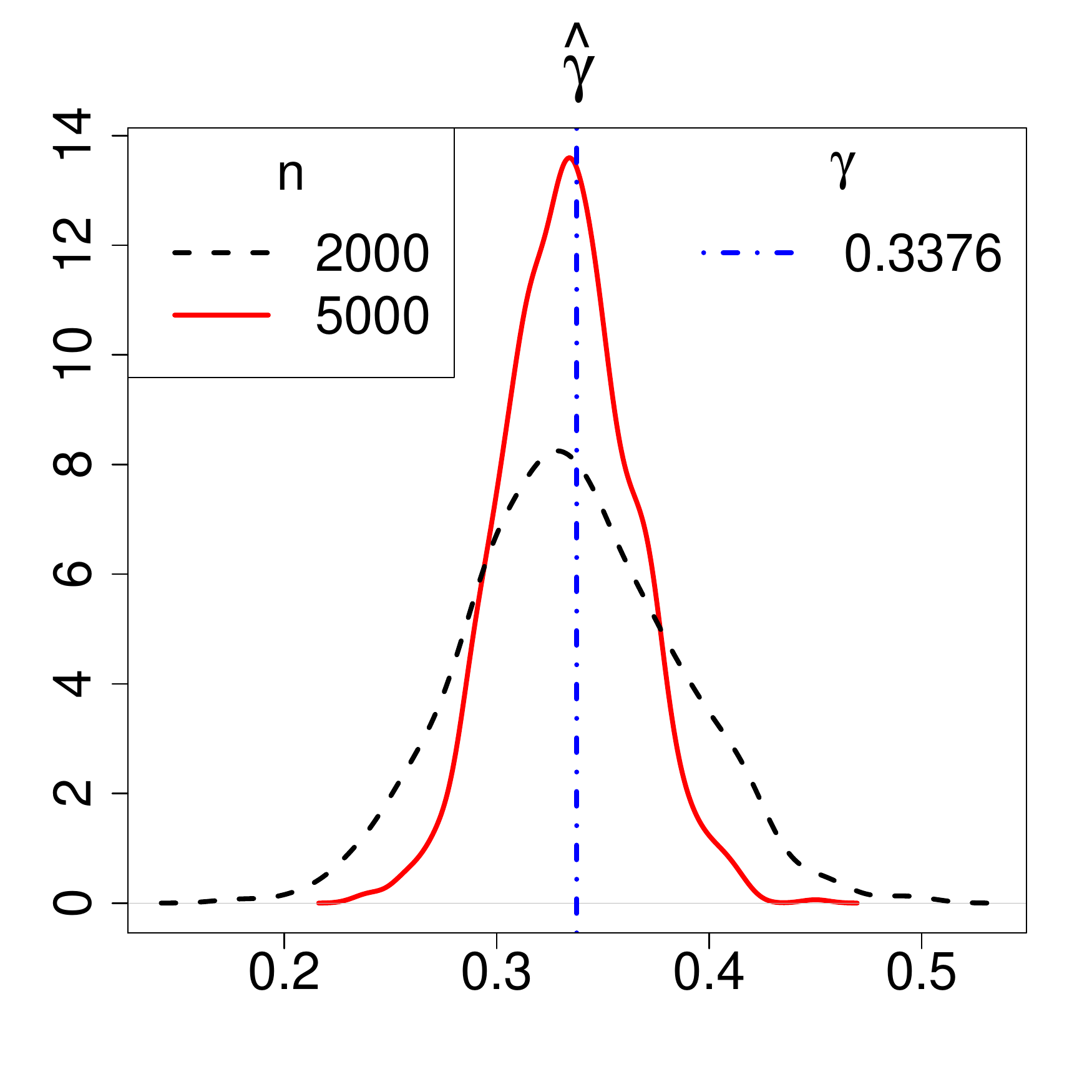}\!\!
  \includegraphics[width = 0.2\textwidth]{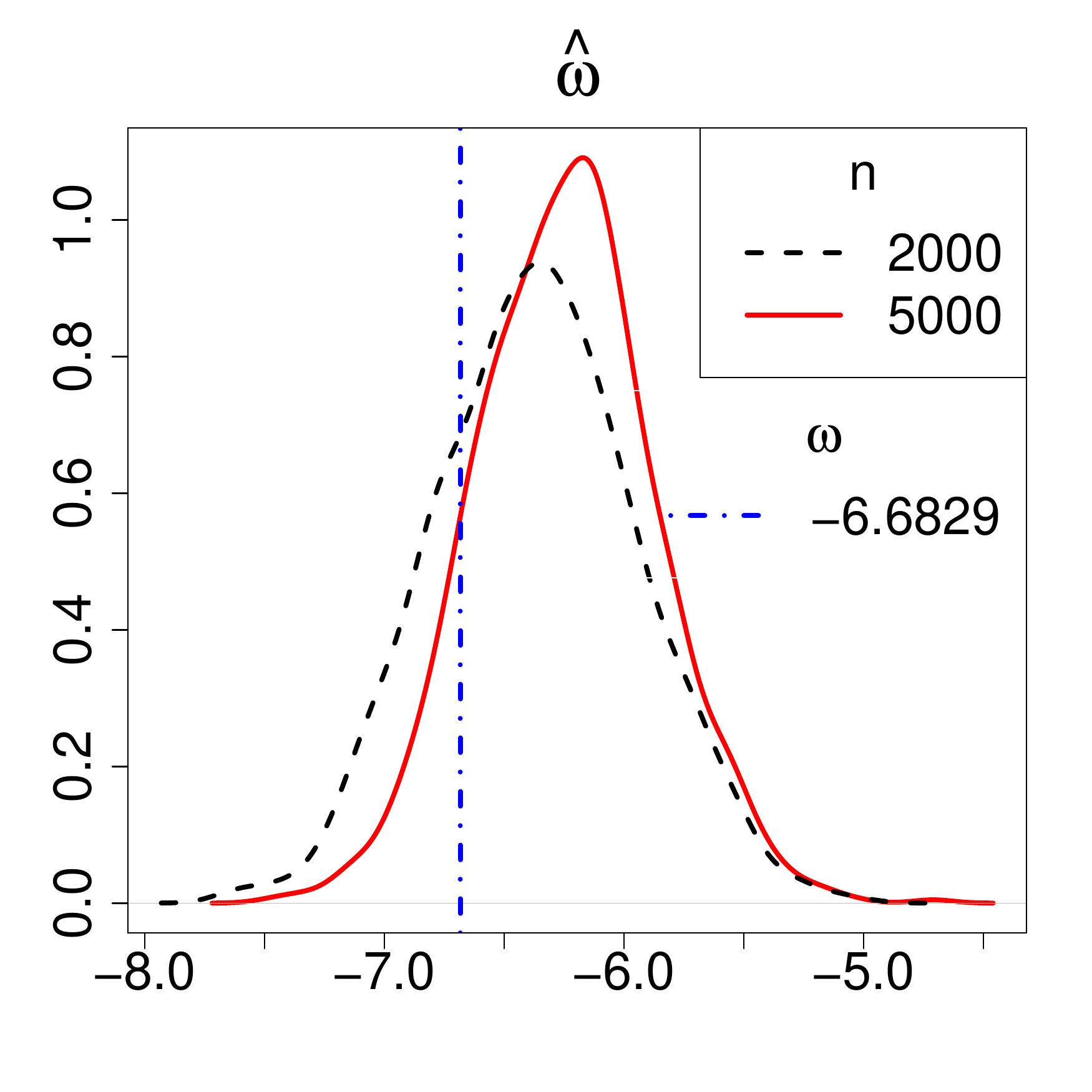}\!\!
  \includegraphics[width = 0.2\textwidth]{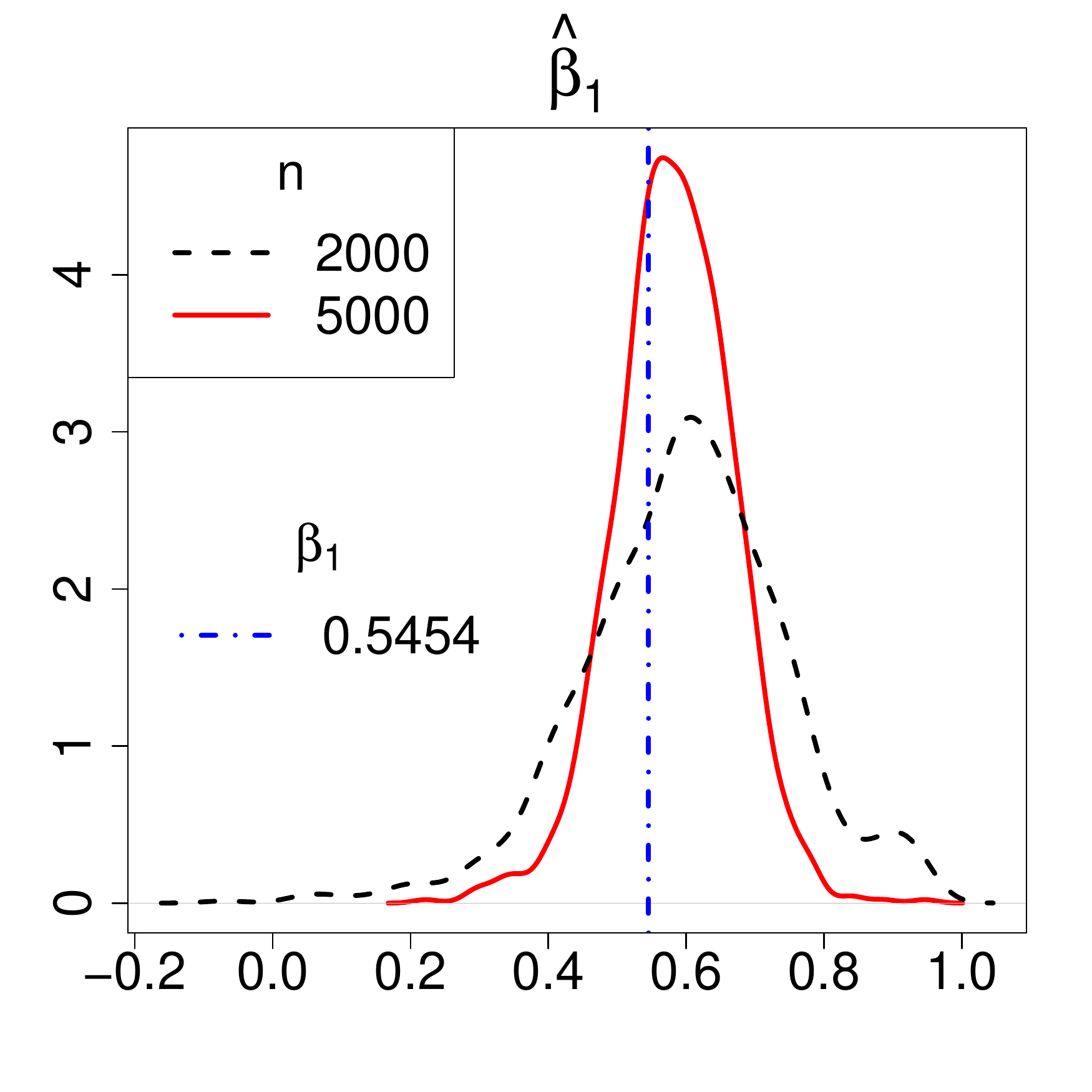}\vspace{-0.2cm}

   \caption{Kernel density function of the estimates for model M3, for $n \in
     \{2,000;\, 5,000\}$.}
 \end{figure}

 \begin{figure}[!htb]
  \vspace{-0.2cm} \centering
  \includegraphics[width = 0.2\textwidth]{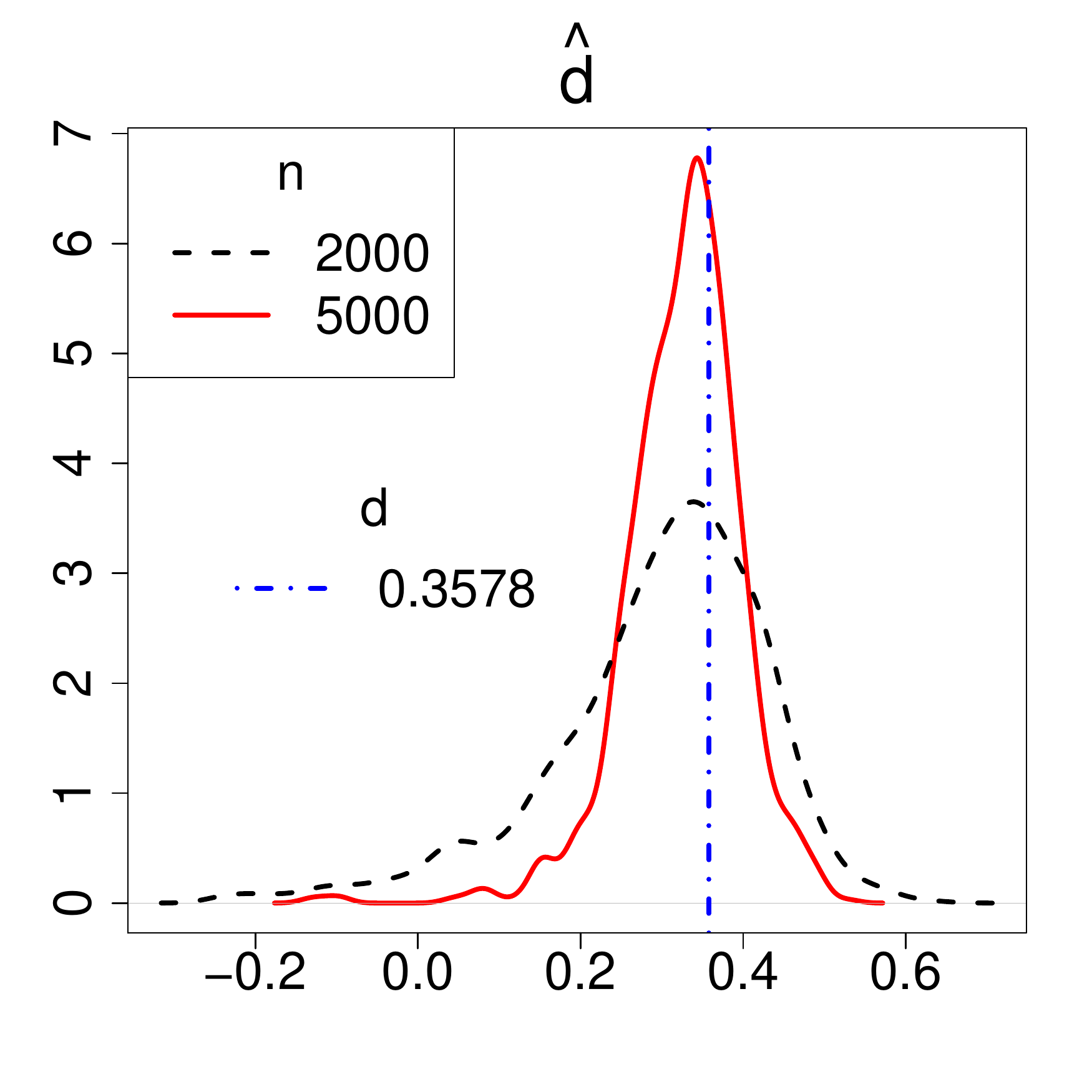}\!\!
  \includegraphics[width = 0.2\textwidth]{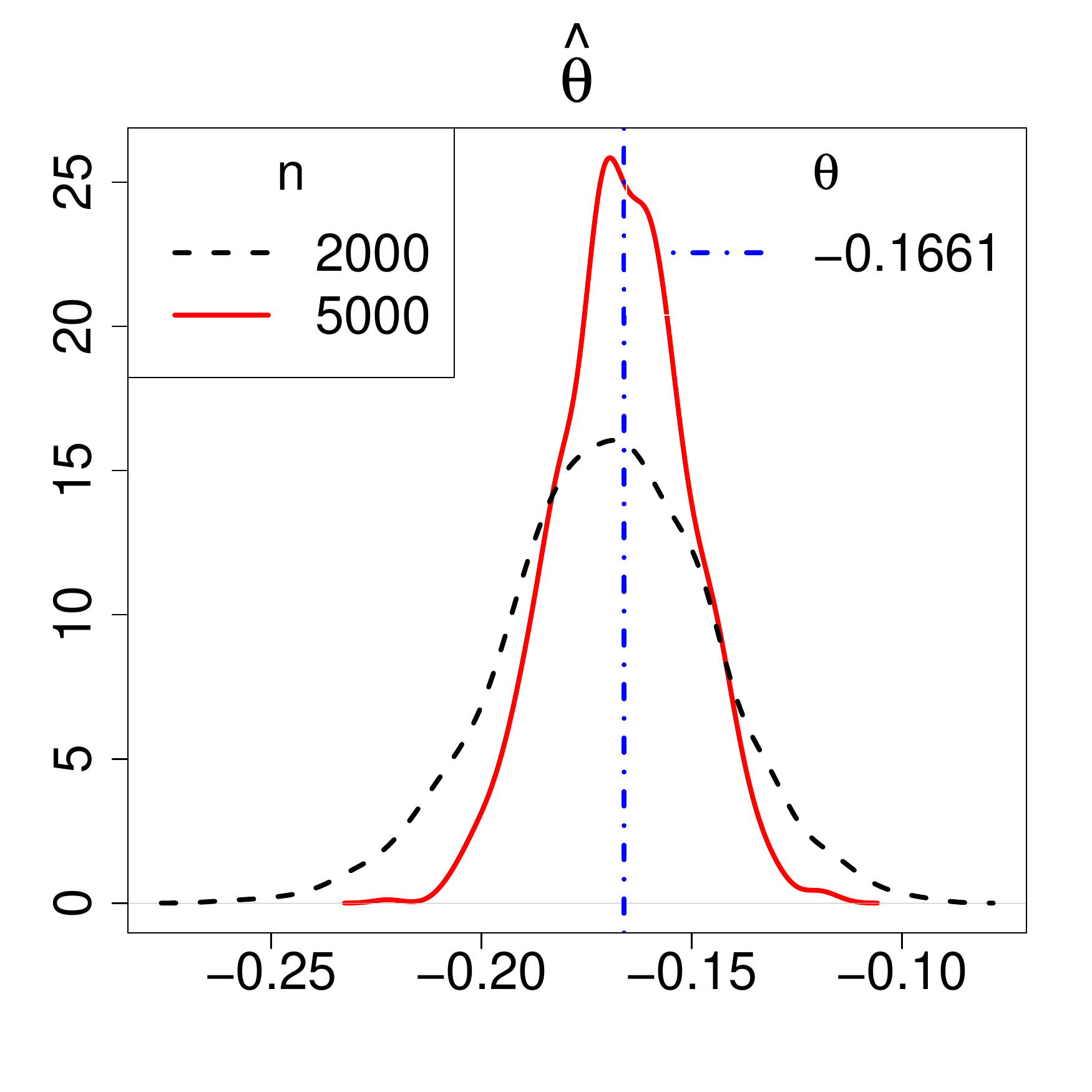}\!\!
  \includegraphics[width = 0.2\textwidth]{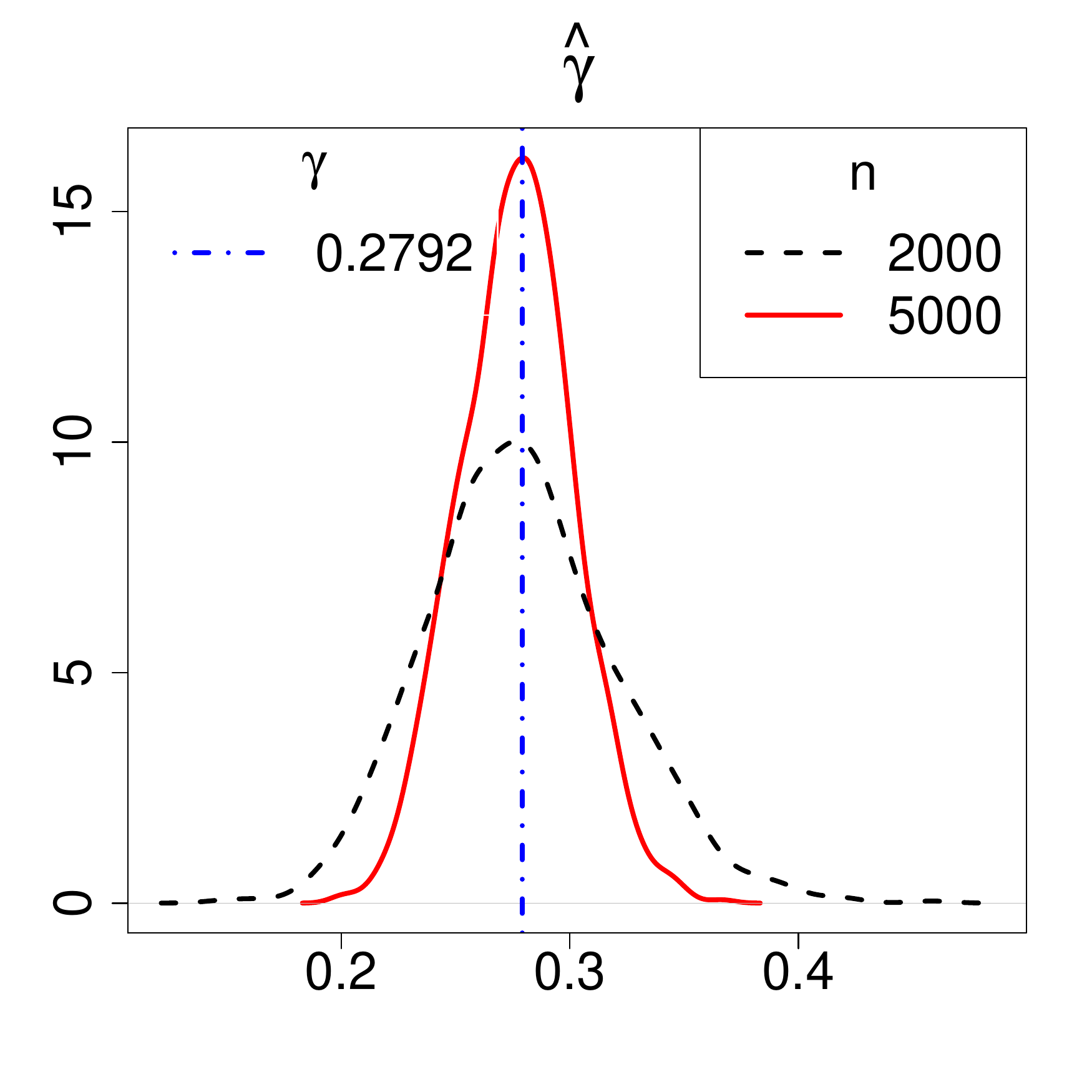}\!\!
  \includegraphics[width = 0.2\textwidth]{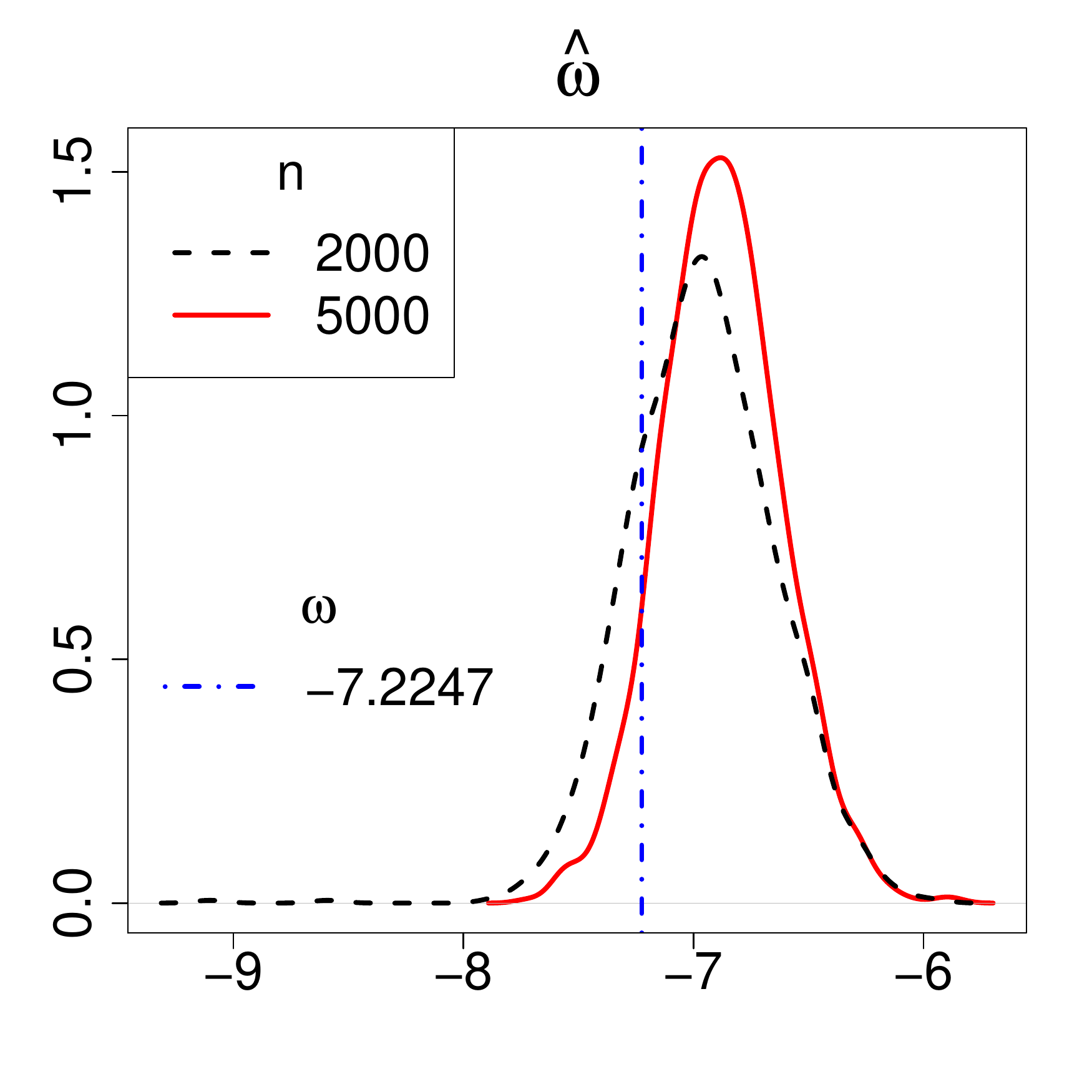}\!\!
  \includegraphics[width = 0.2\textwidth]{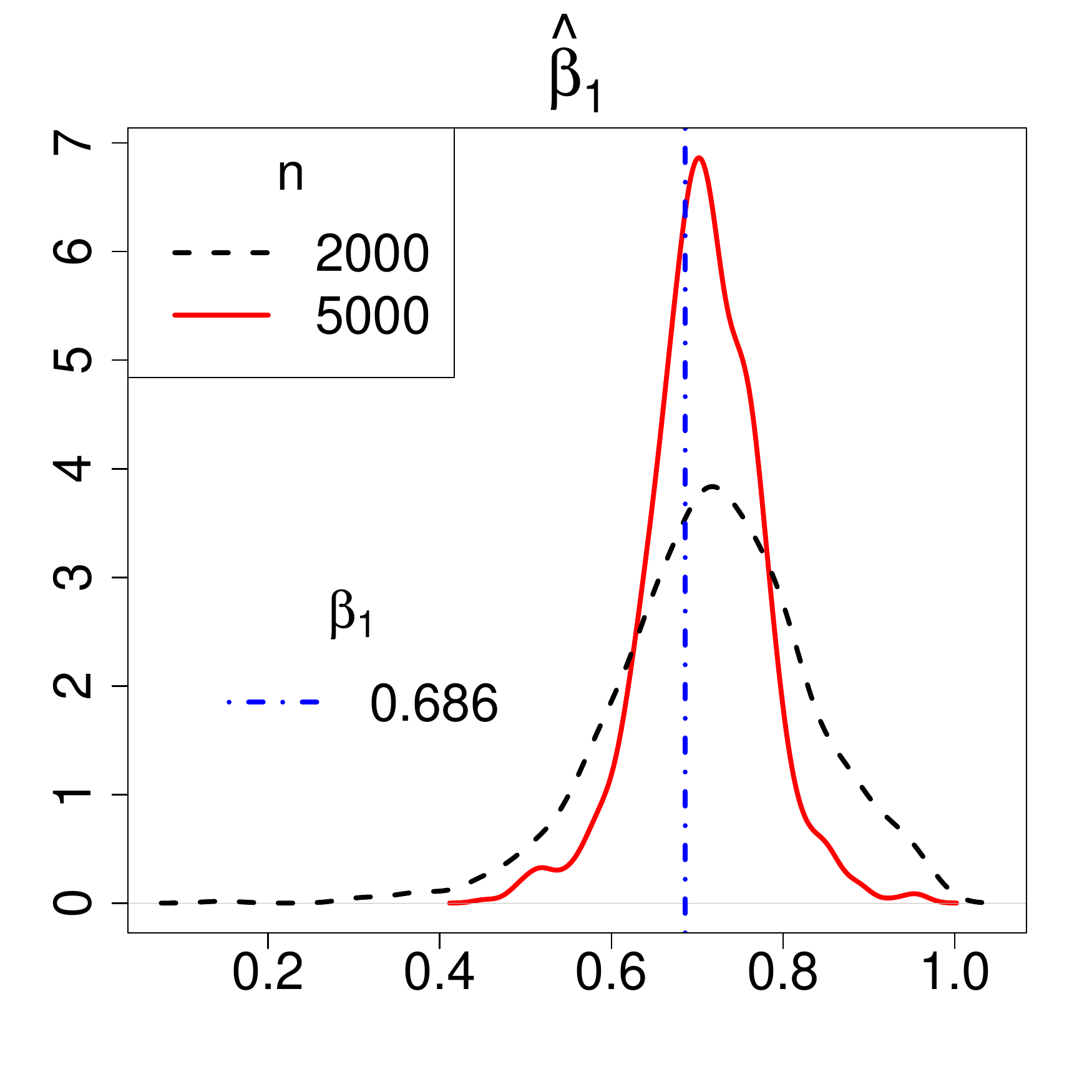}\vspace{-0.2cm}

   \caption{Kernel density function of the estimates for model M4, for $n \in
     \{2,000;\, 5,000\}$. }
 \end{figure}

 \begin{figure}[!htb]
  \vspace{-0.2cm} \centering
  \includegraphics[width = 0.2\textwidth]{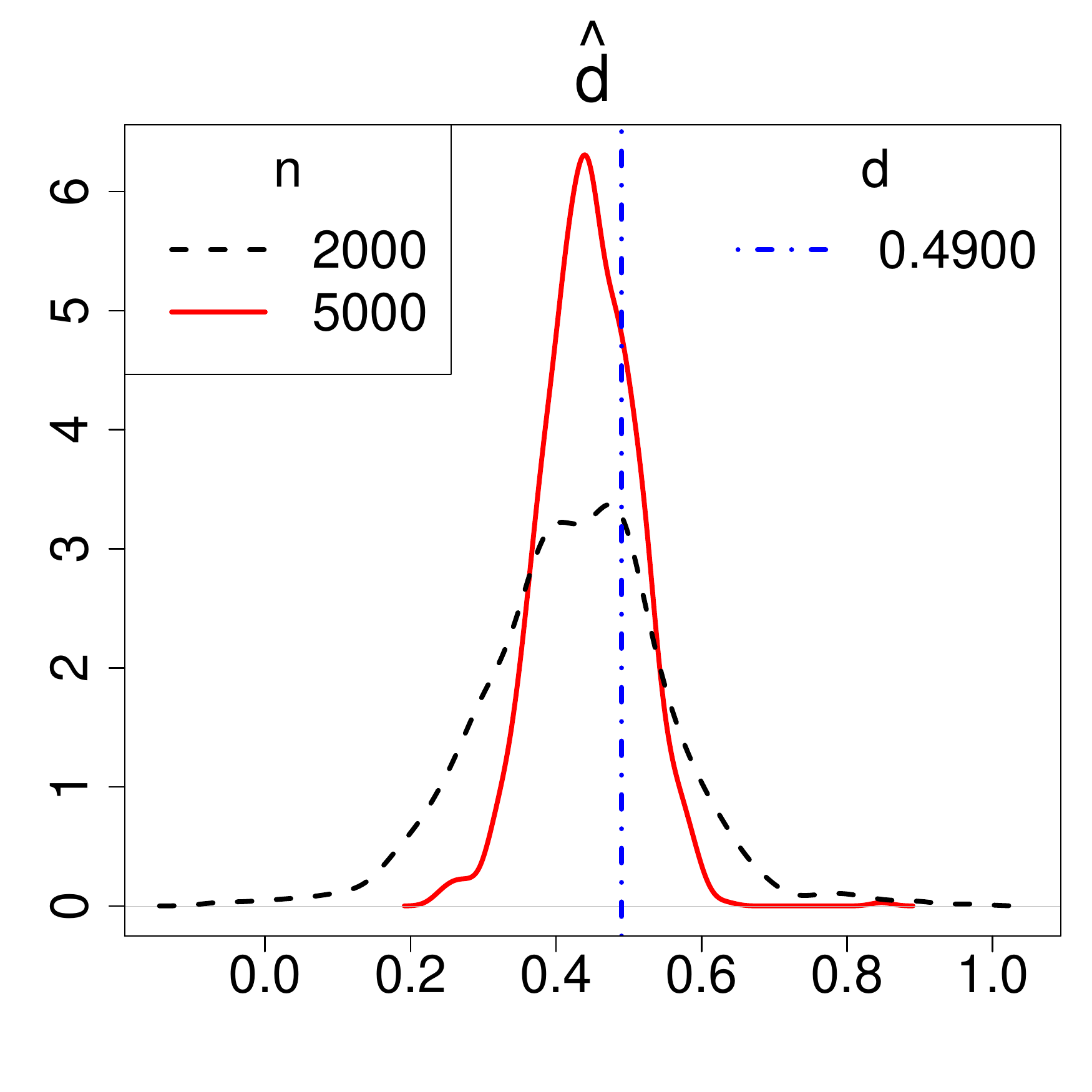}
  \includegraphics[width = 0.2\textwidth]{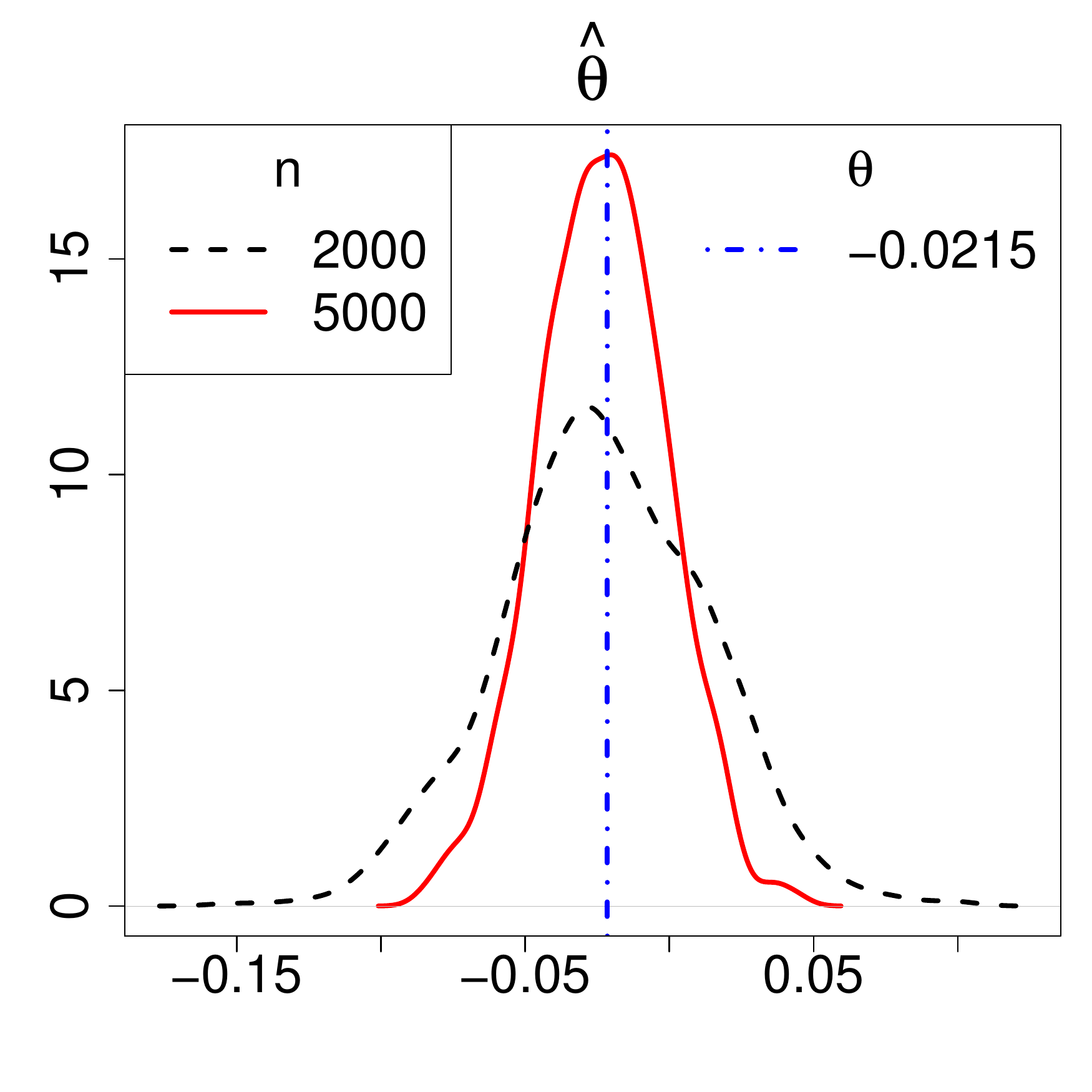}
  \includegraphics[width = 0.2\textwidth]{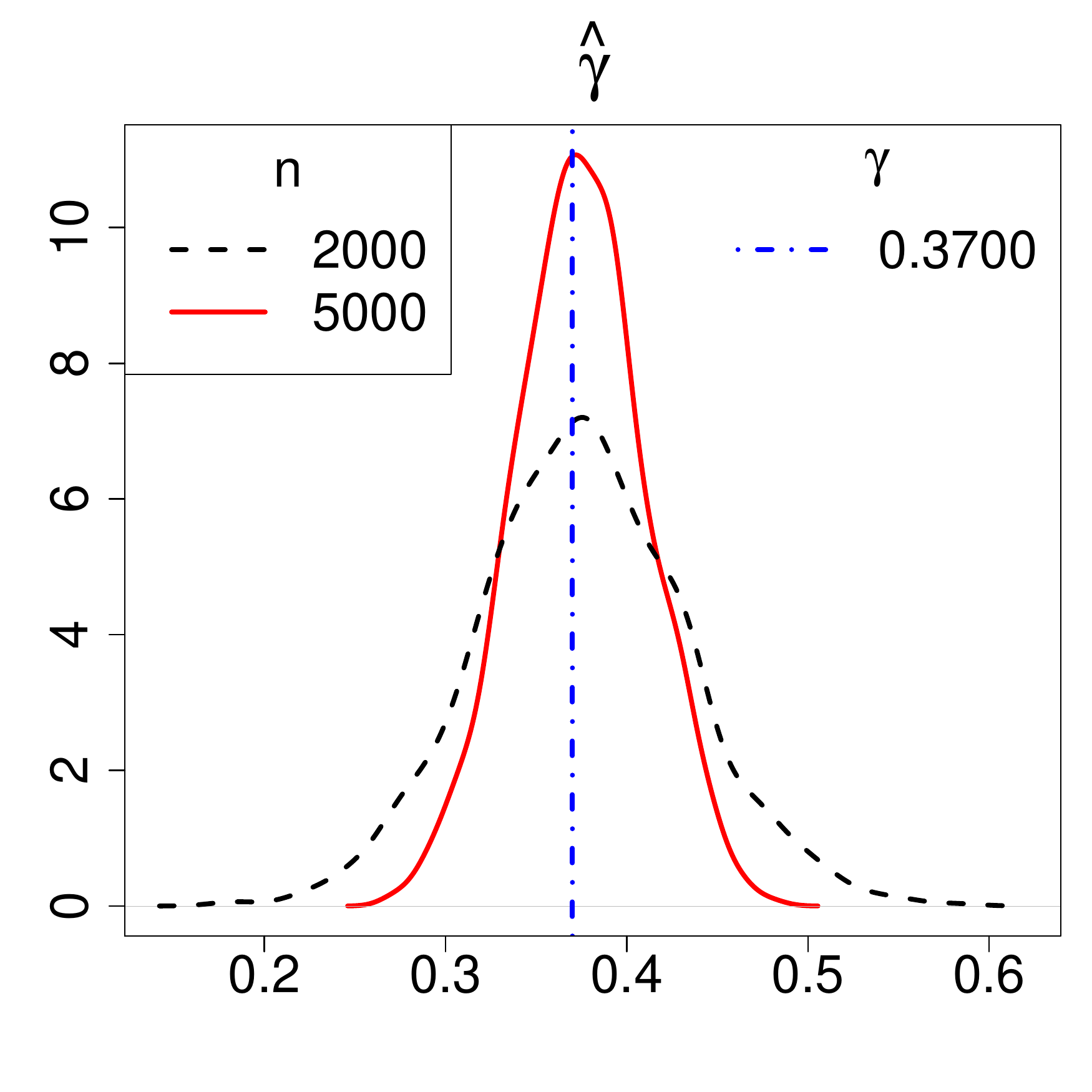}\\
  \includegraphics[width = 0.2\textwidth]{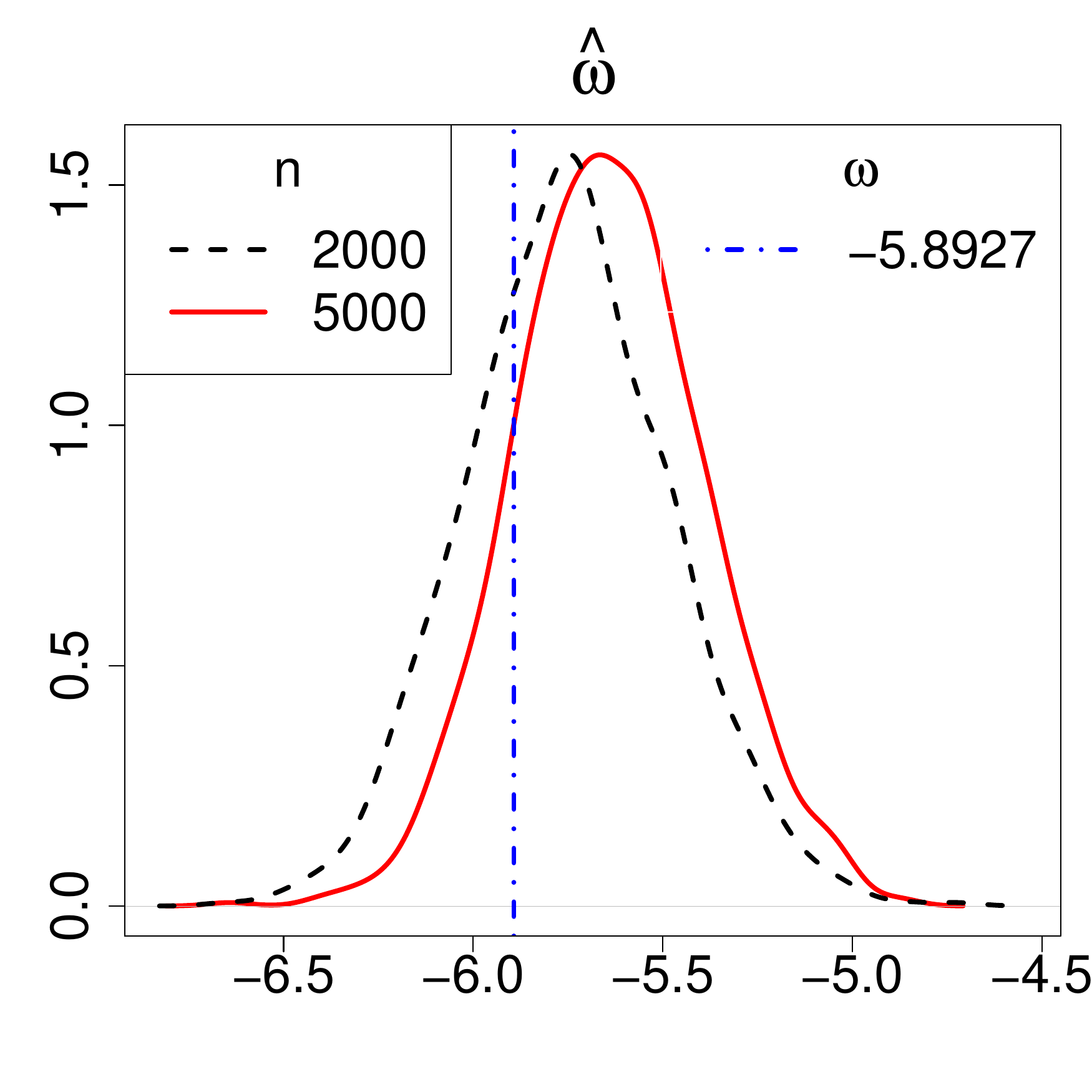}
  \includegraphics[width = 0.2\textwidth]{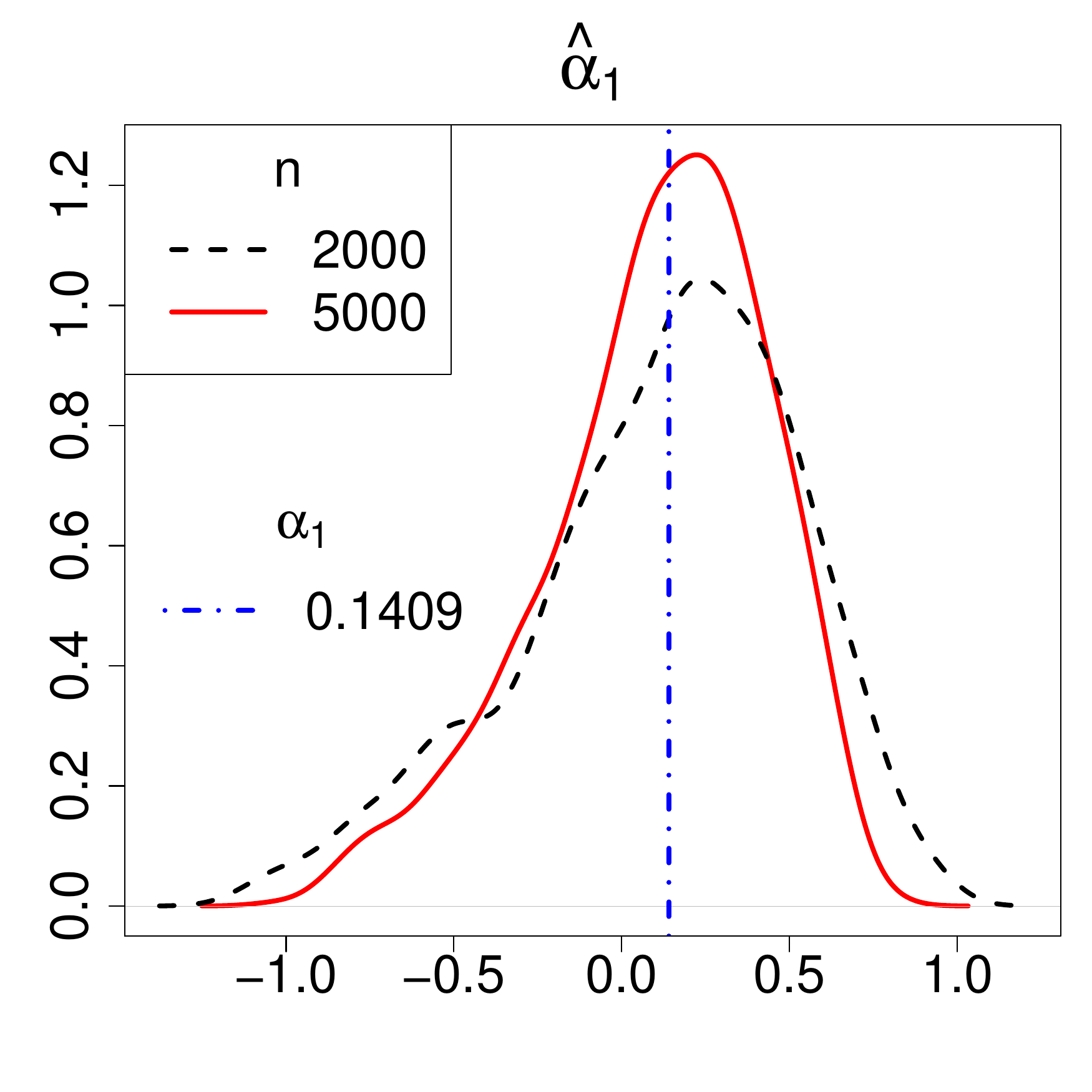}
  \includegraphics[width = 0.2\textwidth]{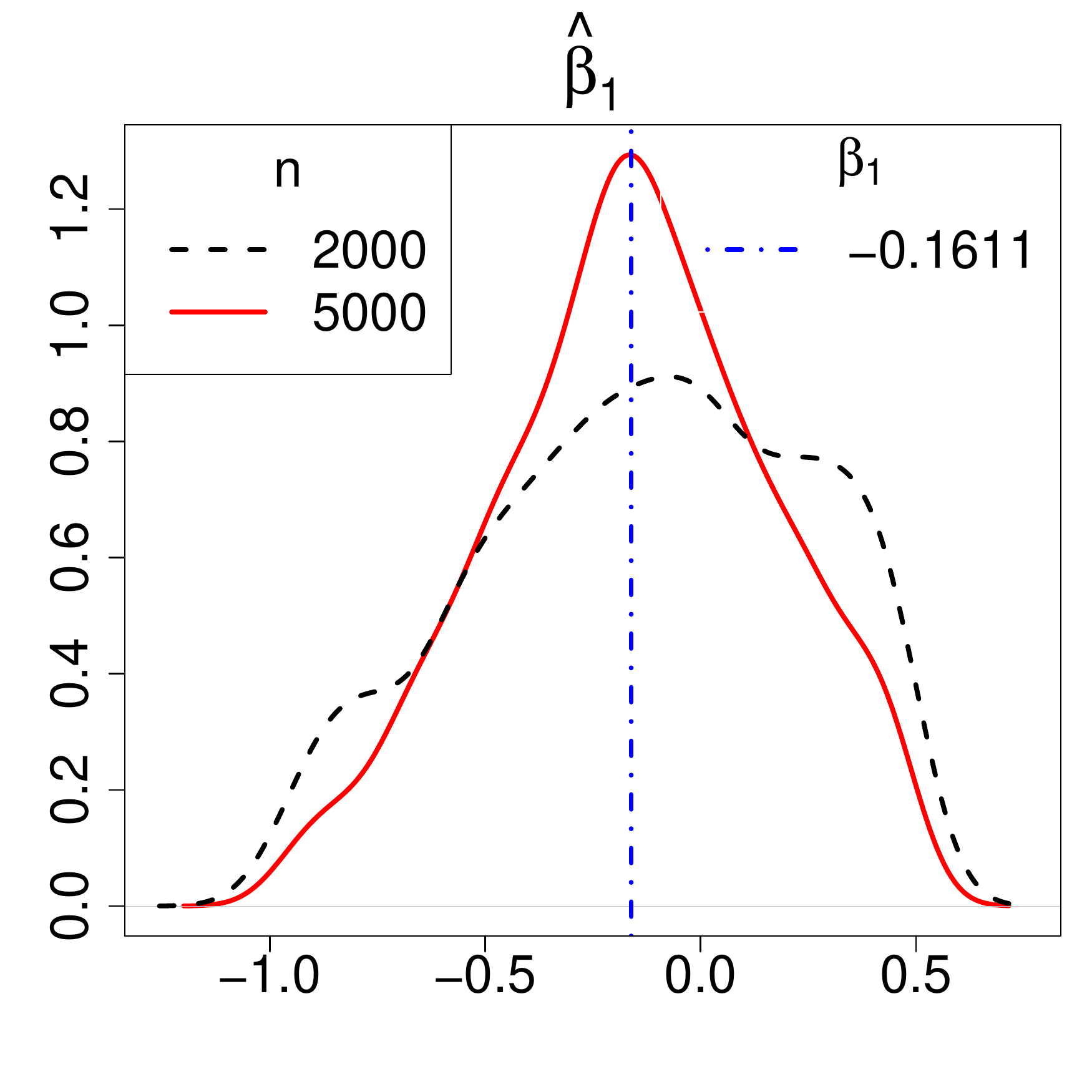}

   \caption{Kernel density function of the estimates for model M5, for $n \in
     \{2,000;\, 5,000\}$. }
 \end{figure}

 \begin{figure}[!htb]
  \vspace{-0.2cm} \centering
  \includegraphics[width = 0.2\textwidth]{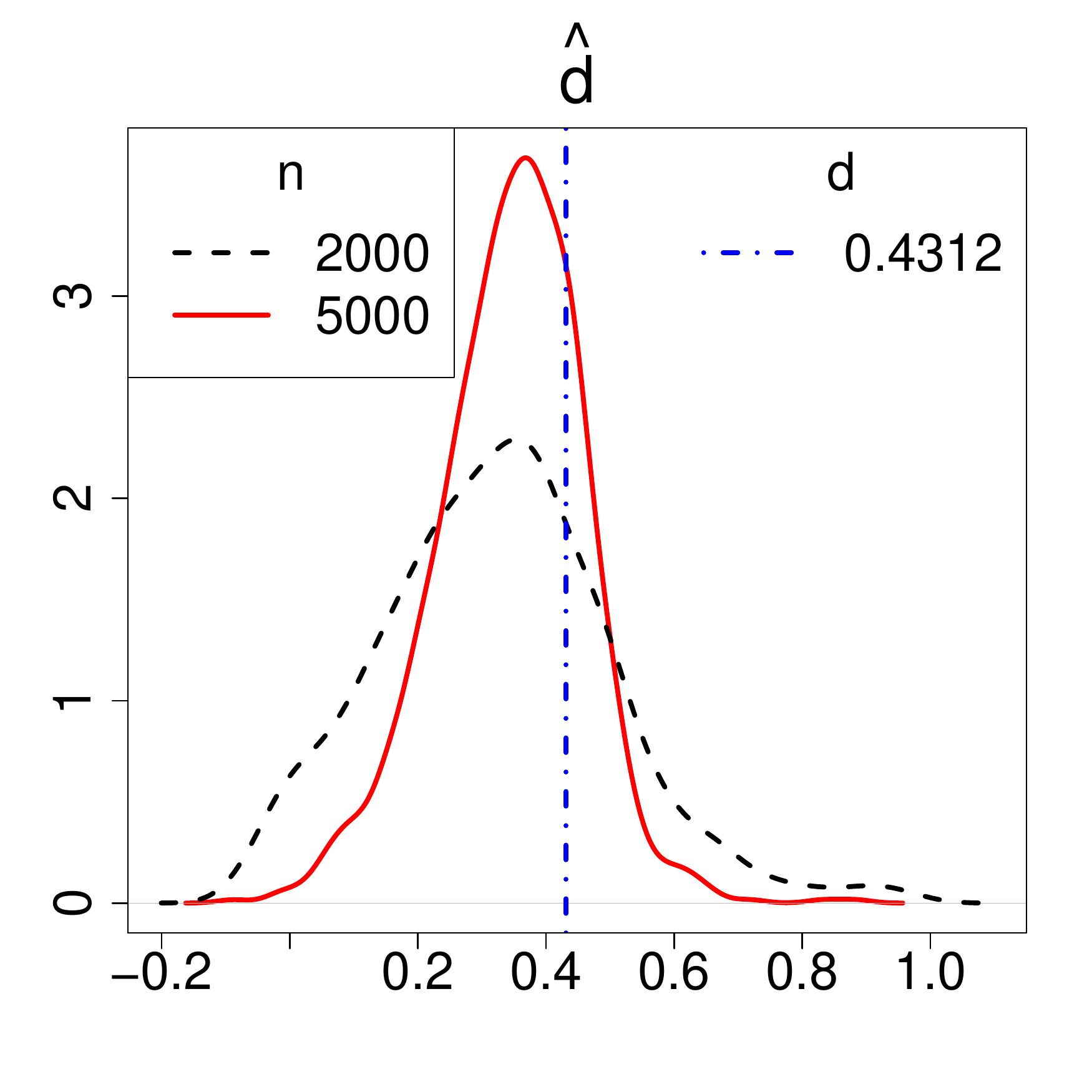}\!\!
  \includegraphics[width = 0.2\textwidth]{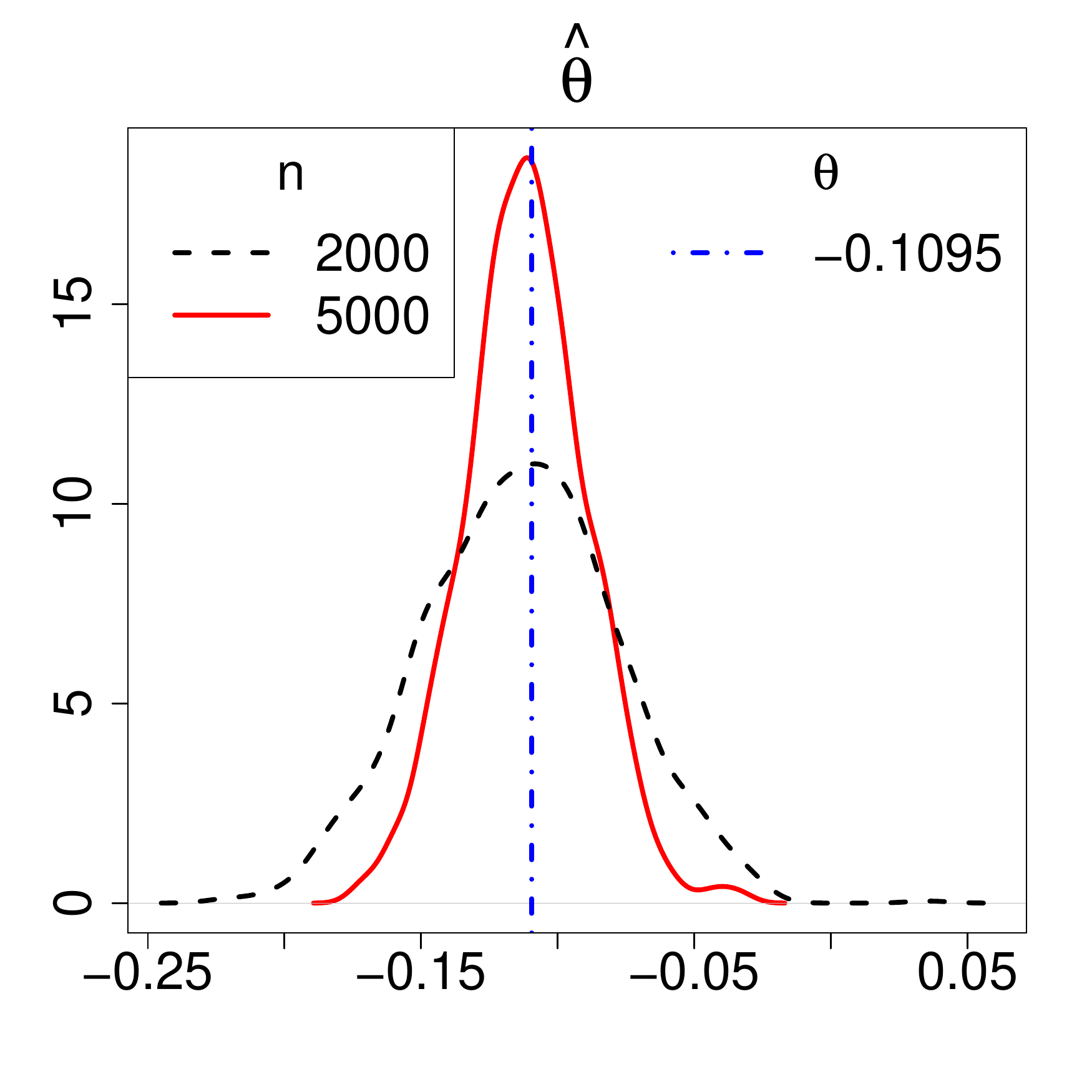}\!\!
  \includegraphics[width = 0.2\textwidth]{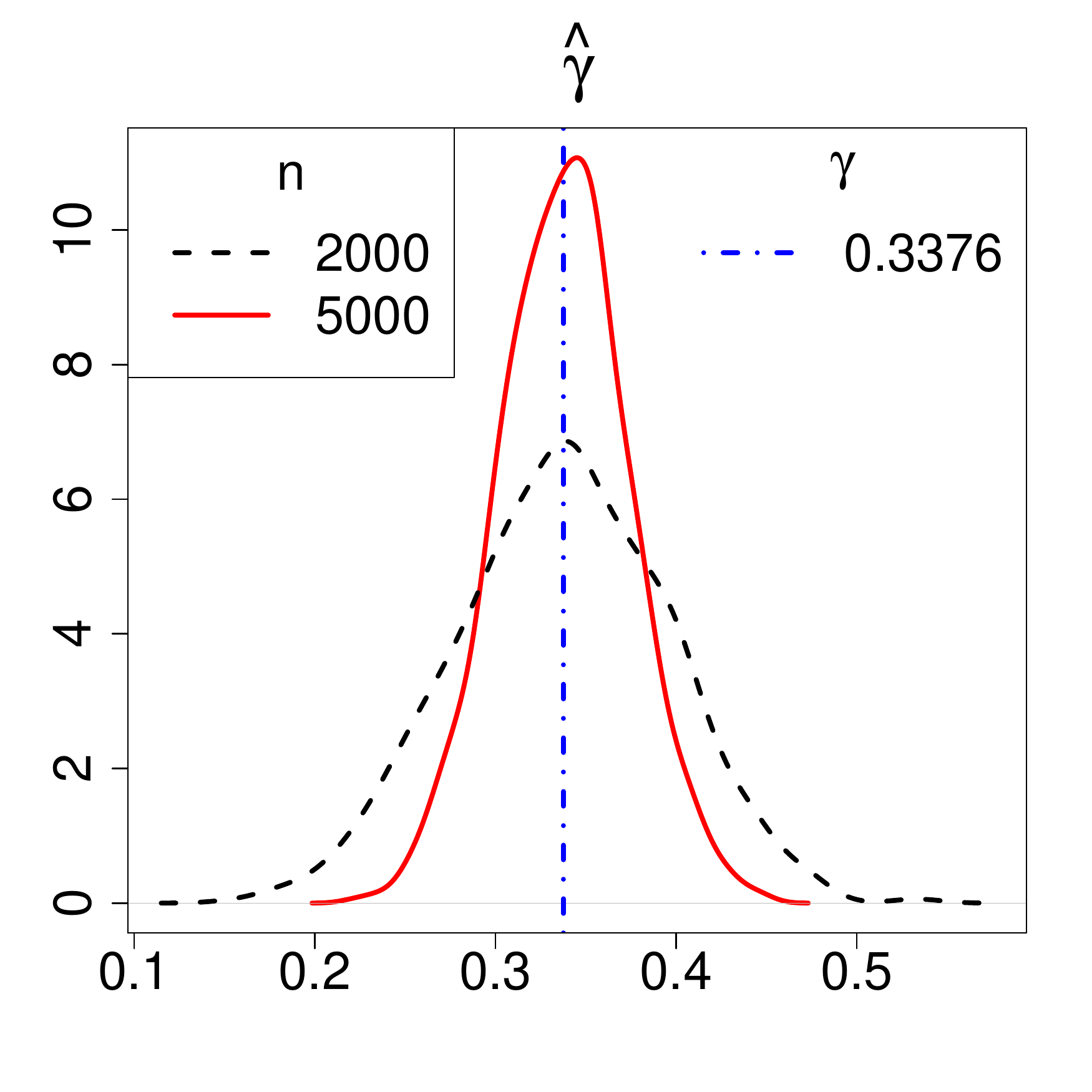}\!\!
  \includegraphics[width = 0.2\textwidth]{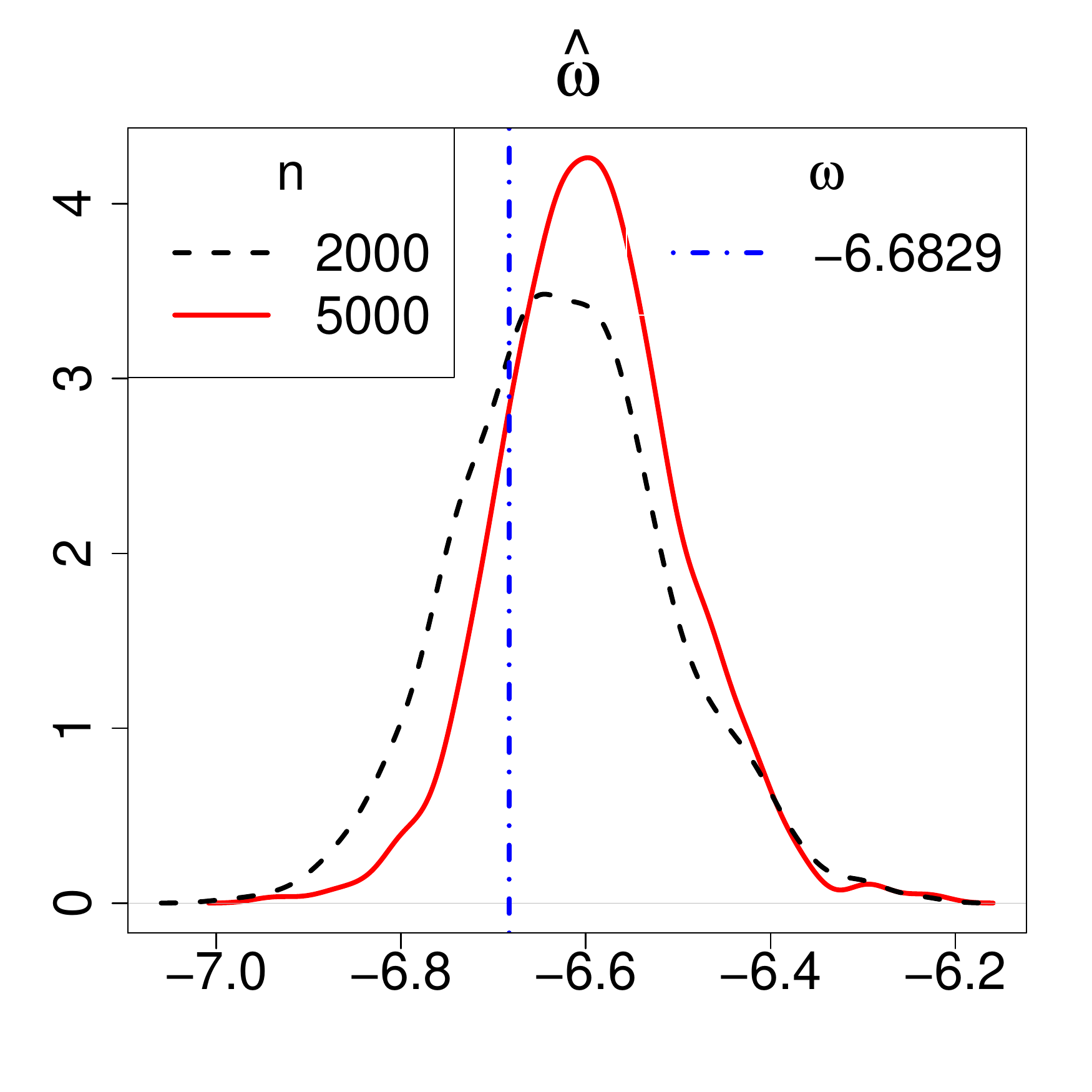}\!\!
  \includegraphics[width = 0.2\textwidth]{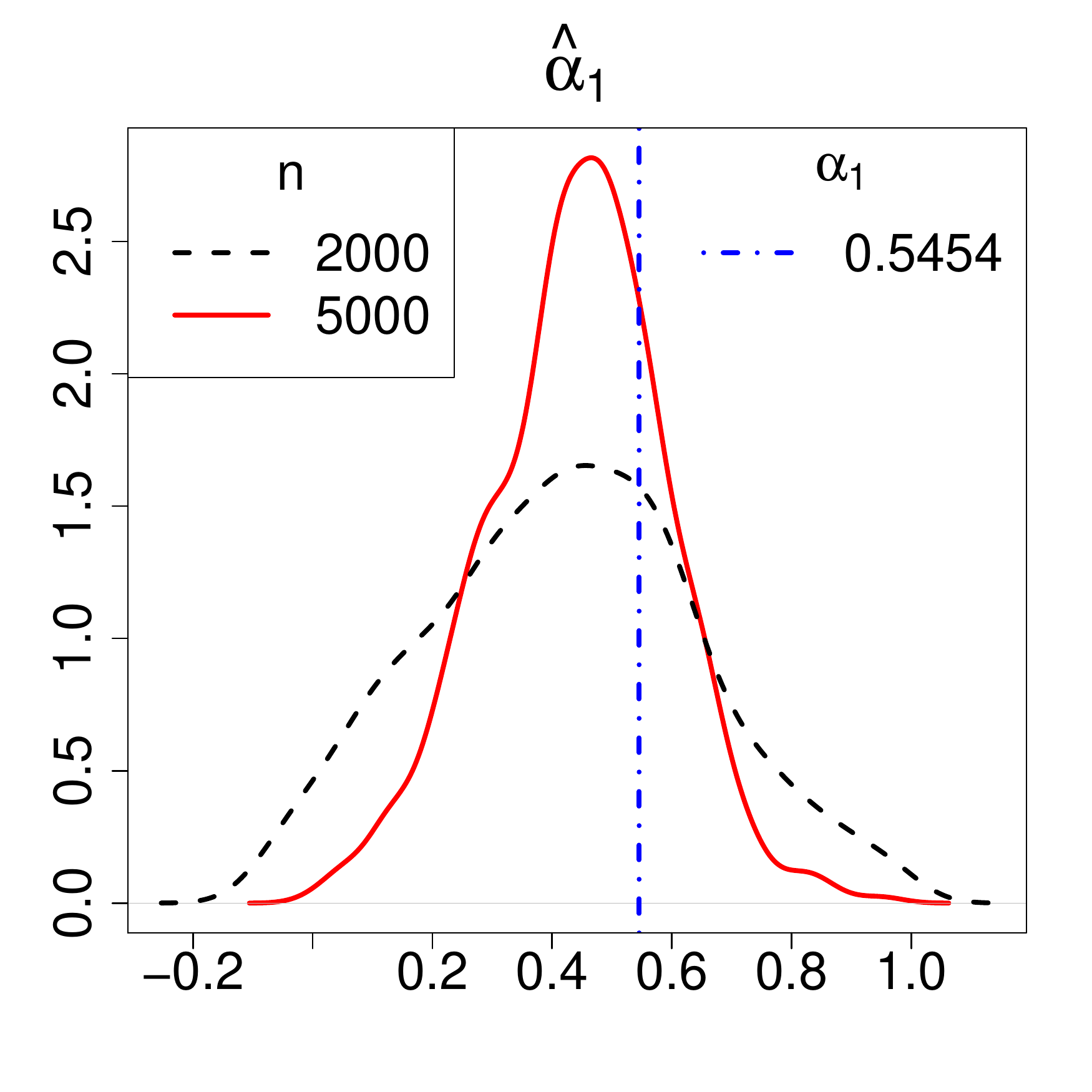}\vspace{-0.2cm}

   \caption{Kernel density function of the estimates for model M6, for $n \in
     \{2,000;\, 5,000\}$.}  \label{figm6}
 \end{figure}

By observing Figures \ref{figm1} - \ref{figm6}, it  is easy to see that, for most estimates, the density function is approximately
symmetric.  For some parameters, we notice the presence of possible outliers,
see for instance the graphs for the parameters $d$ (in particular, models  M2, M3 and
M4), $\alpha_i$ (model M2) and  $\beta_j$ (in particular,  models M1 and M2), with
$i\in \{1,2\}$ and $j\in \{1, 2,3,4\}$.  Although the graphs for
$n=2,000$ and $n=5,000$ are similar, one observes that, as expected, the
observations tend to concentrate closer to the mean when $n=5,000$.

\begin{table}[!htb]
   \centering
   \setlength{\tabcolsep}{2pt}
       \renewcommand{\arraystretch}{1.1}
  \caption{Estimation results for the simulated FIEGARCH models.} \label{resultssim}\vspace{0.2cm}
  {\footnotesize
    \begin{tabular*}{1\textwidth}{@{\extracolsep{\fill}}lrrrrrcrrrrr}
      \hline
      Sample Size ($n$)\phantom{\Big{|}} &     \multicolumn{5}{c}{$2,000$} &&
      \multicolumn{5}{c}{$5,000$}\\
      \cline{2-6}   \cline{8-12}
      Parameter ($\eta$)\phantom{\Big{|}}   &
      \multicolumn{1}{c}{$\bar{\eta}$}  &  \multicolumn{1}{c}{$sd$}  &
      \multicolumn{1}{c}{$bias$} &  \multicolumn{1}{c}{$mae$} &
      \multicolumn{1}{c}{$mse$} &&
      \multicolumn{1}{c}{$\bar{\eta}$}  &
      \multicolumn{1}{c}{$sd$}  & \multicolumn{1}{c}{$bias$} &
      \multicolumn{1}{c}{$mae$} & \multicolumn{1}{c}{$mse$}\\
      \hline
      \vspace{-0.2cm}\\

      \multicolumn{12}{l}{M1 :=   FIEGARCH$(2,d,1)$; \ \ $re = 1,000$}\vspace{0.1cm}\\

 \phantom{-}0.4495 ($d$) &
       0.4022&      0.0854&     -0.0473&      0.0688&      0.0095&&
       0.4309&      0.0468&     -0.0186&      0.0357&      0.0025\\
      -0.1245 ($\theta$) &
      -0.1240&      0.0266&      0.0005&      0.0213&      0.0007&&
      -0.1237&      0.0168&      0.0008&      0.0133&      0.0003\\
      \phantom{-}0.3662 ($\gamma$)&
       0.3612&      0.0543&     -0.0050&      0.0438&      0.0030&&
       0.3610&      0.0337&     -0.0052&      0.0271&      0.0012\\
       -6.5769 ($\omega$)&
      -6.2516&      0.4270&      0.3253&      0.4358&      0.2881&&
      -6.1284&      0.3830&      0.4485&      0.4930&      0.3479\\
       -1.1190 ($\alpha_1$)&
      -0.9067&      0.4519&      0.2123&      0.3567&      0.2492&&
      -1.0344&      0.3259&      0.0846&      0.2010&      0.1134\\
       -0.7619 ($\alpha_2$)&
      -0.6517&      0.4035&      0.1102&      0.2832&      0.1750&&
      -0.7281&      0.2623&      0.0338&      0.1534&      0.0700\\
       -0.6195 ($\beta_1$)&
      -0.3415&      0.4774&      0.2780&      0.3474&      0.3052&&
      -0.5052&      0.3214&      0.1143&      0.1764&      0.1164\vspace{0.1cm}\\

      \multicolumn{12}{l}{M2 :=   FIEGARCH$(0,d,4)$; \ \ $re = 1,000$}\vspace{0.1cm}\\

        \phantom{-}0.2391 ($d$)&
      0.1683&      0.1538&     -0.0708&      0.1216&      0.0287&&
      0.2077&      0.0767&     -0.0314&      0.0650&      0.0069\\
     -0.0456 ($\theta$)&
     -0.0469&      0.0275&     -0.0013&      0.0220&      0.0008&&
     -0.0461&      0.0169&     -0.0005&      0.0134&      0.0003\\
        \phantom{-}0.3963 ($\gamma$)&
      0.3931&      0.0536&     -0.0032&      0.0426&      0.0029&&
      0.3959&      0.0326&     -0.0004&      0.0256&      0.0011\\
     -6.6278 ($\omega$)&
     -6.5525&      0.1146&      0.0753&      0.1075&      0.0188&&
     -6.5077&      0.0905&      0.1201&      0.1253&      0.0226\\
      \phantom{-}0.2289 ($\beta_1$)&
      0.2841&      0.1284&      0.0552&      0.1083&      0.0195&&
      0.2488&      0.0721&      0.0199&      0.0593&      0.0056\\
      \phantom{-}0.1941 ($\beta_2$)&
      0.2078&      0.0865&      0.0137&      0.0657&      0.0077&&
      0.1990&      0.0456&      0.0049&      0.0367&      0.0021\\
      \phantom{-}0.4737 ($\beta_3$)&
      0.4710&      0.0935&     -0.0027&      0.0667&      0.0088&&
      0.4784&      0.0441&      0.0047&      0.0349&      0.0020\\
     -0.4441 ($\beta_4$)&
     -0.4704&      0.1063&     -0.0263&      0.0867&      0.0120&&
     -0.4500&      0.0592&     -0.0059&      0.0466&      0.0035\vspace{0.1cm}\\

     \multicolumn{12}{l}{M3 :=   FIEGARCH$(0,d,1)$; \ \ $re = 1,000$}\vspace{0.1cm}\\
            \phantom{-}0.4312 ($d$) &
      0.3606&      0.1268&     -0.0706&      0.1043&      0.0211&&
      0.3933&      0.0648&     -0.0379&      0.0569&      0.0056\\
     -0.1095 ($\theta$)&
     -0.1111&      0.0255&     -0.0016&      0.0201&      0.0007&&
     -0.1090&      0.0157&      0.0005&      0.0125&      0.0002\\
    \phantom{-}0.3376 ($\gamma$)  &
      0.3346&      0.0493&     -0.0030&      0.0394&      0.0024&&
      0.3331&      0.0300&     -0.0045&      0.0241&      0.0009\\
      -6.6829 ($\omega$)  &
     -6.3686&      0.4230&      0.3143&      0.4271&      0.2778&&
     -6.2413&      0.3715&      0.4416&      0.4814&      0.3330\\
      \phantom{-}0.5454 ($\beta_1$)&
      0.5976&      0.1472&      0.0522&      0.1231&      0.0244&&
      0.5822&      0.0851&      0.0368&      0.0731&      0.0086\vspace{0.2cm}\\

      \multicolumn{12}{l}{M4 :=   FIEGARCH$(0,d,1)$; \ \ $re = 1,000$}\vspace{0.1cm}\\
      \phantom{-}0.3578 ($d$) &
      0.2950&      0.1338&     -0.0628&      0.1056&      0.0218&&
      0.3258&      0.0721&     -0.0320&      0.0569&      0.0062\\
      -0.1661 ($\theta$)&
     -0.1702&      0.0248&     -0.0041&      0.0198&      0.0006&&
     -0.1666&      0.0156&     -0.0005&      0.0124&      0.0002\\
       \phantom{-}0.2792 ($\gamma$)&
      0.2793&      0.0415&      0.0001&      0.0326&      0.0017&&
      0.2769&      0.0248&     -0.0023&      0.0197&      0.0006\\
      -7.2247  ($\omega$)&
     -6.9615&      0.3122&      0.2632&      0.3284&      0.1667&&
     -6.8766&      0.2604&      0.3481&      0.3689&      0.1889\\
        \phantom{-}0.6860 ($\beta_1$)&
      0.7160&      0.1128&      0.0300&      0.0915&      0.0136&&
      0.7067&      0.0665&      0.0207&      0.0535&      0.0048\vspace{0.1cm}\\

      \multicolumn{12}{l}{M5 :=   FIEGARCH$(1,d,1)$; \ \ $re = 1,000$}\vspace{0.1cm}\\
          \phantom{-}0.4900 ($d$) &
      0.4258&      0.1273&     -0.0642&      0.1096&      0.0203&&
      0.4453&      0.0645&     -0.0447&      0.0629&      0.0062\\
    -0.0215 ($\theta$) &
     -0.0229&      0.0355&     -0.0014&      0.0282&      0.0013&&
     -0.0229&      0.0218&     -0.0014&      0.0175&      0.0005\\
       \phantom{-}0.3700 ($\gamma$)&
      0.3751&      0.0577&      0.0051&      0.0455&      0.0034&&
      0.3742&      0.0354&      0.0042&      0.0285&      0.0013\\
    -5.8927 ($\omega$) &
     -5.7507&      0.2688&      0.1420&      0.2415&      0.0924&&
     -5.6414&      0.2494&      0.2513&      0.2902&      0.1253\\
      \phantom{-}0.1409 ($\alpha_1$)&
      0.1152&      0.4082&     -0.0257&      0.3232&      0.1673&&
      0.1012&      0.3310&     -0.0397&      0.2613&      0.1111\\
     -0.1611 ($\beta_1$)&
     -0.1383&      0.3799&      0.0228&      0.3189&      0.1448&&
     -0.1581&      0.3213&      0.0030&      0.2579&      0.1032\vspace{0.1cm}\\

      \multicolumn{12}{l}{M6 :=   FIEGARCH$(1,d,0)$; \ \ $re = 1,000$}\vspace{0.1cm}\\
       \phantom{-}0.4312 ($d$)&
      0.3220&      0.1825&     -0.1092&      0.1706&      0.0452&&
      0.3449&      0.1135&     -0.0863&      0.1107&      0.0203\\
     -0.1095  ($\theta$)&
     -0.1132&      0.0351&     -0.0037&      0.0282&      0.0012&&
     -0.1114&      0.0222&     -0.0019&      0.0176&      0.0005\\
       \phantom{-}0.3376 ($\gamma$)&
      0.3368&      0.0585&     -0.0008&      0.0467&      0.0034&&
      0.3380&      0.0355&      0.0004&      0.0283&      0.0013\\
     -6.6829  ($\omega$) &
     -6.6233&      0.1144&      0.0596&      0.1014&      0.0166&&
     -6.5926&      0.0978&      0.0903&      0.1071&      0.0177\\
      \phantom{-}0.5454  ($\alpha_1$)&
      0.4189&      0.2297&     -0.1265&      0.2109&      0.0688&&
      0.4429&      0.1492&     -0.1025&      0.1428&      0.0328\vspace{0.1cm}\\
  \hline

    \end{tabular*} }
\end{table}

From Table \ref{resultssim} we conclude that, given the models complexity, the
quasi-likelihood method performs relatively well.  Since model M2 presents more
parameters than the other models, which implies a higher dimension maximization
problem, one would expect that the quasi-likelihood method would present the
worst performance in this case.  However,   in terms of $mae$ or  $mse$ values,   the estimation results for model M2
($p = 0$, $d = 0.2391$ and $q = 4$),   M3 ($p = 0$, $d =0.4312$ and $q = 1$),
 M4 ($ p= 0$, $d = 0.3578$ and $q = 1$)  and  M6 ($p = 1$, $d =0.4312$ and $q =
 0$)  are similar (except for the parameter $d$ in model M6)  and the quasi-likelihood method
performs better  for model M2  (except for the parameter $d$)  than
for models M1 ($p = 2$, $d = 0.4495$ and $q = 1$) and M5 ($p = 1$, $d = 0.49$
and $q = 1$).

Table \ref{resultssim}  also indicates that the quasi-likelihood procedure may perform better
for $p=0$ and $q > 0$ than for $p > 0$ and $q = 0$ (we shall investigate this in a future work). This conclusion is based on
the fact that  models M3 and M6 have the same parameter values (with the
necessary adjustments in $\alpha_1$ and $\beta_1$) and  all parameters, except $\omega$,  were better estimated in model
M3 than  M6.

By comparing the $mae$ and $mse$ values, given in
Table \ref{resultssim}, we conclude that the worst performance occurs for models
M1 and M5  (in particular, see the estimation results for $\omega$, $\alpha_i$ and $\beta_j$, $i
= 1,\cdots,p$ and $j = 1, \cdots, q$).   This outcome is explained by the fact that the parameter $d$ is very
close to the non-stationary region for model M5 and, for model M1, not only
$p=2$ but also $d = 0.4495$, which implies a more complex model with stronger
long-range dependence.  The small $bias$ values indicate that, for all
parameters, the mean estimated value is very close to the true value.  Although
for $n= 2,000$ the standard deviation of several estimates is high if compared
with the mean estimated value, as expected, the estimators performance improves
as the sample size increases.

 \subsection{Forecasting Procedure}

 To obtain the predicted values, for each  replication of model M$i$, with $i \in
\{1,\cdots, 6\}$, and each sub-sample  $\{x_t\}_{t=1}^n$, with $n\in\{2,000;
5,000\}$,  we repeat steps {\bf F1} - {\bf F5} below.

  \vspace{1\baselineskip}
\noindent {\bf F1:}   Replace the true parameters values  $\boldsymbol{\eta}=(d;\omega;\theta;\lambda;\alpha_1,\cdots,\alpha_p;\beta_1,\cdots,\beta_q)'$
  by the estimated ones, namely,  $\boldsymbol{\hat\eta}=(\hat
 d;\hat\omega;\hat\theta;\hat\lambda;\hat\alpha_1,\cdots,\hat\alpha_p;\hat\beta_1,\cdots,\hat\beta_q)'$,
  and use the  recurrence formula given in Proposition
  \ref{coefficientsfiegarch} to
calculate  the corresponding  coefficients $\{\hat \lambda_{d,k}\}_{k = 0}^{n+50}$.

  \vspace{1\baselineskip}
\noindent {\bf F2:}  Obtain  the time series  $\{z_t\}_{t=1}^n$ (which
corresponds to the residuals of the fitted model) and
$\{\sigma_t\}_{t=1}^n$. To do so, let  $g(z_t) = 0$, whenever $t<0$, and
calculate $\sigma_t$ and $z_t$ recursively as follows:
\[
 \sigma_1 = e^{\hat \omega 0.5};   \quad   z_1 = \frac{x_1}{\sigma_1}; \quad
\sigma_t  =\exp\bigg\{\frac{\hat \omega}{2} + \frac{1}{2}\sum_{k=0}^{n-1}\hat\lambda_{d,k}\left[\hat\theta
z_{t-1-k} + \hat\gamma (|z_{t-1-k}| - \sqrt{2/\pi})\right]\bigg\}  \quad
\mbox{ and  } \quad  z_t = \frac{x_t}{\sigma_t},
\]
for all $t=2,\cdots, n$.

  \vspace{1\baselineskip}
\noindent {\bf F3:} In expression \eqref{eq:sigmag},  replace $\mathds{E}(|Z_0|)$ and $\mathds{E}(Z_0|Z_0|)$
  by their   respective   sample estimates,   and  obtain an estimate $\hat\sigma_g^2$ for $\sigma_g^2$ given by
  \[
\hat \sigma^2_g = \hat\theta^2 + \hat\gamma^2 -\hat\gamma^2\left[\frac{1}{n}\sum_{t=1}^n|z_t|\right]^2 + 2
\, \hat\theta\,\hat\gamma \left[\frac{1}{n}\sum_{t=1}^nz_t|z_t|\right].
  \]

  \vspace{1\baselineskip}
\noindent {\bf F4:}  By considering  expressions \eqref{exp1}  and
\eqref{sigmatilde},  obtain
  the  predicted values $\{\tilde\sigma_{N+h}^2\}_{h=1}^{50}$,
\[
  \tilde \sigma_{N+1}^2  =  \check{\sigma}_{N+1}^2 \quad  \quad \mbox{and} \quad
  \quad
  \tilde \sigma_{N+h}^2  =     \check{\sigma}_{N+h}^2\bigg( 1
  +\frac{1}{2}\hat\sigma^2_g\sum_{k=0}^{h-2}\hat\lambda_{d,k}^2\bigg),  \quad
  \mbox{for all }   h > 1,
\]
where
\[
\check \sigma_{N+h}^2 = \exp\bigg\{\hat\omega +
    \sum_{k=0}^{n-1}\hat\lambda_{d,k+h-1}\left[ \hat \theta z_{n-k} +
    \hat\gamma (|z_{n-k}| - \hat{\mu}_{|z|}) \right] \bigg\}, \quad \mbox{for all }\,  h>0,
    \]
      with $\hat{\mu}_{|z|} := \frac{1}{n}\sum_{t=1}^n|z_t|$.

    \vspace{1\baselineskip}
\noindent {\bf F5:}    Based on the fact that  $\mathds{E}(X_{N+h}^2|\mathcal{F}_{N}) =
 \mathds{E}(\sigma_{N+h}^2|\mathcal{F}_{N})$,   set $\tilde X_{N+h}^2 : =
 \tilde\sigma_{N+h}^2$, for all $h>0$.

 \subsection{Forecasting Results}

  In what follows we discuss the simulation results related to forecasting based
 on the fitted FIEGARCH models.  To access the models forecast performance,
 during the generating process, we create 50 extra values for each simulated time
 series. Those values are used here to compare with the $h$-step ahead forecast,
 for $h \in \{1,\cdots, 50\}$.

 Table \ref{forecasting} presents the mean over $1,000$ simulated values of
 $\sigma_{N+h}^2$ and $X_{N+h}^2$ obtained from model M$i$,  for  each $i\in\{1,
 \cdots, 6\}$,  and the corresponding  $h$-step ahead predicted
 values   $\tilde{\sigma}_{N+h}^2:=\tilde{X}_{N+h}^2$,  for $h\in
 \{1,\cdots,5\}$,  forecasting origin  $N = 5,000$ and sub-samples  $n \in
 \{2,000; 5,000 \}$.  This table  also reports   the mean
 square error ($mse$) of forecast,   defined as
 \[
mse(Y_{N+h}) :=  \frac{1}{re}\sum_{k =1}^{re}\big(Y_{N+h}^{(k)} - \check{Y}_{N+h}^{(k)}(n)\big)^2,
\quad \mbox{for any }  \, \, h \in \{1,\cdots, 5\}  \, \,  \mbox{ and } \, \,
n \in \{2,000; 5,000\},
\]
where $re = 1,000$ is the number of replications,  $Y_{N+h}$ denotes the true
value of $\sigma_{N+h}^2$ (or  $X_{N+h}^2$) and $ \check{Y}_{N+h}^{(k)}(n)$ is
the predicted value obtained in the $k$-th replication, for $k\in\{1, \cdots, re\}$,   based on the model
fitted to the sub-sample with size $n$.   Notice that, due to the small
magnitude of the sample  means,  all values in Table
\ref{forecasting} are multiplied by 100.

\begin{table}[!ht]
  \centering
      \renewcommand{\arraystretch}{1.1}
\setlength{\tabcolsep}{2pt}
   \caption{Mean simulated values for $\sigma_{N+h}^2$ and $X_{N+h}^2$,
     obtained  from   model M$i$,    the
     corresponding mean predicted values $\tilde\sigma_{N+h}^2 =
     \tilde X_{N+h}^2$  and the  mean square error of
     forecast,   for  $h\in\{1,\cdots,5\}$ and  $i\in\{1,\cdots,6\}$.  The forecasting origin is  $N= 5,000$ and $n
     \in\{2,000; 5,000\}$ is the sub-sample size used to fit the models and to
     obtain the predicted values.   All
     values reported correspond to the calculated values multiplied by a scaling constant  (except $h$).  The
     scaling constant is   equal to $10^2$, for    $\sigma_{N+h}^2$,
     $X_{N+h}^2$ and $\tilde\sigma_{N+h}^2$, and to  $10^4$,  for the $mse$ values. } \label{forecasting}\vspace{0.2cm}
   {\footnotesize
     \begin{tabular*}{1\textwidth}{@{\extracolsep{\fill}}ccccccccccc}
       \hline
         \phantom{\Big{|}}$n$ &&&&\multicolumn{3}{c}{2,000}    &&\multicolumn{3}{c}{5,000}\\
         \cline{5-7}\cline{9-11}

           \phantom{\Big{|}}$\phantom{xx}h\phantom{xx}$ &
           \multicolumn{1}{c}{$\phantom{xx}\sigma_{N+h}^2\phantom{xx}$ } &
         \multicolumn{1}{c}{$\phantom{xx}X_{N+h}^2\phantom{xx}$ } &&
         \multicolumn{1}{c}{Predictor} &
         \multicolumn{1}{c}{$mse(\sigma_{N+h}^2)$} &
         \multicolumn{1}{c}{$mse(X_{N+h}^2)$ } &&
         \multicolumn{1}{c}{Predictor} &
         \multicolumn{1}{c}{$mse(\sigma_{N+h}^2)$} &
         \multicolumn{1}{c}{$mse(X_{N+h}^2)$} \\

   \hline
  \vspace{-0.1cm}\\
  \multicolumn{11}{l}{M1 :=   FIEGARCH$(2,d,1)$; \ \ $re = 1,000$}
  \vspace{0.1cm}\\
   1 &      0.1698&      0.1575&&      0.1652&      0.0010&      0.0993&&      0.1634&      0.0003&      0.0969\\
  2 &      0.1635&      0.1473&&      0.1640&      0.0038&      0.0900&&      0.1611&      0.0032&      0.0875\\
  3 &      0.1636&      0.1540&&      0.1655&      0.0078&      0.1122&&      0.1632&      0.0075&      0.1116\\
  4 &      0.1629&      0.1490&&      0.1662&      0.0122&      0.1117&&      0.1633&      0.0114&      0.1101\\
  5 &      0.1641&      0.1542&&      0.1665&      0.0147&      0.1906&&   0.1642&      0.0141&      0.1903\vspace{0.1cm}\\

       \multicolumn{11}{l}{M2 :=   FIEGARCH$(0,d,4)$; \ \ $re =
         1,000$}  \vspace{0.1cm}\\

        1 &      0.1387&      0.1284&&      0.1359&      0.0004&      0.0521&&      0.1369&      0.0002&      0.0515\\
  2 &      0.1383&      0.1246&&      0.1395&      0.0024&      0.0506&&      0.1398&      0.0021&      0.0501\\
  3 &      0.1357&      0.1299&&      0.1374&      0.0027&      0.0551&&      0.1381&      0.0024&      0.0547\\
  4 &      0.1378&      0.1276&&      0.1390&      0.0029&      0.0562&&      0.1399&      0.0028&      0.0559\\
  5 &      0.1356&      0.1253&&      0.1409&      0.0036&      0.0568&&      0.1414&      0.0034&      0.0570\vspace{0.1cm}\\

       \multicolumn{11}{l}{M3 :=   FIEGARCH$(0,d,1)$; \ \ $re =
         1,000$}  \vspace{0.1cm}\\

        1 &      0.1487&      0.1380&&      0.1452&      0.0007&      0.0833&&      0.1439&      0.0002&      0.0848\\
  2 &      0.1456&      0.1287&&      0.1459&      0.0026&      0.0681&&      0.1442&      0.0022&      0.0674\\
  3 &      0.1426&      0.1350&&      0.1466&      0.0052&      0.0773&&      0.1447&      0.0045&      0.0777\\
  4 &      0.1438&      0.1200&&      0.1473&      0.0075&      0.0619&&      0.1453&      0.0068&      0.0600\\
  5 &      0.1411&      0.1354&&      0.1479&      0.0076&      0.1316&&      0.1459&      0.0069&      0.1309\vspace{0.1cm}\\

       \multicolumn{11}{l}{M4 :=   FIEGARCH$(0,d,1)$; \ \ $re =
         1,000$}  \vspace{0.1cm}\\
  1 &      0.0932&      0.0894&&      0.0918&      0.0005&      0.0411&&      0.0910&      0.0002&      0.0416\\
  2 &      0.0905&      0.0809&&      0.0918&      0.0013&      0.0275&&      0.0908&      0.0010&      0.0270\\
  3 &      0.0885&      0.0810&&      0.0917&      0.0027&      0.0293&&      0.0907&      0.0022&      0.0291\\
  4 &      0.0886&      0.0764&&      0.0918&      0.0040&      0.0251&&      0.0908&      0.0036&      0.0242\\
  5 &      0.0876&      0.0831&&      0.0919&      0.0042&      0.0461&&   0.0909&      0.0037&      0.0456\vspace{0.1cm}\\

       \multicolumn{11}{l}{M5 :=   FIEGARCH$(1,d,1)$; \ \ $re =
         1,000$}  \vspace{0.1cm}\\
 1 &      0.2898&      0.2669&&      0.2808&      0.0012&      0.2096&&      0.2795&      0.0005&      0.2087\\
  2 &      0.2883&      0.2800&&      0.2833&      0.0069&      0.2489&&      0.2817&      0.0065&      0.2494\\
  3 &      0.2908&      0.2836&&      0.2844&      0.0081&      0.2452&&      0.2821&      0.0081&      0.2461\\
  4 &      0.2909&      0.2963&&      0.2847&      0.0077&      0.3178&&      0.2827&      0.0076&      0.3174\\
  5 &      0.2923&      0.2971&&      0.2852&      0.0097&      0.3695&&   0.2832&      0.0096&      0.3704\vspace{0.1cm}\\

     \multicolumn{11}{l}{M6 :=   FIEGARCH$(0,d,1)$; \ \ $re =
       1,000$}  \vspace{0.1cm}\\
 1 &      0.1271&      0.1143&&      0.1242&      0.0001&      0.0367&&      0.1247&      0.0001&      0.0367\\
  2 &      0.1260&      0.1140&&      0.1265&      0.0013&      0.0379&&      0.1265&      0.0013&      0.0377\\
  3 &      0.1259&      0.1228&&      0.1262&      0.0014&      0.0471&&      0.1263&      0.0013&      0.0471\\
  4 &      0.1284&      0.1188&&      0.1263&      0.0018&      0.0479&&      0.1264&      0.0017&      0.0477\\
  5 &      0.1261&      0.1192&&      0.1264&      0.0015&      0.0473&&  0.1265&      0.0014&      0.0474\\
       \hline
     \end{tabular*}}
 \end{table}

 From Table \ref{forecasting}  (see also Figure \ref{figforecast} below) conclude that,
\begin{itemize}
\item  when we consider $\sigma_{N+h}^2$,
the predicted values are relatively close to the simulated ones, which is indicated by the small $mse$
values,  for all models and any
$h\in\{1,\cdots,6\}$;

\item the $mse$ value increases as $h$ increases.  This
result is expected and it is theoretically explained in Proposition
\ref{hstepaheadX} which shows that
\begin{equation*}
    \mathds{E}\big(\big[\ln(\sigma_{n+h}^2)- \hat\ln(\sigma_{n+h}^2)\big]^2\big)
    =\sigma^2_g\sum_{k=0}^{h-2}\lambda_{d,k}^2 \overset{h \to
      \infty}{\longrightarrow} \sigma^2_g\sum_{k=0}^{\infty}\lambda_{d,k}^2,
  \end{equation*}
  where $\sigma^2_g := \mathds{E}([g(Z_0)]^2)$ is given in \eqref{eq:sigmag};

\item  when we consider  $X_{N+h}^2$, the $mse$ is usually high,  if compared to the mean simulated and mean predicted
values.  Therefore, we conclude that $\tilde
{X}_{n+h}^2:=\tilde{\sigma}_{n+h}^2$ is a poor estimator for $X_{n+h}^2$.  This
result is not a surprise since the main purpose of FIEGARCH models is to
estimate the logarithm of the conditinal variance of the process and not the
process itself;

\item  as expected, in all cases, the models' forecasting performance improves as $n$
increases. Notice, however, that the difference in the $mse$ values, from
$n=2,000$ to $n=5,000$, is small (recall that the values are multiplied by 100).
This is so because the coefficients
$\lambda_{d,k}$ converges to zero, as $k$ goes to infinity. Therefore, it is
expected that, for some $m\in\mathds{N}$ and any $M>0$, using the last $m$ or
the last $m+M$ known values to calculate the $h$-step ahead forecast value for
the process will not considerably change the results.
\end{itemize}

Figure \ref{figforecast} shows the mean taken over 1,000 replications for:
\begin{itemize}
\item  the simulated values  $\sigma_{N+h}^2$ and $X_{N+h}^2$ obtained from
  model M$i$, for each   $i\in\{1,\cdots,6\}$, $N = 5,000$ and $h\in\{1,\cdots,50\}$;

\item the  one-step ahead forecast  values    $\check\sigma_{N^*+1}^2 := \tilde\sigma_{N^*+1}^2$ (denoted in the graphs by
$\hat\sigma_{N+h-1}^2(1)$),   for $N^* = N+h$,  $N= 5,000$  and
$h\in\{1,\cdots,50\}$.   The predictor  $\check\sigma_{N}^2(1)$ is obtained
directly from the sub-sample  $\{x_t\}_{t=1}^{n}$, by following steps
\textbf{F1} -\textbf{F5} (this figure only reports the graphs for the  case
$n=5,000$).
The  remaining  predicted values $\{\check\sigma_{N+h-1}^2(1)\}_{h=2}^{50}$ are
calculated by updating the
forecasting origin from $N = 5,000$ to   $N^*= N +h -1 $, that is,
by  introducing the observations $\{X_{N+h}\}_{h=1}^{49}$, one at a time, and
following steps \textbf{F2} -\textbf{F5};

\item  the  $h$-step  ahead forecast   values  considering the
    predictors  $\tilde  \sigma_{N+h}^2$  and  $\check \sigma_{N+h}^2$
    (denoted in the graphs by   $\sigma_{N}^2(h)$).  These values are obtained
   by following steps \textbf{F1} -\textbf{F5}    with
    forecasting origin     $N = 5,000$ (without update).  For all graphs  the size of
    the sub-sample used for parameter estimation and forecasting  is  $n =
    5,000$.
  \end{itemize}
The dashed lines in  Figure \ref{figforecast} correspond to the limiting
constants $L_1(i)$ and $L_2(i)$, for  $i\in\{1,\cdots,6\}$,  described in the sequel.

\begin{figure}[!ht]
  \centering
  \mbox{
   \includegraphics[width = 0.5\textwidth]{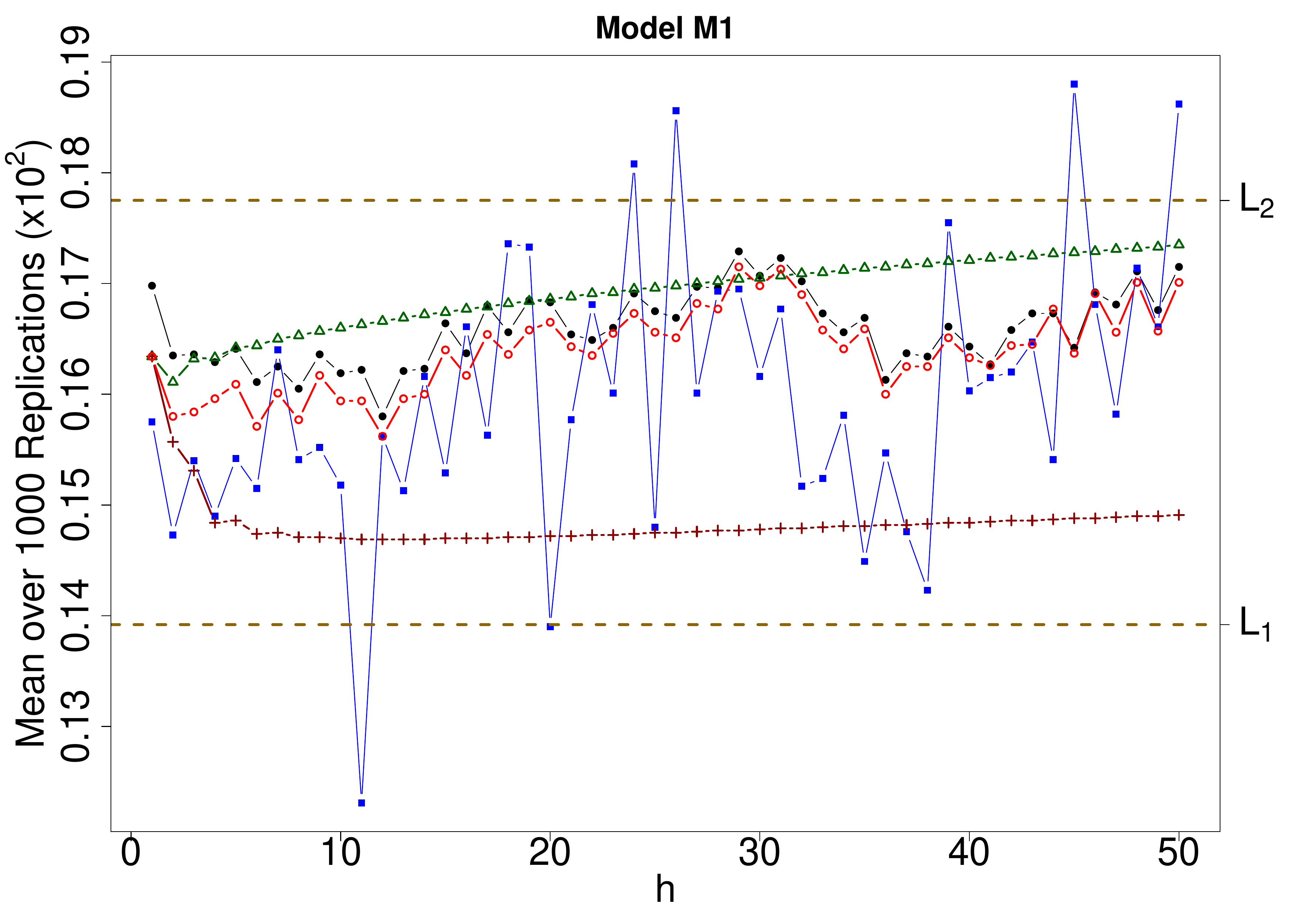}
   \includegraphics[width = 0.5\textwidth]{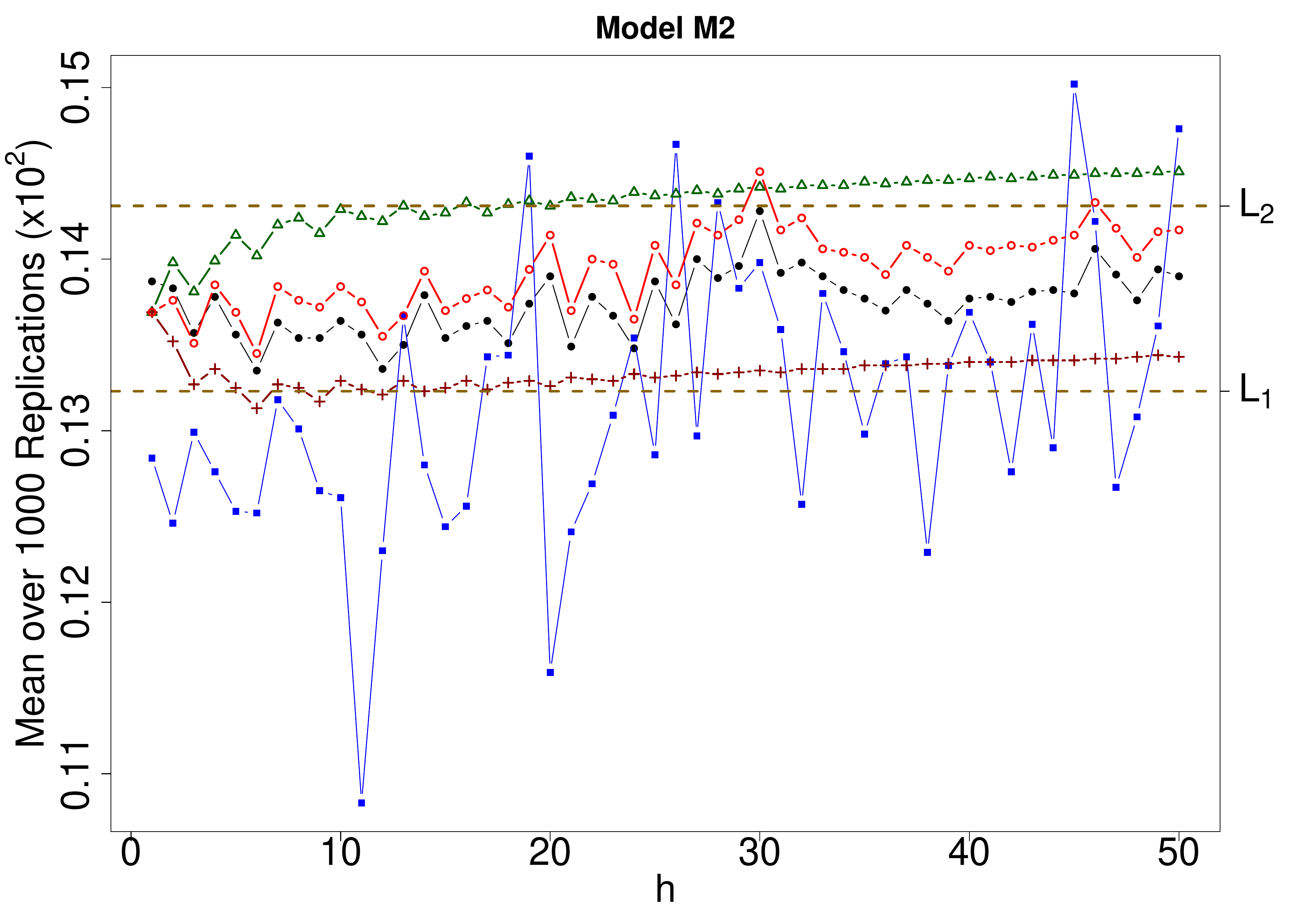}
 }
 \mbox{
    \includegraphics[width = 0.5\textwidth]{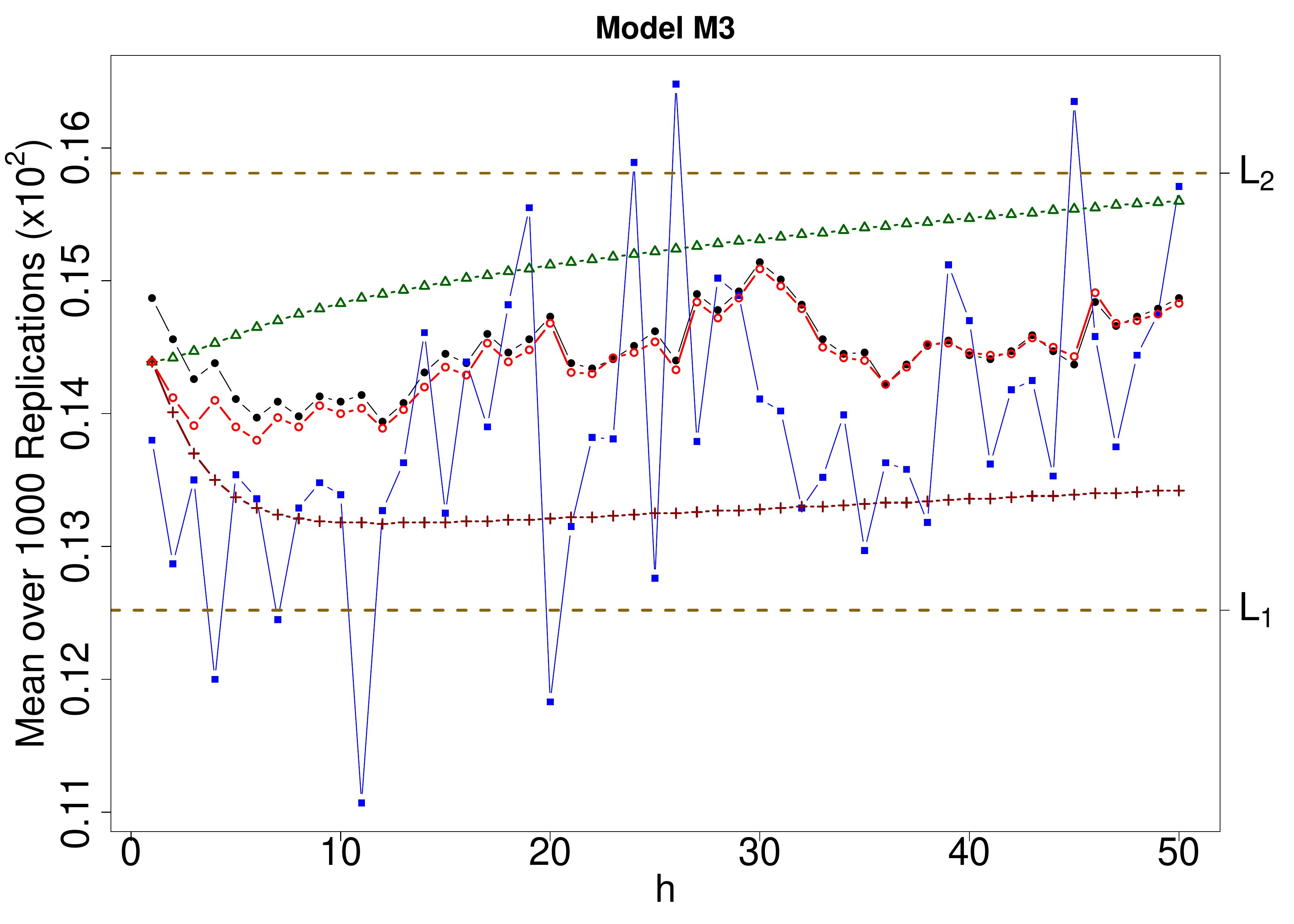}
    \includegraphics[width = 0.5\textwidth]{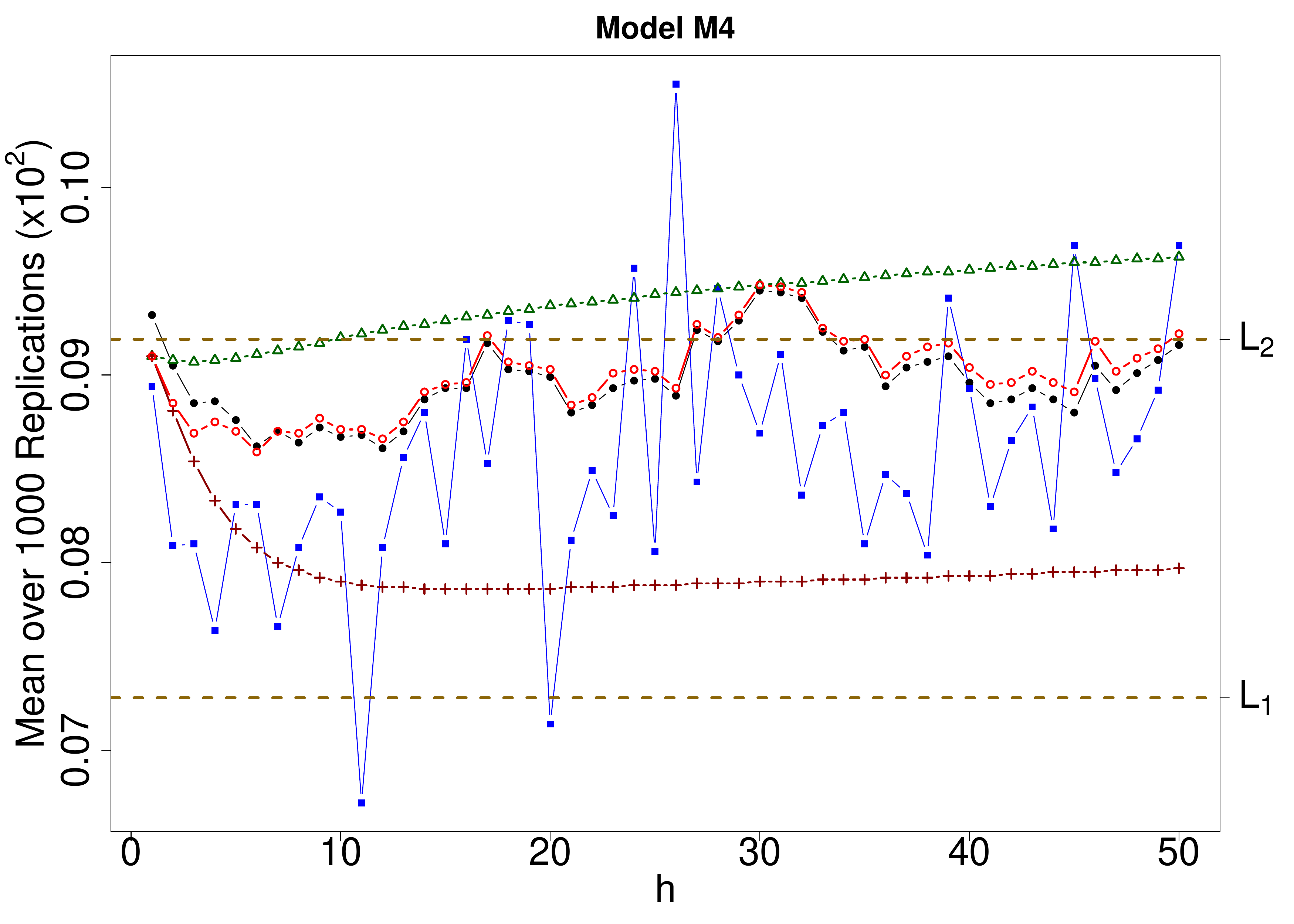}
      } \mbox{
    \includegraphics[width = 0.5\textwidth]{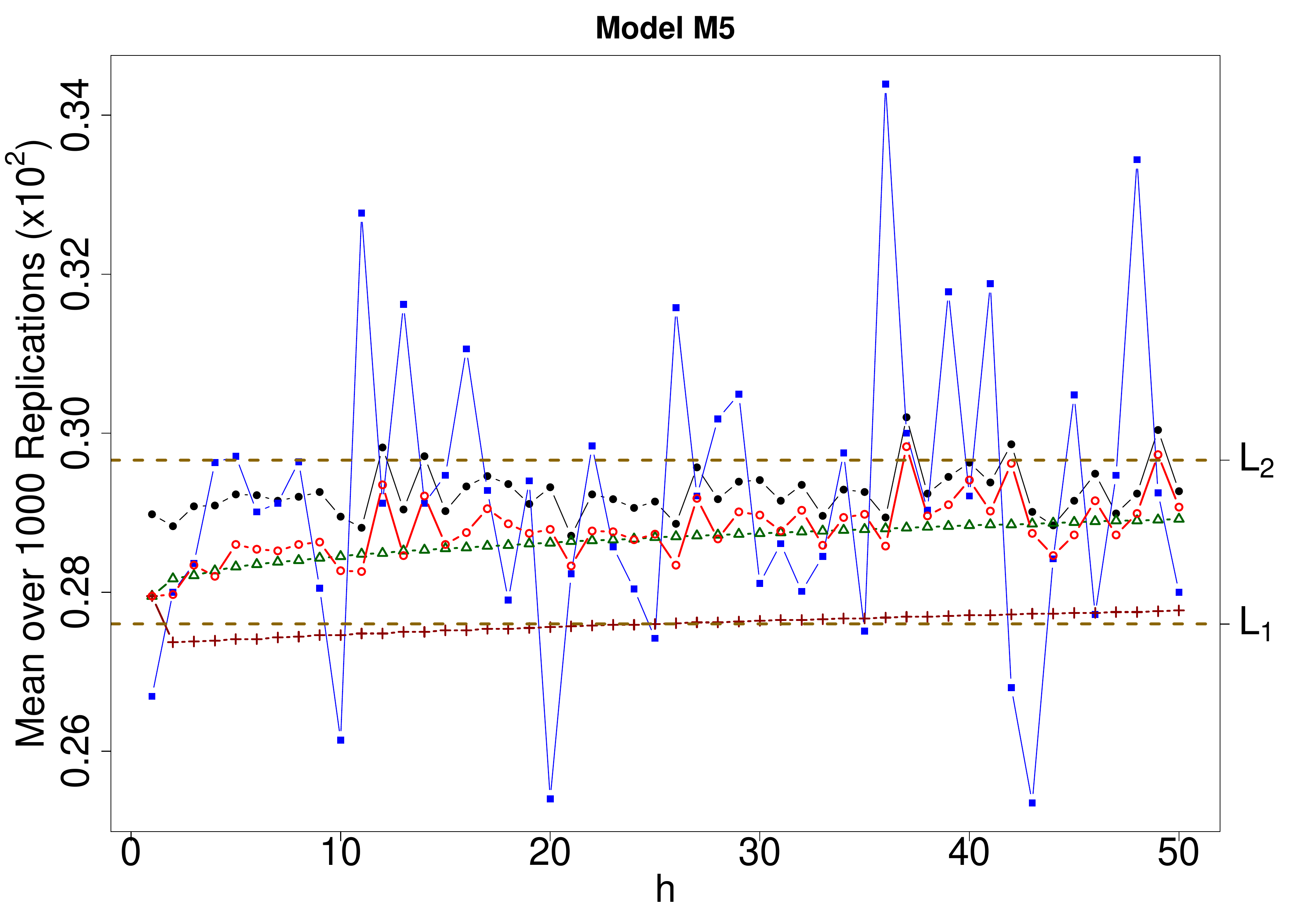}
    \includegraphics[width = 0.5\textwidth]{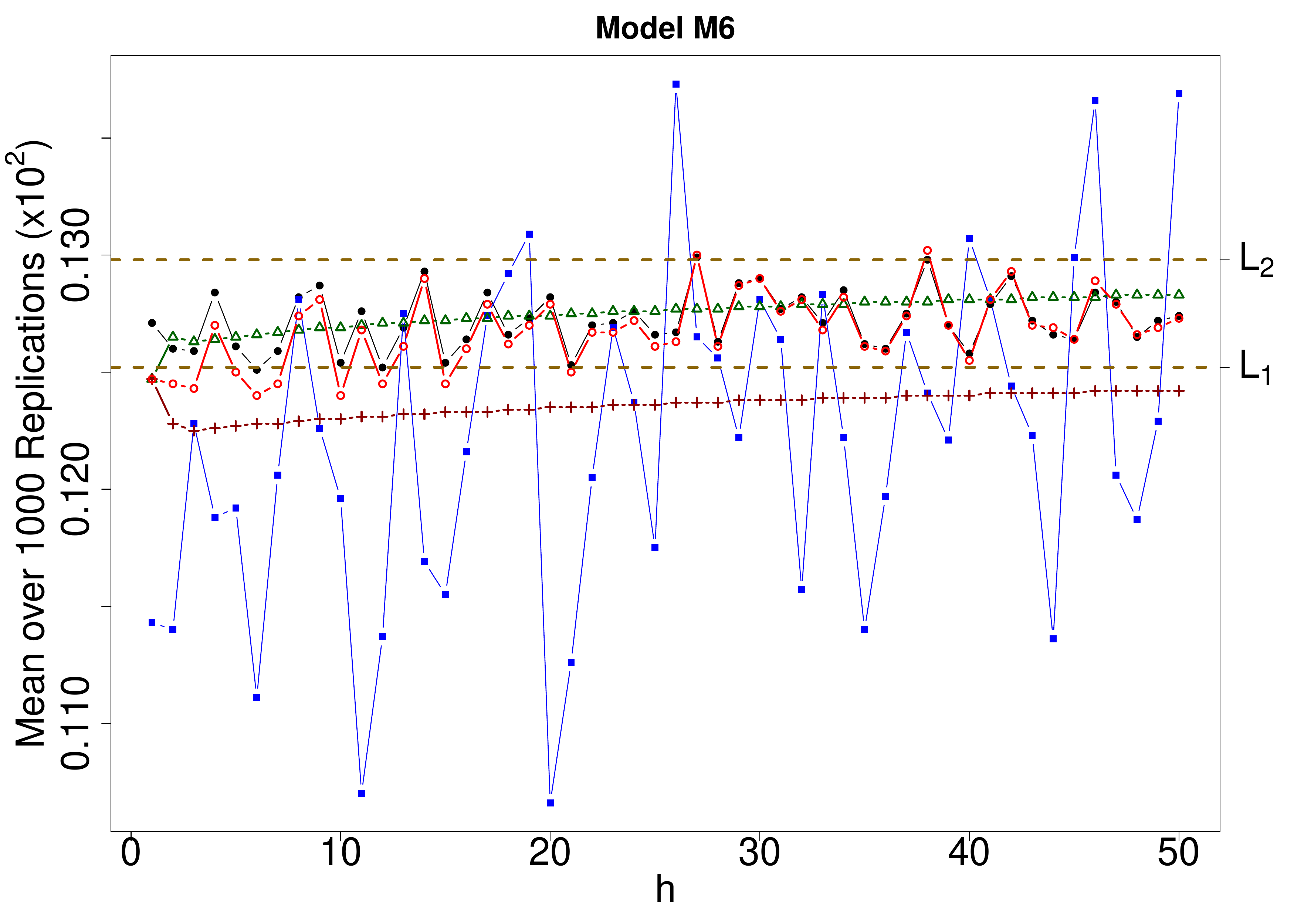}
  }\vspace{-1cm}
  \mbox{
    \includegraphics[width = 0.4\textwidth]{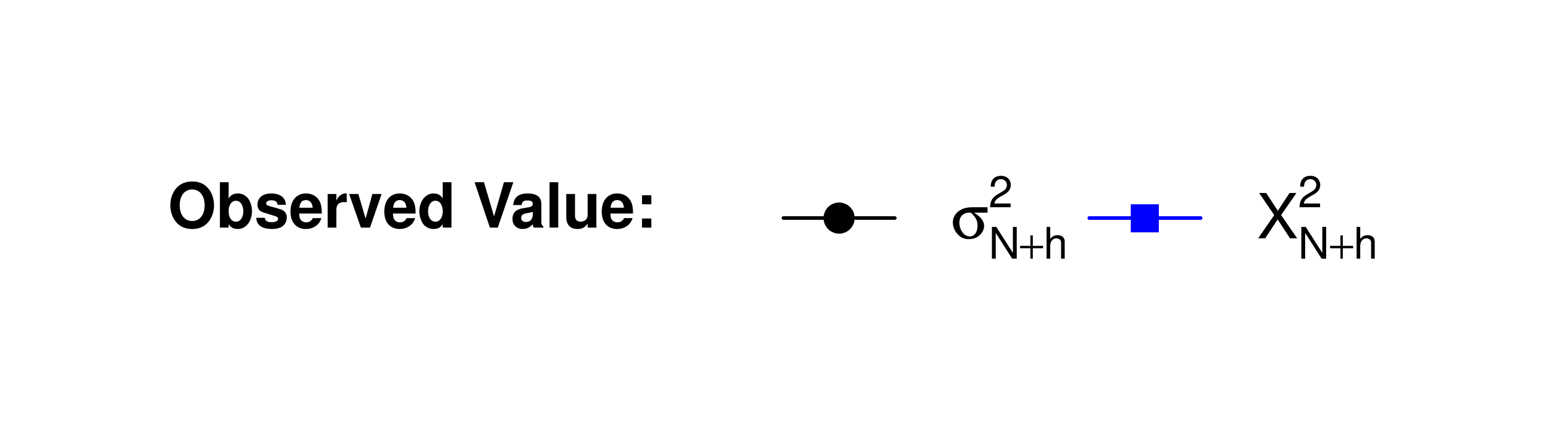}
  \includegraphics[width = 0.5\textwidth]{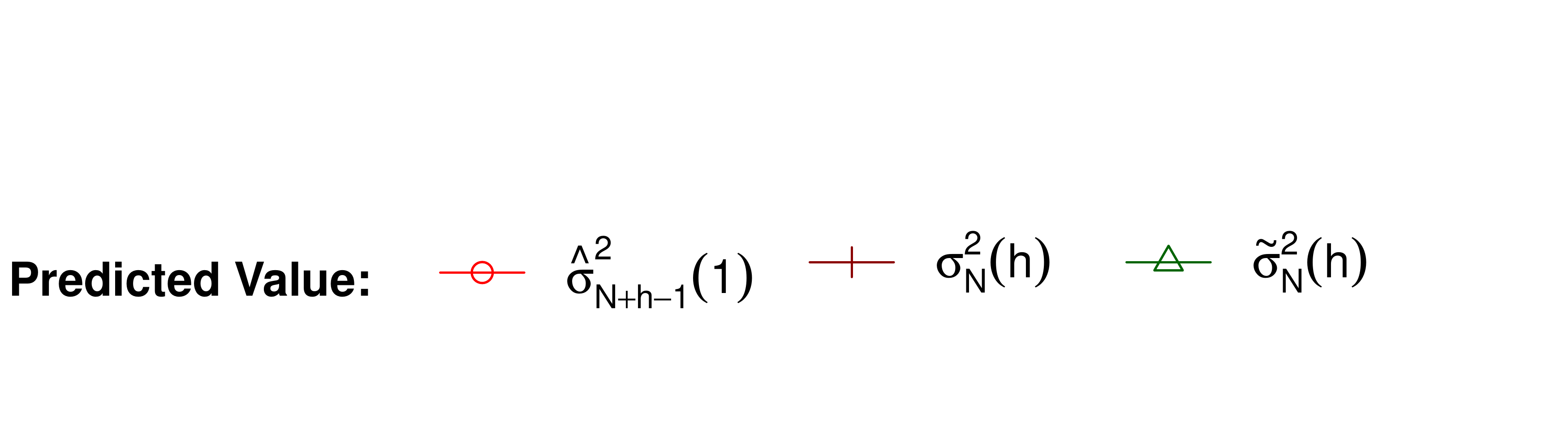}} \vspace{-0.6cm}\\
     \caption{For each model M$i$, $i\in\{1,\cdots,6\}$: the
    simulated values for $\sigma_{N+h}^2$; the  one-step ahead forecast
    $\check\sigma_{N^*+1}^2 := \tilde\sigma_{N^*+1}^2$ (denoted in the graphs by
    $\hat\sigma_{N+h-1}^2(1)$),   obtained by updating the forecasting origin to
    $N^*     = N +h - 1 $;  the  $h$-step  ahead forecast   values  considering the
    predictors  $\tilde     \sigma_{N+h}^2$ and  $\check \sigma_{N+h}^2$
    (denoted in the graphs by   $\sigma_{N}^2(h)$),   with  forecasting origin
    $N$.  For all models    $h\in\{1,\cdots,50\}$,  $N= 5,000$ and the size of
    the sub-sample used for parameter estimation and forecasting  is  $n =
    5,000$.  All values in the graphs correspond to the mean taken over 1,000 replications.
  }\label{figforecast}
\end{figure}

From Figure \ref{figforecast} we observe that,   for all models,  the means for
the one-step ahead  forecast values $\check\sigma_{N^*+1}^2$,
show the same behavior over the time as the means for the  true values
$\sigma_{N^*+1}^2$,
where $N^* = N+h-1$, $N = 5,000$ and $h\in\{1,\cdots,
50\}$.    As expected,  due to the error carried from the parameter estimation
(specially, from the distribution misspecification),  we observe a small
forecasting bias, which decreases as $h$ increases.  The decrease in the
forecasting bias,  as the forecasting origin is updated,  can be attributed to the
fact that we start the recurrence formula (step \textbf{F2}) assuming
$\mathds{E}(|Z_0|) = \sqrt{2/\pi}$ and as the new observations $X_{N+h}$ are
introduced, the constant $\mathds{E}(|Z_0|) $ is replaced by its  sample
estimate (step  \textbf{F3}), which provides more accurate values for  $g(Z_t)$
as $t$ increases ($t> N$).

Regarding the $h$-step ahead predictors $\check{\sigma}_{i,n+h}^2$ and  $\tilde{\sigma}_{i,n+h}^2$,
Figure \ref{figforecast} shows that   the estimation bias  is higher if we
consider the former one.  This figure also shows that,  for all models, the predicted
value converges to a constant as $h$ increases. This is expected since the
$h$-step ahead predictor is defined in terms of the conditional expectation.  In
fact, from expression \eqref{sigmacheckconvergence}, $\check{\sigma}_{N+h}^2$ converges to
$L_1(i) := e^{\omega(i)}$ as $h$ goes to infinity, where $\omega(i)$ denotes the parameter
$\omega$ for model M$i$ and hence,  from  expression \eqref{relation},
\begin{align}
  \tilde\sigma_{N+h}^2 :=  \check{\sigma}_{N+h}^2\bigg( 1
  +\frac{1}{2}\sigma^2_g\sum_{k=0}^{h-2}\lambda_{d,k}^2\bigg) \overset{h\to
    \infty}{-\!\!\!\longrightarrow}  & \, e^{\omega(i)}\bigg( 1
  +\frac{1}{2}\sigma^2_g(i)\sum_{k=0}^{\infty}\lambda_{d,k}^2(i)\bigg)
  \nonumber\\
   \approx & \,  e^{\omega(i)}\bigg( 1
  +\frac{1}{2}\sigma^2_g(i)\sum_{k=0}^{m}\lambda_{d,k}^2(i)\bigg) := L_2(i),\label{convergence}
\end{align}
for each $i \in \{1,\cdots,6\}$  and $m$ sufficiently large.  The values of
 $\omega(i)$ (also given in Table \ref{simpar}),  $L_1(i)$  and  $L_2(i)$,  for $m = 50,000$ and $i \in \{1,\cdots,6\}$, are presented
in Table \ref{values2}.

\begin{table}[!ht]
  \renewcommand{\arraystretch}{1.2}
  \caption{Values of  $\omega(i)$,  $L_1(i) :=e^{\omega(i)}$  and $L_2(i)$, defined in \eqref{convergence},  for $m = 50,000$ and $i \in
    \{1,\cdots,6\}$. }\label{values2}\vspace{0.2cm}
  {\footnotesize
  \begin{tabular*}{1\textwidth}{@{\extracolsep{\fill}}crrrrrr}
        \hline
        $i$       & \multicolumn{1}{c}{1} & \multicolumn{1}{c}{2}&
        \multicolumn{1}{c}{3} & \multicolumn{1}{c}{4}&
        \multicolumn{1}{c}{5} &\multicolumn{1}{c}{6} \\
    \cline{1-1} \cline{2-7}
    $\omega(i)$ & -6.5769 & -6.6278 & -6.6829 & -7.2247 & -5.8927 & -6.6829 \\
    $L_1(i)\times 100$ &      0.1392&      0.1323&      0.1252&      0.0728&
    0.2760&      0.1252\\
        $L_2(i)\times 100$  &     0.1775&      0.1431&      0.1581&      0.0919&      0.2966&      0.1298\\
        \hline
    \end{tabular*}}
  \end{table}

 Upon comparing the values of $L_1(i)$ and $L_2(i)$, given  in Table
 \ref{values2} (also reported in Figure  \ref{figforecast} as $L_1$ and $L_2$),
 for each $i\in\{1,\cdots, 6\}$,  respectively,   with the limits  $\lim_{h\to
   \infty} \check{\sigma}_{N+h}^2$  and  $\lim_{h\to \infty}
 \tilde{\sigma}_{N+h}^2$ (see Figure \ref{figforecast}), we conclude that these
 values are close to each other,  for all models.    A small difference between
 $L_1(i)$ and $\lim_{h\to    \infty} \check{\sigma}_{N+h}^2$ (respectively,
 $L_2(i)$ and $\lim_{h\to    \infty} \tilde{\sigma}_{N+h}^2$) is expected since
 the former one is calculated using  the true parameter values while
 $\check{\sigma}_{N+h}^2$ is obtained by considering the estimates for the  parameter.

\section{Analysis of an Observed Time Series}\label{analysisObserved}

This section presents the analysis of the S\~ao Paulo Stock Exchange Index
(Bovespa Index or IBovespa) log-return time series. We consider the FIEGARCH
model, fully described in this paper, and we compare its forecasting performance
with other ARCH-type models.  The total number of observations for the IBovespa
time series is $n = 1737$. We consider the first 1717 observations to fit the
models and we reserve the last 20 ones to compare with the out-of-sample
forecast.

Figure \ref{figibov1} (a) presents IBovespa time series $\{P_t\}_{t=1}^{1718}$,
in the period of January/1995 to December/2001. We observe a strong decay in the
index value close to $t=1,000$ (that is, January 15, 1999). This period is
characterized by the Real (the Brazilian currency) devaluation.  Figures
\ref{figibov1} (b) and (c) present, respectively, the IBovespa log-return time
series, $\{r_t\}_{t=1}^{1717}$, and the square of the log-return time series,
$\{r_t^2\}_{t=1}^{1717}$, in the same period.  Observe that the log-return
series presents the stylized facts of financial time series such as apparent
stationarity, mean around zero and clusters of volatility. Also, in Figure
\ref{figibov2} we observe that, while the log-return series presents almost no
correlation, the sample autocorrelation of the square of the log-return series
assumes high values for several lags, pointing to the existence of
heteroskedasticity and possibly long memory. Notice that the periodogram of
$\{\ln(r_t^2)\}_{t=1}^{1717}$, presented in Figure \ref{spectral} (c), also
indicates possibly long-memory in the conditional variance. Regarding the
histogram and the QQ-Plot, we observe that the distribution of the log-return
series seems approximately symmetric and leptokurtic.

\begin{figure}[htbp]
  \centering \mbox{ \hspace{-6pt}\subfigure[ $\{P_t\}_{t=1}^{1718}$]
    {\includegraphics[width=0.3\textwidth, height =
      0.15\textheight]{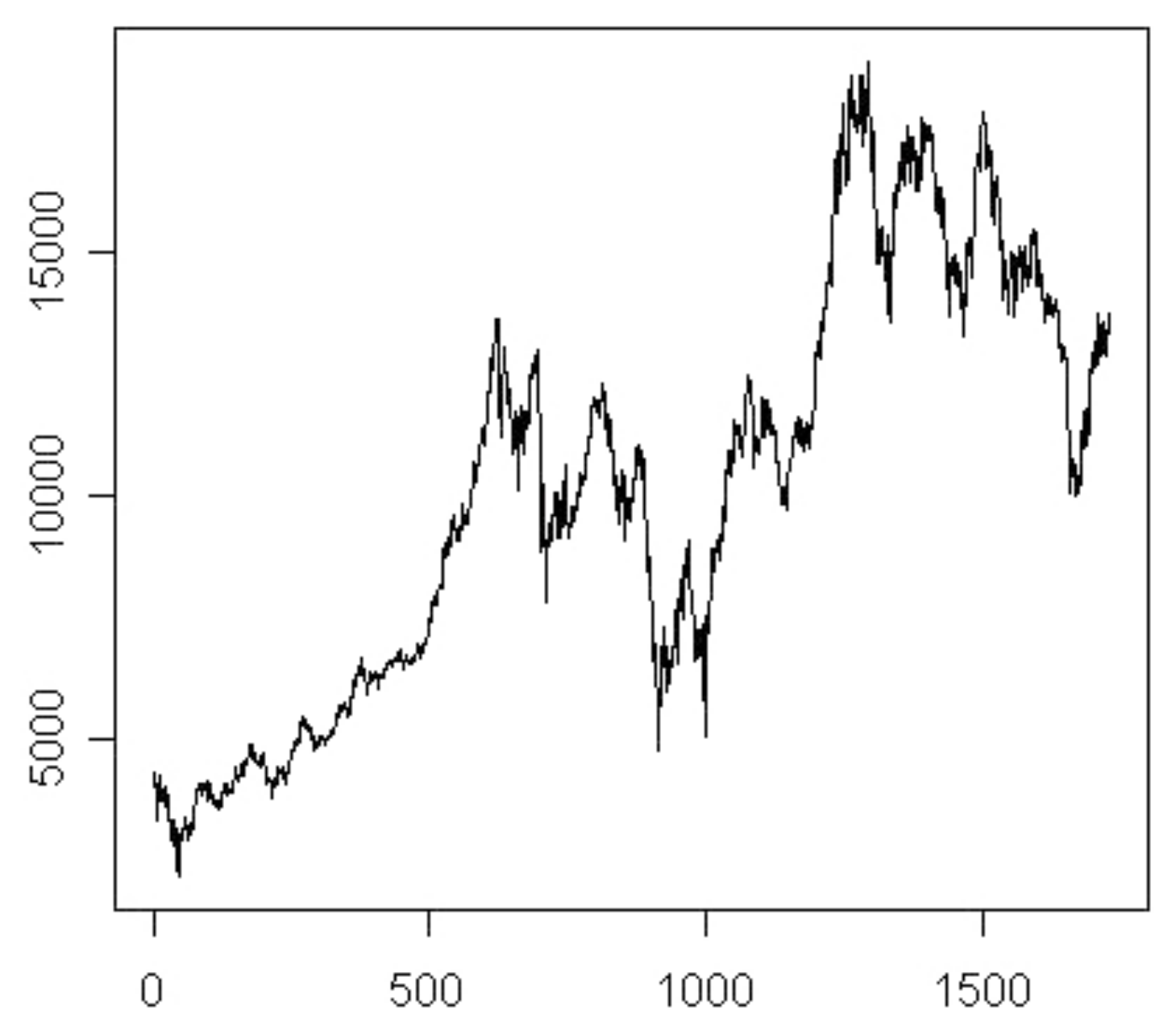}}\hspace{10pt} \subfigure[
    $\{r_t\}_{t=1}^{1717}$] {\includegraphics[width=0.3\textwidth, height =
      0.15\textheight]{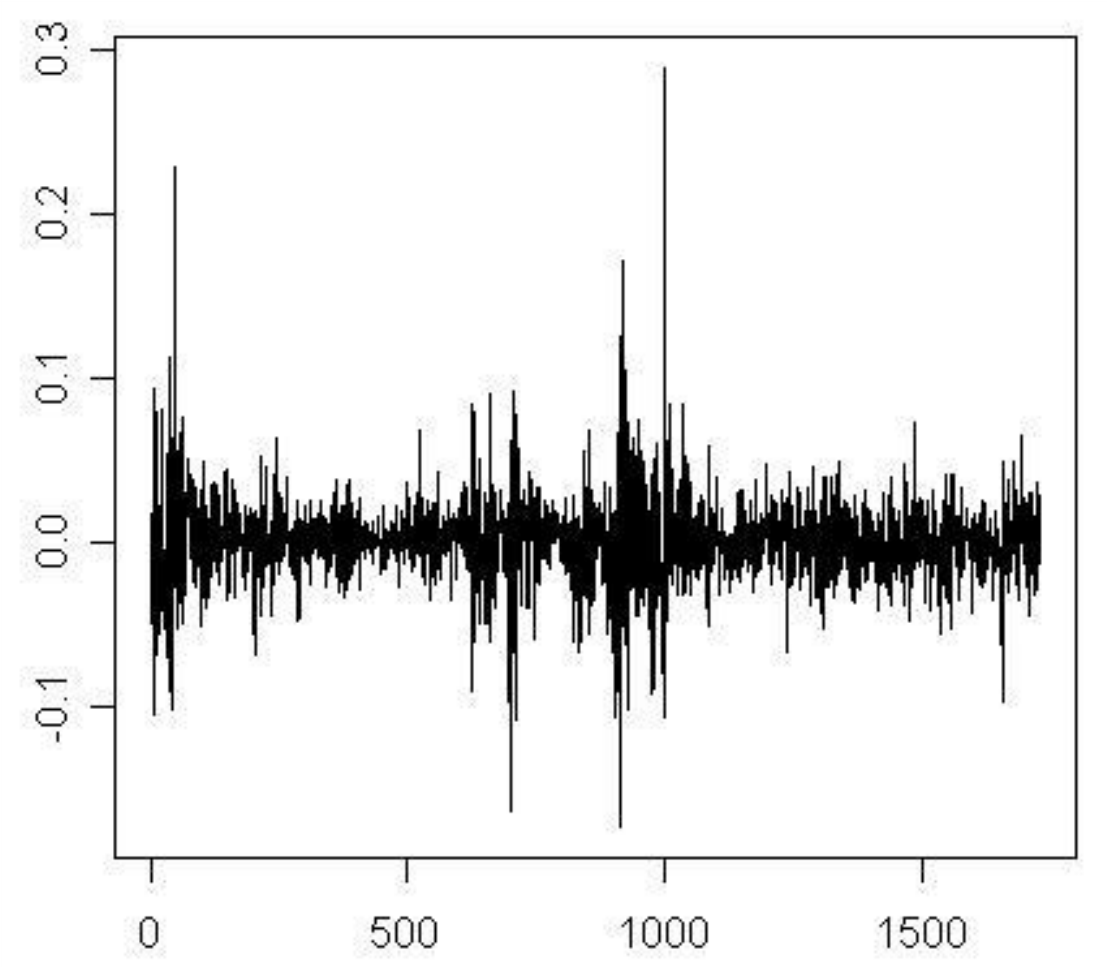}}\hspace{10pt} \subfigure[
    $\{r_t^2\}_{t=1}^{1717}$] {\includegraphics[width=0.3\textwidth, height =
      0.15\textheight]{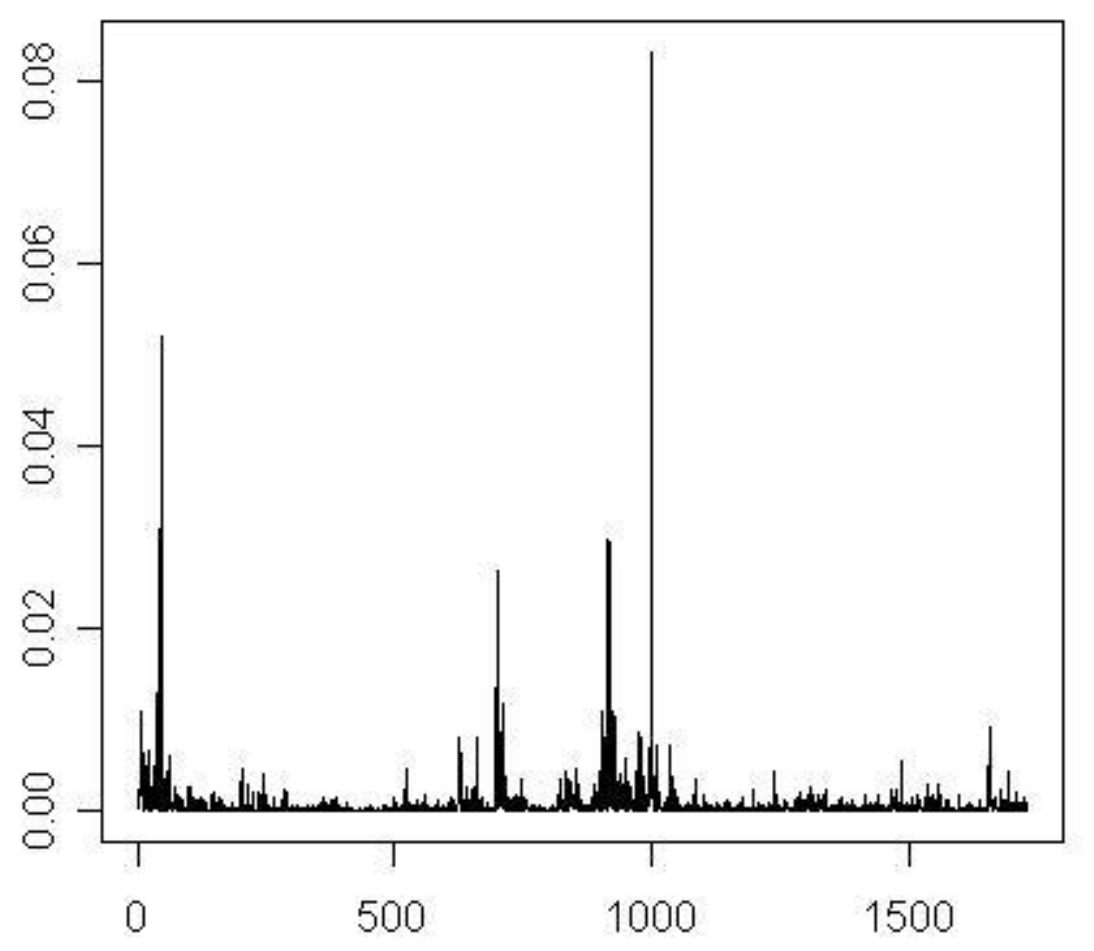}} }
  \caption{Time series: (a) Bovespa index; (b) IBovespa log-returns; (c) square
    of the IBovespa log-returns, in the period of January/1995 to
    December/2001.}\label{figibov1}
\end{figure}

\begin{figure}[h]
  \centering \mbox{ \subfigure[]{\includegraphics[width=0.23\textwidth, height =
      0.15\textheight]{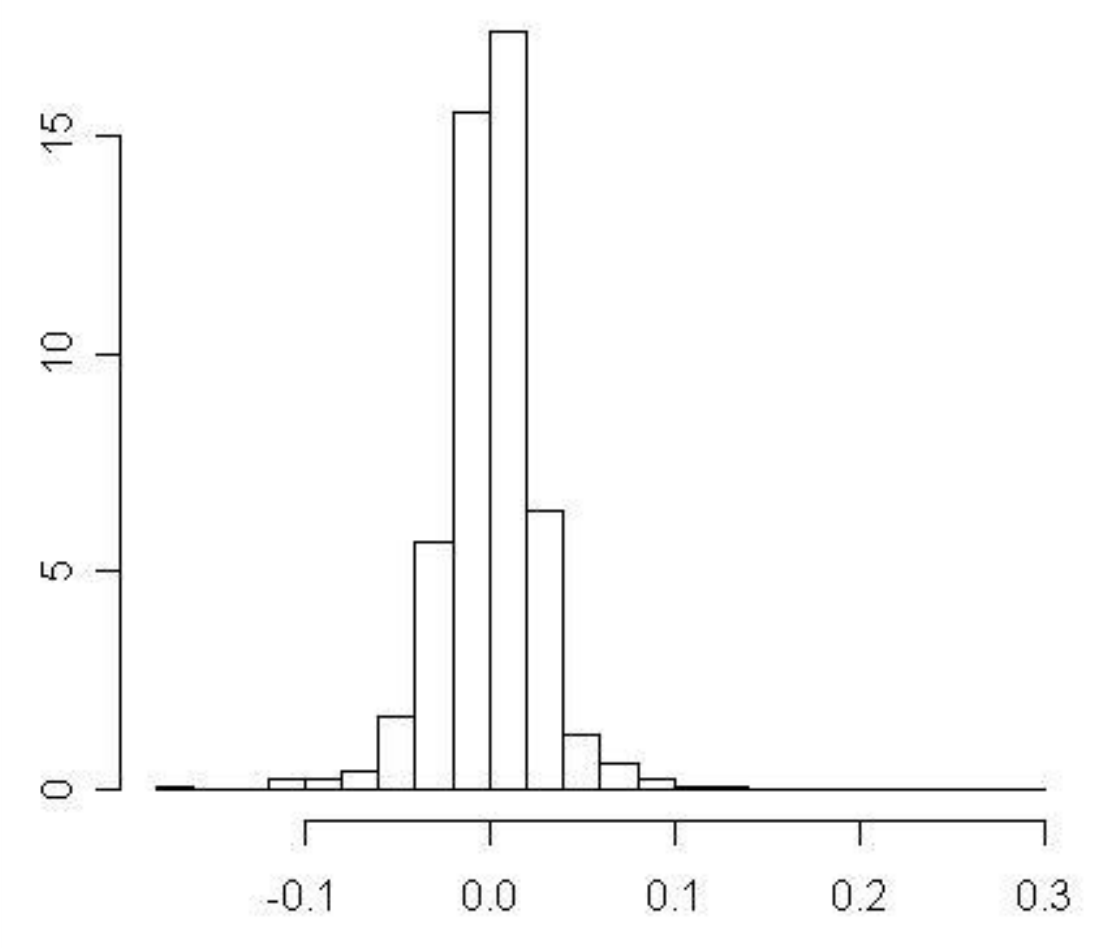}}
    \subfigure[]{\includegraphics[width=0.23\textwidth, height =
      0.15\textheight]{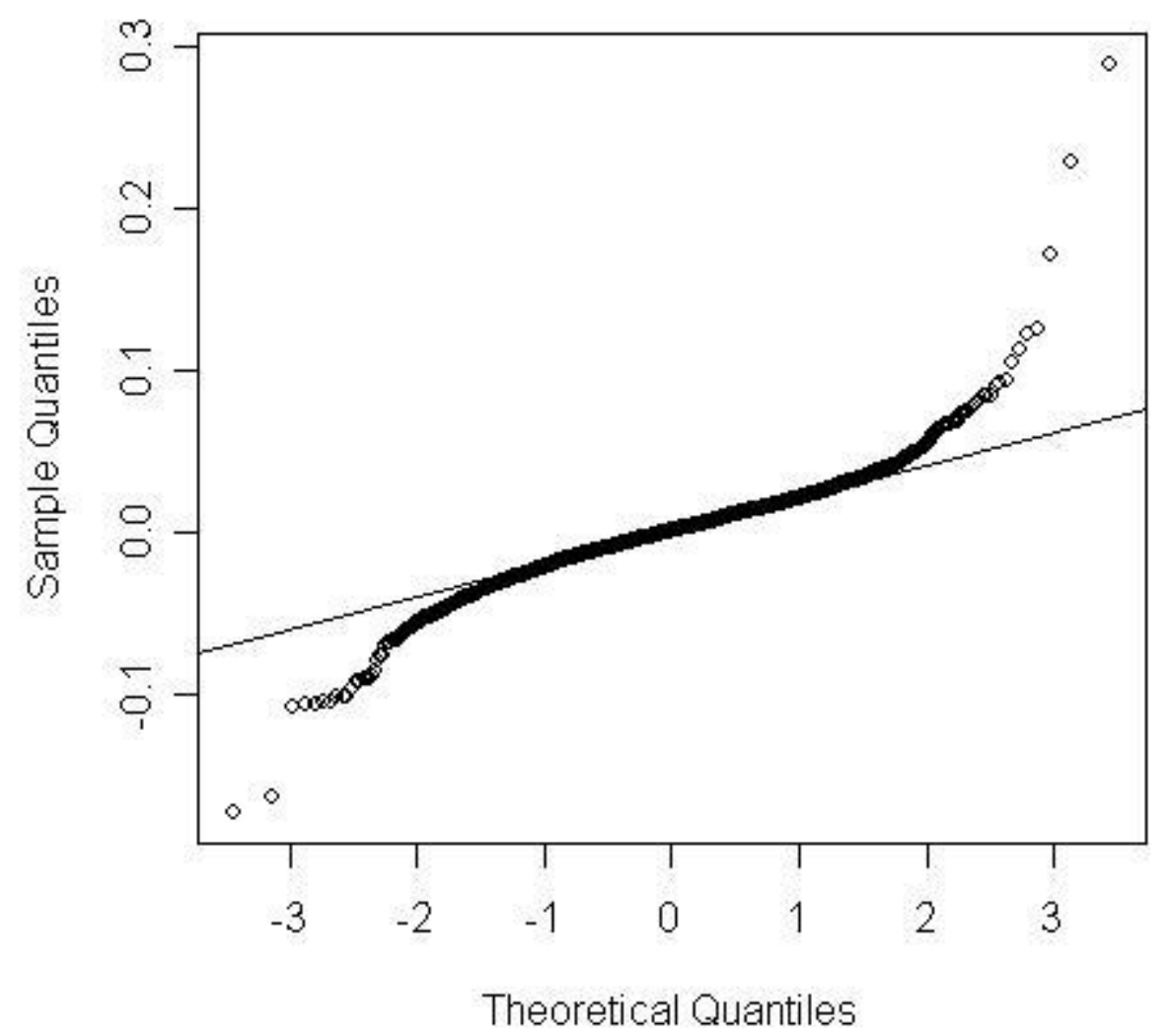}}
    \subfigure[]{\includegraphics[width=0.23\textwidth, height =
      0.15\textheight]{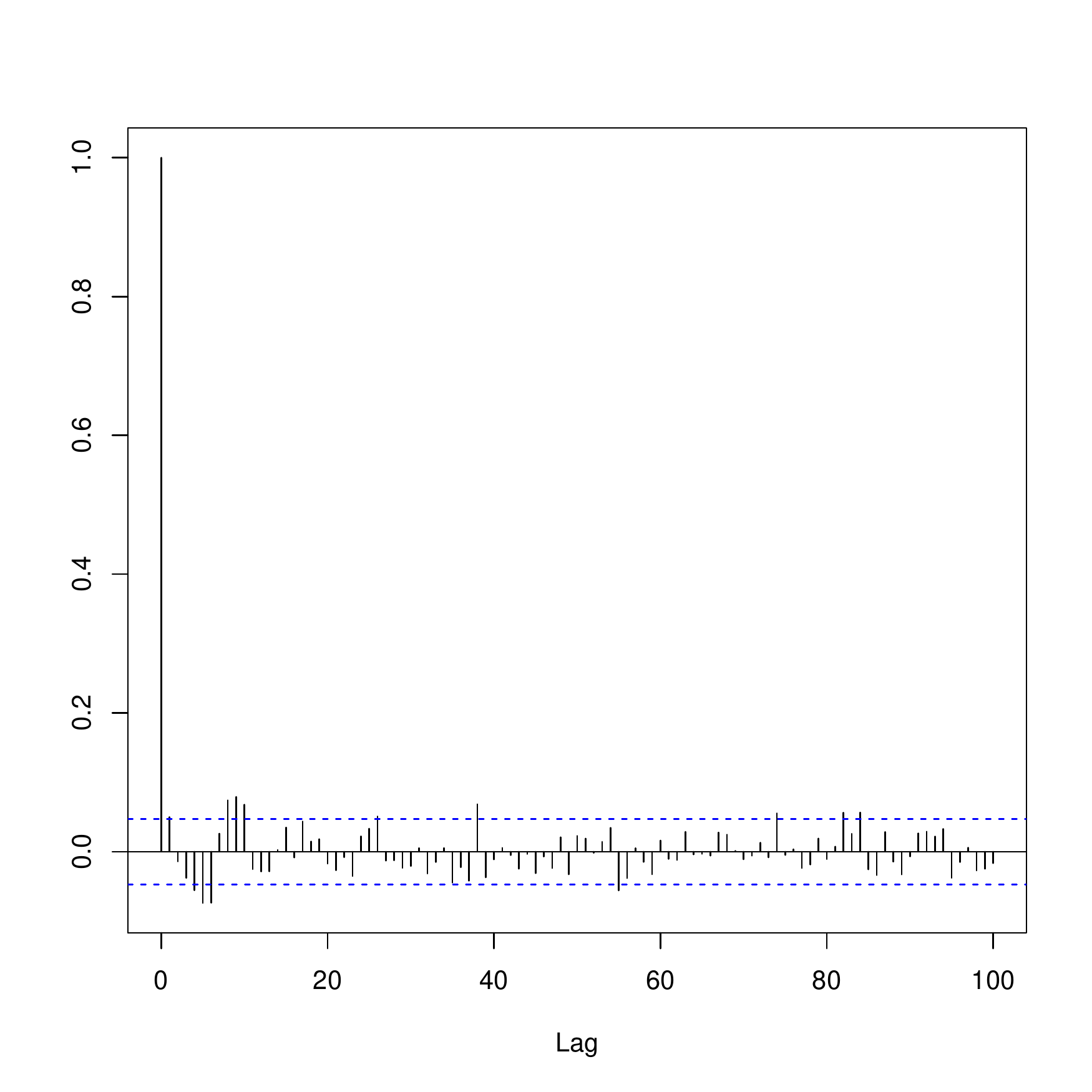}}
    \subfigure[]{\includegraphics[width=0.23\textwidth, height =
      0.15\textheight]{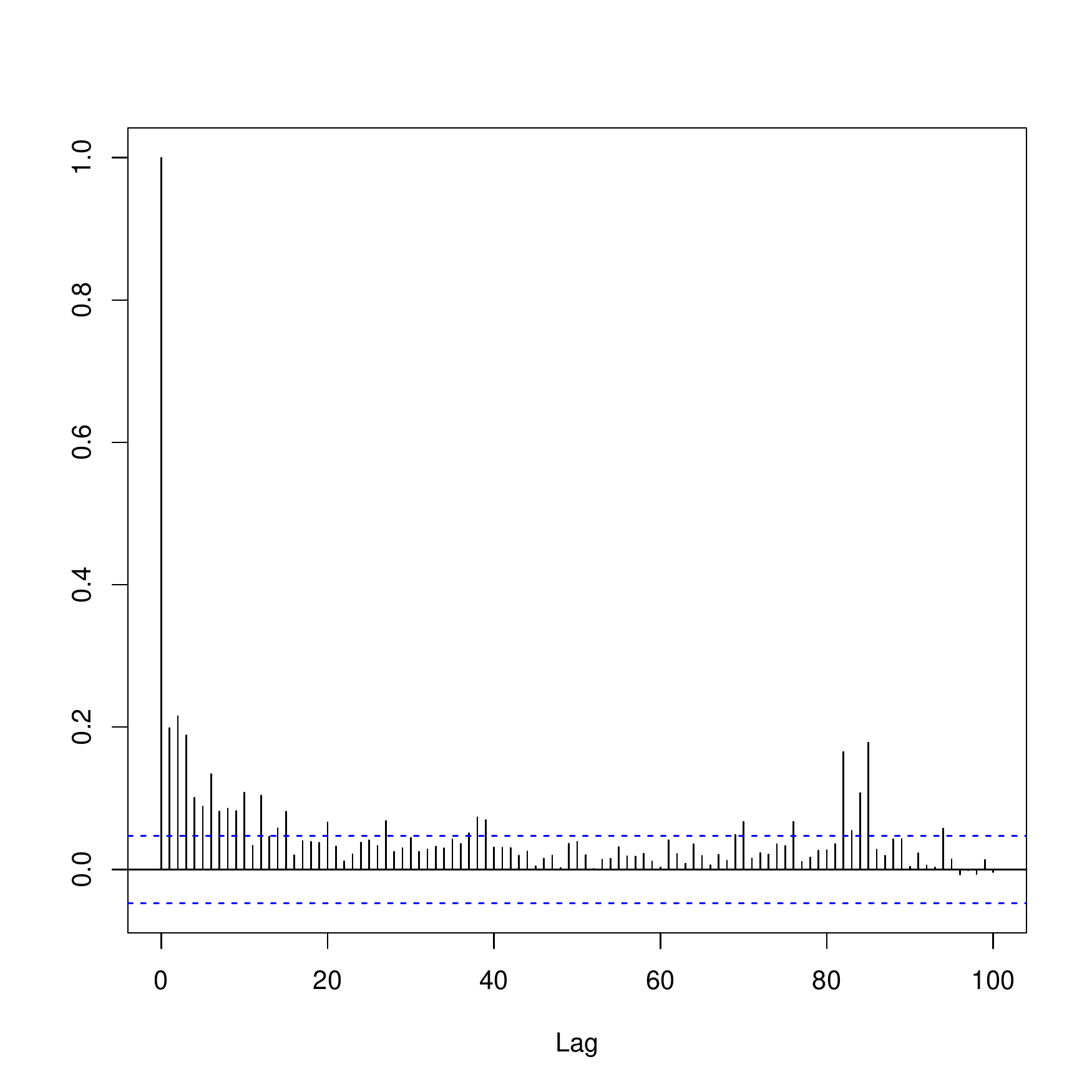}} }
  \caption{(a) Histogram; (b) QQ-Plot and (c) sample autocorrelation of the
    IBovespa log-return series and (d) sample autocorrelation of the square of
    the IBovespa log-return series.}\label{figibov2}
\end{figure}

To investigate whether the stationarity property holds  for the  time series $\{r_t\}_{t=1}^{1717}$ we apply the runs test (or Wald-Wolfwitz test), as described in  \cite{GR2012}. Due to the magnitude of the data   we multiply the time series values by 100 before applying the test.
The p-values for the test considering the moments of order\footnote{For $r>10$ the values of $\{r_t^r\}_{t=1}^{1,717}$   are too close to zero and the test always returns the same p-value as $r=10$. } $r\in\{1, \cdots, 10\}$  are reported in Figure \ref{GRtest}.     For comparison, this figure also shows the p-values of the  test applied to the simulated time series presented in Figure \ref{fiegsim}.  Notice that, for all $r\in\{1, \cdots, 10\}$  the test does not reject the null hypothesis of stationarity.

\begin{figure}[h]
  \centering
  \includegraphics[width=0.5\textwidth]{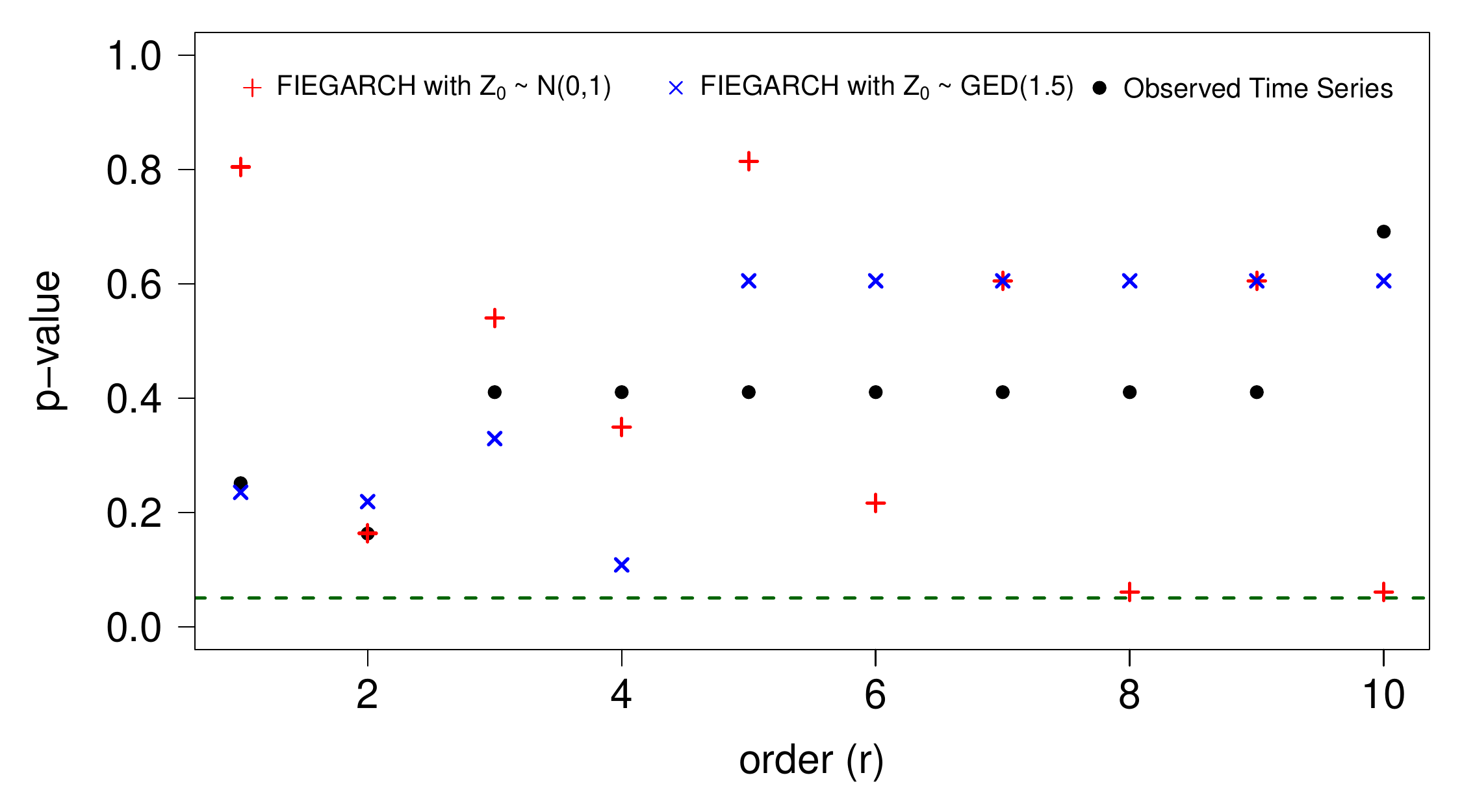}
  \caption{The p-values for the stationarity hypothesis  test considering the moments of order $r\in\{1, \cdots, 10\}$. The dashed line corresponds to p-value = 0.05.}\label{GRtest}
\end{figure}

To analyze if the ergodicity property holds for the time series $\{r_t\}_{t=1}^{1717}$ we perform the  test  described in \cite{DEL2001}.   For comparison, we also apply this  test to the simulated time series (only for sample size $n=2,000$) considered in Section \ref{simulationsection}.   The test results are given in Table \ref{ergodicityTest}.  The reported values are the proportion of p-values smaller than 0.05 and 0.10 in a total of 100 repetitions of  step 3 of the  Algorithm 1  given in  \cite{DEL2001}.  Moreover,   for  the simulated time series, the values in Table  \ref{ergodicityTest} correspond to the mean taken over 1,000 replications.    Notice that  the proportion of p-values smaller than  0.05 (equivalently, 0.10) is always higher for the simulated time series (known to be ergodic)  then for the observed time series. Given that the proportion of p-values smaller than 0.05 and 0.10 is close to the expected, we conclude that the ergodicity property holds for $\{r_t\}_{t=1}^{1717}$.

\begin{table}[!htbp]
    \renewcommand{\arraystretch}{1.2}
  \centering
  \caption{Proportion of p-values smaller than 0.05 and 0.10 in a total of 100 repetitions of  step 3 of the  Algorithm 1  given in  \cite{DEL2001} for the simulated time series obtained from model M$i$, with $i\in\{1,\cdots,6\}$, and for the observed time series $\{r_t\}_{t=1}^{1717}$.} \label{ergodicityTest}\vspace{0.2cm}
  {\footnotesize
    \begin{tabular*}{1\textwidth}{@{\extracolsep{\fill}}cccccccc}
      \hline
p-values &  M1 & M2 & M3& M4 & M5& M6 & $\{r_t\}_{t=1}^{1717}$\\
   \cline{1-1}  \cline{2-7} \cline{8-8}
0.05 & 0.10&  0.08&  0.09&  0.09&  0.07&  0.07& 0.05           \\
0.10 & 0.17&  0.13&  0.14&  0.15&  0.13&  0.12& 0.11\\
\hline
    \end{tabular*}}
\end{table}

The analysis of the sample autocorrelation function suggests an
ARMA$(p_1,q_1)$-FIEGARCH$(p_2,d,q_2)$ model. While an ARMA model accounts for
the correlation among the log-returns, a FIEGARCH model take into account the
long memory (in the conditional variance) and the heteroskedasticity
characteristics of the time series.   To select the best ARMA$(p_1,q_1)$-FIEGARCH$(p_2,d,q_2)$ model for  the data   we initially
considered all possible models with $p_1,q_1\in\{0,1,2,3\}$ and
$p_2,q_2 \in\{0,1,2\}$ and applied the quasi-likelihood method to estimate the
unknown parameters.  Then we eliminate the models with correlated residuals  and
selected the best models, with respect to the log-likelihood, Bayesian  (BIC),  Akaike  (AIC) and
    Hannan-Quinn (HQC)  information  criteria   (in this step three models were selected).  The models order and the
corresponding AIC, BIC and HQC values are reported in Table \ref{aic}.
Boldface indicates that the model was the best with respect to the corresponding   the criterion.

\begin{table}[!ht]
      \renewcommand{\arraystretch}{1.2}
  \caption{Log-likelihood value and   Bayesian  (BIC),  Akaike  (AIC) and
    Hannan-Quinn (HQC)  information  criteria values   for \blue three  \black competitive ARMA$(p_1,q_1)$-FIEGARCH$(p_2,d,q_2)$ models
    fitted to the IBovespa log-return time series.}\label{aic}\vspace{0.2cm}
  {\footnotesize
\begin{tabular*}{1\textwidth}{@{\extracolsep{\fill}}ccccccccc}
  \hline
  \multicolumn{5}{c}{Order} & \multicolumn{4}{c}{Criterion}\\
    \cline{1-5} \cline{6-9}
    $p_1$ & $q_1$& $p_2$& $d$ & $q_2$ &Log-likelihood&  BIC & AIC & HQC \\
        \cline{1-5} \cline{6-9}
    3	&2	&1	&0.3651	&1	&\textbf{4142.260}	& -8202.588 & -8262.520 & -8240.344\\
    0	&1	&0	&0.3578	&1	&4138.552	& \textbf{-8232.414}	&-8265.104&\textbf{	-8253.008}\\
    0	&  2	& 0	&0.3785   &	1	&4141.197	& -8230.256 & \textbf{-8268.394} & -8254.282\\
  \cline{1-5} \cline{6-9}
  \hline
  \end{tabular*}}
 \renewcommand{\arraystretch}{1}
\begin{flushleft}\vspace{-0.2cm}
{\footnotesize Note: Boldface indicates that the model was the best, among all combinations of $p_1,q_1\in\{0,1,2,3\}$ and
$p_2,q_2 \in\{0,1,2\}$,\linebreak
 \phantom{Note:} with respect to the corresponding criterion. }
\end{flushleft}
\end{table}
As shown in Table \ref{aic}, the  values of the  selection criteria  did not vary
much amongst the tested models so we choose the most parsimonious one, namely, ARMA(0,1)-FIEGARCH$(0,d,1)$.  We compare the forecasting
performance of this model with other ARCH-type models  and with a radial basis function model.   For this comparison
the  order of the   ARMA$(p_1,q_1)$ part of the  model was not changed, that is,
we fixed $p_1 = 0$ and $q_1 =1$ for all ARCH-type models.  The EGARCH$(p_2,q_2)$
model was   set to have the same values for  $p_2$ and $q_2$ as the FIEGARCH
model  so we could investigate the influence of the long memory parameter $d$.
For the GARCH($p_2,q_2$) model we choose the smallest values of $p_2$ and
$q_2$ for which the residuals of the model are not correlated. The same was done
for the ARCH$(p_2)$ model (which resulted in $p_2 = 6$).  The ARCH(1) model was
presented only for comparison.     The estimated coefficients for the  ARCH-type models are given in Table \ref{allmodels},
with the corresponding log-likelihood value.  Notice that, the FIEGARCH model
fitted to this time series present the same parameters values as model M4
considered in the simulated study in Section \ref{simulationsection}.

\begin{table}[!htbp]
  \renewcommand{\arraystretch}{1.1}
  \centering
  \caption{Fitted models and their respective log-likelihood, BIC, AIC and HQC values.  The number in parenthesis corresponds to
    the  standard error of the estimate.}\label{allmodels}\vspace{0.2cm}
  {\footnotesize
    \begin{tabular*}{1\textwidth}{@{\extracolsep{\fill}}cccccc}
      \hline
      \multirow{2}{*}{Estimate}     & {\scriptsize ARMA(0,1) + } & {\scriptsize ARMA(0,1) + } & {\scriptsize ARMA(0,1) + } & {\scriptsize ARMA(0,1) + } & {\scriptsize ARMA(0,1) + }\\
      &  {\scriptsize ARCH(1)}   &   {\scriptsize ARCH(6) }  &  {\scriptsize GARCH(1,1) }  &  {\scriptsize EGARCH(0,1)}   &
      {\scriptsize FIEGARCH(0,$d$,1)}\\
      \hline
      \vspace{-0.2cm}\\
      $\hat\theta_1$   &   -0.1138 (0.0200)   &  -0.0642 (0.0267)  &   -0.0647 (0.0266)
      &   -0.0751   (0.0254) &   -0.0776  (0.0257)\\

      $\hat\omega$   &   0.0004  (0.0000)  &   0.0002  (0.0000)  &   0.0000  (0.0000)
      &  -7.4694 (0.0969)   &   -7.2247  (0.2143)\\

      $\hat\alpha_1$   &   0.6071 (0.0581)   &   0.2307 (0.0417)  &   0.2019 (0.0247)  &   -   &   -\\

      $\hat\alpha_2$   &   -    &   0.1540  (0.0333)   &   -   &   -   &   - \\

      $\hat\alpha_3$   &   -    &   0.1852  (0.0390)   &   -   &   -   &   - \\

      $\hat\alpha_4$   &   -   &   0.1145 (0.0348)  &   -   &   -   &   - \\

      $\hat\alpha_5$   &   -   &   0.0641 (0.0290)   &   -   &   -   &   -\\

      $\hat\alpha_6$   &   -   &   0.0635 (0.0257)   &   -   &   -   &   -\\

      $\hat\beta_1$   &   -   &   -   &   0.7659 (0.0271)   &   0.9373  (0.0103)  &
      0.6860 (0.0986)\\

      $\hat d$   &   -   &   -   &   -   &   -   &  0.3578 (0.0810)\\

      $\hat\theta$  &   -   &   -   &   -   &   -0.1653 (0.0197)   &   -0.1661 (0.0224)\\

      $\hat\gamma$ &   -   &   -   &   -   &  0.2782 (0.0300)  &   0.2972 (0.0332)\vspace{0.1cm}\\

      \hline

     log-likelihood & 3934.337 & 4060.372 & 4072.622 &
      4137.625 & 4138.552\\
BIC & -7846.329 &-8061.157 &-8115.451& -8238.008& -8232.414\\
AIC & -7862.674 &-8104.744 &-8137.244 &-8265.250 & -8265.104\\
HQC & -7856.626 &-8088.616 &-8129.180 &-8255.170 &-8253.008\\

      \hline

    \end{tabular*}}
\end{table}

To fit a radial basis model to the data (no exogenous variables are considered) we assume that  $\{r_t\}_{t\in\mathds{Z}}$ can be written as    (see \cite{GI2003},  \cite{DB2010})
\[
r_t = \phi(\boldsymbol{y}_{t-1}) + \psi(\boldsymbol{y}_{t-1})Z_t :=  \phi(\boldsymbol{y}_{t-1})  + \varepsilon_t, \quad \mbox{for all} \, \,  t\in\mathds{Z},
\]
with $\boldsymbol{y}_{t-1} =(r_{t-1}, \cdots, r_{t-p})$, for some $p > 0$, $\varepsilon_t :=  \psi(\boldsymbol{y}_{t-1})Z_t$, $\mathds{E}(Z_t) = 0$ and $\mathds{E}(Z_t^2) = 1$.    Under these assumptions,   $\mathds{E}(r_t | \boldsymbol{y}_{t-1})  = \phi(\boldsymbol{y}_{t-1})$ and $\mathds{E}(\varepsilon_t^2 | \boldsymbol{y}_{t-1}) = \psi^2( \boldsymbol{y}_{t-1})$,  for all $t\in\mathds{Z}$.   Therefore,   we use  neural networks   $\Phi_{n}$  and  $\boldsymbol{\Psi}_{n}$  to approximate, respectively,  $\phi(\boldsymbol{y})$ and $\psi^2(\boldsymbol{y})$,  and obtain
    \[
\hat\phi(\boldsymbol{y}) = \Phi_{n}(\boldsymbol{y}; \boldsymbol{\hat w}_1) \quad \mbox{and} \quad \hat\psi^2(\boldsymbol{y}) = \Psi_{n}(\boldsymbol{y}; \boldsymbol{\hat w}_2), \quad \mbox{for all } \,\, \boldsymbol{y} \in \mathds{R}^p,
\]
where
\small
\[
\boldsymbol{\hat w}_1 = \mbox{arg min}\bigg\{\frac{1}{n-p}\sum_{t = p+1}^n\Big[r_t - \Phi_{n}(\boldsymbol{y}_{t-1}; \boldsymbol{w})\Big]^2\bigg\}
\quad \mbox{\normalsize and} \quad
 \boldsymbol{\hat w}_2 = \mbox{arg min}\bigg\{\frac{1}{n-p}\sum_{t = p+1}^n\Big[\hat\varepsilon_t^2 - \Psi_{n}(\boldsymbol{y}_{t-1}; \boldsymbol{w})\Big]^2\bigg\},
 \]\normalsize
with $\hat \varepsilon_t =  r_t - \hat \phi( \boldsymbol{y}_{t-1})$,  for all $ t\in \mathds{Z}$.  In both cases, we consider one hidden layer containing $N$  neurons, for some $N\in\mathds{N}$,  that is,
\[
\Phi_{n}(\boldsymbol{y}; \boldsymbol{ w}_1) = \sum_{i = 1}^N a_i\rho_{i}(||\boldsymbol{y} - \boldsymbol{c}_i||) \quad \mbox{and }\quad  \Psi_{n}(\boldsymbol{y}; \boldsymbol{ w}_2) = \sum_{i = 1}^N a_i^* \rho_{i}^*(||\boldsymbol{y} - \boldsymbol{c}_i^*||), \quad \mbox{for all } \boldsymbol{y} \in\ \mathds{R}^p,
\]
with   $\boldsymbol{ w}_1 = (a_1,\cdots, a_N, b_1,\cdots,b_N, \boldsymbol{c}_1, \cdots, \boldsymbol{c}_N)$,  $\boldsymbol{ w}_2 = (a_1^*,\cdots, a_N^*,  b_1^*,\cdots,b_N^*, \boldsymbol{c}_1^*, \cdots, \boldsymbol{c}_N^*)$,  $a_i, b_i, a_i^*,b_i^* \in\mathds{R}$,  $\boldsymbol{c}_i, \boldsymbol{c}_i^*\in\mathds{R}^p$,  $||\cdot ||$ the Euclidean norm, $\rho_{i}(z) = e^{-(b_ iz)^2 }$  and  $\rho_{i}^*(z) = e^{-(b_i^{*}z)^2 }$, for any $z\in\mathds{R}$ and $i\in\{1,\cdots,N\}$.

To obtain  a $h$-step ahead predictor  for $r_{n+h}^2$ given $\{r_t\}_{t=1}^n$,   we observe that, for all $t\in\mathds{Z}$,
\[
\mathds{E}\big(r_t | \{r_k\}_{k < t}\big)  = \mathds{E}(r_t | \boldsymbol{y}_{t-1})  = \phi(\boldsymbol{y}_{t-1}) \quad \mbox{and } \quad
\mbox{\rm Var}\big(r_t | \{r_k\}_{k < t}\big)  =  \mbox{\rm Var}(r_t|\boldsymbol{y}_{t-1}) = \psi^2( \boldsymbol{y}_{t-1}).
\]
Therefore,  $\mathds{E}\big(r_t^2 | \{r_k\}_{k < t}\big)  = \mathds{E}(r_t^2 | \boldsymbol{y}_{t-1})  = \varphi(\boldsymbol{y}_{t-1}) =  \psi^2( \boldsymbol{y}_{t-1}) +  \phi^2( \boldsymbol{y}_{t-1})$,  for some $\varphi: \mathds{R}^p \to  \mathds{R}^p$. Thus,  once  $\phi(\cdot)$ and  $\psi^2(\cdot)$ are estimated, the predictors
 $\hat r_{n+h}$  and   $\hat r_{n+h}^2$  can be obtained recursively  as
\begin{align*}
\hat r_{n+1} = \hat\phi( \boldsymbol{y}_n)  \quad & \mbox{and} \quad \hat r_{n+1}^2 =  \hat\psi^2( \boldsymbol{y}_n) + \hat\phi^2(\boldsymbol{y}_n),\\
 \hat r_{n+h} = \hat\phi( \boldsymbol{\hat y}_{n+h}) \quad & \mbox{and} \quad     \hat r_{n+h}^2 =  \hat\psi^2( \boldsymbol{\hat y}_{n+h}) + \hat\phi^2(\boldsymbol{\hat y}_{n+h}), \quad \mbox{for all } h > 1,
\end{align*}
where  $\boldsymbol{\hat y}_{n+h} = (\hat r_{n+h-1}, \cdots, \hat r_{n+h-1-p})$, with  $\hat r_{n+h-1-k} = r_{n+h-1-k}$, whenever $n+h-1-k \leq n$.

 Tables \ref{allmodelsmae} and  \ref{allmodelsmae2} present some statistics to access the out-of-sample forecasting performance, respectively,  of ARCH-type  and radial basis models.  The values in these tables correspond to the mean absolute error ($mae$), the mean percentage
error ($mpe$) and the maximum absolute error ($max_{ae}$) of forecast,
respectively defined as
\[
mae = \frac{1}{20}\sum_{h=1}^{20}|e_{n+h}|, \quad mpe :=
\frac{1}{20}\sum_{h=1}^{20}\frac{|e_{n+h}|}{r_{n+h}^2} \quad \mbox{and} \quad
max_{ae} := \max_{h\in\{1,\cdots,20\}}\{|e_{n+h}|\}
\]
where, $e_{n+h} := \hat{r}_{n+h}^2 - r_{n+h}^2$, for $h\in \{1,\cdots, 20\}$ and
$n = 1717$, is the $h$-step ahead forecast error.   Note that, when considering the ARMA combined with ARCH-type models,
from the ARMA(0,1)
part of the models, $r_t = X_t -\theta_1X_{t-1}$, where $X_t = \sigma_tZ_t$, for
all $t\in\mathds{Z}$.  Since we define $\hat r_{t+h}^2 =
\mathds{E}(r_{t+h}^2|\mathcal{F}_{t})$ and $\sigma_t^2$ is
$\mathcal{F}_{t-1}$-measurable, for all $t\in\mathds{Z}$, by elementary
calculations we conclude that, $\hat r_{n+1}^2 = \sigma_{n+1}^2 +
\theta_1^2X_{n}^2$ and $\hat r_{n+h}^2 = \hat\sigma_{n+h}^2 +
\theta_1^2\hat\sigma_{n+h-1}^2$, for all $h>1$, with $\hat\sigma_{n+1}^2 =
\sigma_{n+1}^2$.  For EGARCH and FIEGARCH models, $\hat\sigma_{n+1}^2$ is
replaced by $\tilde\sigma_{n+1}^2$, given in expression \eqref{sigmatilde}, and
$\check{\sigma}_{n+h}^2 := \exp\{\hat{\ln}(\sigma_{n+h}^2)\}$, where
$\hat{\ln}(\sigma_{n+h}^2)$ is defined in Proposition \ref{hstepaheadX}.

\begin{table}[!htbp]
    \renewcommand{\arraystretch}{1.1}
  \centering
  \caption{Mean absolute error ($mae$), mean percentage error ($mpe$) and
    maximum absolute error ($max_{ae}$)  of forecasting  for the models in Table \ref{allmodels}. }\label{allmodelsmae}\vspace{0.2cm}
  {\footnotesize
    \begin{tabular*}{1\textwidth}{@{\extracolsep{\fill}}crrrrrrr}
      \hline
      \multirow{2}{*}{Model}     & \multicolumn{1}{c}{\scriptsize ARMA(0,1) + } &
      \multicolumn{1}{c}{\scriptsize ARMA(0,1) +} &
      \multicolumn{1}{c}{\scriptsize ARMA(0,1) + } & \multicolumn{2}{c}{\scriptsize ARMA(0,1) + } & \multicolumn{2}{c}{\scriptsize ARMA(0,1) + }\\
      &  \multicolumn{1}{c}{\scriptsize ARCH(1)}   &   \multicolumn{1}{c}{\scriptsize ARCH(6) }  &  \multicolumn{1}{c}{\scriptsize GARCH(1,1) }  &  \multicolumn{2}{c}{\scriptsize EGARCH(1,1)}   &
      \multicolumn{2}{c}{\scriptsize FIEGARCH(1,$d$,1)}\\
      \cline{2-2}\cline{3-3}\cline{4-4}\cline{5-6}\cline{7-8}

      Predictor\phantom{\Big{|}} & $\hat \sigma_{t+h}^2$ &$\hat \sigma_{t+h}^2$ &$\hat \sigma_{t+h}^2$
      & $\tilde \sigma_{t+h}^2$&$\check \sigma_{t+h}^2$   & $\tilde \sigma_{t+h}^2$&
      $\check \sigma_{t+h}^2$ \\
      \hline
      \vspace{-0.2cm}\\

      $mae$ &  $\phantom{00}$0.00053 & $\phantom{0}$0.00045 & $\phantom{0}$0.00043 &0.00045&$\phantom{0}$0.00044 & 0.00045 &$\phantom{0}$0.00043\\
      $mpe$ &  109.40844 & 68.97817 & 60.29677 & 71.33057 &61.26625 &68.42884 &59.88066\\
      $ max_{ae}$ & $\phantom{00}$0.00094 & $\phantom{0}$0.00094 & $\phantom{0}$0.00094 & 0.00082&$\phantom{0}$0.00087 & 0.00084&$\phantom{0}$0.00088\vspace{0.1cm}\\
      \hline
      \multicolumn{8}{l}{\footnotesize Note: The high $mpe$ values  are
        due to 5 observations close to zero.}
    \end{tabular*}}
\end{table}

From Table \ref{allmodelsmae} we conclude that, given its high $mpe$ value, the
ARMA(0,1)-ARCH(1) does not fit the data well. In fact, the square of the
residuals from this model are still correlated and we use the model only for
comparison.  The ARMA(0,1)-ARCH(6) model performed similar to the
ARMA(0,1)-GARCH(1,1) model, in terms of both, $mae$ and $max_{ae}$ values,
presenting a higher $mpe$ value.  However, the latter is more parsimonious.
Although the log-likelihood value is higher (and the $max_{ae}$ value is
smaller) for the ARMA(0,1)-EGARCH(0,1) model, the $mae$ and the $mpe$ values are
smaller for the ARMA(0,1)-GARCH(0,d,1) model.  Overall, the
ARMA(0,1)-FIEGARCH(0,d,1) performs slightly better than the other models.

The fact that all models present a similar perfomance confirms the following,
already known in the literature.
\begin{itemize}
\item In practice, ARCH$(p)$ models perform relatively well for most
  applications.

\item GARCH$(p,q)$ models are more parsimonious than the ARCH ones. For
  instance, notice that similar results were obtained here by considering an
  ARCH$(6)$ model and a GARCH$(1,1)$ model.

\item For EGARCH$(p,q)$ models the conditional variance is defined in terms of
  the logarithm function and less (usually none) restrictions have to be imposed
  during parameter estimation.  Moreover, EGARCH models are not necessarily more
  parsimonious than ARCH/GARCH ones since it also carries information on the
  returns' asymmetry ($\theta$ and $\gamma$ parameters).

\item FIEGARCH$(p,d,q)$ models can describe not only the same characteristics as
  ARCH, GARCH and EGARCH models do, but also the long-memory in the volatility.
  Also, the performance of all models will be very similar if the volatility
  presents high persistence.  For instance, notice that for the ARCH$(6)$ model
  $\alpha_1+\cdots+\alpha_6 = 0.812$, for the GARCH$(1,1)$ model $\alpha_1 +
  \beta_1 = 0.9678$ and for the EGARCH model $\beta_1 = 0.9373$, which imply
  high persistence in the volatility.  Moreover, for the FIEGARCH model, we
  found $d = 0.3578$ with standard error equal to 0.0810, which indicates that
  the parameter $d$ is statistically different from zero and thus, there is
  evidence of long-memory in the volatility.

\item Given their definition, it is expected that EGARCH and FIEGARCH models
  will provide better forecasts for $\ln(\sigma_{t+h}^2)$ than for
  $\sigma_{t+h}^2$ and, consequently, for $X_{t+h}^2$.
\end{itemize}

\begin{table}[!htbp]
 \renewcommand{\arraystretch}{1.1}
 \centering
 \caption{Mean absolute error ($mae$), mean percentage error ($mpe$) and
  maximum absolute error ($max_{ae}$) of forecasting  for radial basis models
 with  $N \in\{5,10, \cdots, 45\}$ hidden  neurons and $p \in\{1,5,10,15\}$.}
 \label{allmodelsmae2}\vspace{0.2cm}
 {\footnotesize
  \begin{tabular*}{1\textwidth}{@{\extracolsep{\fill}}crrrrccrrrr}
  \hline
 $p$ & N & $mae$ &$mpe$ & $ max_{ae}$ &&  $p$ & N & $mae$ &$mpe$ & $ max_{ae}$\\
  \cline{1-2}  \cline{3-5}  \cline{7-8}  \cline{9-11}
 1 &  5 &     0.00189&   168.16694&     0.00276 && 10 &  5 &     0.00046&    84.07916&     0.00096\\
    & 10 &     0.00464&   360.57740&     0.02105 &&   & 10 &     0.00209&   211.49929&     0.00288\\
        & 15 &     0.00306&   205.95363&     0.01798 &&    & 15 &     0.00076&    40.16931&     0.00156\\
    & 20 &     0.00284&   405.17466&     0.00406   &&    & 20 &     0.00251&   353.29510&     0.00329\\
    & 25 &     0.00106&    69.24385&     0.00193   &&    & 25 &     0.00099&    65.04972&     0.00177\\
    & 30 &     0.00077&    35.08914&     0.00165   &&    & 30 &     0.00214&   309.03589&     0.00292\\
    & 35 &     0.00117&    81.84698&     0.00204  &&    & 35 &     0.00047&    60.11370&     0.00083\\
    & 40 &     0.00082&    40.86115&     0.00169  &&   & 40 &     0.00224&   214.27183&     0.00302\\
    & 45 &     0.00044&     7.76332&     0.00130  &&    & 45 &     0.00043&    46.54092&     0.00084\\
  \cline{1-2}  \cline{3-5}  \cline{7-8}  \cline{9-11}
 5 &  5 &     \textbf{0.00040}&    49.60723&     0.00090 && 15 &  5 &     \textbf{0.00040}&    20.88682&     0.00111\\
    & 10 &     0.00050&    92.13256&     0.00092 && & 10 &     0.00063&    42.05418&     0.00164\\
    & 15 &     0.00058&   111.93650&     0.00109 && & 15 &     0.00110&   185.41861&     0.00212\\
    & 20 &     0.00040&    21.32100&     0.00116 && & 20 &     0.00277&   326.16372&     0.00378\\
    & 25 &     0.00052&     5.95880&     0.00138 && & 25 &     0.00045&    63.80141&     0.00082\\
    & 30 &     0.00046&     4.61686&     0.00129 && & 30 &     0.00047&     \textbf{3.95304}&     0.00123\\
    & 35 &     0.00041&    19.79905&     0.00116 && & 35 &     0.00045&    72.19310&     0.00098\\
    & 40 &    \textbf{0.00040}&    31.93826&     0.00107 && & 40 &     0.00044&    63.04763&     \textbf{0.00079}\\
    & 45 &     0.00120&    88.07146&     0.00207 && & 45 &     0.00271&   363.47039&     0.00343\\
    \hline
              \multicolumn{11}{l}{\footnotesize Note: Boldface indicates the best model for each criterion.}
 \end{tabular*}}
 \end{table}

From Table \ref{allmodelsmae2} we observe that
\begin{itemize}
\item   in terms of $mae$ or $max_{ae}$, both radial basis  and  ARCH-type (see Table \ref{allmodelsmae}) models have a similar performance.    In this case, ARCH-type models seem a better choice given the smaller number of parameter to be estimated;

\item for each $p$ there exists at least one $N$ for which the $mpe$ value for the  radial basis model is much smaller then any ARCH-type models.  However,  given the similarity regarding  $mae$,  the small  $mpe$ values only indicate that   radial basis models provide a better forecast for observations too  close to zero.
\end{itemize}

\section{Conclusions}\label{conclusions}

Here we show complete mathematical proofs for the stationarity, the ergodicity,
the conditions for the causality and invertibility properties, the
autocorrelation and spectral density functions decay and the convergence order
for the polynomial coefficients that describe the volatility for any
FIEGARCH$(p,q,d)$ process. We prove that if $\{X_t\}_{t \in \mathds{Z}}$ is a
FIEGARCH$(p,d,q)$ process and
$\mathds{E}(\left[\ln(Z_0^2)\right]^2)<\infty$, then
$\{\ln(X_t^2)\}_{t\in\mathds{Z}}$ is an ARFIMA$(q,d,0)$ process with correlated
innovations.  Expressions for the kurtosis and the asymmetry measures of any
stationary FIEGARCH$(p,d,q)$ process were also provided.

We also prove that if $\{X_t\}_{t \in \mathds{Z}}$ is a FIEGARCH$(p,d,q)$
process then  it is a martingale difference with respect to the filtration
$\{\mathcal{F}_{t}\}_{t\in\mathds{Z}}$, where $\mathcal{F}_{t} :=
\sigma(\{Z_s\}_{s\leq t})$. The $h$-step ahead forecast for the processes
$\{X_t\}_{t \in \mathds{Z}}$, $\{\ln(\sigma_t^2)\}_{t\in\mathds{Z}}$ and
$\{\ln(X_t^2)\}_{t\in\mathds{Z}}$ are given with their respective mean square
error forecast.  Since $\mathds{E}(\sigma_{t+h}^2|\mathcal{F}_{t})$ cannot be
easily calculated for FIEGARCH models, we also discuss some alternative
estimators for the $h$-step ahead forecast of $\sigma_{t+h}^2$, for all $h>0$.

We present a Monte Carlo simulation study showing how to perform the generation,
the estimation and the forecasting of six different FIEGARCH models. The
parameter selection of these six models are related to the real time series
analyzed in \cite{PRLO1}.  Parameter estimation was performed by considering the
well known quasi-likelihood method.  We conclude that, given the complexity of
FIEGARCH models, the quasi-likelihood method performs relatively well, which is
indicated by the small $bias$, $mae$ and $mse$ values for the estimates.
Regarding the $h$-step ahead forecast for the processes
$\{\sigma_t^2\}_{t\in\mathds{Z}}$ and $\{X_t^2\}_{t\in\mathds{Z}}$, we observe
that the mean square error of forecast decreases as the sample size
increases. However, while the conditional variance is well estimated, which is
indicated by the small $mae$ values, the estimator $\tilde{X}_{n+h}^2 := \tilde
\sigma_{t+h}^2$, which is an approximation for $\hat X_{t+h}^2 :=
\mathds{E}(X_{n+h}^2|\mathcal{F}_{n}) = \hat{\sigma}_{n+h}^2$, does not perform
well in predicting $X_{n+h}^2$. This result is expected since the purpose of the
model is to forecast the logarithm of the conditional variance and not the
process $\{X_t\}_{t \in \mathds{Z}}$ itself.

Finally, we present the analysis of the S\~ao Paulo Stock Exchange Index
(Bovespa Index or IBovespa) log-return time series. We compared the forecasting
performance of FIEGARCH models, fully described in this paper, with other
ARCH-type models.  All models presented a similar performance which was
attributed to the fact that the ARCH, GARCH and EGARCH models indicated high
persistence in the volatility.  We also  compared the forecasting performance of ARCH-type with radial basis models.
 Given the similarity regarding the mean (and maximum) absolute error of forecast we conclude that both classes show a similar forecasting performance. Comparing the mean percentage error of forecasts we concluded that radial basis models provide  a better forecast for observations too  close to zero.

\section*{Acknowledgments}

S.R.C. Lopes was partially supported by CNPq-Brazil, by CAPES-Brazil, by INCT em
\emph{Ma\-te\-m\'a\-ti\-ca} and by Pronex {\it Probabilidade e Processos
  Estoc\'asticos} - E-26/170.008/2008 -APQ1.  T.S. Prass was supported by
CNPq-Brazil. The authors are grateful to the (Brazilian) National Center of
Super Computing (CESUP-UFRGS) for the computational resources.


\end{document}